\documentclass[10pt]{amsart}

\usepackage{latexsym,exscale,enumerate,amsfonts,amssymb,xparse, mathtools}
\usepackage[normalem]{ulem}
\usepackage{amsmath,amsthm,amsfonts,amssymb,amscd, stmaryrd,textcomp,bbm,euscript}
\usepackage[bookmarks=false]{hyperref}
\usepackage{bbold}
\usepackage[usenames,dvipsnames]{xcolor}
\usepackage{enumitem}
\usepackage{verbatim}
\usepackage{ytableau}

\allowdisplaybreaks
\usepackage{array}

\definecolor{ourblue}{RGB}{109, 156, 179}

\usepackage{xcolor}
\definecolor{myorange}{rgb}{1,0.647,0}

\definecolor{mypurple}{cmyk}{.6,.9,0, .11}

\numberwithin{equation}{section}

\theoremstyle{definition}
\newtheorem{thm}{Theorem}[section]
\newtheorem{cor}[thm]{Corollary}

\newtheorem{conj}[thm]{Conjecture}
\newtheorem{lem}[thm]{Lemma}
\newtheorem{rem}[thm]{Remark}
\newtheorem{conv}[thm]{Convention}

\newtheorem{inter}[thm]{Interregnum}
\newtheorem{prop}[thm]{Proposition}
\newtheorem{defn}[thm]{Definition}
\newtheorem{example}[thm]{Example}

\usepackage{tikz}
\usetikzlibrary{cd}
\usetikzlibrary{snakes}
\usetikzlibrary{decorations.markings}
\usetikzlibrary{decorations.pathreplacing}
\usetikzlibrary{arrows,shapes,positioning}
\usetikzlibrary{decorations.pathmorphing}

\tikzset{anchorbase/.style={baseline={([yshift=-0.5ex]current bounding box.center)}},
tinynodes/.style={font=\tiny,text height=0.75ex,text depth=0.15ex},
smallnodes/.style={font=\scriptsize,text height=0.75ex,text depth=0.15ex},
>={Latex[length=1mm, width=1.5mm]},
overcross/.style={line width=4pt,color=white},
1label/.style={very thick, black},
2label/.style={very thick, ourblue},
m1label/.style={very thick, double, black},
m2label/.style={very thick, double, ourblue},
glabel/.style={very thick, gray},
mglabel/.style={very thick, double, gray},
Klabel/.style={line width=5pt, mypurple},
webs/.style={line width=.9,color=black},
}

\tikzstyle directed=[postaction={decorate,decoration={markings,
    mark=at position #1 with {\arrow{>}}}}]
\tikzstyle rdirected=[postaction={decorate,decoration={markings,
    mark=at position #1 with {\arrow{<}}}}]

\newcommand{\Hom}{{\rm Hom}}

\newcommand{\End}{{\rm End}}
\newcommand{\END}{{\rm End}}

\renewcommand{\to}{\rightarrow}

\let\hat=\widehat
\let\tilde=\widetilde

\newcommand{\strand}[1]{\mathbb{1}_{#1}}

\newcommand{\SG}{\mathfrak{S}}
\newcommand{\BG}{\mathrm{Br}}
\newcommand{\BGpd}[2]{\BG^{#2}_{#1}}

\newcommand{\AS}{\EuScript A}
\newcommand{\BS}{\EuScript B}
\newcommand{\CS}{\EuScript C}
\newcommand{\DS}{\EuScript D}

\newcommand{\YS}{\EuScript Y}
\newcommand{\US}{\EuScript U}
\newcommand{\Bim}{\mathrm{Bim}}
\newcommand{\SBim}{\mathrm{SBim}}
\newcommand{\SSBim}{\mathrm{SSBim}}

\renewcommand{\aa}{\mathbf{a}}
\newcommand{\bb}{\mathbf{b}}
\newcommand{\cc}{\mathbf{c}}
\newcommand{\blam}{\boldsymbol{\lambda}}
\newcommand{\bnu}{\boldsymbol{\nu}}

\newcommand{\rfray}[2]{\mathrm{RedFray}^{#1}_{#2}}
\newcommand{\fray}[2]{\mathrm{Fray}^{#1}_{#2}}
\newcommand{\yfray}[2]{\YS\mathrm{Fray}^{#1}_{#2}}
\newcommand{\ufray}[2]{\US\mathrm{Fray}^{#1}_{#2}}
\newcommand{\yufray}[2]{\YS\US\mathrm{Fray}^{#1}_{#2}}

\newcommand{\finproj}[1]{K^{#1}_{(1^n)}}
\newcommand{\yfinproj}[1]{K^{#1, y}_{(1^n)}}
\newcommand{\infproj}[1]{( P^{#1}_{(1^n)} )^{\vee}}
\newcommand{\yinfproj}[1]{( P^{#1, y}_{(1^n)} )^{\vee}}

\definecolor{revisions}{RGB}{0,0,0}
\definecolor{revisions2}{RGB}{0,0,0}

\begin{document}
%

\title{Fray functors and equivalence of colored HOMFLYPT homologies}
\author{Luke Conners}
\address{Department of Mathematics, University of North Carolina, 
Phillips Hall CB \#3250, UNC-CH, Chapel Hill, NC 27599-3250, USA}
\email{lconners@live.unc.edu}

\begin{abstract}
We construct several families of functors on the homotopy category of singular Soergel bimodules that mimic cabling and insertion of column-colored projectors. We use these functors to identify the intrinsically-colored homology of Webster--Williamson and the projector-colored homology of Elias--Hogancamp for an arbitrary link, up to multiplication by a polynomial in the quantum degree $q$. Combined with the results of \cite{Con23}, this establishes parity results for the intrinsic column-colored homology of positive torus knots, partially resolving a conjecture of Hogancamp--Rose--Wedrich in \cite{HRW21}.
\end{abstract}

\maketitle

\setcounter{tocdepth}{2}
\tableofcontents

\section{Introduction}

There has been a great deal of interest in categorifying classical link invariants since Khovanov's categorification of the Jones polynomial at the turn of the century \cite{Kh00}. One particularly rich line of research in this direction has been the categorification of the Type A Reshetikhin--Turaev polynomials and their common generalization, the HOMFLYPT polynomial. The HOMFLYPT polynomial was first categorified by Khovanov-Rozansky in 2005 (\cite{KhR08}) using matrix factorizations, producing a triply-graded link homology theory. In \cite{Kh07}, Khovanov produced an \textit{a priori} distinct categorification of the HOMFLYPT polynomial using the Hochschild homology of Soergel bimodules and showed that this invariant is isomorphic to the original triply-graded Khovanov--Rozansky homology.

Many authors have since constructed categorifications of the \textcolor{revisions}{HOMFLYPT} polynomial from a variety of perspectives. In each of these instances, it is an interesting question whether the resulting invariant is isomorphic to the original triply-graded Khovanov--Rozansky homology. We mention in particular the constructions by Galashin--Lam via cohomology of positroid varieties (\cite{GL20}); by Oblomkov--Rozansky using coherent sheaves on the Hilbert scheme of points in $\mathbb{C}^2$ (\cite{OR17, OR18, OR19a, OR19b, OR20}); and by Mellit, Shende--Treumann--Zaslow, and Trinh using positive braid varieties (\cite{Mel19, STZ17, Tri21}); each of which is known to agree with triply-graded Khovanov--Rozansky homology. 

Oblomkov--Rasmussen--Shende in \cite{ORS18} conjecturally relate the triply-graded Khovanov--Rozansky homology of an algebraic link to a categorification of the HOMFLYPT polynomial defined using the Hilbert scheme of its plane-curve singularity, and Gorsky--Oblomkov--Rasmussen--Shende in \cite{GORS14} recover the same invariant for positive torus knots using the representation theory of Cherednik algebras. Matching this invariant with triply-graded Khovanov--Rozansky homology remains a major outstanding conjecture. The list goes on; rather than attempt (and inevitably fail) to be comprehensive here, we refer the reader to the excellent survey \cite{GKS21} and the references therein.

\subsection{Colored HOMFLYPT Homologies} \label{sec: col_homs}

There has also been a great deal of interest in categorifying the \textit{colored} HOMFLYPT polynomial, in which each strand of a link is labeled by a Young diagram. Here the landscape of homologies is more disordered and naturally separates into two types, roughly corresponding to two distinct formulations of colored HOMFLYPT homology via MOY graphs (\cite{MOY98}) and via cabling and insertion (\cite{Ais96}). We discuss each type separately.

The first type of colored HOMFLYPT homology associates the Hochschild homology of a complex of singular Soergel bimodules to a colored link. We call these invariants \textit{intrinsically-colored} homologies. An intrinsically-colored homology for links with components labeled by a single column of length $1$ or $2$ was constructed by Mackaay--Stošić--Vaz in \cite{MSV11}. This invariant was extended to links with components labeled by a single column of arbitrary length by Webster--Williamson in \cite{WW17} and to components labeled by arbitrary Young diagrams by Cautis in \cite{Cau17}. We describe this invariant in detail in Section \ref{sec: int_col_hom}.

The second type of colored HOMFLYPT homology replaces a component of a link labeled by a Young diagram $\lambda$ with $n$ boxes with $n$ parallel strands, then inserts a chain complex of Soergel bimodules associated to $\lambda$ called a \textit{projector} before taking Hochschild homology. We call these invariants \textit{projector-colored} homologies\footnote{The invariant constructed by Cautis in \cite{Cau17} also proceeds by cabling and inserting a projector. \textcolor{revisions}{When $\lambda$ is a single column, the relevant projector is exactly an identity singular Soergel bimodule, so this construction agrees with Webster-Williamson's intrinsically-colored homology. For this reason, w}e label Cautis's invariant as \textit{intrinsically-colored} rather than projector-colored.}. The appropriate projectors were constructed by Hogancamp in \cite{Hog18} for $\lambda$ a single row, by Abel--Hogancamp in \cite{AH17} for $\lambda$ a single column, and by Elias--Hogancamp in \cite{EH17b} for arbitrary $\lambda$. Each such projector has a finite and infinite version giving rise to distinct homologies, which we refer to as \textit{finite projector-colored} and \textit{infinite projector-colored}, respectively. We describe these invariants in detail in Section \ref{sec: proj_link_hom}.

Each of the homologies described above assigns a triply-graded vector space to a colored link, which we describe up to isomorphism by giving its Hilbert--Poincar\'e series in three variables $a, q, t$. Already for the unknot colored by a single column of length $k$, these invariants disagree, as we demonstrate below\footnote{The intrinsically-colored homology of the $(1^k)$-colored unknot was computed in \cite{WW17}. To the best of our knowledge, the finite projector-colored homology has not been explicitly written down before, but it can be distilled from the results of Section \ref{sec: comp_hom} of this work. The infinite projector-colored homology was computed in \cite{AH17}.}:

\renewcommand{\arraystretch}{3}

\begin{center}
\begin{tabular}{c|c}
\textbf{Homology} & \textbf{Invariant of $(1^k)$-colored unknot} \\ \hline
Intrinsically-colored & $\prod_{j = 1}^k \left( \cfrac{1 + aq^{-2j}}{1 - q^{2j}} \right) $ \\[1ex] \hline
Finite projector-colored & $[k]! \prod_{j = 1}^k \left( \cfrac{(1 + tq^{-2j})(1 + aq^{-2j})}{1 - q^{2j}} \right)$ \\[1ex] \hline
Infinite projector-colored & $\left( \cfrac{1 - t^2q^{-2}}{1 - q^2} \right)^k \prod_{j = 1}^k \left( \cfrac{1 + aq^{-2j}}{1 - t^2q^{-2j}} \right)$
\end{tabular}
\end{center}

Here and throughout, $[k]!$ is the quantum factorial $[k]! = \prod_{j = 1}^k [j]$, where $[j] = q^{j - 1} + q^{j - 3} + \dots + q^{-j + 3} + q^{-j + 1}$ is the usual quantum integer.

Each of these invariants also admits \textit{deformations}, assigning to each colored link a curved complex of singular/non-singular Soergel bimodules before taking Hochschild homology. These deformations were originally constructed by Gorsky--Hogancamp in \cite{GH22} for the trivial coloring $\lambda = (1)$ and extended to intrinsically-colored homology for $\lambda$ a single column by Hogancamp--Rose--Wedrich in \cite{HRW21}. We discuss the deformed intrinsically-colored invariant in detail in Section \ref{sec: def_int_col_hom}. In \cite{Con23}, we define the deformed projector-colored homology for $\lambda$ a single column in both the finite and infinite case and for $\lambda$ a single row in the infinite case, and we compute these invariants for all positive torus knots. We discuss these deformed projector-colored invariants for $\lambda$ a single column in Section \ref{sec: proj_link_hom}.

Again, already for the unknot colored by a column of length $k$, all of these deformed invariants disagree\footnote{The deformed intrinsically-colored homology of the $(1^k)$-colored unknot was computed in \cite{HRW21}. Both the deformed finite projector-colored homology and the deformed infinite projector-colored homology are computed in \cite{Con23}.}:

\renewcommand{\arraystretch}{3}

\begin{center}
\begin{tabular}{c|c}
\textbf{Homology} & \textbf{Invariant of $(k)$-colored unknot} \\ \hline
Deformed intrinsically-colored & $\prod_{j = 1}^k \left( \cfrac{1 + aq^{-2j}}{(1 - t^2q^{-2j})(1 - q^{2j})} \right) $ \\[1ex] \hline
Deformed finite projector-colored & $[k]! \prod_{j = 1}^k \left( \cfrac{1 + aq^{-2j}}{1 - q^{2j}} \right)$ \\[1ex] \hline
Deformed infinite projector-colored & $[k]! \prod_{j = 1}^k \left( \cfrac{1 + aq^{-2j}}{(1 - t^2q^{-2j})(1 - q^{2j})} \right)$
\end{tabular}
\end{center}

\renewcommand{\arraystretch}{1}

Several of these invariants are related by a simple multiplicative factor. Specifically, both the finite projector-colored invariant and the deformed finite projector-colored invariant are multiples of the intrinsically-colored invariant, and the deformed infinite projector-colored invariant is a multiple of the deformed intrinsically-colored invariant\footnote{The deformed intrinsically-colored invariant is also a multiple of the (undeformed) intrinsically-colored invariant; this is true for all knots by Proposition 5.33 of \cite{HRW21}.}. We conjectured that this behavior persists in \cite{Con23}. Our first main theorem verifies this conjecture.

\begin{thm} \label{thm: intrinsic_vs_cabled_intro}
	Let $\mathcal{L}$ be an arbitrary framed, oriented, $r$-component link with components colored by single-column Young diagrams $\lambda_1 = (1^{k_1}), \dots, \lambda_r = (1^{k_r})$. Then:
	\begin{itemize}
	\item The (graded dimension of the) finite projector-colored homology of $\mathcal{L}$ is obtained from the (graded dimension of the) intrinsically-colored homology of $\mathcal{L}$ via multiplication by
	\[
	\prod_{i = 1}^r \left( [k_i]! \prod_{j = 1}^{k_i} (1 + tq^{-2j}) \right)
	\]
	
	\item The (graded dimension of the) deformed finite projector-colored homology of $\mathcal{L}$ is obtained from the (graded dimension of the) intrinsically-colored homology of $\mathcal{L}$ via multiplication by
	\[
	\prod_{i = 1}^r [k_i]!
	\]
	
	\item The (graded dimension of the) deformed infinite projector-colored homology of $\mathcal{L}$ is obtained from the (graded dimension of the) deformed intrinsically-colored homology of $\mathcal{L}$ via multiplication by 
	\[
	\prod_{i = 1}^r [k_i]!
	\]
	\end{itemize}
\end{thm}

Before discussing the proof of Theorem \ref{thm: intrinsic_vs_cabled_intro}, we point out an important Corollary. In \cite{HRW21}, Hogancamp--Rose--Wedrich conjecture that the intrinsically-colored homology of positive torus links is concentrated in even $t$-degree. We compute the deformed finite projector-colored homology of all positive torus knots in \cite{Con23} and show that the result is concentrated in even $t$-degree. Combined with the second relationship of Theorem \ref{thm: intrinsic_vs_cabled_intro}, our computation confirms this conjecture for positive torus knots:

\begin{cor}
	The intrinsically-colored homology of all positive torus knots is concentrated in even $t$-degree.
\end{cor}

\subsection{Proof Strategy} \label{sec: proof_strat}

We sketch the main ideas involved in the proof of Theorem \ref{thm: intrinsic_vs_cabled_intro} assuming some familiarity with the notion of singular Soergel bimodules; for background on these objects, see Section \ref{sec: ssbim_main}. Our main techincal tool is a family of functors that mimic the passage from a single $n$-colored strand $\strand{(n)} \in \SSBim_{(n)}^{(n)}$ to the infinite Abel--Hogancamp projector\textcolor{revisions}{, which we now explain}.

\subsubsection{Fray Functors: Motivation} In \cite{AH17}, the undeformed infinite projector is built in the following way:

\begin{enumerate}
	\item Pass from $\strand{(n)}$ to the full merge-split Soergel bimodule $B_{w_0}$. Graphically, this corresponds to splitting an $n$ labeled strand into $n$ $1$-labeled edges at the top and bottom.
	
	\item Pass from $B_{w_0}$ to the Koszul complex on $B_{w_0}$ for the action of the polynomials $x_2 - x_2', \dots, x_n - x_n'$. The result is the finite Abel--Hogancamp projector, which we denote by $K_{(1^n)}$.
	
	\item Pass to the formal tensor product $K_{(1^n)} \otimes \mathbb{Z}[u_2, \dots, u_n]$ with $\mathrm{deg}(u_i) = t^2q^{-2i}$, and twist by a family of polynomials $g_{ij}(\mathbb{X}, \mathbb{X}')$ such that $\sum_{j = 2}^n g_{ij}(\mathbb{X}, \mathbb{X}') (x_j -x_j')$ vanishes on $B_{w_0}$. The result is the infinite Abel--Hogancamp projector, which we denote by $P_{(1^n)}^{\vee}$.
\end{enumerate}

\textcolor{revisions}{We offer some brief motivation at the decategorified level for the generalization of this construction we employ here. The Abel-Hogancamp projector $P^{\vee}_{(1^n)}$ categorifies a certain idempotent element $p_{(1^n)}$ in the Hecke algebra $H_n$, which itself governs the behavior of the colored HOMFLYPT link \textit{polynomial}; this is the approach due to Aiston, developed in \cite{Ais96}, referenced above. On the other hand, the MOY approach to the colored HOMFLYPT polynomial developed in \cite{MOY98} can be roughly considered as assigning to an $n$-colored strand the trivial idempotent $1_{(n)}$ in a certain endomorphism algebra of the locally unital $q$-\textbf{Schur} category, as studied e.g. by Brundan in \cite{Bru24}. The Hecke algebra $H_n$ itself appears as an endomorphism algebra in this category, and there is a whole family of objects with endomorphism algebras interpolating between the two mentioned here, each with their own incarnation of such an idempotent. At the categorified level, this family roughly corresponds to changing the domain and codomain sequences in the category $\SSBim_{\aa}^{\bb}$, and it is natural to look for analogues of the Abel-Hogancamp projector in these settings which categorify these idempotents.}

\textcolor{revisions}{With this goal in mind, w}e generalize each of these steps in the following way. Fix $n \geq 0$, and let $\blam = (\lambda_1, \dots, \lambda_m)$ be a sequence of positive integers satisfying $\sum_{i = 1}^m \lambda_i = n$. \textcolor{revisions}{Such a sequence is called a \textit{composition} of $n$; we denote compositions by} $\blam \vdash n$. Given such a \textcolor{revisions}{composition}, we consider the following procedure:

\begin{enumerate}
	\item Pass from $\strand{(n)}$ to a partial merge-split singular Soergel bimodule $W_{\blam}$. Graphically, this corresponds to splitting an $n$-labeled strand into $m$ edges with labels $\lambda_1, \dots, \lambda_m$ at the top and bottom.
	
	\item Pass from $W_{\blam}$ to a Koszul complex on $W_{\blam}$ for the action of the polynomials $e_k(\mathbb{X}_j) - e_k(\mathbb{X}'_j)$ for each pair of external alphabets $\mathbb{X}_j, \mathbb{X}'_j$, $1 \leq j \leq m$ and each polynomial degree $1 \leq k \leq \lambda_j$. We denote the resulting complex by $\finproj{\blam}$ and call $\finproj{\blam}$ a \textit{finite frayed projector}.
	
	\item Pass to the formal tensor product $\finproj{\blam} \otimes \mathbb{Z}[u_1, u_2, \dots, u_n]$ with $\mathrm{deg}(u_i) = t^2q^{-2i}$, and \textcolor{revisions}{add to the differential a term $\sum_{i, j, k} a_{ijk}^{\blam}(\mathbb{X}, \mathbb{X}') \otimes u_i$, where $a_{ijk}^{\blam}$ are} a family of polynomials satisfying
	
	\begin{align} \label{cond: interpolation}
	\sum_{j = 1}^m \sum_{k = 1}^{\lambda_j} a_{ijk}^{\blam}(\mathbb{X}, \mathbb{X}') (e_k(\mathbb{X}_j) - e_k(\mathbb{X}'_j)) = e_i(\mathbb{X}') - e_i(\mathbb{X}).
	\end{align}
	
	Here $\mathbb{X}$, $\mathbb{X}'$ denote the total external alphabets $\mathbb{X} = \sum_{j = 1}^m \mathbb{X}_j$, $\mathbb{X}' = \sum_{j = 1}^m \mathbb{X}'_j$. \textcolor{revisions}{With this convention, the endomorphism given by the action of $e_i(\mathbb{X}') - e_i(\mathbb{X})$ vanishes identically on all singular Soergel bimodules, providing a natural generalization of step (3) in the Abel--Hogancamp construction above.} The result is a chain complex, which we denote by $\infproj{\blam}$ and call an \textit{infinite frayed projector}.
\end{enumerate}

\begin{rem}

A few comments are in order. First, when $\blam = (1^n) = (1, 1, \dots, 1) \vdash n$, our $\finproj{\blam}$ is nearly identical to the finite Abel--Hogancamp projector $K_{(1^n)}$; the only difference is that we pass to a Koszul complex on $x_1 - x_1'$, while Abel--Hogancamp do not. This also accounts for our extra polynomial variable $u_1$ in Step 3.

Second, Abel--Hogancamp pick their polynomials $g_{ij}$ so that there are natural inclusion maps relating $P_{(1^n)}^{\vee}$ and $P_{(1^{n + 1})}^{\vee}$; they then check that $\sum_{j = 2}^n g_{ij} (x_j - x_j')$ vanishes on $B_{w_0}$, which guarantees that their twist is a well-defined chain complex. Conversely, we \textit{insist} that our $a_{ijk}^{\blam}$ satisfy Equation \eqref{cond: interpolation}. \textcolor{revisions}{This sacrifices the inclusion maps present in the Abel--Hogancamp construction but generalizes more naturally to arbitrary compositions (and ensures better functoriality properties, as discussed below).} Since $e_i(\mathbb{X}') - e_i(\mathbb{X})$ vanishes on all singular Soergel bimodules, our $\infproj{\blam}$ is also a well-defined chain complex. We show in Section \ref{subsec: idem} that when $\blam = (1^n) \vdash n$, $\infproj{\blam} \simeq P_{(1^n)}^{\vee}$ up to a shift in $q$-degree.

\end{rem}

With these modifications, it turns out that each of these steps becomes functorial. More precisely, let $C$ be any chain complex of singular Soergel bimodules with a distinguished $n$-labeled edge in the domain and codomain. Then we can apply Steps 1 and 2 \textit{locally} at these $n$-labeled edges in each chain bimodule; the result is a new chain complex of singular Soergel bimodules in which the distinguished $n$-labeled edges in the domain and codomain have been replaced with $m$ edges labeled $\lambda_1, \dots, \lambda_m$. We show in Section \ref{sec: fray_functors} that these first two steps assemble to give a well-defined functor for each $\blam \vdash n$; we call this functor $\fray{n}{\blam}$.

Step 3 is a bit more subtle, but also functorial. Equation \eqref{cond: interpolation} guarantees that for any chain complex $C$ of singular Soergel bimodules, the result of applying Steps 1-3 to $C$ is a \textit{curved} complex\textcolor{revisions}{\footnote{\textcolor{revisions}{Roughly speaking, a curved complex is a generalization of a chain complex in which the differential squares not to $0$, but to some other natural endomorphism called the \textit{curvature}. See Section \ref{sec: curved_complexes} for more details.}}} with total curvature

\[
-F_u^{(n)} := \sum_{i = 1}^n (e_i(\mathbb{X}') - e_i(\mathbb{X})) \otimes u_i.
\]

We call this curved complex $\ufray{n}{\blam}(C)$, and we prove in Section \ref{sec: fray_functors} that $\ufray{n}{\blam}$ gives a well-defined functor from chain complexes of singular Soergel bimodules to curved complexes of singular Soergel bimodules for each \textcolor{revisions}{composition} $\blam \vdash n$. In fact, this functor extends to \textcolor{revisions}{certain} curved complexes \textcolor{revisions}{as well. Step 3 in our construction is an example of an operation on chain complexes called \textit{twisting}, which accepts a chain complex $C$ and produces a curved complex whose differential involves some formal parameters $u_1, \dots, u_n$ with some homological degree. We explain this construction in detail in Section \ref{sec: twists}. For now, we point out that twisting is additive, in the sense that a curved complex $\overline{C}$ which has been constructed by twisting a chain complex $C$ can be twisted again. In this way,} given a curved twist $\overline{C} := \mathrm{tw}(C \otimes \mathbb{Z}[u_1, \dots, u_n])$ of $C$ with total curvature $F$, we may consider $\ufray{n}{\blam}(\overline{C})$ as a curved complex with total curvature $F - F_u^{(n)}$ by identifying the corresponding $u$ variables. When $F = F_u^{(n)}$ to begin with, then $\ufray{n}{\blam}(\overline{C})$ is a \textit{bona fide} chain complex.

\subsubsection{Fray Functors: Compatibility with Rickard Complexes}

\textcolor{revisions}{In Section \ref{sec: fray_functors}, we show that both $\fray{n}{\blam}$ and $\ufray{n}{\blam}$ interact particularly well with the complexes involved in computing colored HOMFLYPT homology. To explain these results requires a brief digression detailing the constructions of intrinsically- and projector-colored HOMFLYPT homologies and their deformations; the reader familiar with these constructions can safely skip to Theorem \ref{thm: functorial_cables_intro} below.}

\textcolor{revisions}{In defining the intrinsically-colored homology of a colored link $\mathcal{L}$, one begins by choosing a presentation of $\mathcal{L}$ as the closure\footnote{\textcolor{revisions}{By this we mean the link $\mathcal{L}$ is obtained from the braid $\beta$ by gluing the top of each strand to the bottom of that strand. A classical theorem of Alexander \cite{Al23} guarantees that this is always possible.}} of a colored braid $\beta$. One can associate to $\beta$ a bounded chain complex $C(\beta)$ of singular Soergel bimodules, called the \textit{Rickard complex} of $\beta$, in such a way that the colored braid relations are satisfied up to canonical homotopy equivalence. Applying Hochschild homology termwise to the complex $C(\beta)$ results in a chain complex of doubly-graded vector spaces, the homology of which is exactly the intrinsically-colored HOMFLYPT homology of $\mathcal{L}$. We discuss Rickard complexes in more detail in Section \ref{sec: col_braids} and their Hochschild homology in Section \ref{sec: int_col_hom}.}

\textcolor{revisions}{In \cite{HRW21}, Hogancamp--Rose--Wedrich show that the Rickard complexes $C(\beta)$ admit canonical curved analogues $C^u(\beta)$ which also satisfy the colored braid relations up to homotopy equivalence. The curved complexes $C^u(\beta)$ have curvature given by a  sum of terms $F_u^{(n)}$ as defined above associated to each strand\footnote{\textcolor{revisions}{For the reader familiar with \cite{HRW21}, we emphasize that we will frequently work with complexes $C^u(\beta)$ which are curved along \textit{only a single strand}, in contrast with their fully-curved complexes.}}; as such, we refer to these complexes as \textit{$\Delta e$-deformed Rickard complexes}. There is again a natural notion of Hochschild homology for these curved complexes, and Hogancamp--Rose--Wedrich prove that applying Hochschild homology termwise to the curved complexes $C^u(\beta)$ results in a well-defined link invariant, which we call \textit{deformed intrinsically-colored HOMFLYPT homology}. We discuss $\Delta e$-deformed Rickard complexes in more detail in Section \ref{sec: curv_rick} and their Hochschild homology in Section \ref{sec: def_int_col_hom}.}

\textcolor{revisions}{To obtain projector-colored homology requires a procedure called \textit{cabling and insertion}. Roughly speaking, cabling and insertion proceeds by replacing the distinguished $n$-colored strand of $\beta$ with a ``$\blam$-colored cable" of parallel strands to obtain a new braid $\beta^{\blam}$ (cabling), then tensoring the Rickard complex $C(\beta^{\blam})$ with an auxiliary complex (insertion). This procedure results in a well-defined link invariant whenever the auxiliary complex satisfies a technical condition called ``sliding past crossings". We discuss cabling and insertion and the resulting link invariants in more detail in Section \ref{sec: proj_link_hom}.}

\textcolor{revisions}{If the curved braid $\beta$ has an $n$-colored strand, then one can apply $\fray{n}{\blam}$ to the Rickard complex $C(\beta)$ for any \textcolor{revisions}{composition} $\blam \vdash n$. Similarly, one can apply $\ufray{n}{\blam}$ to the $\Delta e$-deformed Rickard complex $C^u(\beta)$ with curvature along the distinguished strand; the result is a \textit{bona fide} chain complex by the discussion above. Our second main theorem gives an alternate characterization of these complexes $\fray{n}{\blam}(C(\beta))$ and $\ufray{n}{\blam}(C^u(\beta))$ in terms of cabling and inserting frayed projectors:}

\begin{thm} \label{thm: functorial_cables_intro}
	\textcolor{revisions}{\textcolor{revisions2}{Suppose $\beta$ is a colored braid} with a distinguished $n$-colored strand, and let $\blam \vdash n$ be an arbitrary \textcolor{revisions}{composition}. Then $\fray{n}{\blam}(C(\beta))$ is homotopy equivalent to the complex obtained by cabling and inserting a finite frayed projector $\finproj{\blam}$ into $C(\beta)$. Similarly, $\ufray{n}{\blam}(C^u(\beta))$ is homotopy equivalent to the complex obtained by cabling and inserting an infinite frayed projector $\infproj{\blam}$ into $C(\beta)$.}
\end{thm}

In other words, up to homotopy equivalence, $\fray{n}{\blam}$ and $\ufray{n}{\blam}$ are \textcolor{revisions}{\textit{functorial} incarnations} of cabling and inserting finite/infinite projectors. When $\blam = (1^n) \vdash n$, Theorem \ref{thm: functorial_cables_intro} allows us to prove that $\infproj{\blam}$ is a unital idempotent \textcolor{revisions}{in the sense of \cite{Hog17}} as a purely formal consequence of a computation involving the $(n - 1, 1)$-colored full twist braid on two strands, significantly streamlining the main technical computations of \cite{Con23}. Uniqueness of these idempotents then guarantees that $\infproj{\blam} \simeq P_{(1^n)}^{\vee}$ up to an overall shift in $q$-degree.

\begin{rem} \label{rem: def_intro_thm}
This whole story admits \textcolor{revisions}{another natural deformation. Both $\fray{n}{\blam}$ and $\ufray{n}{\blam}$ admit further $\Delta e$-like deformations along each $\lambda_j$-colored strand, called $\yfray{n}{\blam}$ and $\yufray{n}{\blam}$, which mimic the passage from the Abel--Hogancamp projectors to their deformations.} Theorem \ref{thm: functorial_cables_intro} also lifts to a deformed version, in which the frayed projectors and cabled Rickard complexes are replaced with their deformations. \textcolor{revisions}{We explain these constructions in more detail and make these statements precise in Section \ref{sec: fray_functors}.}
\end{rem}

\subsubsection{Frayed Projector-Colored Homologies}

As a consequence of Theorem \ref{thm: functorial_cables_intro} and its deformed version, both the frayed projectors $\finproj{\blam}$, $\infproj{\blam}$ and their deformations $\yfinproj{\blam}$, $\yinfproj{\blam}$ slide past crossings and therefore give rise to well-defined link invariants by cabling and insertion. In fact, by varying the \textcolor{revisions}{composition}s $\blam \vdash n$ for each strand color, we obtain four \textit{families} of link invariants\textcolor{revisions}{, which we refer to collectively as \textit{frayed projector-colored homologies}}. We collect these statements as our third main theorem:

\begin{thm} \label{thm: hom_family_intro}

Let $\mathcal{L}$ be an arbitrary $r$-component link with components colored by positive integers $k_1, \dots, k_r$. Then upon taking Hochschild homology, each of the following procedures gives rise to a family of well-defined link invariants parametrized by \textcolor{revisions}{composition}s $\blam_i \vdash k_i$:

\begin{itemize}
	\item Replace each $k_i$-colored strand of $\mathcal{L}$ with a $\blam_i$-colored cable, and insert a $\blam_i$-colored \textit{finite} frayed projector into the Rickard complex of the resulting cable.
	
	\item Replace each $k_i$-colored strand of $\mathcal{L}$ with a $\blam_i$-colored cable, and insert a $\blam_i$-colored \textit{deformed finite} frayed projector into the \textit{deformed} Rickard complex of the resulting cable.
	
	\item Replace each $k_i$-colored strand of $\mathcal{L}$ with a $\blam_i$-colored cable, and insert a $\blam_i$-colored \textit{infinite} frayed projector into the Rickard complex of the resulting cable.
	
	\item Replace each $k_i$-colored strand of $\mathcal{L}$ with a $\blam_i$-colored cable, and insert a $\blam_i$-colored \textit{deformed infinite} frayed projector into the \textit{deformed} Rickard complex of the resulting cable.
\end{itemize}
\end{thm}

When $\blam_i = (1^{k_i}) \vdash k_i$ for each $i$, these invariants recover the (deformed/undeformed) (finite/infinite) projector-colored invariants described above. In Section \ref{sec: comp_hom}, we show that when $\blam_i = (k_i) \vdash k_i$, the second invariant described above recovers the intrinsically-colored homology of $\mathcal{L}$, and the fourth recovers the deformed intrinsically-colored homology of $\mathcal{L}$. As a consequence, we view the families of Theorem \ref{thm: hom_family_intro} as interpolating between the intrinsically-colored homologies and projector-colored homologies described in Section \ref{sec: col_homs}.

\begin{rem} \label{rem: yify_is_good}
\textcolor{revisions}{Notice that it is the \textit{deformed} frayed projector-colored homologies which recover both undeformed and deformed intrinsically-colored homology. We view this phenomenon as further evidence for the utility of deformed HOMFLYPT homology in both structure and computations, as has been observed in e.g. \cite{GH22}, \cite{HRW21}, \cite{Con23}, and \cite{GHM21}.}
\end{rem}

In light of Theorem \ref{thm: functorial_cables_intro}, the functors $\fray{n}{\blam}$, $\yfray{n}{\blam}$, $\ufray{n}{\blam}$, and $\yufray{n}{\blam}$ exactly capture the cabling and insertion procedures described in Theorem \ref{thm: hom_family_intro}. This reduces the computation of projector-colored homology to understanding the interactions between Hochschild homology and each of these functors. We show in Section \ref{sec: comp_hom} that the functors $\fray{n}{\blam}$, $\yfray{n}{\blam}$, and $\yufray{n}{\blam}$ interact especially well with Hochschild homology, finally resulting in Theorem \ref{thm: intrinsic_vs_cabled_intro}.

\begin{rem}
\textcolor{revisions}{Projector-colored link homologies for different choices $C$ of complexes to be inserted have been studied at length. The case in which $C$ is a categorified Young symmetrizer of highest weight $\nu$ was investigated in detail by Cautis in \cite{Cau17} in the setting of singular Soergel bimodules; here one takes $\blam = \nu$. Many authors have studied the corresponding case in the setting of (nonsingular) Soergel bimodules, in which $\blam = (1^k) \vdash k$ regardless of the weight $\nu$; we refer to these invariants as \textit{thin projector-colored triply-graded homologies}. When this weight is $\nu = (k)$, the corresponding complex $C$ is the (unital) \textit{row}-colored categorified Young symmetrizer of \cite{Hog18}, and the corresponding invariant is computed for all positive torus knots in \cite{HM19}. When $\nu = (1^k)$, the corresponding complex $C$ is the \textit{column}-colored Young symmetrizer of \cite{AH17} discussed above. It is mentioned in \cite{AH17} that this complex gives rise to a link invariant, though this invariant is not computed in that work except when $\mathcal{L}$ is a colored unknot.}

\textcolor{revisions}{The deformed versions of each of these thin projector-colored invariants for all positive torus knots are computed in \cite{Con23}, and their graded dimensions for a fixed knot are shown to agree up to a change of variables $q \leftrightarrow tq^{-1}$. In \cite{EH17b}, Elias and Hogancamp construct a family of finite and infinite categorified Young symmetrizers $K_{\lambda}$, $P_{\lambda}$ for arbitrary partitions $\lambda$ which slide past crossings, though they do not study the resulting invariant. Finally, in the context of Khovanov homology, the case in which $C$ is a highest-weight categorified Jones--Wenzl projector is investigated by Cooper-Krushkal in \cite{CK12}, the lowest-weight categorified Jones--Wenzl projector is investigated by Rozansky in \cite{Roz10b}, and the case in which $C$ is a categorified Jones--Wenzl of arbitrary weight is studied by Cooper-Hogancamp in \cite{CH15}.}
\end{rem}

\subsection{Outlook}

In this section, we briefly describe potential extensions of this work.

\subsubsection{Categorical Idempotence}

In Section \ref{sec: top_rec}, we show that when $\blam = (1^n) \vdash n$, the infinite frayed projector $\infproj{\blam}$ and its deformation $\yinfproj{\blam}$ are unital idempotents in the sense of \cite{Hog17}.
Theorem \ref{thm: functorial_cables_intro} allows us to reduce this proof to a straightforward computation involving the $(n - 1, 1)$-colored full twist braid in a purely formal way. If we could carry out similar computations for the $(a, b)$-colored full twist for \textit{arbitrary} colors $a, b$, a similar formal procedure would show that $\infproj{\blam}$ and $\yinfproj{\blam}$ are unital idempotents for \textit{arbitrary} \textcolor{revisions}{composition}s $\blam \vdash n$. These computations are currently in progress, and we plan to present this procedure in future work.

\subsubsection{Fray Functors are 2-Functors}

The statement that $\infproj{\blam}$ is idempotent can be rephrased as $\ufray{n}{\blam}(\strand{(n)}) \star \ufray{n}{\blam}(\strand{(n)}) \simeq \ufray{n}{\blam}(\strand{(n)} \star \strand{(n)})$. We conjecture that this compatibility with horizontal composition $\star$ extends to general complexes:

\begin{conj}
Let $A, B$ be curved complexes of singular Soergel bimodules with distinguished $n$-labeled edges in their domains and codomains, and suppose $A \star B$ is well-defined. Then $\ufray{n}{\blam}(A) \star \ufray{n}{\blam}(B) \simeq \ufray{n}{\blam}(A \star B)$ for each \textcolor{revisions}{composition} $\blam \vdash n$.
\end{conj}

We suspect similar functoriality statements hold on the nose for $\yufray{n}{\blam}$ and up to a polynomial in $q, t$ for $\fray{n}{\blam}$ and $\yfray{n}{\blam}$.

\subsubsection{Categorical Diagonalization of Colored Full Twists}

Recall that when $\blam = (1^n)$, our infinite frayed projector $\infproj{\blam}$ is homotopy equivalent to the infinite Abel--Hogancamp projector $P_{(1^n)}^{\vee}$. In \cite{EH17b}, Elias--Hogancamp realize $P_{(1^n)}^{\vee}$ as one of a family of categorical idempotents, indexed by Young tableaux on $n$ boxes, which diagonalize the $n$-strand, $1$-colored full twist. We believe our $\infproj{\blam}$ should also be one of a family of categorical idempotents which diagonalize the $m$-strand, $\lambda_1, \dots, \lambda_m$-colored full twist for each \textcolor{revisions}{composition} $\blam = (\lambda_1, \dots, \lambda_m) \vdash n$.

In fact, another candidate idempotent in this family has already been studied. Cautis in \cite{Cau17} proves that a certain stable limit of infinite full twist braids gives rise to a categorical idempotent, which he calls a HOMFLY clasp. Moreover, this HOMFLY clasp for the $n$-strand, $1$-colored full twist recovers the row projector $P_{(n)}$ of \cite{Hog18}, which lives opposite $P_{(1^n)}^{\vee}$ in Elias--Hogancamp's diagonalization of the full twist. We conjecture that our $\infproj{\blam}$ and Cautis's HOMFLY clasp should similarly live at the opposite ends of a diagonalization of an arbitrary full twist.

Further, Cautis's HOMFLY clasp for a $1$-strand full twist with label $n$ is exactly the identity bimodule $\strand{(n)}$. Applying $\ufray{n}{\blam}$ to this clasp recovers our $\infproj{\blam}$. We conjecture that this behavior persists: given a \textcolor{revisions}{composition} $\bnu = (\nu_1, \dots, \nu_m) \vdash n$ and a \textit{refinement} $\blam = \bnu^{\aa}$ for some $\aa \vdash \nu_i$ (in the sense of Section \ref{sec: symm_poly}), applying $\ufray{\nu_i}{\aa}$ to a version of Cautis's HOMFLY clasp for the $\bnu$-colored full twist should recover an intermediate term in the diagonalization of the $\blam$-colored full twist.

\subsubsection{Algebra Action}

In \cite{GHM21}, Gorsky--Hogancamp--Mellit construct a dg algebra $\mathcal{A}$ and show that $\mathcal{A}$ has a well-defined action on the Rouquier complex of an arbitrary braid $\beta$. Our $\yufray{n}{\blam}(C)$ seems to mimic this action for a similar algebra, with the generators $\xi_i$ acting via our backwards Koszul differentials and the generators $u_i$ acting via differentiation with respect to our bulk deformation variables $u_i$. It would be interesting to investigate this relationship further, and especially to investigate whether there is a $\mathfrak{sl}_2$-action related to our functors.

\subsection{Structure}

In Section \ref{sec: hom_alg}, we recall the necessary homological algebra required to precisely formulate and prove our main theorems. Our treatment of extension of scalars and deformations in this section is rather non-standard but well-adapted to the proofs of our main theorems; to our knowledge, the deformation lifting material in Section \ref{subsec: gauss_elim} is new. In Section \ref{sec: ssbim_main}, we recall the theory of singular Soergel bimodules and deformed Rickard complexes. In Section \ref{sec: fray_functors}, we introduce the family of fray functors described in the Introduction and establish some naturality results, including Theorem \ref{thm: functorial_cables_intro}. In Section \ref{sec: cat_idem}, we use the results of Section \ref{sec: fray_functors} to prove that our $\infproj{\blam}$ agrees with the Abel--Hogancamp infinite projector when $\blam = (1^n)$. Finally, in Section \ref{sec: link_hom}, we discuss applications to link homology and prove Theorems \ref{thm: intrinsic_vs_cabled_intro} and \ref{thm: hom_family_intro}.

\subsection{Acknowledgements}

We would like to thank Sabin Cautis, Ben Elias, and Nicolle Gonz\'alez for helpful discussions at WARTHOG 2023 as this work was taking shape. We would also like to thank Matt Hogancamp and Josh Wang for many helpful discussions throughout the preparation of this work and the math department at MIT for hosting these discussions. We would especially like to thank David Rose for his countless hours of discussion and suggesting that the functors of Section \ref{sec: fray_functors} might exist \textcolor{revisions}{and an anonymous referee, whose careful attention and detailed feedback regarding an earlier draft have dramatically improved the quality of this work.}

During the preparation of this work, the author was partially supported by Simons Collaboration Grant 523992: “Research on knot invariants, representation theory, and categorification,” NSF CAREER Grant DMS-2144463, and the Graduate Summer Research Fellowship at UNC-Chapel Hill. This project was conceived of and partially carried out during the AIM Research Community on Link Homology.

\section{Homological Algebra Background} \label{sec: hom_alg}

\subsection{DG Categories} \label{sec: dg_categories}

Unless otherwise noted, we fix throughout a field $R$ with $\mathrm{char}(R) \neq 2$, an abelian group $\Gamma$, and an $R$-linear category $\AS$. We denote by $\AS[\Gamma]$ the category whose objects are $\Gamma$-indexed collections $X = \{X_{\gamma}\}_{\gamma \in \Gamma}$ of objects of $\AS$ and whose morphism spaces are $\Gamma$-graded $R$-modules

\[
\Hom_{\AS[\Gamma]}(X, Y) = \bigoplus_{\gamma \in \Gamma} \Hom^{\gamma}_{\AS[\Gamma]}(X, Y)
\]
with components

\[
\Hom^{\gamma}_{\AS[\Gamma]}(X, Y) = \prod_{\gamma' \in \Gamma} \Hom_{\AS}(X_{\gamma'}, Y_{\gamma' + \gamma}).
\]

We say a morphism $f \colon X \to Y$ in $\AS[\Gamma]$ is \textit{homogeneous of degree $\gamma$} and write $\mathrm{deg}(f) = \gamma$ if $f \in \Hom^{\gamma}_{\AS[\Gamma]}(X, Y)$.

Suppose that $\Gamma$ is equipped with a symmetric bilinear form $\langle -, - \rangle \colon \Gamma \times \Gamma \to \mathbb{Z}/2\mathbb{Z}$. In this setting, we define the \textit{commutator} of two homogeneous morphisms $f, g \in \End_{\AS[\Gamma]}(X)$ to be

\[
[f, g] := f \circ g - (-1)^{\langle \mathrm{deg}(f), \mathrm{deg}(g) \rangle} g \circ f.
\]

If $\AS$ is additive and monoidal, then given any pair \textcolor{revisions}{$X, Y \in \AS[\Gamma]$, there is a natural tensor product $X \otimes_{\AS[\Gamma]} Y$ given by}

\[
\textcolor{revisions}{(X \otimes_{\AS[\Gamma]} Y)_{\gamma} := \bigoplus_{\gamma_1 + \gamma_2 = \gamma} X_{\gamma_1} \otimes_{\AS} Y_{\gamma_2}}
\]
\textcolor{revisions}{whenever the right-hand side is a well-defined object of $\AS$ for each $\gamma \in \Gamma$. We can similarly define a tensor product of (homogeneous) morphisms $f, g$ by}

\[
(f \otimes_{\AS[\Gamma]} g)(x \otimes_{\AS[\Gamma]} y) := (-1)^{\langle \mathrm{deg}(f), \mathrm{deg}(g) \rangle} f(x) \otimes_{\AS[\Gamma]} g(y).
\]
As a consequence, we have the \textit{middle interchange law}:

\[
(f \otimes_{\AS[\Gamma]} g) \circ (f' \otimes_{\AS[\Gamma]} g') = (-1)^{\langle \mathrm{deg}(f'), \mathrm{deg}(g) \rangle} (f \circ f') \otimes_{\AS[\Gamma]} (g \circ g'). 
\]%
\textcolor{revisions}{If $X \otimes_{\AS[\Gamma]} Y$ is well-defined for every such pair, then these tensor products lift to a monoidal structure on $\AS[\Gamma]$.}

Suppose additionally that $\Gamma$ is equipped with a distinguished element ${\bf{d}} \in \Gamma$ satisfying $\langle {\bf{d}}, {\bf{d}} \rangle = 1 \in \mathbb{Z}/2\mathbb{Z}$. Then we may consider the category $\AS[\Gamma]_{dg}$ of \textit{$\Gamma$-graded complexes in $\AS$}. Objects of $\AS[\Gamma]_{dg}$ are pairs $(X, d_X)$ where $X \in \AS[\Gamma]$ and $d_X \in \End^{{\bf{d}}}_{\AS[\Gamma]}(X)$ satisfies $d_X^2 = 0$. We refer to this pair as a \textit{$\Gamma$-graded complex} and to $d_X$ as the \textit{differential} on $X$.

Morphism spaces in $\AS[\Gamma]_{dg}$ consist of $\Gamma$-graded complexes in $R-\mathrm{Mod}$. The underlying sets are inherited from those of $\AS[\Gamma]$:

\[
\Hom_{\AS[\Gamma]_{dg}}\Big( (X, d_X), (Y, d_Y) \Big) = \Hom_{\AS[\Gamma]}(X, Y).
\]

We denote the differential on morphism spaces by $d_{\AS[\Gamma]}$; it is defined on homogeneous morphisms via the formula

\[
d_{\AS[\Gamma]} \colon f \mapsto [d, f] := d_Y \circ f - (-1)^{\langle {\bf{d}}, \mathrm{deg}(f) \rangle} f \circ d_X
\]
and extended $R$-linearly. \textcolor{revisions}{Again, if $\AS$ is additive and monoidal, then given a pair of complexes $(X, d_X), (Y, d_Y) \in \AS[\Gamma]_{dg}$, 
such that $X \otimes_{\AS[\Gamma]} Y$ is well-defined}, we can define a tensor product complex as follows:

\[
(X, d_X) \otimes_{\AS[\Gamma]_{dg}} (Y, d_y) = \big( X \otimes_{\AS[\Gamma]} Y, d_X \otimes_{\AS[\Gamma]} \mathrm{id}_Y + \mathrm{id}_X \otimes_{\AS[\Gamma]} d_Y \big).
\]

\begin{defn} \label{def: diff_gamm_grad_mod}
    In the special case when $\AS = R-\mathrm{Mod}$, we call objects in $\AS[\Gamma]_{dg}$ \textit{differential $\Gamma$-graded $R$-modules}. Given a differential $\Gamma$-graded $R$-module $(X, d_X)$, we call $X$ a \textit{differential $\Gamma$-graded $R$-algebra} \textcolor{revisions}{(or just a \textit{dg $R$-algebra} if $\Gamma$ is clear from context)} if $X$ is equipped with an $R$-bilinear, $\Gamma$-graded multiplication satisfying the \textit{graded Leibniz rule}:

    \[
    d_X(ab) = \left( d_X(a) \right) b + (-1)^{\langle {\bf{d}}, \mathrm{deg}(a) \rangle} a \left( d_X(b) \right).
    \]

\end{defn}

\begin{example}
    Let $\AS$ be an arbitrary $R$-linear category. For each $X \in \AS[\Gamma]_{dg}$, the endomorphism space $\End_{\AS[\Gamma]_{dg}}(X)$ is equipped with an $R$-linear, $\Gamma$-graded multiplication given by composition. The graded Leibniz rule in this setting takes the form

    \[
    d_{\AS[\Gamma]}(f \circ g) = d_{\AS[\Gamma]}(f) \circ g + (-1)^{\langle {\bf{d}}, \mathrm{deg}(f) \rangle} f \circ d_{\AS[\Gamma]}(g).
    \]

    This identity is easily verified, so each such space is a differential $\Gamma$-graded $R$-algebra.
\end{example}

We call a morphism $f \in \Hom_{\AS[\Gamma]_{dg}}(X, Y)$ \textit{closed} if $d_{\AS[\Gamma]}(f) = 0$ and \textit{exact} (or \textit{nullhomotopic}) if $f = d_{\AS[\Gamma]}(g)$ for some $g \in \Hom_{\AS[\Gamma]_{dg}}(X, Y)$. We call a closed map of degree $0$ a \textit{chain map}. If $f$ and $g$ are homogeneous morphisms with $\mathrm{deg}(f) = \mathrm{deg}(g)$, we say $f$ and $g$ are \textit{homotopic} if $f - g$ is exact; in this case, we write $f \sim g$.

\begin{defn} \label{def: dg_cat}
    We call $\AS$ a \textit{$\Gamma$-graded category over $R$} if $\AS$ is enriched over $R-\mathrm{Mod}[\Gamma]$. We call $\AS$ a \textit{differential $\Gamma$-graded category over $R$} if $\AS$ is enriched over $R-\mathrm{Mod}[\Gamma]_{dg}$ and for each $X, Y, Z \in \AS$, composition defines a chain map

    \[
    \Hom_{\AS}(X, Y) \otimes_{R-\mathrm{Mod}[\Gamma]_{dg}} \Hom_{\AS}(Y, Z) \to \Hom_{\AS}(X, Z).
    \]

    \vspace{1em}
\end{defn}

We specialize for the remainder of this work to the situation in which $\Gamma$ is a finitely-generated free abelian group as in the following examples.

\begin{example}
    Let $\Gamma = \mathbb{Z}$ with $\langle j, j' \rangle = 0$. We denote the generator of $\Gamma$ by $q$ and indicate this by writing $\Gamma = \mathbb{Z}_q$; we also often write elements of $\mathbb{Z}_q$ multiplicatively, denoting e.g. $-3 \in \mathbb{Z}_q$ as $q^{-3}$. We refer to $\mathbb{Z}_q$ throughout as the \textit{quantum} grading.
    
    Since there can be no element ${\bf{d}} \in \mathbb{Z}_q$ satisfying $\langle {\bf{d}}, {\bf{d}} \rangle = 0$, we may only consider the category $\AS[\mathbb{Z}_q]$. When $\AS = R-\mathrm{Mod}$, this is the usual category of graded $R$-modules in which we do not restrict ourselves to degree $0$ morphisms. We will also denote by $q^{\pm 1}$ the \textit{quantum grading shift} functors on $\mathcal{A}[\mathbb{Z}_q]$, defined on objects by $(q^jX)_i = X_{i - j}$ and as the identity on morphisms.
\end{example}

\begin{example} \label{ex: homological_gradings}
    Let $\Gamma = \mathbb{Z}$ with $\langle k, k' \rangle = kk'$ and ${\bf{d}} = 1$. Then $\AS[\Gamma]_{dg}$ is the usual category of (co)chain complexes over $\AS$ and morphisms of all degree. We denote the generator of $\Gamma$ by $t$ and write elements of $\Gamma = \mathbb{Z}_t$ multiplicatively as in the previous example. We refer to $\mathbb{Z}_t$ as the \textit{homological} grading.
    
    Since this case is so common, we employ the shortened notation $\CS{\AS} := \AS[\Gamma]_{dg}$. We also often replace the phrase ``differential $\Gamma$-graded" with the letters ``dg", giving rise to the notion of e.g. a dg $R$-module, a dg $R$-algebra, a dg category over $R$, etc.
    
    We denote by $t^{\pm 1}$ the homological grading shift functors on $\AS[\mathbb{Z}_t]$ and $\AS[\mathbb{Z}_t]_{dg}$. The choice of bilinear form associated to $\Gamma$ requires certain sign conventions in defining these functors. More precisely, $t$ is defined on objects of $\AS[\mathbb{Z}_t]_{dg}$ by

    \[
    t(X, d_X) = (tX, -d_X); \quad (tX)_i = X_{i - 1}
    \]
    and on morphisms as the identity on the nose. This sign convention ensures that the obvious degree $t$ morphism $X \to tX$ given by the identity on each chain object of $X$ is closed.
\end{example}

\begin{example}
    Combining the previous two examples, let $\Gamma = \mathbb{Z}_q \times \mathbb{Z}_t$ with $\langle (j, k), (j', k') \rangle = kk'$ and ${\bf{d}} = (0, 1)$. When $\AS = R-\mathrm{Mod}$, objects of $\AS[\Gamma]$ are simply $\mathbb{Z}_q \times \mathbb{Z}_t$-graded $R$-modules, while objects of $\AS[\Gamma]_{dg}$ are chain complexes of $\mathbb{Z}_q$-graded $R$-modules. The differentials on such complexes are quantum degree-preserving $R$-module homomorphisms of homological degree $1$; by the multiplicative conventions above, we often notate this fact as $\mathrm{deg}(d) = q^0t^1 = t$.
\end{example}

More generally, given any grading group $\Gamma$ equipped with a (potentially trivial) symmetric bilinear form $\langle - , - \rangle_{\Gamma}$ and $R$-linear category $\AS$, we can form the category $\AS[\Gamma \times \mathbb{Z}_t]_{dg}$ as follows. We define a bilinear form $\langle - , - \rangle_{\Gamma \times \mathbb{Z}_t}$ by

\[
\langle (\gamma, k) , (\gamma', k') \rangle_{\Gamma \times \mathbb{Z}_t} := \langle \gamma, \gamma' \rangle_{\Gamma} + kk'
\]
and let ${\bf{d}} := (0, 1)$. The resulting category $\AS[\Gamma \times \mathbb{Z}_t]_{dg}$ is the (dg) category of chain complexes of $\Gamma$-graded objects of $\AS$. Since such constructions are so common, we will employ the notation

\[
\CS(\AS[\Gamma]) := \AS[\Gamma \times \mathbb{Z}_t]_{dg}.
\]

The category $\CS(\AS[\Gamma])$ comes equipped with grading shift endofunctors $(-)[(\gamma, n)]$ for each $(\gamma, n) \in \Gamma \times \mathbb{Z}_t$. These functors act on objects by

\[
(X, d_X)[(\gamma, n)] = \left( X[(\gamma, n)], (-1)^n d_X \right); \quad (X[(\gamma, n)])_{(\gamma', m)} = X_{(\gamma - \gamma', m - n)}
\]
and as the identity on morphisms. As in Example \ref{ex: homological_gradings}, this sign convention ensures that the obvious degree $(\gamma, n)$ morphism $X \to X[(\gamma, n)]$ given by the identity on each chain object of $X$ is closed.

When considering morphisms in $\CS(\AS[\Gamma])$, we will often only be interested in their \textit{homological} degree. As such, given two complexes $X, Y \in \CS(\AS[\Gamma])$, we also reserve the shorthand notation

\[
\Hom^k_{\CS(\AS[\Gamma])}(X, Y) := \bigoplus_{\gamma \in \Gamma} \Hom^{(\gamma, k)}_{\CS(\AS[\Gamma])}(X, Y).
\]

Whenever morphism spaces are written with a single superscript as above, that index should be interpreted as specifying the homological degree of the relevant morphism, even when the underlying grading group $\Gamma$ contains additional factors.

\begin{rem} \label{rem: implicit_t_grading}
We will often implicitly consider $\AS[\Gamma]$ as a $\Gamma \times \mathbb{Z}_t$-graded category over $R$ with trivial $\mathbb{Z}_t$-grading. From this perspective, there is a natural full inclusion $\AS[\Gamma] \hookrightarrow \CS(\AS[\Gamma])$ of $\Gamma \times \mathbb{Z}_t$-graded categories over $R$ sending an object $X = \{X_{\gamma}\}_{\gamma \in \Gamma}$ to the complex $(X, 0)$ satisfying $X_{\gamma, k} = X_{\gamma}$ for $k = 0$ and $X_{\gamma, k} = 0$ for $k \neq 0$.
\end{rem}

\begin{example} \label{ex: tri_grading}
    Let $\Gamma = \mathbb{Z}_a \times \mathbb{Z}_q \times \mathbb{Z}_t$ with $\langle (i, j, k), (i', j', k') \rangle = ii' + kk'$ and ${\bf{d}} = (0, 0, 1)$. This choice of $\Gamma$ will naturally arise when applying derived functors (specifically Hochschild homology) to each chain object in categories of the form $\CS(\AS[\mathbb{Z}_q])$; as such, we refer to the factor $\mathbb{Z}_a$ as the \textit{Hochschild} or \textit{derived} grading. Notice that the Hochschild grading participates in homological sign conventions, but degrees of differentials on chain complexes in $\CS(\AS[\mathbb{Z}_a \times \mathbb{Z}_q])$ still have degree $t$. As above, we write elements of $\mathbb{Z}_a$ multiplicatively and denote Hochschild grading shift functors by $a^{\pm 1}$.
\end{example}

For any $R$-linear category $\AS$, we denote by $\mathrm{Ch}(\AS)$ the category with the same objects as $\CS(\AS)$ but with morphisms restricted to chain maps. We restrict the use of the word \textit{isomorphism} in $\CS(\AS)$ to describe invertible maps in $\mathrm{Ch}(\AS)$ and write $X \cong Y$ in case $X, Y \in \CS(\AS)$ are isomorphic in this sense.

With $\AS$ as above, we denote by $K(\AS)$ the category with the same objects and morphisms as $\CS(\AS)$, but we declare two morphisms $f, g$ in $K(\AS)$ to be equal if $f \sim g$. We will often refer to $K(\AS)$ as the \textit{homotopy category} of $\CS(\AS)$. We call two complexes $X, Y \in K(\AS)$ \textit{homotopy equivalent} and write $X \simeq Y$ if $X$ and $Y$ are isomorphic in $K(\AS)$; we call $X$ \textit{contractible} if $X \simeq 0$. The full data of a homotopy equivalence is a collection of morphisms

\begin{equation} \label{eq: hom_eq}
\begin{tikzcd}
X \arrow[rr, "f", shift left] \arrow["h"', loop, distance=2em, in=215, out=145] &  & Y \arrow[ll, "g", shift left] \arrow["k"', loop, distance=2em, in=35, out=325]
\end{tikzcd}
\end{equation}

\vspace{1em}

\noindent in $\CS(\AS)$ with $\mathrm{deg}(f) = \mathrm{deg}(g) = 0$, $\mathrm{deg}(h) = \mathrm{deg}(k) = -1$, $d_{\CS(\AS)}(f) = d_{\CS(\AS)}(g) = 0$, $d_{\CS(\AS)}(h) = \mathrm{id}_X - g \circ f$, and $d_{\CS(\AS)}(k) = \mathrm{id}_Y - f \circ g$. We will often abuse terminology by referring to just $f$ or $g$ as a homotopy equivalence.

\begin{defn} \label{def: sdr}
In the special case when $k = 0$ in \eqref{eq: hom_eq}, we call the data $\{f, g, h\}$ a \textit{strong deformation retraction} from $X$ onto $Y$. We say this strong deformation retraction satisfies the \textit{side conditions} if $h^2 = 0$, $fh = 0$, and $hg = 0$.
\end{defn}

In the special case that $\AS$ is abelian, then a morphism $f \in \Hom_{\mathrm{Ch}(\AS)}(X, Y)$ induces a morphism $f^* \colon H^{\bullet}(X) \to H^{\bullet}(Y)$ in cohomology. We call $f$ a \textit{quasi-isomorphism} if $f^*$ is an isomorphism.

We will often use superscripts to denote homological bounds in the various categories considered above. For example, $K^b(\AS)$ (resp. $K^+(\AS)$, $K^-(\AS)$) will denote the homotopy category of bounded (resp. bounded below, bounded above) chain complexes over $\AS$. $\AS^b[\mathbb{Z}_t]$, $\CS^b(\AS)$, $\mathrm{Ch}^b(\AS)$, etc. are defined similarly.

\subsection{Curved Complexes} \label{sec: curved_complexes}

In the previous section, given an $R$-linear category $\AS$, we formed the category $\CS(\AS)$ of complexes over $\AS$ by considering $\mathbb{Z}_t$-graded sequences $X = \bigoplus_{k \in \mathbb{Z}} X_k$ of objects of $\AS$ together with a degree $t$ differential $d_X$ satisfying $d_X^2 = 0$. In this section, we relax the requirement $d_X^2 = 0$, instead recovering the category of \textit{curved} complexes over $\AS$.

\begin{defn} \label{def: cat_center}
    Let $\BS$ be a $\Gamma \times \mathbb{Z}_t$-graded category over $R$. The \textit{center} of $\BS$ is the $\Gamma \times \mathbb{Z}_t$-graded algebra $Z(\BS)$ of natural transformations of the identity functor $\mathrm{id}_{\BS}$.
\end{defn}

Given an $R$-linear category $\AS$ and an element $F \in Z(\AS[\mathbb{Z}_t])$, we would like to replace the defining equation $d_X^2 = 0$ in $\CS(\AS)$ with the condition $d_X^2 = F$. This runs into a technical difficulty, as a well-known argument shows that $Z(\AS[\mathbb{Z}_t]) \cong Z(\AS)$ as $\mathbb{Z}_t$-graded algebras (see e.g. Lemma 3.9 of \cite{HRW21}). In particular, $Z(\AS[\mathbb{Z}_t])$ has no elements of degree $t^2$. A standard remedy is to formally extend scalars along some $\Gamma \times \mathbb{Z}_t$-graded $R$-module $S$.

\subsubsection{Extension of Scalars}

\begin{defn} \label{def: extending_scalars_obj_general}
Fix an abelian group $\Gamma$ equipped with a symmetric bilinear form as above, and let $\BS$ be an additive, $\Gamma \times \mathbb{Z}_t$-graded category over $R$. Suppose $S$ is a free $\Gamma \times \mathbb{Z}_t$-graded $R$-module with basis $\mathcal{S} = \{s_i\}_{i \in \mathcal{I}}$. Then for each $X \in \BS$, we denote by $X \otimes_R S$ the formal direct sum

\[
X \otimes_R S := \bigoplus_{i \in \mathcal{I}} X[\mathrm{deg}(s_i)] \in \BS
\]
whenever this direct sum exists. In this situation, given a basis element $s_i \in \mathcal{S}$, we often denote the summand $X[\mathrm{deg}(s_i)]$ of $X \otimes_R S$ by $X \otimes s_i$ or simply $s_i X$.
\end{defn}

\begin{defn}
Given a free $\Gamma \times \mathbb{Z}_t$-graded $R$-module $S$, we denote by $\End_R(S)$ the $\Gamma \times \mathbb{Z}_t$-graded $R$-algebra $\End_{R-\mathrm{Mod}[\Gamma \times \mathbb{Z}_t]}(S)$.
\end{defn}

Since $\BS$ is $R$-linear, given any $X, Y \in \BS$, we can consider the $\Gamma \times \mathbb{Z}_t$-graded $R$-bimodule $\Hom_{\BS}(X, Y) \otimes_R \End_R(S)$, where here $\otimes_R$ denotes the usual tensor product of $\Gamma \times \mathbb{Z}_t$-graded $R$-modules. Whenever $X\otimes_RS, Y \otimes_R S$ are well-defined, we can construct a graded $R$-module inclusion 

\begin{equation} \label{eq: extend_inclusion}
\Hom_{\BS}(X, Y) \otimes_R \End_R(S) \hookrightarrow \Hom_{\BS}(X \otimes_R S, Y \otimes_R S)
\end{equation}
as follows. Let $f \in \Hom_{\BS}(X, Y)$ and $\varphi \in \End_R(S)$ be given. For each basis element $s_i \in \mathcal{S}$, expand $\varphi(s_i)$ along the basis $\mathcal{S}$ as $\varphi(s_i) = \sum_{j \in \mathcal{I}} r_{ij} s_j$; then $r_{ij} = 0$ for all but finitely many $j$. It follows that the direct sum

\[
(f \otimes \varphi)_i := \oplus_{j \in I} r_{ij} f \colon X[s_i] \to Y \otimes_R S
\]
is a well-defined morphism in $\BS$ for each $i \in \mathcal{I}$. Then the inclusion \eqref{eq: extend_inclusion} is given on pure tensors $f \otimes \varphi \in \Hom_{\BS}(X, Y) \otimes_R \End_R(S)$ by

\[
f \otimes \varphi \mapsto \bigoplus_{i \in \mathcal{I}} (f \otimes \varphi)_i
\]
and extended $R$-linearly. We will always suppress this inclusion in the sequel.

\begin{rem}
A straightforward verification shows that \eqref{eq: extend_inclusion} preserves the middle interchange law

\[
(f \otimes s) \circ (f' \otimes s') = (-1)^{\langle \mathrm{deg}(f'), \mathrm{deg}(s) \rangle} (f \circ f') \otimes (s's).
\]
\end{rem}

\vspace{1em}

There is a natural graded $R$-module inclusion of $\Hom_{\BS}(X, Y)$ into $\Hom_{\BS}(X \otimes_R S, Y \otimes_R S)$ sending a morphism $f \colon X \to Y$ to $f \otimes \mathrm{id}_S$. Additionally, post-composition with morphisms of the form $\mathrm{id}_Y \otimes \varphi$ defines a right action of the $\Gamma \times \mathbb{Z}_t$-graded $R$-algebra $\End_R(S)$ on $\Hom_{\BS}(X \otimes_R S, Y \otimes_R S)$.

\begin{defn} \label{def: extending_scalars_cat_general}
    Let $\Gamma, R, S, \BS$ be as in Definition \ref{def: extending_scalars_obj_general}, and suppose $X \otimes_R S$ is well-defined for each $X \in \BS$. We define a subcategory $\BS \otimes_R S$, called the \textit{full object-driven extension of scalars of $\BS$ by $S$}, as follows. The objects of $\BS \otimes_R S$ are those of the form $X \otimes_R S$ for some $X \in \BS$. Given two objects $X, Y \in \BS$, we define the morphism space $\Hom_{\BS \otimes_R S}(X \otimes_R S, Y \otimes_R S)$ to be the $\End_R(S)$-submodule of $\Hom_{\BS}(X \otimes_R S, Y \otimes_R S)$ consisting of (potentially infinite) $\End_R(S)$-linear combinations of morphisms in $\Hom_{\BS}(X, Y)$.
\end{defn}

When extending scalars, it is often useful to restrict the allowable endomorphisms of $S$. We make this precise below.

\begin{defn} \label{def: extending_scalars_cat_restricted}
Let $\Gamma, R, S, \BS$ be as in Definition \ref{def: extending_scalars_cat_general}. Let $\Phi$ be a $\Gamma \times \mathbb{Z}_t$-graded sub $R$-algebra of $\End_R(S)$. Then the \textit{$\Phi$-restricted object-driven extension of scalars of $\BS$ by $S$}, denoted $(\BS \otimes_R S)_{\Phi}$, is the subcategory of $\BS \otimes_R S$ with the same objects as $\BS \otimes_R S$ and morphisms $\Hom_{(\BS \otimes_R S)_{\Phi}}(X, Y)$ potentially infinite $\Phi$-linear combinations of morphisms in $\Hom_{\BS}(X, Y)$.
\end{defn}

\begin{rem} \label{rem: extension_obj_mor}
	There are two common notions of categorical extension of scalars. The first is the \textit{object-driven} extension of Definition \ref{def: extending_scalars_cat_restricted}; this is also the approach taken in \cite{GH22, Con23}. The second is the \textit{morphism-driven} extension for $S$ a graded $R$-algebra; the resulting category $(\BS \otimes_R S)_{mor}$ has objects the same as those of $\BS$ and morphisms $\Hom_{(\BS \otimes_R S)_{mor}}(X, Y) := \Hom_{\BS}(X, Y) \otimes_R S$. This is the approach taken in \cite{HRW21}; see Definition 3.10 of that work.

    Each approach has its advantages. The morphism-driven extension is succinct to define and automatically restricts morphisms in $\BS \otimes_R S$ to be finite $S$-linear combinations of morphisms in $\BS$. On the other hand, the object-driven extension in principle allows one to consider potentially infinite $\End_R(S)$-linear combinations of morphisms in $\BS$, but one must typically restrict the allowable morphisms by hand to get a well-behaved category. We prefer the object-driven extension precisely because we wish to explicitly allow infinite $\End_R(S)$-linear combinations of morphisms between curved complexes. Insisting that $\BS \otimes_R S$ be a subcategory of $\BS$ ensures that these infinite sums are well-defined morphisms in $\BS$.
    
When $\BS = \AS^b[\mathbb{Z}_t]$ consists of \textit{bounded} complexes, $S$ is a graded-commutative $R$-algebra whose generators all have positive $\mathbb{Z}_t$-degree, and $\Phi$ is exactly the class of endomorphisms of $S$ given by mutliplication by some element $s \in S$, then $(\BS \otimes_R S)_{\Phi}$ is canonically isomorphic to the morphism-driven extension $(\BS \otimes_R S)_{mor}$. We almost exclusively work in that situation in the present work, and the reader who prefers the morphism-driven extension is welcome to imagine in such instances that we take that approach; we will always indicate when the distinction is essential.
\end{rem}

\subsubsection{Curved Complexes}

Fix an additive $\Gamma$-graded category $\AS$ over $R$ and a well-defined extension of scalars $\DS := (\AS[\mathbb{Z}_t] \otimes_R S)_{\Phi}$ as in Definition \ref{def: extending_scalars_cat_restricted}.

\begin{prop} \label{prop: ext_scal_nat}
The map $\eqref{eq: extend_inclusion}$ restricts to a natural inclusion $Z(\AS) \otimes_R Z(\Phi) \hookrightarrow Z(\DS)$ of graded $R$-algebras.
\end{prop}

\begin{proof}
Suppose $z \in Z(\AS)$, $\psi \in Z(\Phi)$. Any morphism in $\DS$ is a sum of morphisms of the form $f \otimes \varphi$ for some morphism $f$ in $\AS$ and $\varphi \in \Phi$. We check the required commutativity directly using the middle interchange law:

\begin{align*}
(z \otimes \psi) \circ (f \otimes \varphi) & = (-1)^{\langle \mathrm{deg}(\psi), \mathrm{deg}(f) \rangle} zf \otimes \psi \varphi \\
& = (-1)^{\langle \mathrm{deg}(\psi), \mathrm{deg}(f) \rangle + \langle \mathrm{deg}(z), \mathrm{deg}(f) \rangle + \langle \mathrm{deg}(\psi), \mathrm{deg}(\varphi) \rangle} fz \otimes \varphi \psi \\
& = (-1)^{\langle \mathrm{deg}(z) + \mathrm{deg}(\psi), \mathrm{deg}(f) + \mathrm{deg}(\varphi) \rangle - \langle \mathrm{deg}(z), \mathrm{deg}(\varphi) \rangle} fz \otimes \varphi \psi \\
& = (-1)^{\langle \mathrm{deg}(z \otimes \psi), \mathrm{deg}(f \otimes \varphi) \rangle} (f \otimes \varphi) \circ (z \otimes \psi).
\end{align*}
\end{proof}

\begin{defn} \label{def: curved_complexes}
    Fix a degree $t^2$ element $F \in Z(\DS)$. The category of \textit{$F$-curved complexes over $\AS$}, denoted $\YS_F(\AS; S_{\Phi})$, is the dg category whose objects are pairs $(X \otimes_R S, \delta_X)$ with $X \otimes_R S \in \DS$ and $\delta_X \in \End^1_{\DS}(X \otimes_R S)$ satisfying $\delta_X^2 = F|_{X \otimes_R S}$. We say $(X \otimes_R S, \delta_X)$ has \textit{connection} $\delta_X$ and \textit{curvature }$F$.

    Given two $F$-curved complexes $(X \otimes_R S, \delta_X)$, $(Y \otimes_R S, \delta_Y) \in \YS_F(\AS; S_{\Phi})$, their morphism space is the (uncurved) chain complex
    \[
    \Hom_{\YS_F(\AS; S_{\Phi})}(X \otimes_R S, Y \otimes_R S) = \Hom_{\DS}(X \otimes_R S, Y \otimes_R S); \quad d \colon f \mapsto [\delta, f] = \delta_Y \circ f - (-1)^{\mathrm{deg}(f)} f \circ \delta_X.
    \]
    
We denote by $K\YS_F(\AS; S_{\Phi})$ the \textit{homotopy category of $F$-curved complexes over $\AS$}; this category has the same objects as $\YS_F(\AS; S_{\Phi})$ and morphism sets consisting of homotopy classes of degree $0$ closed morphisms of $\YS_{F}(\AS; S_{\Phi})$.
\end{defn}

\begin{rem}
Implicit in the above definition is the claim that $\YS_F(\AS; S_{\Phi})$ is indeed a dg category over $R$. This is easily verified.
\end{rem}

\begin{defn} \label{def: curv_tensor}
If $\AS$ is monoidal, then in favorable circumstances, there is a notion of tensor product of curved complexes over $\AS$. This arises from an underlying tensor product on $\DS$, defined by

\[
(X \otimes_R S) \otimes_{\DS} (Y \otimes_R S) := (X \otimes_{\AS[\mathbb{Z}_t]} Y) \otimes_R S.
\]
Now let $F_1, F_2 \in Z(\DS)$ be two curvature elements, and suppose $F_1 \otimes 1 + 1 \otimes F_2 \in Z(\DS)$. Then given curved complexes $(X \otimes_R S, \delta_X) \in \YS_{F_1}(\AS; S_{\Phi})$ and $(Y \otimes_R S, \delta_Y) \in \YS_{F_2}(\AS; S_{\Phi})$, their tensor product is the pair

\[
((X \otimes_R S) \otimes_{\DS} (Y \otimes_R S), \delta_X \otimes 1 + 1 \otimes \delta_Y) \in \YS_{F_1 \otimes 1 + 1 \otimes F_2}(\AS; S_{\Phi}).
\]
That the connection above squares to $F_1 \otimes 1 + 1 \otimes F_2$ is an immediate application of the middle interchange law.
\end{defn}

\begin{conv} \label{con: implicit_curv}
    There is a faithful functor from $\CS(\AS)$ to $\YS_0(\AS; S_{\Phi})$ given by taking complexes $(X, d_X) \in \CS(\AS)$ to $0$-curved complexes $(X \otimes_R S, d_X \otimes 1)$ and morphisms $f \colon X \to Y$ to morphisms $f \otimes \mathrm{id}_S \colon X \otimes_R S \to Y \otimes_R S$. This allows us to regard chain complexes over $\AS$ as $0$-curved complexes over $\AS$, and we do so implicitly throughout without further comment. When $S = R$ with trivial $\mathbb{Z}_t$-grading and $\Phi = \End_R(R) \cong R$, this functor gives an equivalence of categories $\CS(\AS) \cong \YS_0(\AS; S_{\Phi})$.
\end{conv}

\begin{rem} \label{rem: unrolling}
    In the special case when $F = 0$, we may include $\YS_F(\AS; S)$ as a (non-full) subcategory of $\CS(\AS)$, as the defining equation $\delta^2 = 0$ is exactly the defining equation of $\CS(\AS)$. Explicitly, this inclusion sends a $0$-curved complex $(X \otimes_R S, \delta_X) \in \YS_0(\AS; S)$ to the chain complex $(X \otimes_R S, \delta_X) \in \CS(\AS)$. We refer to this as the \textit{unrolling} functor.
\end{rem}

\subsection{Twists and Convolutions} \label{sec: twists}
We will frequently wish to modify the connection on a curved complex without changing the underlying chain objects. To set the stage, fix an additive $\Gamma$-graded category $\AS$ over $R$, and let $\DS := (\AS \otimes_R S)_{\Phi}$ be a well-defined extension of scalars. Suppose $F_1$, $F_2 \in Z(\DS)$ are two elements of degree $t^2$. Let $(X \otimes_R S, \delta_X) \in \YS_{F_1}(\AS; S_{\Phi})$ be given, and suppose $X \otimes_R S$ is equipped with an endomorphism $\alpha \in \End^1_{\DS}(X \otimes_R S)$ satisfying the \textit{Maurer--Cartan equation}

\begin{equation} \label{eq: maurer_cartan}
[\delta_X, \alpha] + \alpha^2 = F_2.
\end{equation}

\textcolor{revisions}{Recalling that $[-, -]$ is a graded commutator, a} straightforward calculation gives that $(\delta_X + \alpha)^2 = F_1 + F_2$; in fact, this is equivalent to \eqref{eq: maurer_cartan}.

\begin{defn}
    We call the curved complex

    \[
    \mathrm{tw}_{\alpha}(X \otimes_R S) := (X \otimes_R S, \delta_X + \alpha) \in \YS_{F_1 + F_2}(\AS; S_{\Phi})
    \]
    the \textcolor{revisions}{\textit{curved twist of} $(X \otimes_R S, \delta_X)$ \textit{by} $\alpha$} or simply a \textcolor{revisions}{\textit{curved twist}} of $X \otimes_R S$. \textcolor{revisions}{Similarly, w}e call $\alpha$ a \textit{Maurer--Cartan element} for $F_2$ or just a \textit{twist}.
\end{defn}

\begin{example} \label{ex: cones}
One familiar class of twisted complexes is the \textit{mapping cone} of a closed degree $0$ morphism $f \colon (X, \delta_X) \to (Y, \delta_Y)$ between $F$-curved complexes over $\AS$. To see this, observe that $f$ can be naturally considered as a degree $t$ endomorphism of the direct sum complex $(\textcolor{revisions}{t^{-1}(X)} \oplus Y, -\delta_X + \delta_Y)$. From this perspective $f^2 = 0$, and Equation \eqref{eq: maurer_cartan} reads

\[
[-\delta_X + \delta_Y, f] + f^2 = \delta_Y \circ f + f \circ (- \delta_X) = 0.
\]

As a consequence, we can consider the twist of $\textcolor{revisions}{t^{-1}(X)} \oplus Y$

\[
\mathrm{Cone}(f) := \mathrm{tw}_f \left( \textcolor{revisions}{t^{-1}(X)} \oplus Y \right) \in \YS_F(\AS; S_{\Phi})
\]
\end{example}

\begin{rem} \label{rem: cone_morphism_commute}
\textcolor{revisions}{Observe that there is an honest equality $\mathrm{Cone}(\Phi(f)) = \Phi(\mathrm{Cone}(f))$ for any functor $\Phi$ and morphism $f$. We will use this fact repeatedly beginning in Section \ref{sec: cat_idem}.}
\end{rem}

More generally, we will often wish to consider complexes built iteratively out of mapping cones in this fashion. The language of convolutions gives us a convenient way to formulate this notion.

\begin{defn}
    Let $(I, \leq)$ be an indexing poset, and for each $i \in I$, let $(X_i \otimes_R S, \delta_i) \in \YS_{F_1}(\AS; S_{\Phi})$ be an $F_1$-curved complex over $\AS$. Set\footnote{Here we should also insist that $\DS$ admits direct sums over the set underlying $I$.} $(X, \delta_X) := \bigoplus_{i \in I} (X_i \otimes_R S, \delta_i)$. Suppose $\alpha \in \End^1_{\DS}(X)$ is a Maurer--Cartan element for $F_2$ such that the components $\alpha_{ij} \colon X_j \to X_i$ of $\alpha$ satisfy $\alpha_{ij} = 0$ whenever $i \leq j$.

    In this case, we call $\alpha$ a \textit{one-sided twist} of $X$ and the curved complex $\mathrm{tw}_{\alpha}(X, \delta_X) \in \YS_{F_1 + F_2}(\AS; S_{\Phi})$ a (one-sided) \textit{convolution} of the curved complexes $\{(X_i, \delta_i)\}$ indexed by $I$.
\end{defn}

We will typically denote convolutions explicitly by drawing the twist as an arrow between the complexes $X_i$. For instance, we depict the mapping cone of Example \ref{ex: cones} as

\begin{center}
\begin{tikzcd}
\mathrm{Cone}(f) = \mathrm{tw}_f \left(\textcolor{revisions}{t^{-1}(X)} \oplus Y \right) := t^{-1}X \arrow[r, "f"] & Y
\end{tikzcd}
\end{center}

\vspace{1em}

We can also view (potentially curved) complexes $(X, \delta_X)$ as convolutions of the form 

\[
X = \mathrm{tw}_{\delta_X} \left( \bigoplus_{k \in \mathbb{Z}} t^k X_k, 0 \right)
\]
Here we implicitly consider objects in $\AS$ as chain complexes concentrated in homological degree $0$. In diagrammatic form, this becomes

\begin{center}
\begin{tikzcd}
    (X, \delta_X) = \dots \arrow[r, "\delta_X"] & t^{-1} X_{-1} \arrow[r, "\delta_X"] & X_0 \arrow[r, "\delta_X"] & t X_1 \arrow[r, "\delta_X"] & t^2 X_2 \arrow[r, "\delta_X"] & \dots
\end{tikzcd}
\end{center}
Notice in particular the explicit shifts $t^k$ indicating homological degree; we adopt this convention throughout rather than the more typical underlining of homological degree $0$.

\textcolor{revisions}{We will often wish to lift a homotopy equivalence of curved complexes $X \cong Y$ to a homotopy equivalence between one-sided convolutions $\mathrm{tw}_{\alpha}(X) \cong \mathrm{tw}_{\beta}(Y)$ for some one-sided twists $\alpha \in \End(X)$, $\beta \in \End(Y)$. The study of when this lift can be achieved is called \textit{homological perturbation theory}; see e.g. \cite{Mar01, Hog20}. We will require only the following two special cases of this general theory.}

\begin{prop} \label{prop: hpt_curv_iso}
	Let $\tilde{X} = (X \otimes_R S, \delta_X)$, $\tilde{Y} = (Y \otimes_R S, \delta_Y) \in \YS_{F_1}(\AS; S_{\Phi})$ be given, and suppose $f \in \Hom(\tilde{X}, \tilde{Y})$, $g \in \Hom(\tilde{Y}, \tilde{X})$ constitute an isomorphism in $\YS_{F_1}(\AS; S_{\Phi})$. Let $\alpha \in \End^1_{\DS}(\tilde{X})$ be a Maurer--Cartan element for some $F_2 \in Z(\DS)$. Then $f, g$ lift without modification to an isomorphism of curved complexes $\mathrm{tw}_{\alpha}(\tilde{X}) \cong \mathrm{tw}_{f \alpha g}(\tilde{Y})$ in $\YS_{F_1 + F_2}(\AS; S_{\Phi})$.
\end{prop}

\begin{proof}
This is Proposition 4.1 in \cite{HRW21}, restricted to the special case in which $f$ and $g$ are isomorphisms \textcolor{revisions}{(so that the relevant homotopies vanish)}.
\end{proof}

\begin{prop} \label{prop: cone_invariance}
\textcolor{revisions}{Let $\alpha \colon (X, \delta_X) \to (Y, \delta_Y)$ be a closed degree $0$ morphism, and suppose there is a homotopy equivalence of the form}

\begin{equation*}
\begin{tikzcd}
(Y, \delta_Y) \arrow[rr, "f", shift left] \arrow["h"', loop, distance=2em, in=215, out=145] &  & (Y', \delta_Y') \arrow[ll, "g", shift left] \arrow["k"', loop, distance=2em, in=35, out=325]
\end{tikzcd}
\end{equation*}

\textcolor{revisions}{Then $\mathrm{Cone}(\alpha) \simeq \mathrm{Cone}(f \circ \alpha)$.}

\end{prop}

\begin{proof}
\textcolor{revisions}{Consider the morphism $\Psi \colon \mathrm{Cone}(\alpha) \to \mathrm{Cone}(f \circ \alpha)$ depicted in \textcolor{blue}{blue} below:}

\begin{equation*}
\begin{tikzcd}
t^{-1}X \arrow[rr, "\alpha"] \arrow[dd, "1", color=blue] & & Y \arrow[dd, "f", shift left, harpoon, color=blue]\\
\\
t^{-1}X \arrow[rr, "f \alpha"] & & Y'
\end{tikzcd}
\end{equation*}

\textcolor{revisions}{It is easily verified that $\Psi$ is closed of degree $0$. It is well-known that $\Psi$ is a homotopy equivalence if and only if $\mathrm{Cone}(\Psi) \simeq 0$, so it sufffices to prove the latter. Gaussian elimination along the identity component of the differential of $\mathrm{Cone}(\Psi)$ from $X$ to itself leaves $\mathrm{Cone}(\Psi) \simeq \mathrm{Cone}(f)$. The latter is contractible because $f$ is a homotopy equivalence by assumption.}
\end{proof}

\subsection{Deformations} \label{subsec: def}

Throughout this Section, we fix a $\Gamma$-graded category $\AS$ over $R$ as usual and assume $\langle -, - \rangle_{\Gamma} = 0$ for simplicity.

\begin{defn} \label{def: poly_alg}
Let $\mathbb{U} = \{u_i\}_{i = 1}^n$ be a collection of formal $\Gamma \times \mathbb{Z}_t$-graded variables with homological degree $|u_i| := \mathrm{deg}_{\mathbb{Z}_t}(u_i)$. We denote by $R[\mathbb{U}]$ the graded-commutative polynomial algebra in the variables $\mathbb{U}$ with coefficients in $R$. Formally speaking, this is the collection of formal $R$-linear combinations of words in the alphabet $\mathbb{U}$ subject to the relations $u_i u_j = (-1)^{|u_i| |u_j|} u_ju_i$.
\end{defn}

Given an alphabet $\mathbb{U}$ as in Definition \ref{def: poly_alg}, we may decompose $\mathbb{U} = \mathbb{U}_{even} \sqcup \mathbb{U}_{odd}$, where $u_i \in \mathbb{U}_{even}$ (respectively $\mathbb{U}_{odd}$) if $|u_i|$ is even (respectively odd). We refer to elements of $\mathbb{U}_{even}$ (respectively $\mathbb{U}_{odd}$) as \textit{even} variables (respectively \textit{odd} variables). Since $R$ is a field not of characteristic $2$, each monomial in $R[\mathbb{U}]$ is degree $1$ or $0$ in each odd variable $u_i \in \mathbb{U}_{odd}$. In particular, for each odd variable $u_i$, there is an endomorphism $u_i^{\vee} \in \End_R(R[\mathbb{U}])$ of degree $\mathrm{deg}(u_i^{\vee}) = \mathrm{deg}(u_i)^{-1}$ given by \textit{contraction} against $u_i$. Formally, $u_i^{\vee}$ sends monomials $m$ which can be written as $m = u_i m'$ to $m'$ and all other monomoials to $0$. 

\begin{defn} \label{def: poly_alg_good_mors}
Let $\mathbb{U}$ be as in Definition \ref{def: poly_alg} with decomposition $\mathbb{U} = \mathbb{U}_{even} \sqcup \mathbb{U}_{odd}$ into even and odd variables. We denote by $\Phi(\mathbb{U})$ the graded $R$-subalgebra of $\End_R(R[\mathbb{U}])$ generated by multiplication by all variables in $\mathbb{U}$ and contraction against odd variables.
\end{defn}

The endomorphisms $u_i \in \Phi(\mathbb{U})$ for $u_i$ even and $u_j, u_j^{\vee} \in \Phi(\mathbb{U})$ satisfy a number of relations; we quote a few here without proof.

\begin{prop} \label{prop: poly_props}
Let $\Phi(\mathbb{U})$ be as in Definition \ref{def: poly_alg_good_mors}. Then:

\begin{itemize}
\item $u_i^2 = (u_i^{\vee})^2 = 0 \in \Phi(\mathbb{U})$ for each $u_i \in \mathbb{U}_{odd}$.

\item $u_i^{\vee} u_i + u_i u_i^{\vee} = \mathrm{id}_{R[\mathbb{U}]}$ for each $u_i \in \mathbb{U}_{odd}$.

\item The subalgebra $R[\mathbb{U}_{even}, \mathbb{U}^{\vee}_{odd}] \subset \Phi(\mathbb{U})$ generated by multiplication by even variables and contraction against odd variables is graded commutative.

\item $Z(\Phi(\mathbb{U})) = R[\mathbb{U}_{even}]$.
\end{itemize}
\end{prop}

\begin{defn} \label{def: poly_alg_ext}
We denote by $\AS[\mathbb{U}] := (\AS[\mathbb{Z}_t] \otimes_R R[\mathbb{U}])_{\Phi(U)}$ the $\Phi(\mathbb{U})$-restricted object-driven extension of scalars of $\AS$ by $R[\mathbb{U}]$.
\end{defn}

\begin{prop} \label{prop: def_curv_central}
Let $\varphi_1, \dots, \varphi_n \in Z(\AS[\mathbb{Z}_t])$ be given with $\mathrm{deg}(\varphi_i) := d_i \in \Gamma \times \mathbb{Z}_t$. For each $1 \leq i \leq n$, let $u_i$ be a formal variable of degree $\mathrm{deg}(u_i) = d_i^{-1}t^2$, and let $\mathbb{U} := \{u_1, \dots, u_n\}$ be the collection of all such variables. Set

\begin{equation} \label{eq: def_curv}
F := \sum_{i = 1}^n \varphi_i \otimes u_i
\end{equation}

Then $F \in Z(\AS[\mathbb{U}])$.
\end{prop}

\begin{proof}
Recall from the discussion following Definition \ref{def: cat_center} that $Z(\AS[\mathbb{Z}_t]) \cong Z(\AS)$ as $\mathbb{Z}_t$-graded $R$-algebras. In particular, every element of $Z(\AS[\mathbb{Z}_t])$ with nonzero homological degree must vanish, and so $\varphi_i = 0$ unless $d_i = (\gamma_i, 0)$ for some $\gamma_i \in \Gamma$. But then $\mathrm{deg}(u_i) = (\gamma^{-1}, t^2)$; in particular, $u_i$ is an even variable, so $\varphi_i \otimes u_i \in Z(\AS[\mathbb{U}])$ by Propositions \ref{prop: poly_props} and \ref{prop: ext_scal_nat}.
\end{proof}

\begin{defn} \label{def: deformations}
Let $\varphi_1, \dots, \varphi_n \in Z(\AS[\mathbb{Z}_t])$, $\mathbb{U}$, and $F$ be as in Proposition \ref{prop: def_curv_central}. We call $\YS_F(\AS; R[\mathbb{U}]_{\Phi(\mathbb{U})})$ the category of \textit{$F$-deformations} over $\AS$ and $\mathbb{U}$ the alphabet of \textit{deformation parameters}. We often suppress the subscript $\Phi(\mathbb{U})$, writing just $\YS_F(\AS; R[\mathbb{U}])$.

By convention, when $n = 0$, we set $\mathbb{U} = \emptyset$, $F = 0$, and $\YS_F(\AS; R[\mathbb{U}]) = \CS(\AS)$.
\end{defn}

\begin{rem} \label{rem: odd_def_is_uncurved}
Notice that any $F$-deformation with an alphabet $\mathbb{U}$ of all odd deformation parameters must satisfy $F = 0$. We implicitly consider such deformations as chain complexes by unrolling as in Remark \ref{rem: unrolling}.
\end{rem}

We can decompose morphisms in $\YS_F(\AS; R[\mathbb{U}])$ according to their $\mathbb{U}$-degree in much the same way as we build objects. To do this, we decompose $\mathbb{U}$ into an alphabet of even variables $\mathbb{U}_{even} = \{u_1, \dots, u_m\}$ and odd variables $\mathbb{U}_{odd} = \{u_{m + 1}, \dots, u_n\}$ as usual. Then for each multi-index ${\bf{v}} = (v_1, \dots, v_n) \in \mathbb{Z}_{\geq 0}^m \times \{-1, 0, 1\}^{n - m}$, set \textcolor{revisions}{$u^{\bf{v}} := u_1^{v_1} u_2^{v_2} \dots u_n^{v_n}$.} Given any homogeneous morphism $f \in \Hom_{\YS_F(\AS; R[\mathbb{U}])}(X \otimes_R R[\mathbb{U}], Y \otimes_R R[\mathbb{U}])$, we may separate $f$ into $\mathbb{U}$-degrees

\begin{equation} \label{eq: morphism_components}
f = \sum_{{\bf{v}} \in \mathbb{Z}_{\geq 0}^m \times \{-1, 0, 1\}^{n - m}} f_{a^{\bf{v}}} \otimes a^{\bf{v}}
\end{equation}
where $f_{a^{\bf{v}}} \in \Hom_{\AS[\Gamma \times \mathbb{Z}_t]}(X, Y)$ is a homogeneous morphism of degree $\mathrm{deg}(f_{a^{\bf{v}}}) = \mathrm{deg}(f) - \mathrm{deg}(a^{\bf{v}})$ for each multi-index ${\bf{v}}$.

\textcolor{revisions}{In particular, given a deformation $(X \otimes_R R[\mathbb{U}], \delta_X) \in \YS_F(\AS; R[\mathbb{U}])$, we may consider $\delta_X$ as an endomorphism of $X \otimes_R R[\mathbb{U}]$ and study its $\mathbb{U}$-degree $0$ component $d_X$. Then $d_X$ is an endomorphism of the $\mathbb{Z}_t$-graded object $X \in \AS[\mathbb{Z}_t]$. In particular, if $d_X^2 = 0$, then $(X, d_X)$ is a chain complex.}

\begin{defn} \label{def: curved_lift}
	\textcolor{revisions}{We call $(X \otimes_R R[\mathbb{U}], \delta_X) \in \YS_F(\AS; R[\mathbb{U}])$ an \textit{$F$-deformation} of the chain complex $(X, d_X)$ if $(\delta_X)_{\vec{0}} = d_X$ and $d^2_X = 0$}. Similarly, given two $F$-deformations of complexes $X$ and $Y$ and a map $\tilde{f}$ between these deformations, we call $\tilde{f}$ a \textit{curved lift} of $f$ \textcolor{revisions}{if the $\mathbb{U}$-degree $0$ component $f$ of $\tilde{f}$ is a chain map from $X$ to $Y$.}
\end{defn}

We often \textcolor{revisions}{emphasize that $(X \otimes_R R[\mathbb{U}], \delta_X)$ is an $F$-deformation} by writing
	
	\[
    (X \otimes_R R[\mathbb{U}], \delta_X) = \mathrm{tw}_{\Delta_X}(X \otimes_R R[\mathbb{U}])
    \]
    where $\Delta_X = \delta_X - d_X \otimes 1$. 	We point out that $(X \otimes_R R[\mathbb{U}], \delta_X)$ is an $F$-deformation of $(X, d_X)$ if and only if $\delta_X$ is a curved lift of $d_X$.
    
    The identity $d_X^2 = 0$ is typically enforced by considering the $\mathbb{U}$-degree $0$ component of the equation $\delta_X^2 = 0$ whenever $\mathbb{U}$ consists entirely of even variables of positive homological degree. \textcolor{revisions}{In this case, all curved complexes $(X \otimes_R R[\mathbb{U}], \delta_X)$ in $\YS_F(\AS; R[\mathbb{U}])$ are $F$-deformations of some underlying chain complex $(X, d_X) \in \CS(\AS)$.} Passing to $\mathbb{U}$-degree $0$ in all relevant data constitutes a forgetful functor from $\YS_F(\AS; R[\mathbb{U}])$ to $\CS(\AS)$ taking each curved complex to its underlying chain complex and each morphism $\tilde{f}$ to its $\mathbb{U}$-degree $0$ component $f$. We emphasize that this forgetful functor is \textbf{\textit{not}} the unrolling functor of Remark \ref{rem: unrolling}. Unrolling preserves the underlying object $X \otimes_R R[\mathbb{U}] \in \AS[\mathbb{Z}_t]$, while this forgetful functor does not.

A given complex $(X, d_X)$ may \textit{a priori} have many distinct deformations, even for a fixed curvature $F$. In their original work on deformed link homology \cite{GH22}, Gorsky--Hogancamp establish some obstruction-theoretic tools for comparing these deformations. \textcolor{revisions}{We develop some modest extensions of those tools below.}

\begin{rem}
\textcolor{revisions}{In Proposition \ref{prop: lifting_maps} and Corrolary \ref{cor: inv_uniq_yify}, we assume the relevant alphabet of deformation parameters $\mathbb{U}$ consists only of even variables with positive homological degree. We do this mostly for convenient access to degree arguments as in Definition \ref{def: curved_lift} and invite the interested reader to consider the extent to which the proofs we cite here extend to more general alphabets of deformation parameters.}
\end{rem}

\begin{prop} \label{prop: lifting_maps}
	\textcolor{revisions}{Let $X, Y \in \CS^+(\AS)$ be given, and suppose $H^{\bullet}(\Hom_{\CS(\AS)}(X, Y))$ is concentrated in non-negative degrees. Let $f \colon X \to Y$ be a chain map, and suppose $\tilde{X} = \mathrm{tw}_{\Delta_X}(X \otimes_R R[\mathbb{U}])$ and $\tilde{Y} = \mathrm{tw}_{\Delta_Y}(Y \otimes_R R[\mathbb{U}])$ are $F$-deformations of $X$ and $Y$, respectively. Then there is a chain map $f^u \colon \tilde{X} \to \tilde{Y}$ which is a curved lift of $f$. Further, if $f$ is a homotopy equivalence, then so is $f^u$.}
\end{prop}

\begin{proof}
    \textcolor{revisions}{This is Corollary 4.18 of \cite{Con23}, extended from the bounded category $\CS^b(\AS)$ to the bounded below category $\CS^+(\AS)$. The same proof using Proposition 4.17 of \cite{Con23} applies in this more general setting.}
\end{proof}

\begin{defn} \label{def: quasi-inv}
\textcolor{revisions}{Let $\AS$ be a monoidal category, and let $(X, d_X) \in \CS(\AS)$ be a bounded complex. We call $X$ \textit{right quasi-invertible} if there exists some complex $X^{\vee} \in \CS(\AS)$ such that $X^{\vee}$ is invertible up to homotopy equivalence (that is, there exists some $Y \in \CS(\AS)$ such that $X^{\vee} \otimes Y$ and $Y \otimes X^{\vee}$ are homotopy equivalent to the identity in their home chain categories) and $X \otimes X^{\vee}$ is homotopy equivalent to a complex concentrated in homological degree $0$. In this case, we call $X^{\vee}$ a \textit{right quasi-inverse} of $X$. We define a \textit{left quasi-invertible} complex and its \textit{left quasi-inverse} similarly.}
\end{defn}

\begin{lem} \label{lem: quasi-inv_ends}
\textcolor{revisions}{Suppose $X \in \CS(\AS)$ is left or right quasi-invertible. Then $H^{\bullet}(\End_{\CS(\AS)}(X))$ is concentrated in degree $0$.}
\end{lem}

\begin{proof}
\textcolor{revisions}{Let $X^{\vee}$ be a right quasi-inverse of $X$, and let $C$ be a complex concentrated in homological degree $0$ satisfying $X \otimes X^{\vee} \simeq C$. Since $X^{\vee}$ is invertible up to homotopy equivalence, the functor $- \otimes X^{\vee}$ on $\CS(\AS)$ induces the first in a sequence of homotopy equivalences of complexes $\mathrm{End}(X) \simeq \mathrm{End}(X \otimes X^{\vee}) \simeq \mathrm{End}(C)$, which in turn induce isomorphisms on homology. Since $C$ is concentrated in homological degree $0$, so is $\mathrm{End}(C)$, and therefore so is its homology. The case of left quasi-invertible complexes is exactly analogous.}
\end{proof}

\begin{cor} \label{cor: inv_uniq_yify}
\textcolor{revisions}{Suppose $X \in \CS(\AS)$ is left or right quasi-invertible. Then any two $F$-deformations of $X$ are homotopy equivalent as curved complexes.}
\end{cor}

\begin{proof}
\textcolor{revisions}{Let $Y = (X \otimes_R R[\mathbb{U}], \delta_Y)$ and \textcolor{revisions2}{$Y' = (X \otimes_R R[\mathbb{U}], \delta_Y')$} be two $F$-deformations of $X$. By Proposition \ref{prop: lifting_maps} and Lemma \ref{lem: quasi-inv_ends}, the identity map on $X$ has a curved lift to a map $\tilde{id} \colon Y \to Y'$. Again by Proposition \ref{prop: lifting_maps}, since $id_X$ is a homotopy equivalence, so is $\tilde{id}$.}
\end{proof}

\subsection{Strict Deformations} \label{subsec: strict_def}

Throughout this section, we fix $\varphi_1, \dots, \varphi_n \in Z(\AS[\mathbb{Z}_t])$ with $\mathrm{deg}(\varphi_i) = d_i$, a deformation alphabet $\mathbb{U} = \{u_1, \dots, u_n\}$, and curvature $F = \sum_{i = 1}^n \varphi_i \otimes u_i$ as in Proposition \ref{prop: def_curv_central} unless otherwise noted. Most of the deformations we consider will be of the following form.

\begin{defn}
Let $\tilde{X} := \mathrm{tw}_{\Delta_X}(X \otimes_R R[\mathbb{U}]) \in \YS_F(\AS; R[\mathbb{U}])$ be an $F$-deformation of a chain complex $(X, d_X) \in \CS(\AS)$. We call $\tilde{X}$ a \textit{strict} $F$-deformation if $\Delta_X$ is linear in the alphabet $\mathbb{U}$\textcolor{revisions}{; that is, $(\Delta_X)_{\vec{v}} = 0$ unless $\vec{v} \in \mathbb{Z}^n_{\geq 0}$ contains a single entry $1$ and all other entries $0$.} 
\end{defn}

Given a strict deformation $\tilde{X}$ of $X$, the decomposition of $\Delta_X$ into components as in \eqref{eq: morphism_components} simplifies dramatically, giving

\begin{equation} \label{eq: strict_def}
\Delta_X = \sum_{i = 1}^n \xi_i u_i
\end{equation}
for some family of endomorphisms $\xi_i \in \End_{\AS[\mathbb{Z}_t]}(X)$\textcolor{revisions}{. Notice that, since $\Delta_X$ must be homogeneous of degree $t$, $\xi_i$ must be homogeneous of degree $\mathrm{deg}(\xi_i) = t^{-1} d_i$.} We are naturally led to consider what conditions on the family $\{\xi_i\}$ are necessary and sufficient to ensure that $\mathrm{tw}_{\Delta_X}(X \otimes_R R[\mathbb{U}])$ is an $F$-deformation of $X$ for $\Delta_X$ as defined in \eqref{eq: strict_def}.

\begin{prop} \label{prop: strict_def}
For each $1 \leq i \leq n$, let $\xi_i \in \End_{\AS[\mathbb{Z}_t]}(X)$ be given with $\mathrm{deg}(\xi_i) = t^{-1}d_i$. Define $\Delta_X$ as in \eqref{eq: strict_def}. Then $\mathrm{tw}_{\Delta_X}(X \otimes_R R[\mathbb{U}])$ is an $F$-deformation of $X$ if and only if the following two conditions hold:

\begin{enumerate}
\item $[d_X, \xi_i] = \varphi_i$ for each $i$;
\item The family $\{\xi_i\}$ pairwise graded-commute.
\end{enumerate}
\end{prop}

\begin{proof}
	Recall that $\mathrm{tw}_{\Delta_X}(X \otimes_R R[\mathbb{U}])$ is an $F$-deformation of $X$ if and only if $\Delta_X$ satisfies the Maurer--Cartan Equation \eqref{eq: maurer_cartan}; in our setting, this reads
	
	\begin{equation} \label{eq: Kosz_mc}
	[d_X \otimes 1, \Delta_X] + \Delta_X^2 = \sum_{i = 1}^n \varphi_i \otimes u_i.
	\end{equation}
	
	We expand the left-hand side of \eqref{eq: Kosz_mc}. Beginning with the first term, we obtain
	
	\begin{align*}
	[d_X \otimes 1, \Delta_X] & = (d_X \otimes 1) \left( \sum_{i = 1}^n \xi_i \otimes u_i \right) + \left( \sum_{i = 1}^n \xi_i \otimes u_i \right) (d_X \otimes 1) \\
	& = \sum_{i = 1}^n \left( d_X \xi_i + (-1)^{|u_i|}  \xi_i d_X \right) \otimes u_i
	\end{align*}
	Here the \textcolor{revisions}{first line arises by definition of the graded commutator $[-, -]$, and the} second line arises from the middle interchange law. Since $\mathrm{deg}(\xi_i) = t^{-1}d_i$ and $\mathrm{deg}(u_i) = t^2 d_i^{-1}$, the homological degrees $|\xi_i|$ and $|u_i|$ satisfy $|\xi_i| = 1 - |u_i|$. In particular, we have
	
	\[
	d_X \xi_i + (-1)^{|u_i|} \xi_i d_X = d_X \xi_i - (-1)^{|\xi_i|} \xi_i d_X = [d_X, \xi_i].
	\]
	Making this substitution above gives
	
	\[
	[d_X \otimes 1, \Delta_X] = \sum_{i = 1}^n [d_X, \xi_i] \otimes u_i.
	\]
	Meanwhile, the second term in \eqref{eq: Kosz_mc} gives
	
	\[
	\Delta_X^2 = \left( \sum_{i = 1}^n \xi_i \otimes u_i \right)^2 = \left( \sum_{i = 1}^n \xi_i \otimes u_i \right) \left( \sum_{j = 1}^n \xi_j \otimes u_j \right) = \sum_{i = 1}^n \sum_{j = 1}^n (-1)^{|\xi_j| |u_i|} \xi_i \xi_j \otimes u_i u_j.
	\]
	
	When $i = j$, \textcolor{revisions}{recall from above that $|\xi_i| = 1 - |u_i|$. In particular, either $|\xi_i|$ or $|u_i|$ is odd, and therefore either $\xi_i^2 = 0$ or $u_i^2 = 0$ by graded commutativity.} We can rewrite the remainder of the sum using the middle interchange law as
	
	\begin{align*}
	\Delta_X^2 & = \sum_{i < j} \left( (-1)^{|\xi_j| |u_i|} \xi_i \xi_j \otimes u_i u_j + (-1)^{|\textcolor{revisions}{\xi_i}| |u_j|} \xi_j \xi_i \otimes u_j u_i \right) \\
	& = \sum_{i < j} (-1)^{|\xi_j||u_i|} \left(\xi_i \xi_j + (-1)^{(|\xi_i| + |u_i|)(|\xi_j| + |u_j|) - |\xi_i| |\xi_j|} \xi_j \xi_i \right) \otimes u_i u_j.
	\end{align*}
	Here the second equality follows from graded commutativity of $R[\mathbb{U}]$. \textcolor{revisions}{Because} $|\xi_i| + |u_i| = 1$ for each $i$, this sum simplifies as
	
	\begin{align*}
	\Delta_X^2 & = \sum_{i < j} (-1)^{|\xi_j||u_i|} \left(\xi_i \xi_j + (-1)^{|\xi_i| |\xi_j|} \xi_j \xi_i \right) \otimes u_i u_j \\
	& = \sum_{i < j} (-1)^{|\xi_j||u_i|} [\xi_i, \xi_j] \otimes u_i u_j.
	\end{align*}
	
	Plugging back into \eqref{eq: Kosz_mc}, we obtain
	
	\begin{equation} \label{eq: Kosz_mc_final}
	[\textcolor{revisions}{d_X \otimes 1}, \Delta_X] + \Delta_X^2 = \sum_{i = 1}^n [d_X, f_i] \otimes u_i + \sum_{i < j} (-1)^{|\xi_j||u_i|} [\xi_i, \xi_j] \otimes u_i u_j = \sum_{i = 1}^n \varphi_i \otimes u_i.
	\end{equation}
	
	Equation \eqref{eq: Kosz_mc_final} holds if and only if it holds in each $\mathbb{U}$-component; the linear component holds if and only if (1) is satisfied, and the quadratic component holds if and only if (2) is satisfied.
\end{proof}

\begin{defn} \label{def: strict_def_family}
Let $\{\xi_i\} \subset \End(X)$ be a family of endomorphisms satisfying the conditions of Proposition \ref{prop: strict_def}. Then we call $\{\xi_i\}$ an \textit{$F$-deforming family} on $X$ and $\mathrm{tw}_{\Delta_X}(X \otimes_R R[\mathbb{U}])$ the \textit{strict deformation of $X$ arising from $\{\xi_i\}$.}
\end{defn}

\begin{example}[Koszul Complex] \label{ex: kosz_comp}
Let $\{\xi_1, \dots, \xi_n\}$ be any pairwise-commuting family of closed endomorphisms of $X$ of \textit{even} homological degree. Let $\Theta := \{\theta_1, \dots, \theta_n\}$ denote the corresponding alphabet of deformation parameters \textcolor{revisions}{with degrees $d_i = \mathrm{deg}(\theta_i) = t * \mathrm{deg}(\xi_i)$; notice that this forces the homological degree of each $\theta_i$ to be \textit{odd}.} For each $i$, let $\varphi_i := 0 \in \End_{\CS(\AS)}(X)$, considered as living in degree $d_i$, and set $F = \sum_{i = 1}^n \varphi_i \otimes \theta_i = 0$ as usual. Then $\{\xi_1, \dots, \xi_n\}$ is an $F$-deforming family on $X$.

In this case, we call the strict deformation of $X$ arising from $\{\xi_i\}$ the \textit{Koszul complex} on $X$ for the family $\{\xi_i\}$, which we denote $K(X; \xi_1, \dots, \xi_n)$. Explicitly, we have

\[
K(X; \xi_1, \dots, \xi_n) = \mathrm{tw}_{\Delta_X}(\textcolor{revisions}{X \otimes_R R[\Theta]}) \in \YS_0(\AS; R[\Theta]); \quad \Delta_X := \sum_{i = 1}^n \xi_i \otimes \theta_i.
\]

We often consider $K(X; \xi_1, \dots, \xi_n)$ as a complex in $\CS(\AS)$ by unrolling as in Remark \ref{rem: unrolling}.
\end{example}

The following Proposition is standard and is the setting in which Koszul complexes are most frequently considered; see e.g. \cite{Wei94} for a proof.

\begin{prop} \label{prop: kosz_reg_seq}
Suppose $X$ is a commutative dg $R$-algebra \textcolor{revisions}{(see Definition \ref{def: diff_gamm_grad_mod})} with trivial differential, and let $f_1, \dots, f_n$ be a regular sequence\footnote{That is, $f_i$ is not a zero divisor in $X/(f_1, \dots, f_{i - 1})$ for each $i$.} of elements of $X$ of even homological degree. Set $\gamma := \prod_{i = 1}^n t^{-1} \mathrm{deg}(f_i) \in \Gamma \times \mathbb{Z}_t$. Then $\gamma K(X; f_1, \dots, f_n)$ is quasi-isomorphic to $X/(f_1, \dots, f_n)$ as dg $R$-modules.
\end{prop}

\begin{rem} \label{rem: kosz_grading_shift}
There is a filtration by $\Theta$-degree on $K(X; f_1, \dots, f_n)$ with each subquotient isomorphic to $X$. The grading shift $\gamma$ in Proposition \ref{prop: kosz_reg_seq} ensures that the \textcolor{revisions}{homology $H^{\bullet}(K(X; f_1, \dots, f_n))$ of the Koszul complex} is concentrated in the smallest nontrivial component of this filtration (i.e. highest $\Theta$-degree) whenever $f_1, \dots, f_n$ form a regular sequence.

An equivalent (and more common) way to enforce this is to define the differential on the Koszul complex as $\sum_{i = 1}^n f_i \otimes \theta_i^{\vee}$, contracting against the variables $\theta_i$ rather than multiplying. After unrolling as in Remark \ref{rem: unrolling}, the resulting complexes are isomorphic up to regrading, as is easily checked. We use the conventions of Example \ref{ex: kosz_comp} here for closer agreement with our usual notion of strict deformation.
\end{rem}

Under certain conditions, we may replace an endomorphism in a deforming family with a homotopic morphism without affecting the resulting strict deformation.

\begin{prop} \label{prop: kosz_base_change}
	Let $(X, d_X) \in \CS(\AS)$ be a chain complex and $\{\xi_i\}_{i = 1}^n$ an $F$-deforming family on $X$. Suppose $\{\xi_1, \dots, \xi_{n - 1}, \xi_n'\}$ is also an $F$-deforming family on  $X$ for some $\xi_n' \in \End_{\AS[\mathbb{Z}_t]}(X)$, and suppose there exists some $h \in \End_{\AS[\mathbb{Z}_t]}(X)$ \textcolor{revisions}{of degree $\mathrm{deg}(h) = t^{-1} \mathrm{deg}(\xi_n)$} satisfying the following conditions:
	
	\begin{enumerate}
	\item $[d_X, h] = \xi_n - \xi_n'$;
	\item $h^N = 0$ for sufficiently large $N$;
	\item $R[\xi_1, \dots, \xi_n, \xi'_n, h] \subset \End_{\AS[\mathbb{Z}_t]}(X)$ is graded-commutative.
	\end{enumerate}
	
	Then the strict deformations $\tilde{X}$, $\tilde{X}'$ of $X$ arising from $\{\xi_i\}_{i = 1}^n$ and $\{\xi_1, \dots, \xi_{n - 1}, \xi_n'\}$ are isomorphic as $F$-curved complexes via the mutually inverse morphisms
	
	\[
	\Psi := \sum_{k = 0}^{\infty} \frac{1}{k!} h^k \otimes u_n^k \colon \tilde{X} \to \tilde{X}'; \quad \Psi^{-1} := \sum_{k = 0}^{\infty} \frac{1}{k!} (-h)^k \otimes u_n^k \colon \tilde{X}' \to \tilde{X}
	\]
\end{prop}

\begin{proof}
	First, observe that the degrees of $h$ and $u_n$ exactly cancel, so $\Psi$ is degree $0$. Additionally, (2) ensures that the infinite sum defining $\Psi$ is actually finite\textcolor{revisions}{, so} $\Psi$ is well-defined. It remains to check that $\Psi$, $\Psi^{-1}$ are mutually inverse closed morphisms.
	
	Set $\Delta_X := \sum_{i = 1}^n \xi_i \otimes u_i$, $\Delta_X' := \left( \sum_{i = 1}^{n - 1} \xi_i \otimes u_i \right) + \xi'_n \otimes u_n$. We write $\tilde{X}$, $\tilde{X}'$ explicitly as
	
	\[
	\tilde{X} = \mathrm{tw}_{\Delta_X}(X \otimes_R R[\mathbb{U}]); \quad \tilde{X}' = \mathrm{tw}_{\Delta'_X}(X \otimes_R R[\mathbb{U}]).
	\]
	
	One can verify directly that $\Psi$ is closed by computing $[d, \Psi]$; we omit the details, as they are very similar to the proof of \ref{prop: strict_def}. What needs to be checked is that the expression

    \begin{align*}
    [d, \Psi] & = (d_X \otimes 1 + \Delta_X') \circ \Psi - \Psi \circ (d_X \otimes 1 + \Delta_X) \\
    & = \left( d_X \otimes 1 + \sum_{i = 1}^{n - 1} \xi_i \otimes u_i + \xi'_n \otimes u_n \right) \left( \sum_{k = 0}^{\infty} \frac{1}{k!} h^k \otimes u_n^k \right) - \left( \sum_{k = 0}^{\infty} \frac{1}{k!} h^k \otimes u_n^k \right) \left(d_X \otimes 1 + \sum_{i = 1}^n \xi_i \otimes u_i \right)
    \end{align*}
	vanishes identically; this can be achieved by repeated application of (1), (3), and graded commutativity of $R[\mathbb{U}]$.
    
    Swapping the roles of $\xi_n$ and $\xi_n'$, we see immediately that $\Psi^{-1}$ is closed as well. We claim that $\Psi^{-1} \Psi = \mathrm{id}_{\tilde{X}}$ and $\Psi \Psi^{-1} = \mathrm{id}_{\tilde{X}'}$. By again swapping the roles of $\xi_n$ and $\xi'_n$, it suffices to show the latter. We compute directly:

    \begin{equation} \label{eq: junk_1}
        \Psi \Psi^{-1} = \left( \sum_{j = 0}^{\infty} \frac{1}{j!} h^j \otimes u_n^j \right) \left( \sum_{k = 0}^{\infty} \frac{1}{k!} (-h)^k \otimes u_n^k \right) = \sum_{j, k = 0}^{\infty} \cfrac{(-1)^{|u_n^j| |h^k| + k}}{(j!)(k!)} h^{j + k} \otimes u_n^{j + k}.
    \end{equation}
    
    Upon setting $\ell := j + k$, we can rewrite \eqref{eq: junk_1} as
    
    \begin{equation} \label{eq: junk_2}
        \Psi \Psi^{-1} = \sum_{\ell = 0}^{\infty} \frac{1}{\ell!} \left( \sum_{k = 0}^{\ell} (-1)^{(\ell - k)k|a_r| |h| + k} \binom{\ell}{k} \right) h^{\ell} \otimes u_n^{\ell}.
    \end{equation}
	We claim that \textcolor{revisions}{the coefficient of $h^{\ell} \otimes u_n^{\ell}$ in \eqref{eq: junk_2} vanishes} except when $\ell = k = 0$. Indeed, when $\ell$ is odd, $(\ell - k)k$ is even for all values of $k$. It follows that we may rewrite these summands as

    \begin{equation} \label{eq: junk_3}
    \sum_{k = 0}^{\ell} (-1)^{(\ell - k) k |a_r| |h| + k} \binom{\ell}{k} = \sum_{k = 0}^{\ell} (-1)^{k} \binom{\ell}{k} = \big( 1 + (-1) \big)^{\ell} = 0.
    \end{equation}
    
	On the other hand, suppose $\ell \geq 2$ is even. If $|h|$ is odd, then $h^{\ell} = 0$ by (3). Otherwise, if $|h|$ is even, we can make the same simplification as in \eqref{eq: junk_3}. Then the only surviving term in \eqref{eq: junk_2}, obtained from setting $\ell = k = 0$, is
	
	\[
	\Psi \Psi^{-1} = h^0 \otimes u_n^0 = \mathrm{id}_{\tilde{X}'}.
	\]
\end{proof}

\begin{cor} \label{cor: kosz_base_change_rep}
Let $(X, d_X) \in \CS(\AS)$ be a chain complex, and let $\{\xi_i\}_{i = 1}^n$ and $\{\xi'_i\}_{i = 1}^n$ be two $F$-deforming families on $X$. Suppose that for each $1 \leq i \leq n$, there exists some $h_i \in \End_{\AS[\mathbb{Z}_t]}(X)$ such that $[d_X, h_i] = \xi_i - \xi_i'$. Suppose further that each such $h_i$ is nilpotent and the subalgebra of $\End_{\AS[\mathbb{Z}_t]}$ generated by $\xi_i, \xi'_i$, and $h_i$ for each $i$ is graded-commutative.
	
	Then the strict deformations $\tilde{X}$, $\tilde{X}'$ of $X$ arising from $\{\xi_i\}_{i = 1}^n$ and $\{\xi_i'\}_{i = 1}^n$ are isomorphic as $F$-curved complexes via the mutually inverse morphisms
	
	\[
	\Psi := \prod_{i = 1}^n \sum_{k = 0}^{\infty} \frac{1}{k!} h_i^k \otimes u_i^k \colon \tilde{X} \to \tilde{X}'; \quad \Psi^{-1} := \prod_{i = 1}^n \sum_{k = 0}^{\infty} \frac{1}{k!} (-h_i)^k \otimes u_i^k \colon \tilde{X}' \to \tilde{X}
	\]
\end{cor}

\begin{rem} \label{rem: kosz_base_order_ind}
Notice that the graded commutativity of the family $\{h_i\}$ in Corollary \ref{cor: kosz_base_change_rep} allows the products defining $\Psi$ and \textcolor{revisions}{$\Psi^{-1}$} to be applied in any order.
\end{rem}

\subsection{Deformation Lifting and Gaussian Elimination} \label{subsec: gauss_elim}

We have already discussed several technical tools for comparing deformations of isomorphic complexes in Proposition \ref{prop: hpt_curv_iso}, quasi-invertible complexes in \textcolor{revisions}{Corollary} \ref{cor: inv_uniq_yify}, complexes whose morphism complex has homology concentrated in non-negative homological degrees in Proposition \ref{prop: lifting_maps}, and distinct strict deformations of the same underlying complex in Proposition \ref{prop: kosz_base_change}. Our final tool, Proposition \ref{prop: sdr_deformation_lifting}, allows us to lift a strong deformation retraction satisfying the side conditions to a homotopy equivalence between strict deformations.

\begin{lem} \label{lem: sdr_lifting}
Let $A, B \in \CS(\AS)$ be given, and let $\{f, g, h\}$ be a strong deformation retraction from $A$ to $B$ satisfying the side conditions. Suppose $\{\xi_i\} \subset \End^{-1}_{\CS(\AS)}(B)$ is an $F$ deforming family on $B$. For each $i$, set

\[
\xi_i' := g \xi_i f + \varphi_i h \in \End^{-1}_{\CS(\AS)}(A).
\]
Then $\{\xi'_i\}$ is an $F$-deforming family on $A$ satisfying the following conditions:

\begin{enumerate}
\item $h \xi'_i = \xi'_ih = 0$ for each $i$;
\item $f \xi'_i g = \xi_i$ for each $i$.
\end{enumerate}
\end{lem}

\begin{proof}
We need to verify that $\{\xi'_i\}$ form a graded-commutative family satisfying $[d_A, \xi'_i] = \varphi_i$ and conditions (1) and (2). We check the condition $[d_A, \xi'_i] = \varphi_i$ below by an explicit computation; the other checks are similar, and we leave the details to the interested reader.

By additivity of $d_A$ and the graded Leibniz rule, we have
\begin{align*}
[d_A, \xi'_i] = [d_A, g \xi_i f] + [d_A, \varphi_i h] = [d, g] \xi_i f + g [d_B, \xi_i] f + g \xi_i [d, f] + [d_A, \varphi_i] h + \varphi_i [d_A, h].
\end{align*}
Recall that by assumption, $g, f$, and $\varphi_i$ are closed, $[d_B, \xi_i] = \varphi_i$, and $[d_A, h] = \mathrm{id}_A - gf$. Making these substitutions, we obtain

\[
[d_A, \xi'_i] = g \varphi_i f + \varphi_i (\mathrm{id}_A - gf).
\]
Since $\varphi_i \in Z(\AS[\mathbb{Z}_t])$, it commutes with everything in sight, leaving

\[
[d_A, \xi'_i] = \varphi_i (gf + \mathrm{id}_A - gf) = \varphi_i.
\]
\end{proof}

\begin{defn}
We say the family $\{\xi'_i\}$ of Lemma \ref{lem: sdr_lifting} is obtained by \textit{lifting} the family $\{\xi_i\}$ along the strong deformation retraction $\{f, g, h\}$.
\end{defn}

\begin{prop} \label{prop: sdr_deformation_lifting}
Retain notation as in Lemma \ref{lem: sdr_lifting}, and let $\tilde{A} := \mathrm{tw}_{\alpha}(A \otimes_R S)$, $\tilde{B} := \mathrm{tw}_{\beta}(B \otimes_R S)$ be the strict deformations of $A$ and $B$ arising from the $F$-deforming families $\{\xi'_i\}$ and $\{\xi_i\}$. Then the morphisms $\{f \otimes 1, g \otimes 1, h \otimes 1\}$ constitute a strong deformation retraction from $\tilde{A}$ to $\tilde{B}$.
\end{prop}

\begin{proof}
We need to verify that $f \otimes 1$, $g \otimes 1$ are closed degree $0$ morphisms between $\tilde{A}$ and $\tilde{B}$ satisfying $(f \otimes 1)(g \otimes 1) = \mathrm{id}_{\tilde{B}}$ and $[d_{\tilde{A}}, (h \otimes 1)] = (g \otimes 1)(f \otimes 1)]$. We check that $f \otimes 1$ is closed by an explicit computation below; the other checks are similar, and we leave the details to the interested reader.

\begin{align*}
[d, f \otimes 1] & = \left( \sum_i \xi_i \otimes u_i + d_B \otimes 1 \right) \circ (f \otimes 1) - (f \otimes 1) \circ \left( \sum_i (g \xi_i f + \varphi_i h) \otimes u_i + d_A \otimes 1 \right) \\
& = \left(\sum_i \xi_i f \otimes u_i \right) + d_B f \otimes 1 - \left( \sum_i (fg \xi_i f + f \varphi_i h) \otimes u_i \right) - f d_A \otimes 1 \\
& = \left( \sum_i \varphi_i f h \otimes u_i \right) + (d_B f - f d_A) \otimes 1 \\
& = 0.
\end{align*}
\end{proof}

Proposition \ref{prop: sdr_deformation_lifting} is not of much use without a good method for constructing strong deformation retractions satisfying the side conditions. Our primary means of doing so will be \textit{Gaussian elimination}:

\begin{prop}[Gaussian Elimination] \label{prop: gauss_elim}
Let $A, B_1, B_2, D, E, F$ be (potentially curved) chain complexes over $\AS$, and let $\varphi: B_1 \rightarrow B_2$ be an isomorphism. Consider the diagram below.

\vspace{1.5em}

\begin{equation} \label{eq: gauss_elim}
\begin{tikzcd}[ampersand replacement=\&,column sep=huge]
    A \arrow[rr, "\begin{pmatrix} \alpha \\ \beta \end{pmatrix}"] \arrow[dd, "\mathrm{id}_A", harpoon, shift left]
        \& \& B_1 \oplus D \arrow[dd, "{\setlength\arraycolsep{2pt} \begin{pmatrix} \mathrm{id}_{B_1} & \varphi^{-1} \kappa \\ 0 & \mathrm{id}_D \end{pmatrix}}", harpoon, shift left] \arrow[rr, "{\begin{pmatrix} \varphi & \kappa \\ \gamma & \epsilon \end{pmatrix}}"]
            \& \& B_2 \oplus E \arrow[dd, "{\setlength\arraycolsep{1pt} \begin{pmatrix} \mathrm{id}_{B_2} & 0 \\ -\gamma \varphi^{-1} & \mathrm{id}_E \end{pmatrix}}", harpoon, shift left] \arrow[rr, "{\begin{pmatrix} \rho & \psi \end{pmatrix}}"]
                \& \& F \arrow[dd, "\text{id}_F", harpoon, shift left]\\
    \\
    A \arrow[rr, "\begin{pmatrix} 0 \\ \beta \end{pmatrix}"'] \arrow[uu, "\text{id}_A", harpoon, shift left]
        \& \& B_1 \oplus D \arrow[rr, "{\setlength\arraycolsep{2pt} \begin{pmatrix} \varphi & 0 \\ 0 & \epsilon - \gamma \varphi^{-1} \kappa \end{pmatrix}}"'] \arrow[uu, "{\setlength\arraycolsep{2pt} \begin{pmatrix} \mathrm{id}_{B_1} & -\varphi^{-1} \kappa \\ 0 & \mathrm{id}_D \end{pmatrix}}", harpoon, shift left]
            \& \& B_2 \oplus E \arrow[uu, "{\setlength\arraycolsep{1pt} \begin{pmatrix} \mathrm{id}_{B_2} & 0 \\ \gamma \varphi^{-1} & \mathrm{id}_E \end{pmatrix}}", harpoon, shift left] \arrow[rr, "{\begin{pmatrix} 0 & \psi \end{pmatrix}}"']
                \& \& F \arrow[uu, "\text{id}_F", harpoon, shift left]
\end{tikzcd}
\end{equation}

\vspace{1.5em}

Suppose the top row of \eqref{eq: gauss_elim} is a segment of a curved complex $\tilde{C} \in \YS_G(\AS; S)$ for some curvature $G \in Z(\AS[\mathbb{Z}_t] \otimes_R S)$. Extend \eqref{eq: gauss_elim} to all of $\tilde{C}$ by including vertical identity arrows on the rest of $\tilde{C}$. Let $K$ denote the curved complex

\begin{center}
\begin{tikzcd}
    K = B_1 \arrow[r, "\varphi"] & B_2
\end{tikzcd}
\end{center}
and $C$ denote the complementary summand to $K$ in the bottom row of the extended version of \eqref{eq: gauss_elim}. Then $C \in \YS_G(\AS; S)$, and the morphisms specified by the vertical arrows of the extended version of \eqref{eq: gauss_elim} constitute an isomorphism of curved complexes $\tilde{C} \cong K \oplus C$.
\end{prop}

\begin{proof}
    This is a standard computation; see e.g. Lemma 3.2 of \cite{BN07}\footnote{The proof in \cite{BN07} deals only with uncurved chain complexes, but the computations in the curved case are identical.}.
\end{proof}

\begin{cor} \label{cor: gauss_elim_removal}
In the situation of Proposition \ref{prop: gauss_elim}, there is a strong deformation retraction from $\tilde{C}$ onto $C$ satisfying the side conditions.
\end{cor}

\begin{proof}
Let $\Phi \colon \tilde{C} \to C \oplus K$ denote the isomorphism of Proposition \ref{prop: gauss_elim}; $\pi_C, \pi_K$ be the projections from $C \oplus K$ onto $C$ and $K$, respectively; and $\iota_C, \iota_K$ the corresponding inclusions. Let $h \in \End^{-1}(K)$ be the endomorphism whose only nonzero component is $\varphi^{-1} \colon B_2 \to B_1$, and observe that $[d_K, h] = \mathrm{id}_K$. Then the morphisms constituting the strong deformation retraction from $\tilde{C}$ onto $C$ are as follows:

\begin{center}
\begin{tikzcd}
\tilde{C} \arrow[rr, "\pi_C \Phi", shift left] \arrow["\Phi^{-1} \iota_K h \pi_K \Phi"', loop, distance=2em, in=215, out=145] &  & C \arrow[ll, "\Phi^{-1} \iota_C", shift left]
\end{tikzcd}
\end{center}

Verifying that these morphisms constitute a strong deformation retraction satisfying the side conditions is a straightforward computation.
\end{proof}

\section{Singular Soergel Bimodules} \label{sec: ssbim_main}

\subsection{Symmetric Polynomials} \label{sec: symm_poly}

Fix $N > 0$, and let $\mathbb{X} = \{x_1, \dots, x_N\}$ be an alphabet of formal variables. As usual, we denote by $R[\mathbb{X}]$ the polynomial ring in the alphabet $\mathbb{X}$. There is an action of the symmetric group $\SG_N$ on $R[\mathbb{X}]$ given by permuting the indices of the variables; we denote by $\mathrm{Sym}_R(\mathbb{X}) := R[\mathbb{X}]^{\SG_N} \subset R[\mathbb{X}]$ the ring of invariants under this action. We will often suppress the coefficient ring $R$ when it is clear from context, writing just $\mathrm{Sym}(\mathbb{X})$. We call elements of $\mathrm{Sym}(\mathbb{X})$ \textit{symmetric polynomials}.

We will make use of two distinguished generating sets of $\mathrm{Sym}(\mathbb{X})$.

\begin{defn} \label{def: sym_poly_bases}
    The \textit{elementary symmetric polynomials} $e_i(\mathbb{X})$ and \textit{power sum polynomials} $p_i(\mathbb{X})$ are defined via generating functions as follows:

    \begin{align*}
        E(\mathbb{X}, t) & := \prod_{x \in \mathbb{X}} (1 + xt) =: \sum_{i \geq 0} e_i(\mathbb{X})t^i; \\
        P(\mathbb{X}, t) & := \sum_{x \in \mathbb{X}} \frac{xt}{1 - xt} =: \sum_{i \geq 0} p_i(\mathbb{X})t^i
    \end{align*}
\end{defn}
 
By the fundamental theorem of symmetric polynomials, we have ring isomorphisms $\mathrm{Sym}(\mathbb{X}) \cong R[e_1(\mathbb{X}), \dots, e_N(\mathbb{X})] \cong R[p_1(\mathbb{X}), \dots, p_N(\mathbb{X})]$.

\begin{defn} \label{def: decomp}
    Fix $N > 0$, and let $\bb = (b_1, \dots, b_m)$ be a sequence of positive integers satisfying $\sum_{i = 1}^m b_i = N$. Then we call $\bb$ a \textit{\textcolor{revisions}{composition}} of $N$ and write $\bb \vdash N$.

    Given a \textcolor{revisions}{composition} $\bb \vdash N$ and a positive integer $n$, we denote by $(\bb, n)$ the \textcolor{revisions}{composition} of $N + n$ given by concatenating the \textcolor{revisions}{composition}s $\bb \vdash N$ and $(n) \vdash n$. Similarly, given two \textcolor{revisions}{composition}s $\bb \vdash N$ and $\blam \vdash N'$, we denote by $(\bb, \blam)$ the \textcolor{revisions}{composition} of $N + N'$ given by concatenating $\bb$ and $\blam$. Finally, if $\blam \vdash b_i$ for some $1 \leq i \leq m$, we denote by $\bb^{\blam}$ the \textcolor{revisions}{composition} of $N$ obtained from $\bb$ by replacing $b_i$ with $\blam$; we refer to the passage from $\bb$ to $\bb^{\blam}$ as \textit{refining} a \textcolor{revisions}{composition}.
\end{defn}

\begin{rem} \label{rem: fray_ambig}
    If $b_i = b_j$ for some $i \neq j$, given a \textcolor{revisions}{composition} $\blam \vdash b_i$, the notation $\bb^{\blam}$ is in principle ambiguous. In practice, we will always specify which element of $\bb$ is to be replaced.
\end{rem}

Given a \textcolor{revisions}{composition} $\bb \vdash N$, we may consider the parabolic subgroup $\SG_{\bb} := \SG_{b_1} \times \dots \times \SG_{b_m} \subset \SG_N$. Then $\SG_{\bb}$ also acts on $R[\mathbb{X}]$ by permuting variables; we denote the ring of invariants under this action by $\mathrm{Sym}^{\bb}(\mathbb{X})$. Note that $\mathrm{Sym}^{\aa}(\mathbb{X}) \subset \mathrm{Sym}^{\bb}(\mathbb{X})$ if and only if $\SG_{\aa} \supset \SG_{\bb}$ if and only if $\bb$ can be obtained by repeatedly refining $\aa$.

\begin{example}
    When $\aa = (N) \vdash N$, we have $\mathrm{Sym}^{\aa}(\mathbb{X}) = \mathrm{Sym}(\mathbb{X})$. On the other extreme, when $\bb = (1^N) \vdash N$, we have $\mathrm{Sym}^{\bb}(\mathbb{X}) = R[\mathbb{X}]$.
\end{example}

\begin{example}
    Let $N = 4$, $\aa = (3, 1) \vdash N$, $\bb = (1, 2, 1) \vdash N$. Then $\bb = \aa^{\blam}$, where $\blam$ is the \textcolor{revisions}{composition} $(1, 2) \vdash 3$. We also have $\mathrm{Sym}^{\aa}(\mathbb{X}) \subset \mathrm{Sym}^{\bb}(\mathbb{X})$ and $\SG_{\aa} \supset \SG_{\bb}$.
\end{example}

Equivalently, for each $1 \leq i \leq m$, let $\mathbb{X}_i$ be an alphabet of size $b_i$. We will identify $\mathbb{X}$ with the disjoint union $\mathbb{X}_1 \sqcup \dots \sqcup \mathbb{X}_m$. Under this identification, $R^{\bb}$ is exactly the ring of polynomials partially symmetric in each of the alphabets $\mathbb{X}_i$ separately, and we have a ring isomorphism

\[
\mathrm{Sym}^{\bb}(\mathbb{X}) = \mathrm{Sym}(\mathbb{X}_1 | \dots | \mathbb{X}_m) \cong \mathrm{Sym}(\mathbb{X}_1) \otimes_R \dots \otimes_R \mathrm{Sym}(\mathbb{X}_m)
\]

\vspace{1em}

As a consequence of the fundamental theorem of symmetric polynomials, we obtain a ring isomorphism

\[
\mathrm{Sym}^{\bb}(\mathbb{X}) \cong R[e_1(\mathbb{X}_1), \dots, e_{b_1}(\mathbb{X}_1)] \otimes_R \dots \otimes_R R[e_1(\mathbb{X}_m), \dots, e_{b_m}(\mathbb{X}_m)].
\]

\vspace{1em}

We will often be interested in polynomials which are simultaneously symmetric in two disjoint alphabets $\mathbb{X}$, $\mathbb{X}'$, each of size $N$. In this case, given a \textcolor{revisions}{composition} $\bb = (b_1, \dots, b_m) \vdash N$, we abbreviate

\[
\mathrm{Sym}^{\bb}(\mathbb{X}, \mathbb{X}') := \mathrm{Sym}^{\bb}(\mathbb{X}) \otimes_R \mathrm{Sym}^{\bb}(\mathbb{X}') \cong \mathrm{Sym}^{(\bb, \bb)}(\mathbb{X} | \mathbb{X}').
\]

\vspace{1em}

More generally, given a permutation $\sigma \in \SG_m$, we may \textit{twist} the symmetry requirement on $\mathbb{X}$ from $\bb$ to $\sigma(\bb) = (b_{\sigma^{-1}(1)}, \dots, b_{\sigma^{-1}(m)})$ (or equivalently, permute the sequence of alphabets $(\mathbb{X}_1, \dots, \mathbb{X}_m)$ by $\sigma^{-1}$ before enforcing $\bb$-symmetry). In this case, we write

\[
\mathrm{Sym}^{\bb}_{\sigma}(\mathbb{X}, \mathbb{X}') := \mathrm{Sym}^{\sigma(\bb)}(\mathbb{X}) \otimes_R \mathrm{Sym}^{\bb}(\mathbb{X}') \cong \mathrm{Sym}^{(\sigma(\bb), \bb)}(\mathbb{X} | \mathbb{X}').
\]

\vspace{1em}

We will make heavy use of the following fact.

\begin{lem} \label{lem: a_ijk_exist}
    Fix a positive integer $N > 0$ and a \textcolor{revisions}{composition} $\bb = (b_1, \dots, b_m) \vdash N$, and let $\mathbb{X}, \mathbb{X}'$ be disjoint alphabets each of size $N$. Then there exist a family of polynomials $a_{ijk}^{\bb}(\mathbb{X}, \mathbb{X}') \in \mathrm{Sym}^{\bb}(\mathbb{X}, \mathbb{X}')$ indexed by triples of positive integers $1 \leq i \leq N$, $1 \leq j \leq m$, $1 \leq k \leq b_j$ satisfying

    \[
    \sum_{j = 1}^m \sum_{k = 1}^{b_j} a_{ijk}^{\bb}(\mathbb{X}, \mathbb{X}')
    \big( e_k(\mathbb{X}_j) - e_k(\mathbb{X}'_j) \big) = e_i(\mathbb{X}) - e_i(\mathbb{X}')
    \]
\end{lem}

\begin{proof}
    By the fundamental theorem of symmetric polynomials, we may consider $\mathrm{Sym}^{\bb}(\mathbb{X})$, \textcolor{revisions}{$\mathrm{Sym}^{\bb}(\mathbb{X}')$}, and $\mathrm{Sym}^{\bb}(\mathbb{X}, \mathbb{X}')$ explicitly as polynomial rings over (subsets of) the $2N$ elementary symmetric polynomials $e_k(\mathbb{X}_j), e_k(\mathbb{X}'_j)$. Since $e_i(\mathbb{X})$ (respectively $e_i(\mathbb{X}')$) is simultaneously symmetric in the alphabets
    $\mathbb{X}_1, \dots, \mathbb{X}_m$ (respectively the corresponding primed alphabets), there exists some $N$-variable integer polynomial $P$ such that $P(\{e_k(\mathbb{X}_j)\}) =
    e_i(\mathbb{X})$ and $P(\{e_k(\mathbb{X}'_j)\}) = e_i(\mathbb{X}')$.
    
    Now, consider the ideal
    
    \[
    J = \langle e_k(\mathbb{X}_j) - e_k(\mathbb{X}'_j)
    \rangle_{\substack{1 \leq j \leq m \\ 1 \leq k \leq \lambda_j}} \subset \mathrm{Sym}^{\bb}(\mathbb{X}, \mathbb{X}')
    \]
    
    Then $e_i(\mathbb{X}) - e_i(\mathbb{X}') = P(\{e_k(\mathbb{X}_j)\}) - P(\{e_k(\mathbb{X}'_j)\})$ vanishes on the affine variety $V(J) \subset \mathrm{Spec}(\mathrm{Sym}^{\bb}(\mathbb{X}, \mathbb{X}'))$, so $e_i(\mathbb{X}) - e_i(\mathbb{X}') \in \sqrt{J}$ by the Nullstellensatz. Since $\mathrm{Sym}^{\bb}(\mathbb{X}, \mathbb{X}')/J \cong \mathrm{Sym}^{\bb}(\mathbb{X})$ has no zero divisors, $J$ is a radical ideal, and in fact $e_i(\mathbb{X}) - e_i(\mathbb{X}') \in J$.
\end{proof}

Note that Lemma \ref{lem: a_ijk_exist} is merely an existence statement; there may be some freedom in the specific choice of $a_{ijk}^{\bb}(\mathbb{X}, \mathbb{X}')$ satisfying the desired identity. The next Lemma asserts that, for the family of \textcolor{revisions}{composition}s $\bb_n = (1^n) \vdash n$, we can make a coherent choice of these polynomials as $n$ varies.

\begin{lem} \label{lem: thin_a_compare}
        For each $n \geq 1$, set $\bb_n := (1^n) \vdash n$, and let $\mathbb{X_n}$ be an alphabet of size $n$. Suppose $a_{ij1}^{\textcolor{revisions}{\bb_n}}(\mathbb{X}_n, \mathbb{X}'_n)$ satisfy the conditions of Lemma \ref{lem: a_ijk_exist} for $\bb = \bb_n \vdash n$
        for each $1 \leq i, j \leq n$. Then the family

        \[
        a^{\bb_{n + 1}}_{ij1}(\mathbb{X}_{n + 1}, \mathbb{X}'_{n + 1}) = 
        \begin{cases}
            e_{i - 1}(\mathbb{X}_n) & \text{for} \ j = n + 1; \\
            x'_{n + 1} a^{\bb_n}_{i - 1, j, 1}(\mathbb{X}_n, \mathbb{X}'_n)
            + a^{\bb_n}_{ij1}(\mathbb{X}_n, \mathbb{X}'_n) & \text{for} \ 1 \leq j \leq n
        \end{cases}
        \]

        \vspace{1em}
        
        \noindent satify the conditions of Lemma \ref{lem: a_ijk_exist} for $\bb = \bb_{n + 1} \vdash n + 1$.
    \end{lem}

    \begin{proof}
        By assumption, we have $\sum_{j = 1}^n a_{ij1}^{\bb_n}(\mathbb{X}_n, \mathbb{X}'_n)
        (x_j - x_j') = e_i(\mathbb{X}_n) - e_i(\mathbb{X}'_n)$. Using this assumption,
        we verify explicitly that the choices of $a^{\bb_{n + 1}}_{ij1}$ above
        satisfy the necessary identity:

        \begin{align*}
            \sum_{j = 1}^{n + 1} a_{ij1}^{\bb_{n + 1}} (x_j - x_j')
            & = \left( \sum_{j = 1}^n \left( x'_{n + 1}
            a^{\bb_n}_{i - 1, j, 1}(\mathbb{X}_n, \mathbb{X}'_n)
            + a^{\bb_n}_{ij1}(\mathbb{X}_n, \mathbb{X}'_n) \right)
            (x_j - x_j') \right) + e_{i - 1}(\mathbb{X}_n) (x_{n + 1} - x'_{n + 1}) \\
            & = x'_{n + 1} (e_{i - 1}(\mathbb{X}_n) - e_{i - 1}(\mathbb{X}'_n))
            + (e_i(\mathbb{X}_n) - e_i(\mathbb{X}'_n))
            + e_{i - 1}(\mathbb{X}_n) (x_{n + 1} - x'_{n + 1}) \\
            & = e_i(\mathbb{X}_n) + x_{n + 1}e_{i - 1}(\mathbb{X}_n)
            - (e_i(\mathbb{X}'_n) + x'_{n + 1}e_{i - 1}(\mathbb{X}'_n)) \\
            & = e_i(\mathbb{X}_{n + 1}) - e_i(\mathbb{X}'_{n + 1}).
        \end{align*}

        Here the final equality follows from the easily verified identity $e_i(\mathbb{X}_{n + 1}) = e_i(\mathbb{X}_n) + x_{n + 1} e_{i - 1}(\mathbb{X}_n)$.
    \end{proof}
    
\begin{example}
\textcolor{revisions}{When $\bb = (1, 1, 1) \vdash 3 = n$, we may choose}

\begin{gather*}
\textcolor{revisions}{a_{111}^{\bb} = a_{121}^{\bb} = a_{131}^{\bb} = 1;} \\
\textcolor{revisions}{a_{211}^{\bb} = x_2' + x_3', \ a_{221}^{\bb} = x_1 + x_3', \ a_{231}^{\bb} = x_1 + x_2;} \\
\textcolor{revisions}{a_{311}^{\bb} = x_2' x_3', \ a_{321}^{\bb} = x_1 x_3', \ a_{331}^{\bb} = x_1x_2.}
\end{gather*}

\textcolor{revisions}{The relevant identities in each degree then read}

\begin{gather*}
\textcolor{revisions}{ 1(x_1 - x_1') + 1(x_2 - x_2') + 1(x_3 - x_3') = (x_1 + x_2 + x_3) - (x_1' + x_2' + x_3');} \\
\textcolor{revisions}{ (x_2' + x_3')(x_1 - x_1') + (x_1 + x_3')(x_2 - x_2') + (x_1 + x_2)(x_3 - x_3') = (x_1x_2 + x_1x_3 + x_2x_3) - (x_1'x_2' + x_1'x_3' + x_2'x_3');} \\
\textcolor{revisions}{ x_2'x_3'(x_1 - x_1') + x_1x_3'(x_2 - x_2') + x_1x_2(x_3 - x_3') = x_1x_2x_3 - x_1'x_2'x_3'.}
\end{gather*}
\end{example}

\subsection{Singular Soergel Bimodules} \label{sec: SSBim}

Fix $N > 0$ and an alphabet $\mathbb{X} = \{x_1, \dots, x_N\}$ as in Section \ref{sec: symm_poly}. We now consider $R[\mathbb{X}]$ as $\mathbb{Z}_q$-graded by declaring $\mathrm{deg}(x_i) = q^2$ for each $i$. Let $\mathrm{Bim}_N$ denote the $R$-linear $\mathbb{Z}_q$-graded 2-category defined as follows:
\begin{itemize}
    \item Objects of $\mathrm{Bim}_N$ are \textcolor{revisions}{composition}s $\aa \vdash N$.

    \item $1$-morphisms $\aa \to \bb$ are $\mathbb{Z}_q$-graded $\left( \mathrm{Sym}^{\bb}(\mathbb{X}), \mathrm{Sym}^{\aa}(\mathbb{X}) \right)$-bimodules.

    \item $2$-morphisms are graded bimodule homomorphisms.
\end{itemize}

Horizontal composition along objects $\aa \to \bb \to \cc$ is given by tensor product over $\mathrm{Sym}^{\bb}(\mathbb{X})$ and denoted by $\star$, while vertical compositon of $2$-morphisms is the usual composition of bimodule homomorphisms and denoted by $\circ$. We will often write the identity $1$-morphism for an object $\aa \vdash N$ as $\strand{\aa} := \mathrm{Sym}^{\aa}(\mathbb{X})$. The collection $\mathrm{Bim} := \sqcup_{N \geq 0} \mathrm{Bim}_N$ is itself a monoidal $2$-category under the external tensor product

\[
\boxtimes \colon \mathrm{Bim}_{N_1} \times \mathrm{Bim}_{N_2} \to \mathrm{Bim}_{N_1 + N_2}
\]
given on objects by concatenation $\aa \boxtimes \bb = (\aa, \bb)$ and on $1$- and $2$-morphisms by $\otimes_R$.

\begin{rem} \label{rem: bim_implicit_inc}
    For each $N \geq 0$, there is a canonical inclusion $\mathrm{Bim}_N \hookrightarrow \mathrm{Bim}_{N + 1}$ given on objects by $\aa \mapsto \aa \boxtimes (1) = (\aa, 1)$, on $1$-morphisms by $B \mapsto B \boxtimes R[x_{n + 1}]$, and on $2$-morphisms by $f \mapsto f \boxtimes \mathrm{id}_{R[x_{n + 1}]}$. We will implicitly regard data from $\Bim_N$ as living in $\Bim_{N + 1}$ under this inclusion without further comment.
\end{rem}

Recall that when $\SG_{\aa} \subset \SG_{\bb}$, we have a corresponding reverse inclusion $\mathrm{Sym}^{\bb}(\mathbb{X}) \subset \mathrm{Sym}^{\aa}(\mathbb{X})$ of invariant subrings. We regard $\mathrm{Sym}^{\aa}(\mathbb{X})$ as a $(\mathrm{Sym}^{\bb}(\mathbb{X}), \mathrm{Sym}^{\aa}(\mathbb{X}))$-bimodule with left action given by restriction of scalars along this inclusion. In this case, we have distinguished \textit{merge} and \textit{split} bimodules

\[
_{\bb}M_{\aa} := q^{\ell(\bb) - \ell(\aa)} \mathrm{Sym}^{\bb}(\mathbb{X}) \otimes_{\mathrm{Sym}^{\bb}(\mathbb{X})} \mathrm{Sym}^{\aa}(\mathbb{X}); \quad _{\aa}S_{\bb} := \mathrm{Sym}^{\aa}(\mathbb{X}) \otimes_{\mathrm{Sym}^{\bb}(\mathbb{X})} \mathrm{Sym}^{\bb}(\mathbb{X}).
\]
Here $q$ is the usual quantum shift functor and $\ell(\aa)$ (resp. $\ell(\bb)$) is the length of the longest element in $\SG_{\aa}$ (resp. $\SG_{\bb}$).

\begin{defn} \label{def: SSBim}
    The ($R$-linear, $\mathbb{Z}_q$-graded) $2$-category $\SSBim_N$ of \textit{singular Soergel bimodules} is the full $2$-subcategory of $\Bim_N$ generated by $1$-morphisms of the form $\strand{\aa}$, $_{\bb} M _{\aa}$, and $_{\aa} S_{\bb}$ under quantum shifts, horizontal and vertical composition, direct sums, and direct summands. The disjoint union $\sqcup_{N \geq 0} \SSBim_N$ is closed under the external tensor product $\boxtimes$ in the monoidal $2$-category $\mathrm{Bim}$, and so this collection is itself a full monoidal $2$-subcategory of $\mathrm{Bim}$ which we denote by $\SSBim$. We will denote the category of $1$-morphisms from $\aa$ to $\bb$ in $\SSBim$ by $\SSBim_{\aa}^{\bb}$.
\end{defn}

\begin{rem} \label{rem: sbim}
    In the special case when $\aa = \bb = (1^N) \vdash N$, we denote the category of morphisms from $\aa$ to $\bb$ also by $\SBim_N := \SSBim_{(1^n)}^{(1^n)}$ and refer to such morphisms as \textit{Soergel bimodules}. In this case, we have a canonical inclusion functor $\SBim_N \hookrightarrow \SBim_{N + 1}$ defined as in Remark \ref{rem: bim_implicit_inc}.
\end{rem}

There is a useful graphical calculus \textcolor{revisions2}{for $1$-morphisms in $\SSBim$} \textcolor{revisions}{analogous to the Type A webs developed by Cautis--Kamnitzer--Morrison in \cite{CKM14}}. In the special case $\aa = (a, b)$, $\bb = (a + b) \vdash N$, we denote the identity, merge, and split bimodules as follows:

\begin{gather*}
\strand{\bb} := 
\begin{tikzpicture}[anchorbase,scale=.75,tinynodes]
	\draw[webs] (1,0) node[below]{$a + b$}
	 to[out=90,in=270] (1,1) node[above,yshift=-2pt]{$a + b$};
\end{tikzpicture}
; \quad\quad
_{\bb} M_{\aa} := 
\begin{tikzpicture}[anchorbase,scale=.75,tinynodes]
	\draw[webs] (0,0) node[below]{$a$} to[out=90,in=180] (.5,.5);
	\draw[webs] (1,0) node[below]{$b$} to[out=90,in=0] (.5,.5);
	\draw[webs] (.5,.5) to[out=90,in=270] (.5,1) node[above,yshift=-2pt]{$a + b$};
\end{tikzpicture}
; \quad\quad
_{\aa} S_{\bb} :=
\begin{tikzpicture}[anchorbase,scale=.75,tinynodes]
	\draw[webs] (.5,.5) to[out=180,in=270] (0,1) node[above,yshift=-2pt]{$a$};
	\draw[webs] (.5,.5) to[out=0,in=270] (1,1) node[above,yshift=-2pt]{$b$};
	\draw[webs] (.5,0) node[below]{$a + b$} to[out=90,in=270] (.5,.5);
\end{tikzpicture}
\end{gather*}

We depict external tensor product $\boxtimes$ in \textcolor{revisions}{$\SSBim$} by horizontal concatenation of such diagrams and horizontal composition $\star$ of $1$-morphisms by vertical concatenation. More complicated diagrams are built from the above by these operations. For example, for $\aa = (1, 2, 4, 2)$, $\bb = (1, 6, 2)$, $\cc = (1, 6, 1, 1) \vdash 9$, the $1$-morphism $(_{\cc} S_{\bb}) \star (_{\bb} M_{\aa})$ in $SSBim_9$ is depicted as follows:

\begin{gather*}
(_{\cc} S_{\bb}) \star (_{\bb} M_{\aa}) :=  
\begin{tikzpicture}[anchorbase,scale=.75,tinynodes]
	\draw[webs] (1,0) node[below]{$1$}
	 to[out=90,in=270] (1,1) node[above,yshift=-2pt]{$1$};
	 \draw[webs] (2,0) node[below]{$2$} to[out=90,in=180] (2.5,.5);
	\draw[webs] (3,0) node[below]{$4$} to[out=90,in=0] (2.5,.5);
	\draw[webs] (2.5,.5) to[out=90,in=270] (2.5,1) node[above,yshift=-2pt]{$6$};
	\draw[webs] (4.5,0.5) to[out=180,in=270] (4,1) node[above,yshift=-2pt]{$1$};
	\draw[webs] (4.5,0.5) to[out=0,in=270] (5,1) node[above,yshift=-2pt]{$1$};
	\draw[webs] (4.5,0) node[below]{$2$} to[out=90,in=270] (4.5,0.5);
\end{tikzpicture}
\end{gather*}

We interpret more complex web diagrams as follows. Edges labeled with the value $i$ correspond to symmetric polynomial rings in $i$ variables; we refer to these labels as \textit{colors} and the top and bottom such labels in a given diagram as \textit{external colors}, and we refer to edges as \textit{strands}. Every trivalent graph built from these edges represents a $1$-morphism from the object given by reading colors from left to right along the bottom of the diagram to the object given by reading colors from left to right along the top of the diagram. Trivalent vertices indicate a tensor product over the symmetric polynomial ring corresponding to the highest color entering that vertex. We also demand that, for each vertex, the sum of colors entering the bottom and exiting the top of that vertex are equal.

\begin{rem}
    We often suppress colors on strands when there is no ambiguity (for example, when all colors can be determined from those provided by balancing colors above and below each trivalent vertex).
\end{rem}

\textcolor{revisions}{The graphical calculus for \textcolor{revisions2}{$1$-morphisms} is generated by  identity, merge, and split bimodules under the dictionary of operations described above, subject to some relations\footnote{\textcolor{revisions}{One can recover a complete list of relations encoding isomorphisms satisifed by these diagrammatic generators from \cite{CKM14} by considering all webs as upwards oriented, ignoring all tags, and taking $n \to \infty$. We warn the experienced reader that \cite{CKM14} work both with $\mathfrak{sl}_n$ webs (rather than the $\mathfrak{gl}_n$ webs employed here) and at the decategorified level of $\mathfrak{sl}_n$-intertwiners. These differences do not pose any difficulties to the reader only interested in referring to \cite{CKM14} for a list of graphical relations.}}.} We cite here two relations that will be particularly useful to us. To formulate these relations, we will make use of the \textit{quantum integers}:

\[
[j] := q^{j -1} + q^{j - 3} + \dots + q^{-j + 3} + q^{-j + 1}; \quad [-j] := -[j]
\]
as well as the \textit{quantum binomial coefficients}:

\begin{align*}
\begin{bmatrix} j \\ k \end{bmatrix} = \frac{[j]!}{[k]![j - k]!}; \quad \quad [j]! := [j][j - 1] \dots [2][1]
\end{align*}

Given a Laurent polynomial $f(q) \in \mathbb{N}[q, q^{-1}]$ and a singular Soergel bimodule $B$, we denote by $f(q)B$ a direct sum of quantum shifts of $B$ indicated by the corresponding coefficients of $f$; for example, we write

\[
(2q^2 + 1 + q^{-5})B := q^2B \oplus q^2 B \oplus B \oplus q^{-5}B.
\]

\vspace{1em}

Unless otherwise stated, the isomorphisms below hold for all valid colorings of their underlying trivalent graphs.

\begin{prop} \label{prop: web_rels}
There are degree-preserving isomorphisms of $1$-morphisms in $SSBim$ of the following form:

\begin{gather}
\begin{tikzpicture}[anchorbase,scale=.5,tinynodes]  \label{tikzpic: digons}
    \draw[webs] (0.5,0) to (0.5,1);
    \draw[webs] (0.5,1) to[out=180,in=270] (0,1.75) node[left]{$j$};
    \draw[webs] (0.5,1) to[out=0,in=270] (1,1.75) node[right]{$k$};
    \draw[webs] (.5,2.5) to[out=180,in=90] (0,1.75);
    \draw[webs] (.5,2.5) to[out=0,in=90] (1,1.75);
    \draw[webs] (.5,2.5) to (.5,3.5);
\end{tikzpicture}
\cong
\begin{bmatrix} j + k \\ k \end{bmatrix}
\begin{tikzpicture}[anchorbase,scale=.5,tinynodes]
    \draw[webs] (0,0) node[below]{$j + k$} to (0,2) node[above,yshift=-2pt]{$j + k$};
\end{tikzpicture}
; \\
\begin{tikzpicture}[anchorbase,scale=.5,tinynodes] \label{tikzpic: assoc}
    \draw[webs] (0,0) node[below]{$i$} to (0,1);
    \draw[webs] (.5,0) node[below]{$j$} to[out=90,in=180] (1,.5);
    \draw[webs] (1.5,0) node[below]{$k$} to[out=90,in=0] (1,.5);
    \draw[webs] (1,.5) to (1,1);
    \draw[webs] (0,1) to[out=90,in=180] (.5,1.5);
    \draw[webs] (1,1) to[out=90,in=0] (.5,1.5);
    \draw[webs] (.5,1.5) to (.5,2) node[above]{$i + j + k$};
\end{tikzpicture}
\cong
\begin{tikzpicture}[anchorbase,scale=.5,tinynodes]
    \draw[webs] (0,0) node[below]{$i$} to[out=90,in=180] (.5,.5);
    \draw[webs] (1,0) node[below]{$j$} to[out=90,in=0] (.5,.5);
    \draw[webs] (1.5,0) node[below]{$k$} to (1.5,1);
    \draw[webs] (.5,.5) to (.5,1);
    \draw[webs] (.5,1) to[out=90,in=180] (1,1.5);
    \draw[webs] (1.5,1) to[out=90,in=0] (1,1.5);
    \draw[webs] (1,1.5) to (1,2) node[above]{$i + j + k$};
\end{tikzpicture}
; \quad \quad
\begin{tikzpicture}[anchorbase,scale=.5,tinynodes]
    \draw[webs] (0,2) node[above,yshift=-2pt]{$i$} to (0,1);
    \draw[webs] (.5,2) node[above,yshift=-2pt]{$j$} to[out=270,in=180] (1,1.5);
    \draw[webs] (1.5,2) node[above,yshift=-2pt]{$k$} to[out=270,in=0] (1,1.5);
    \draw[webs] (1,1.5) to (1,1);
    \draw[webs] (0,1) to[out=270,in=180] (.5,.5);
    \draw[webs] (1,1) to[out=270,in=0] (.5,.5);
    \draw[webs] (.5,.5) to (.5,0) node[below]{$i + j + k$};
\end{tikzpicture}
\cong
\begin{tikzpicture}[anchorbase,scale=.5,tinynodes]
    \draw[webs] (0,2) node[above,yshift=-2pt]{$i$} to[out=270,in=180] (.5,1.5);
    \draw[webs] (1,2) node[above,yshift=-2pt]{$j$} to[out=270,in=0] (.5,1.5);
    \draw[webs] (1.5,2) node[above,yshift=-2pt]{$k$} to (1.5,1);
    \draw[webs] (.5,1.5) to (.5,1);
    \draw[webs] (.5,1) to[out=270,in=180] (1,.5);
    \draw[webs] (1.5,1) to[out=270,in=0] (1,.5);
    \draw[webs] (1,.5) to (1,0) node[below]{$i + j + k$};
\end{tikzpicture}
\end{gather}

\end{prop}

\begin{proof}
\textcolor{revisions}{\eqref{tikzpic: digons} follows immediately from the fact that $\mathrm{Sym}^{(j, k)}(\mathbb{X})$ is a free graded $\mathrm{Sym}^{(j + k)}(\mathbb{X})$-module of graded rank $\begin{bmatrix} j + k \\ k \end{bmatrix}$ (with degrees shifted to be symmetric about $0$), while \eqref{tikzpic: assoc} follows from the associativity of extension of scalars up to canonical isomorphism.}
\end{proof}

We refer to \eqref{tikzpic: digons} and \eqref{tikzpic: assoc} as the \textit{digon removal} and \textit{associativity} relations, respectively. The associativity relations allow for unambiguous representation \textcolor{revisions}{(up to isomorphism)} of merges and splits along an arbitrary refinement $\aa = \bb^{\blam}$ using vertices of higher valence; for example, given \textcolor{revisions}{composition}s $\aa = (a, b, c, d)$, $\bb = (a + b + c, d) \vdash a + b + c + d$ and $\blam = (a, b, c) \vdash a + b + c$, we have the graphical representation

\begin{gather*}
_{\bb} M_{\aa} =
\begin{tikzpicture}[anchorbase,scale=.5,tinynodes]
    \draw[webs] (-.2,0) node[below]{$a$} to[out=90,in=180] (.5,1);
    \draw[webs] (.5,0) node[below]{$b$} to (.5,1);
    \draw[webs] (1.2,0) node[below]{$c$} to[out=90,in=0] (.5,1);
    \draw[webs] (.5,1) to (.5,2) node[above,yshift=-2pt]{$a + b + c$};
    \draw[webs] (2.2, 0) node[below]{$d$} to (2.2,2) node[above,yshift=-2pt]{$d$};
\end{tikzpicture}
\cong
\begin{tikzpicture}[anchorbase,scale=.5,tinynodes]
    \draw[webs] (0,0) node[below]{$a$} to (0,1);
    \draw[webs] (.5,0) node[below]{$b$} to[out=90,in=180] (1,.5);
    \draw[webs] (1.5,0) node[below]{$c$} to[out=90,in=0] (1,.5);
    \draw[webs] (1,.5) to (1,1);
    \draw[webs] (0,1) to[out=90,in=180] (.5,1.5);
    \draw[webs] (1,1) to[out=90,in=0] (.5,1.5);
    \draw[webs] (.5,1.5) to (.5,2) node[above,yshift=-2pt]{$a + b + c$};
    \draw[webs] (2.2, 0) node[below]{$d$} to (2.2,2) node[above,yshift=-2pt]{$d$};
\end{tikzpicture}
\cong
\begin{tikzpicture}[anchorbase,scale=.5,tinynodes]
    \draw[webs] (0,0) node[below]{$a$} to[out=90,in=180] (.5,.5);
    \draw[webs] (1,0) node[below]{$b$} to[out=90,in=0] (.5,.5);
    \draw[webs] (.5,.5) to (.5,1);
    \draw[webs] (1.5,0) node[below]{$c$} to (1.5,1);
    \draw[webs] (.5,1) to[out=90,in=180] (1,1.5);
    \draw[webs] (1.5,1) to[out=90,in=0] (1,1.5);
    \draw[webs] (1,1.5) to (1,2) node[above,yshift=-2pt]{$a + b + c$};
    \draw[webs] (2.5, 0) node[below]{$d$} to (2.5,2) node[above,yshift=-2pt]{$d$};
\end{tikzpicture}
.
\end{gather*}

\begin{defn} \label{def: full_merge_split}
    Given \textcolor{revisions}{composition}s $\aa, \bb \vdash N$, we denote by $W_{\aa}^{\bb}$ the singular Soergel bimodule $(_{\bb} S_{(N)}) \star (_{(N)} M_{\aa}) \in \SSBim^{\bb}_{\aa}$. When $\aa = \bb$, we write $W_{\aa} := W_{\aa}^{\aa}$.
\end{defn}

\begin{example}
    When $\aa = (3, 1, 1) \vdash 5$, $\bb = (1, 2, 1, 1) \vdash 5$, we have
    \begin{gather*}
    W_{\aa}^{\bb} := 
    \begin{tikzpicture}[anchorbase,scale=.75,tinynodes]
    \draw[webs] (0,0) node[below]{$3$} to[out=90, in=180] (.5,.5);
    \draw[webs] (.5,0) node[below]{$1$} to (.5,.5);
    \draw[webs] (1,0) node[below]{$1$} to[out=90,in=0] (.5,.5);
    \draw[webs] (.5,.5) to (.5,1);
    \draw[webs] (.5,1) to[out=180,in=270] (-.25,1.5) node[above]{$1$};
    \draw[webs] (.5,1) to[out=180,in=270] (.25,1.5) node[above]{$2$};
    \draw[webs] (.5,1) to[out=0,in=270] (.75,1.5) node[above]{$1$};
    \draw[webs] (.5,1) to[out=0,in=270] (1.25,1.5) node[above]{$1$};
    \end{tikzpicture}
    \end{gather*}
\end{example}

\begin{example}
    When $\aa = (N) \vdash N$, then $W_{\aa} = \strand{(N)}$ is the identity bimodule. When $\aa = (1, 1) \vdash 2$, then $W_{\aa} \in \SSBim_{\aa}^{\aa}$ is the usual Bott--Samelson bimodule $B_1$.
\end{example}

The following is a useful generalization of the digon removal relation.

\begin{prop} \label{prop: blamgon_removal}
    Let $\blam = (\lambda_1, \dots, \lambda_n) \vdash N$. Then $(_{(N)} M_{\blam}) \star (_{\blam} S_{(N)}) \cong \cfrac{[N]!}{\prod_{i = 1}^n [\lambda_i]!} \strand{(N)}$.
\end{prop}

\begin{cor}
    Let $\bb = (b_1, \dots, b_n) \vdash N$. Then $W_{\bb}^{\cc} \star W_{\aa}^{\bb} \cong \cfrac{[N]!}{\prod_{i = 1}^n [b_i]!} W_{\aa}^{\cc}$.
\end{cor}

We will also be interested in complexes of singular Soergel bimodules. As $SSBim$ is a $2$-category, the standard constructions of dg categories above do not apply verbatim. We remedy this by applying these constructions to $1$-morphisms as follows.

\begin{defn}
    We denote by $\CS(\SSBim)$ the dg $2$-category of complexes of singular Soergel bimodules whose objects agree with those of $\SSBim$ and with $1$-morphism categories $\Hom_{\CS(\SSBim)}(\aa, \bb) = \CS(\SSBim_{\aa}^{\bb})$. Given a $\mathbb{Z}_q \times \mathbb{Z}_t$-graded polynomial algebra $S$ with coefficients in $R$ and a homological degree $t^2$ element $F \in Z \left( \SSBim_{\aa}^{\bb}[\mathbb{Z}_t] \otimes_R S \right)$, we similarly define the dg $2$-category $\CS_F(\SSBim; S)$ of $F$-curved complexes of singular Soergel bimodules.
\end{defn}

More precisely, given \textcolor{revisions}{composition}s $\aa, \bb \vdash N$, a $1$-morphism $C \in \Hom_{\CS(\SSBim_N)}(\aa, \bb)$ is a chain complex of singular Soergel $\left( \mathrm{Sym}^{\bb}(\mathbb{X}), \mathrm{Sym}^{\aa}(\mathbb{X}) \right)$-bimodules, and a $2$-morphism between two such complexes $C, C'$ is a homomorphism of $\mathbb{Z}_q \times \mathbb{Z}_t$-graded bimodules. Horizontal composition and external tensor product of complexes are defined using the usual conventions for tensor products of chain complexes.

\begin{rem}
\textcolor{revisions}{As usual, we will typically restrict our attention to $F$-deformations with $F$ of the form $F := \sum \varphi_i \otimes u_i$ and $S := R[\mathbb{U}]$ a polynomial algebra over an alphabet of $\mathbb{Z}_q \times \mathbb{Z}_t$-graded deformation parameters; see Definition \ref{def: delta_e_curv} for a precise statement.}
\end{rem}

\subsection{Colored Braids and Rickard Complexes} \label{sec: col_braids}

Let $\BG_n$ denote the $n$-strand braid group with Artin generators $\sigma_1, \dots, \sigma_{n - 1}$. Forgetting the data of over/under crossings gives a homomorphism $\BG_n \to \SG_n$ sending each braid to the corresponding permutation of its strands (read from bottom to top), and we implicitly consider braids as permutations in this way in what follows. 

\begin{defn} \label{def: colored_braids}
    The \textit{colored braid groupoid} on $n$ strands is the category $\mathfrak{Br}_n$ consisting of the following data:

    \begin{itemize}
        \item Objects are sequences $\aa = (a_1, \dots, a_n) \in \mathbb{Z}_{\geq 1}^n$.

        \item Given two such sequences $\aa = (a_1, \dots, a_n)$, $\bb = (b_1, \dots, b_n)$, their morphism space $\BGpd{\aa}{\bb}$ is the collection of braids $\beta \in Br_n$ such that $\beta$ takes $\aa$ to $\bb$ when considered as a permutation. More precisely:

        \[
        \BGpd{\aa}{\bb} := \Hom_{\mathfrak{Br}_n}(\aa, \bb) = \{\beta \in \BG_n \ | \ a_i = b_{\beta(i)} \ \text{for each} \ 1 \leq i \leq n\}
        \]
    \end{itemize}
\end{defn}

We refer to objects of $\mathfrak{Br}_n$ as \textit{sequences of colors} and morphisms as \textit{colored braids} on $n$ strands. We call two sequences of colors $\aa$, $\bb$ compatible if $\BGpd{\aa}{\bb}$ is non-empty. A colored braid on $n$ strands is equivalent to a choice of braid $\beta \in \BG_n$ and an assignment of a positive integer ``color" to each strand of $\beta$. As such, we often denote elements of $\BGpd{\aa}{\bb}$ by $_{\bb} \beta_{\aa}$, or suppressing the codomain or domain sequence, as $_{\bb} \beta$ or $\beta _{\aa}$ (since either sequence can be recovered from the other by applying the permutation corresponding to $\beta$ or its inverse). Finally, we call a sequence $(\underline{\beta})_{\aa}$ of colored Artin generators and their inverses a \textit{colored braid word} representing the corresponding colored braid $\beta_{\aa}$.

We employ the usual graphical depiction of colored braids according to their description as braids with labeled strands. For example, when $\aa = (3, 5, 2)$, $\bb = (5, 3, 2)$, $\cc = (3, 2, 5)$, we have Artin generators

\begin{gather*}
    \BGpd{\aa}{\bb} \ni (\sigma_1)_{\aa} :=
    \begin{tikzpicture}[anchorbase,tinynodes]
    \draw[webs] (0,1) node[above]{$5$} to[out=270,in=90] (.75,0) node[below]{$5$};
    \draw[line width=5pt, color=white] (.75,1) node[above]{$3$} to[out=270,in=90] (0,0);
    \draw[webs] (.75,1) node[above]{$3$} to[out=270,in=90] (0,0) node[below]{$3$};
    \draw[webs] (1.5,0) node[below]{$2$} to (1.5,1) node[above]{$2$};
    \end{tikzpicture}
    ; \quad
    \BGpd{\aa}{\cc} \ni (\sigma_2)_{\aa} := 
    \begin{tikzpicture}[anchorbase,tinynodes]
    \draw[webs] (0,0) node[below]{$3$} to (0,1) node[above]{$3$};
    \draw[webs] (1.5,0) node[below]{$2$} to[out=90,in=270] (.75,1) node[above]{$2$};
    \draw[line width=5pt, color=white] (.75,0) to[out=90,in=270] (1.5,1);
    \draw[webs] (.75,0) node[below]{$5$} to[out=90,in=270] (1.5,1) node[above]{$5$};
    \end{tikzpicture}
\end{gather*}

We will often wish to consider cables of colored braids according to refinements of sequences as in Definition \ref{def: decomp}.

\begin{defn} \label{def: braid_cable}
    Let $\aa = (a_1, \dots, a_n)$, $\bb = (b_1, \dots, b_n)$ be compatible sequences of colors, and suppose $\blam = (\lambda_1, \dots, \lambda_m) \vdash a_k$ is a \textcolor{revisions}{composition} for some distinguished $1 \leq k \leq n$. Then for each colored braid $\beta_{\aa} \in \BGpd{\aa}{\bb}$, we denote by $\beta_{\aa}^{\blam} \in \BGpd{\aa^{\blam}}{\bb^{\blam}}$ the colored braid obtained from $\beta_{\aa}$ by replacing the distinguished $a_k$-labeled strand of $\beta$ with a (blackboard framed) cable of $m$ parallel strands colored by the tuple $(\lambda_1, \dots, \lambda_m)$. We will often refer to this collection of parallel strands as a \textit{$\blam$-colored cable}.
\end{defn}

\begin{example}
    Let $\aa = (3, 5, 2)$, $\bb = (5, 3, 2)$, $\cc = (3, 2, 5)$, and consider the \textcolor{revisions}{composition} $\blam = (3, 1, 1) \vdash 5$. Then the cabled Artin generators $(\sigma_i)_{\aa}^{\blam}$ are as follows, with cabled strands drawn in blue:

    \begin{gather*}
    \BGpd{\aa^{\blam}}{\bb^{\blam}} \ni (\sigma_1)_{\aa}^{\blam} := 
    \begin{tikzpicture}[anchorbase, tinynodes]
    \draw[webs,color=blue] (0,1) node[above]{$3$} to[out=270,in=90] (.5,0) node[below]{$3$};
    \draw[webs,color=blue] (.5,1) node[above]{$1$} to[out=270,in=90] (1,0) node[below]{$1$};
    \draw[webs,color=blue] (1,1) node[above]{$1$} to[out=270,in=90] (1.5,0) node[below]{$1$};
    \draw[line width=5pt, color=white] (0,0) to[out=90,in=270] (1.5,1);
    \draw[webs] (0,0) node[below]{$3$} to[out=90,in=270] (1.5,1) node[above]{$3$};
    \draw[webs] (2,0) node[below]{$2$} to (2,1) node[above]{$2$};
    \end{tikzpicture}
    ; \quad
    \BGpd{\cc^{\blam}}{\aa^{\blam}} \ni (\sigma_2)_{\aa}^{\blam} :=
    \begin{tikzpicture}[anchorbase,tinynodes]
    \draw[webs] (0,0) node[below]{$3$} to (0,1) node[above]{$3$};
    \draw[webs] (.5,1) node[above]{$2$} to[out=270,in=90] (2,0) node[below]{$2$};
    \draw[line width=5pt,color=white] (1,1) to[out=270,in=90] (.5,0);
    \draw[line width=5pt,color=white] (1.5,1) to[out=270,in=90] (1,0);
    \draw[line width=5pt,color=white] (2,1) to[out=270,in=90] (1.5,0);
    \draw[webs,color=blue] (1,1) node[above]{$3$} to[out=270,in=90] (.5,0) node[below]{$3$};
    \draw[webs,color=blue] (1.5,1) node[above]{$1$} to[out=270,in=90] (1,0) node[below]{$1$};
    \draw[webs,color=blue] (2,1) node[above]{$1$} to[out=270,in=90] (1.5,0) node[below]{$1$};
    \end{tikzpicture}
    \end{gather*}
\end{example}

Next, we recall the complexes of singular Soergel bimodules associated to colored braids used in defining triply-graded link homology. We will not be interested in the specific form of the differential; see \cite{HRW21} for details in this respect.

\begin{defn} \label{def: Rickard_complex_gens}
    Fix $a, b \geq 0$. Then the \textit{2-strand Rickard complex} $C_{a, b}$ is the bounded complex of singular Soergel bimodules

    \begin{gather*}
    C_{a, b} := \dots \longrightarrow q^{-k}t^k
    \begin{tikzpicture}[anchorbase,scale=.5,tinynodes]
    \draw[webs] (0,0) node[below]{$a$} to[out=90,in=225] (.5,1.25);
    \draw[webs] (1.5,0) node[below]{$b$} to (1.5,.75);
    \draw[webs] (1.5,.75) to (.5,1.25);
    \draw[webs] (.5,1.25) to (.5,2);
    \draw[webs] (.5,2) to[out=135,in=270] (0,3.25) node[above]{$b$};
    \draw[webs] (.5,2) to (1.5,2.5);
    \draw[webs] (1.5,2.5) to (1.5,3.25) node[above]{$a$};
    \draw[webs] (1.5,.75) to[out=45,in=315] node[right]{$k$} (1.5,2.5);
    \end{tikzpicture}
    \longrightarrow q^{-k - 1}t^{k + 1}
    \begin{tikzpicture}[anchorbase,scale=.5,tinynodes]
    \draw[webs] (0,0) node[below]{$a$} to[out=90,in=225] (.5,1.25);
    \draw[webs] (1.5,0) node[below]{$b$} to (1.5,.75);
    \draw[webs] (1.5,.75) to (.5,1.25);
    \draw[webs] (.5,1.25) to (.5,2);
    \draw[webs] (.5,2) to[out=135,in=270] (0,3.25) node[above]{$b$};
    \draw[webs] (.5,2) to (1.5,2.5);
    \draw[webs] (1.5,2.5) to (1.5,3.25) node[above]{$a$};
    \draw[webs] (1.5,.75) to[out=45,in=315] node[right]{$k + 1$} (1.5,2.5);
    \end{tikzpicture}
    \longrightarrow \dots
    \end{gather*}

    Similarly, we define

    \begin{gather*}
    C_{a, b}^{\vee} := \dots \longrightarrow q^{k + 1}t^{-k - 1}
    \begin{tikzpicture}[anchorbase,scale=.5,tinynodes]
    \draw[webs] (0,0) node[below]{$a$} to[out=90,in=225] (.5,1.25);
    \draw[webs] (1.5,0) node[below]{$b$} to (1.5,.75);
    \draw[webs] (1.5,.75) to (.5,1.25);
    \draw[webs] (.5,1.25) to (.5,2);
    \draw[webs] (.5,2) to[out=135,in=270] (0,3.25) node[above]{$b$};
    \draw[webs] (.5,2) to (1.5,2.5);
    \draw[webs] (1.5,2.5) to (1.5,3.25) node[above]{$a$};
    \draw[webs] (1.5,.75) to[out=45,in=315] node[right]{$k + 1$} (1.5,2.5);
    \end{tikzpicture}
    \longrightarrow q^{k}t^{-k}
    \begin{tikzpicture}[anchorbase,scale=.5,tinynodes]
    \draw[webs] (0,0) node[below]{$a$} to[out=90,in=225] (.5,1.25);
    \draw[webs] (1.5,0) node[below]{$b$} to (1.5,.75);
    \draw[webs] (1.5,.75) to (.5,1.25);
    \draw[webs] (.5,1.25) to (.5,2);
    \draw[webs] (.5,2) to[out=135,in=270] (0,3.25) node[above]{$b$};
    \draw[webs] (.5,2) to (1.5,2.5);
    \draw[webs] (1.5,2.5) to (1.5,3.25) node[above]{$a$};
    \draw[webs] (1.5,.75) to[out=45,in=315] node[right]{$k$} (1.5,2.5);
    \end{tikzpicture}
    \longrightarrow \dots
    \end{gather*}

    In each case, $k$ ranges from $0$ to $\min(a, b)$.
\end{defn}

\begin{defn} \label{def: rickard_complex_words}
    For a sequence $\aa = (a_1, \dots, a_n)$ and an Artin generator $(\sigma_i)_{\aa}$ of the colored braid groupoid, we set
    
    \begin{align*}
    C(\sigma_i)_{\aa} & := \strand{(a_1, \dots, a_{i - 1})} \boxtimes C_{a_i, a_{i + 1}} \boxtimes \strand{(a_{i + 2}, \dots, a_n)} \\
    C(\sigma_i^{-1})_{\aa} & := \strand{(a_1, \dots, a_{i - 1})} \boxtimes C_{a_i, a_{i + 1}}^{\vee} \boxtimes \strand{(a_{i + 2}, \dots, a_n)}
    \end{align*}

    We extend $C$ to arbitrary colored braid words by horizontal composition. Given such a word $(\underline{\beta})_{\aa}$ representing a colored braid $\beta_{\aa}$, we call $C((\underline{\beta})_{\aa})$ the \textit{Rickard complex} associated to $\beta_{\aa}$.
\end{defn}

It is well-known (see e.g. Proposition 3.25 in \cite{HRW21}) that Rickard complexes satisfy the braid relations up to canonical homotopy equivalence. As such, we will often suppress the dependence of Definition \ref{def: rickard_complex_words} on a specific choice of braid word $(\underline{\beta})_{\aa}$ representing $\beta_{\aa}$, writing just $C(\beta_{\aa})$.

In fact we go even further: because we are only interested in the Rickard complexes of colored braids and not the braids themselves, we will often implicitly identify the two throughout. This means we treat colored braids\footnote{And colored braid words.} $\beta_{\aa}$ directly as chain complexes in $K^b(\SSBim)$ and extend the graphical calculus for $\SSBim$ described in Section \ref{sec: SSBim} to braided webs. As an example of this, we have the following well-known \textit{fork-sliding} homotopy equivalence \textcolor{revisions}{(see e.g. \cite{RT21}, Example 3.12)}:

\begin{prop} \label{prop: fork_slide}
    For each $a, b, c \geq 0$, consider the \textcolor{revisions}{composition} $\blam = (a, b) \vdash a + b$. Then there are homotopy equivalences

    \begin{align*}
    (_{(c, a + b)} M_{(c, \blam)}) \star (\sigma_1)_{(a + b, c)}^{\blam} & \simeq (\sigma_1)_{(a + b, c)} \star (_{(a + b, c)} M_{(\blam, c)}); \\
    (_{(a + b, c)}M_{(\blam, c)}) \star (\sigma_1)_{(c, a + b)}^{\blam} & \simeq (\sigma_1)_{(a + b, c)} \star (_{(c, a + b)} M_{(c, \blam)})
    \end{align*}
as well as reflections of these, arising from Gaussian elimination. Graphically, these take the form

    \begin{gather*}
    \begin{tikzpicture}[anchorbase,scale=.5,tinynodes]
    \draw[webs] (0,2) to[out=270,in=90] (2,0) node[below]{$c$};
    \draw[line width=5pt, color=white] (0,0) to[out=90,in=270] (1,2);
    \draw[webs] (0,0) node[below]{$a$} to[out=90,in=270] (1,2);
    \draw[line width=5pt, color=white] (1,0) to[out=90,in=270] (2,2);
    \draw[webs] (1,0) node[below]{$b$} to[out=90,in=270] (2,2);
    \draw[webs] (1,2) to[out=90,in=180] (1.5,2.5);
    \draw[webs] (2,2) to[out=90,in=0] (1.5,2.5);
    \draw[webs] (1.5,2.5) to (1.5,3) node[above]{$a + b$};
    \draw[webs] (0,2) to (0,3) node[above]{$c$};
    \end{tikzpicture}
    \ \ \simeq \ \ 
    \begin{tikzpicture}[anchorbase,scale=.5,tinynodes]
    \draw[webs] (0,0) node[below]{$a$} to[out=90,in=180] (.5,.5);
    \draw[webs] (1,0) node[below]{$b$} to[out=90,in=0] (.5,.5);
    \draw[webs] (2,0) node[below]{$c$} to (2,.5);
    \draw[webs] (0,2) to[out=270,in=90] (2,.5);
    \draw[line width=5pt, color=white] (.5,.5) to[out=90,in=270] (2,2);
    \draw[webs] (.5,.5) to[out=90,in=270] (2,2);
    \draw[webs] (0,2) to (0,3) node[above]{$c$};
    \draw[webs] (2,2) to (2,3) node[above]{$a + b$};
    \end{tikzpicture};
    \quad \quad
    \begin{tikzpicture}[anchorbase,scale=.5,tinynodes]
    \draw[webs] (.5,3) node[above]{$a + b$} to (.5,2.5);
    \draw[webs] (.5,2.5) to[out=180,in=90] (0,2);
    \draw[webs] (.5,2.5) to[out=0,in=90] (1,2);
    \draw[webs] (0,2) to[out=270,in=90] (1,0) node[below]{$a$};
    \draw[webs] (1,2) to[out=270,in=90] (2,0) node[below]{$b$};
    \draw[webs] (2,3) node[above]{$c$} to (2,2);
    \draw[line width=5pt, color=white] (2,2) to[out=270,in=90] (0,0);
    \draw[webs] (2,2) to[out=270,in=90] (0,0) node[below]{$c$};
    \end{tikzpicture}
    \ \ \simeq \ \ 
    \begin{tikzpicture}[anchorbase,scale=.5,tinynodes]
    \draw[webs] (0,3) node[above]{$a + b$} to (0,2);
    \draw[webs] (0,2) to[out=270,in=90] (1.5,.5);
    \draw[webs] (1.5,.5) to[out=180,in=90] (1,0) node[below]{$a$};
    \draw[webs] (1.5,.5) to[out=0,in=90] (2,0) node[below]{$b$};
    \draw[webs] (2,3) node[above]{$c$} to (2,2);
    \draw[line width=5pt, color=white] (2,2) to[out=270,in=90] (0,0.5);
    \draw[webs] (2,2) to[out=270,in=90] (0,0.5);
    \draw[webs] (0,0.5) to (0,0) node[below]{$c$};
    \end{tikzpicture}
    \end{gather*}
\end{prop}

\begin{rem}
\textcolor{revisions}{Recall from Corollary \ref{cor: gauss_elim_removal} that homotopy equivalences arising from Gaussian elimination are in particular strong deformation retractions satisfying the side conditions. We will use this property of fork-sliding repeatedly throughout the rest of this work.}
\end{rem}

\begin{rem}
    In the special case when $\aa = \bb = (1^n) \vdash n$, we will implicitly identify the endomorphism space $\BGpd{\aa}{\bb}$ with the usual (type A) braid group $\BG_n$. We also drop the subscript $\aa$ in our notation for braids, writing e.g. $\beta_{(1^n)} = \beta \in \BG_n$, and suppress the color labels on braid diagrams.
\end{rem}

\subsection{Deformed Rickard Complexes} \label{sec: curv_rick}

In this section we recall the curved complexes associated to braids in defining the deformed, colored, triply-graded link homology of \cite{HRW21}. Those authors give several models for this homology depending on different choices of generators for symmetric polynomials; we find it most convenient to work with the curvature modeled on a difference of \textit{elementary} symmetric polynomials. Throughout, we fix a \textcolor{revisions}{composition} $\aa = (a_1, \dots, a_n) \vdash N$ and a permutation $\sigma \in \SG_n$.

Given a complex $X \in \CS(\SSBim_{\aa}^{\sigma(\aa)})$ of graded bimodules, we may consider $\mathrm{Sym}^{\aa}_{\sigma}(\mathbb{X}, \mathbb{X}')$ as acting on $X$ on the left by $\mathrm{Sym}^{\sigma(\aa)}(\mathbb{X})$ and on the right by $\mathrm{Sym}^{\aa}(\mathbb{X}')$. Any morphism of bimodules must commute with both of these actions, so there is a natural inclusion 
\[
\mathrm{Sym}^{\aa}_{\sigma}(\mathbb{X}, \mathbb{X}') \hookrightarrow Z(\SSBim_{\aa}^{\sigma(\aa)}[\mathbb{Z}_t]).
\]

\begin{defn} \label{def: delta_e_curv}
    Recall the decomposition $\mathbb{X} = \mathbb{X}_1 \sqcup \dots \sqcup \mathbb{X}_n$ of an alphabet $\mathbb{X}$ of size $N$ as in Section \ref{sec: symm_poly}. For each pair of positive integers $1 \leq j \leq n$ and $1 \leq k \leq a_j$, let $u_{jk}$ be a formal deformation parameter of degree $\mathrm{deg}(u_{jk}) = q^{-2k}t^2$. Let $\mathbb{U}_{\aa}$ denote the alphabet of all such deformation parameters. Then the \textit{$\Delta$e-curvature} associated to the pair $(\aa, \sigma)$ is the polynomial

    \[
    F_{\sigma, u}^{\aa}(\mathbb{X}, \mathbb{X}') := \sum_{j = 1}^n \sum_{k = 1}^{a_j} \left(e_k(\mathbb{X}_{\sigma(j)}) - e_k(\mathbb{X}'_j) \right) \otimes u_{jk} \in Z(\SSBim_{\aa}^{\sigma(\aa)}[\mathbb{U}_{\aa}]).
    \]

    We reserve the notation $F_u^{\aa}(\mathbb{X}, \mathbb{X}')$ for the special case $\sigma = e$. We also suppress the alphabets $\mathbb{X}, \mathbb{X}'$, writing just $F_{\sigma, u}^{\aa}$, when they are clear from context.
\end{defn}

\begin{defn} \label{def: y-ification}
    We call a curved complex $\mathrm{tw}_{\Delta_X}(X \otimes_R R[\mathbb{U}_{\aa}]) \in \YS_{F_{\sigma, u}^{\aa}}(\Bim_N; R[\mathbb{U}_{\aa}])$ \textcolor{revisions}{for some Maurer-Cartan element $\Delta_X$} a \textit{$\Delta e$-deformation} of the underlying complex $X$.
\end{defn}

\begin{rem} \label{rem: y_ify_tensors}
Recall from Section \ref{sec: curved_complexes} that given curved complexes $(X \otimes_R S, \delta_X) \in \YS_{F_1}(\AS; S)$ and $(Y \otimes_R S, \delta_Y) \in \YS_{F_2}(\AS; S)$ over a monoidal category $\AS$, if $F_1 \otimes_{\AS} 1 + 1 \otimes_{\AS} F_2 \in Z(\AS)$, we can consider the tensor product $(X \otimes_R S) \otimes_{\AS} (Y \otimes_R S) \in \YS_{F_1 + F_2}(\AS; S)$. This observation applies to $\Delta e$-deformations. Indeed, suppose $\mathrm{tw}_{\Delta_X} (X \otimes_R R[\mathbb{U}_{\aa}]) \in \YS_{F^{\aa}_{\sigma, u}}(\mathrm{Bim}_N; R[\mathbb{U}_{\aa}])$, $\mathrm{tw}_{\Delta_Y} (Y \otimes_R R[\mathbb{U}_{\bb}]) \in \YS_{F^{\bb}_{\rho, u}}(\mathrm{Bim}_N; R[\mathbb{U}_{\bb}])$ are two $\Delta e$-deformations for $\bb = \sigma \cdot \aa$. Then we can express the relevant curvatures as

\[
F_{\sigma, u}^{\aa}(\mathbb{X}', \mathbb{X}'') = \sum_{j = 1}^n \sum_{k = 1}^{a_j} \left( e_k(\mathbb{X}'_{\sigma(j)}) - e_k(\mathbb{X}''_j) \right) \otimes u_{jk} \in \mathrm{Sym}^{\aa}_{\sigma}(\mathbb{X}', \mathbb{X}'') \otimes_R R[\mathbb{U}_{\aa}];
\]

\[
F_{\rho, u}^{\bb}(\mathbb{X}, \mathbb{X}') = \sum_{j = 1}^n \sum_{k = 1}^{a_{\sigma(j)}} \left( e_k(\mathbb{X}_{\rho \sigma(j)}) - e_k(\mathbb{X}'_{\sigma(j)}) \right) \otimes u_{\sigma(j)k} \in \mathrm{Sym}^{\bb}_{\rho}(\mathbb{X}, \mathbb{X}') \otimes_R R[\mathbb{U}_{\bb}].
\]

Recall that the horizontal composition $\star$ on $\Bim_N$ agrees with the usual tensor product $\otimes_{R[\mathbb{X}']}$. Upon also identifying the alphabets $\mathbb{U}_{\aa}$ and $\mathbb{U}_{\bb}$, one can easily verify that the two curvatures above add to

\[
F^{\aa}_{\rho \sigma, u}(\mathbb{X}, \mathbb{X}'') = \sum_{j = 1}^n \sum_{k = 1}^{a_j} \left( e_k(\mathbb{X}_{\rho \sigma(j)}) - e_k(\mathbb{X}''_j) \right) \otimes u_{jk} \in \mathrm{Sym}^{\aa}_{\rho \sigma}(\mathbb{X}, \mathbb{X}'') \otimes_R R[\mathbb{U}_{\aa}]
\]

As a consequence, we can form the horizontal composition

\[
\mathrm{tw}_{\Delta_X} (X \otimes_R R[\mathbb{U}_{\aa}]) \star \mathrm{tw}_{\Delta_Y} (Y \otimes_R R[\mathbb{U}_{\bb}]) :=  \mathrm{tw}_{\Delta_Y \star \mathrm{id}_X + \mathrm{id}_Y \star \Delta_X}((Y \star X) \otimes_R R[\mathbb{U}_{\aa}]) \in \YS_{F^{\aa}_{\rho \sigma, u}}(\mathrm{Bim}_N; R[\mathbb{U}_{\aa}]).
\]

\end{rem}

The following is Lemma 4.20 in \cite{HRW21}:

\begin{prop} \label{prop: dot_slide}
    Fix a sequence $\aa = (a_1, \dots, a_n) \vdash N$ and a colored braid $\beta_{\aa} \in \BGpd{\aa}{\bb}$. Then for each $1 \leq j \leq n$ and each $1 \leq k \leq a_j$, there is a degree $q^{2k}t^{-1}$ endomorphism $\xi_{jk} \in \End^{-1}_{\CS(\SSBim)}(C(\beta_{\aa}))$ satisfying $[d, \xi_{jk}] = e_k(\mathbb{X}_{\beta(j)}) - e_k(\mathbb{X}'_j)$. Moreover, the family of all such morphisms $\xi_{jk}$ can be chosen to pairwise anti-commute and square to $0$.
\end{prop}

We refer to the endomorphisms $\xi_{jk}$ of Proposition \ref{prop: dot_slide} as \textit{dot-sliding homotopies}. Proposition \ref{prop: dot_slide} guarantees that they form an $F_{\sigma, u}^{\aa}$-deforming family on $C(\beta_{\aa})$ as in Definition \ref{def: strict_def_family} and therefore give rise to a strict deformation of $C(\beta_{\aa})$.

\begin{cor} \label{cor: curved_braids}
    Given a colored braid $\beta_{\aa} \in \mathfrak{Br}_n$, set 
    
    \[
    \Delta_{\beta_{\aa}} := \sum_{j = 1}^n \sum_{k = 1}^{a_j} \xi_{jk} \otimes u_{jk} \in \End_{\SSBim[\mathbb{U}_{\aa}]} \left( C(\beta_{\aa}) \otimes_R R[\mathbb{U}_{\aa}] \right).
    \]
    Then $C^u(\beta_{\aa}) := \mathrm{tw}_{\Delta_{\beta_{\aa}}}\left( C(\beta_{\aa}) \otimes_R R[\mathbb{U}_{\aa}] \right)$ is a strict deformation of $C(\beta_{\aa})$.
\end{cor}

We refer to the curved complex $C^u(\beta_{\aa})$ of Corollary \ref{cor: curved_braids} as the \textit{$\Delta e$-deformed Rickard complex} for the colored braid $\beta_{\aa}$. Since Rickard complexes are bounded and invertible up to homotopy equivalence (with inverse $C(\beta_{\aa})^{\vee} = C(\beta^{-1}_{\beta(\aa)})$), \textcolor{revisions}{Corollary} \ref{cor: inv_uniq_yify} guarantees that $C^u(\beta_{\aa})$ is the unique $\Delta$e-deformation of $C(\beta_{\aa})$ up to homotopy equivalence of curved complexes.

\begin{rem} \label{rem: many_rings}
    In practice, all deformations we consider will be curved complexes of singular Soergel bimodules. However, given a complex $X \in \CS(\SSBim_{\aa}^{\sigma(\aa)})$, there are potentially plenty of rings of symmetric polynomials over which $X$ is a complex of bimodules. We apply Definition \ref{def: y-ification} throughout without restricting our attention only to the curvature $F_{\sigma, u}^{\aa}$.
    
    As a concrete example, suppose $\aa = \bb = (2, 3, 5, 6)$. Notice that $\aa = \bnu^{\blam}$ for the \textcolor{revisions}{composition}s $\bnu = (2, 8, 6) \vdash 16$, $\blam = (3, 5) \vdash 8$, so there is an inclusion $\mathrm{Sym}^{\bnu}(\mathbb{X}) \subset \mathrm{Sym}^{\aa}(\mathbb{X})$. Then $X$ is a perfectly good complex of $\mathrm{Sym}^{\bnu}(\mathbb{X})$-bimodules, and we are free to consider $\Delta$e-deformations of $X$ with curvature $F_u^{\bnu}$.

    Similarly, let $\aa$ be as above, and consider a colored braid $\beta_{\aa}$. We are free to forget the action of symmetric polynomials on all but the $2$-colored strand and consider $C(\beta_{\aa})$ as a $ \left( \mathrm{Sym}(\mathbb{X}_{\beta(1)}), \mathrm{Sym}(\mathbb{X}'_1) \right)$-bimodule. We could then consider a deformation of $C(\beta_{\aa})$ with $\Delta$e-curvature along \textit{only that strand}; this curvature has the form

    \[
    F_u^{(2)} = \left( e_1(\mathbb{X}_{\beta(1)}) - e_1(\mathbb{X}'_1) \right) \otimes u_{1, 1} + \left( e_2(\mathbb{X}_{\beta(1)}) - e_2(\mathbb{X}'_1) \right) \otimes u_{1, 2}.
    \]

    We find this freedom very useful throughout and will turn curvature on various strands on and off at will. When there is potential for confusion, we will indicate the particular curvature we consider in a given instance by adding a subscript indicating the ``active" alphabet of deformation parameters.
\end{rem}

\begin{example}
    Let $\beta = \sigma_1^2 \in Br_2$, and suppose $\aa = \bb = (a_{\ell}, a_r)$. Suppose we attach alphabets of deformation parameters $\mathbb{U}_{\ell}$ and $\mathbb{U}_{r}$, of size $a_{\ell}$ and $a_r$ to the colored braid $\beta_{\aa}$ as indicated below.

    
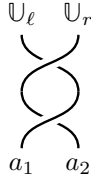
\begin{center}
    \begin{tikzpicture}[anchorbase,scale=.75]
        \draw[webs] (0,1) to[out=270,in=90] (1,0) node[below]{$a_2$};
        \draw[line width=5pt, color=white] (0,0) to[out=90,in=270] (1,1);
        \draw[webs] (0,0) node[below]{$a_1$} to[out=90,in=270] (1,1);
        \draw[webs] (1,1) to[out=90,in=270] (0,2) node[above]{$\mathbb{U}_{\ell}$};
        \draw[line width=5pt, color=white] (0,1) to[out=90,in=270] (1,2);
        \draw[webs] (0,1)  to[out=90,in=270] (1,2) node[above]{$\mathbb{U}_r$};
    \end{tikzpicture}
    \end{center}

    There are (at least) three deformations of $\beta_{\aa}$ we may consider:

    \begin{itemize}
        \item $\beta_{\aa}^{u_{\ell}}$ is a curved complex with $\Delta e$-curvature on the $a_{\ell}$-labeled strand of $\beta_{\aa}$. This curvature is

        \[
        F_{u_{\ell}}^{(a_{\ell})} = \sum_{i = 1}^{a_{\ell}} (e_i(\mathbb{X}_1) - e_i(\mathbb{X}'_1)) \otimes u_{\ell, i}.
        \]

        \item $\beta_{\aa}^{u_r}$ is a curved complex with $\Delta e$-curvature on the $a_r$-labeled strand of $\beta_{\aa}$. This curvature is

        \[
        F_{u_r}^{(a_r)} = \sum_{i = 1}^{a_r} (e_i(\mathbb{X}_2) - e_i(\mathbb{X}'_2)) \otimes u_{r, i}.
        \]

        \item $C^u(\beta_{\aa}) = \beta_{\aa}^{u_{\ell}, u_r}$ is a curved complex with $\Delta e$-curvature on both strands of $\beta_{\aa}$. This curvature is

        \[
        F_u^{\aa} = \sum_{i = 1}^{a_{\ell}} (e_i(\mathbb{X}_1) - e_i(\mathbb{X}'_1)) \otimes u_{\ell, i} + \sum_{i = 1}^{a_r} (e_i(\mathbb{X}_2) - e_i(\mathbb{X}'_2)) \otimes u_{r, i}.
        \]
    \end{itemize}
\end{example}

\subsection{Fork-Slides and Deformations} \label{sec: fork_slyde}

In Section \ref{sec: fray_functors}, we will often make use of a cabled version of the fork-sliding homotopy equivalence of Proposition \ref{prop: fork_slide}. We pause to establish some notation for the complexes involved.

\begin{defn} \label{def: cabled_forks}
	For each pair of positive integers $c, n \geq 0$ and each \textcolor{revisions}{composition} $\blam = (\lambda_1, \dots, \lambda_m) \vdash n$, set 
	
	\[
	MCC_{\blam, c}^n := \left( _{(c, n)}M_{(c, \blam)} \right) \star (\sigma_1)_{(n, c)}^{\blam}; \quad CM_{\blam, c}^n := (\sigma_1)_{(n, c)} \star \left( _{(n, c)}M_{(\blam, c)} \right).
	\]
	Graphically, we have

    \begin{gather*}
    MCC_{\blam, c}^n := 
    \begin{tikzpicture}[anchorbase,scale=.5,tinynodes]
    \draw[webs] (0,2) to[out=270,in=90] (2,0) node[below]{$c$};
    \draw[line width=5pt, color=white] (0,0) to[out=90,in=270] (1,2);
    \draw[webs] (0,0) to[out=90,in=270] (1,2);
    \draw[line width=5pt, color=white] (1,0) to[out=90,in=270] (2,2);
    \draw[webs] (1,0) to[out=90,in=270] (2,2);
    \node at (.5,.2) {$\dots$};
    \node at (1.5,1.8) {$\dots$};
    \draw [decorate,decoration = {brace,mirror}] (0,-.2) -- (1,-.2)
    node[pos=.5,below,yshift=-2pt]{$\blam$};
    \draw[webs] (1,2) to[out=90,in=180] (1.5,2.5);
    \draw[webs] (2,2) to[out=90,in=0] (1.5,2.5);
    \draw[webs] (1.5,2.5) to (1.5,3) node[above,yshift=-2pt]{$n$};
    \draw[webs] (0,2) to (0,3);
    \end{tikzpicture}
    ; \quad CM_{\blam, c}^n :=
    \begin{tikzpicture}[anchorbase,scale=.5,tinynodes]
    \draw[webs] (0,0) to[out=90,in=180] (.5,.5);
    \draw[webs] (1,0) to[out=90,in=0] (.5,.5);
    \node at (.5,.2) {$\dots$};
    \draw [decorate,decoration = {brace,mirror}] (0,-.2) -- (1,-.2)
    node[pos=.5,below,yshift=-2pt]{$\blam$};
    \draw[webs] (2,0) node[below]{$c$} to (2,.5);
    \draw[webs] (0,2) to[out=270,in=90] (2,.5);
    \draw[line width=5pt, color=white] (.5,.5) to[out=90,in=270] (2,2);
    \draw[webs] (.5,.5) to[out=90,in=270] (2,2);
    \draw[webs] (0,2) to (0,3);
    \draw[webs] (2,2) to (2,3) node[above,yshift=-2pt]{$n$};
    \end{tikzpicture}
    \end{gather*}
\end{defn}

\begin{prop} \label{prop: cabled_fork_slides}
Let $c, n, \blam$ be as in Definition \ref{def: cabled_forks}. Then there is a strong deformation retraction from $MCC_{\blam, c}^n$ to $CM_{\blam, c}^n$ of the form
\begin{gather*}
    \begin{tikzpicture}[anchorbase,scale=.5,tinynodes]
    \draw[webs] (0,2) to[out=270,in=90] (2,0) node[below]{$c$};
    \draw[line width=5pt, color=white] (0,0) to[out=90,in=270] (1,2);
    \draw[webs] (0,0) to[out=90,in=270] (1,2);
    \draw[line width=5pt, color=white] (1,0) to[out=90,in=270] (2,2);
    \draw[webs] (1,0) to[out=90,in=270] (2,2);
    \node at (.5,.2) {$\dots$};
    \node at (1.5,1.8) {$\dots$};
    \draw [decorate,decoration = {brace,mirror}] (0,-.2) -- (1,-.2)
    node[pos=.5,below,yshift=-2pt]{$\blam$};
    \draw[webs] (1,2) to[out=90,in=180] (1.5,2.5);
    \draw[webs] (2,2) to[out=90,in=0] (1.5,2.5);
    \draw[webs] (1.5,2.5) to (1.5,3) node[above,yshift=-2pt]{$n$};
    \draw[webs] (0,2) to (0,3);
    \draw (2.5,1.7) to (4,1.7) node[above,yshift=-2pt]{$\nu$};
    \draw[->] (4,1.7) to (5.5,1.7);
    \draw (5.5,1.3) to (4,1.3) node[below]{$\mu$};
    \draw[->] (4,1.3) to (2.5,1.3);
    \draw (-.75,2.25) to[out=180,in=90] (-1.5,1.5) node[left]{$h$};
    \draw[->] (-1.5,1.5) to[out=270,in=180] (-.75,.75);
    \end{tikzpicture}
    \begin{tikzpicture}[anchorbase,scale=.5,tinynodes]
    \draw[webs] (0,0) to[out=90,in=180] (.5,.5);
    \draw[webs] (1,0) to[out=90,in=0] (.5,.5);
    \node at (.5,.2) {$\dots$};
    \draw [decorate,decoration = {brace,mirror}] (0,-.2) -- (1,-.2)
    node[pos=.5,below,yshift=-2pt]{$\blam$};
    \draw[webs] (2,0) node[below]{$c$} to (2,.5);
    \draw[webs] (0,2) to[out=270,in=90] (2,.5);
    \draw[line width=5pt, color=white] (.5,.5) to[out=90,in=270] (2,2);
    \draw[webs] (.5,.5) to[out=90,in=270] (2,2);
    \draw[webs] (0,2) to (0,3);
    \draw[webs] (2,2) to (2,3) node[above,yshift=-2pt]{$n$};
    \end{tikzpicture}
\end{gather*}
as well as reflections of this, satisfying the side conditions.
\end{prop}

\begin{proof}
Repeated applications of Proposition \ref{prop: fork_slide} along with the associativity relations of Proposition \ref{prop: web_rels}.
\end{proof}

We wish to lift the statement of Proposition \ref{prop: cabled_fork_slides} to a $\Delta e$-deformed version. In keeping with the philosophy of Remark \ref{rem: many_rings}, there are many possible lifts we may consider, and we will find it useful to establish some notation for the various alphabets and dot-sliding homotopies involved before proceeding.

\begin{conv} \label{conv: cabled_fork_alphs}
Fix a choice of $c, n, \blam$ as in Definition \ref{def: cabled_forks} for the remainder of this Section. We label alphabets acting in various locations on $MCC_{\blam, c}^n$ and $CM_{\blam, c}^n$ as follows:

\begin{gather} \label{fig: cabled_fork_alphs}
MCC_{\blam, c}^n =
\begin{tikzpicture}[scale=.5,anchorbase,tinynodes]
    \draw[webs] (-1,2) to[out=270,in=90] (2,0) node[below,yshift=-2pt]{\large $\mathbb{X}'_{\ell}$};
    \draw[line width=5pt, color=white] (-1,0) to[out=90,in=270] (.75,1.75);
    \draw[webs] (-1,0) node[below,yshift=-2pt]{\large $\mathbb{X}'_1$}
    to[out=90,in=270] (.75,1.75) node[left]{\large $\mathbb{X}_1$};
    \draw[line width=5pt, color=white] (.5,0) to[out=90,in=270] (2.25,1.75);
    \draw[webs] (.5,0) node[below,yshift=-2pt]{\large $\mathbb{X}'_m$}
    to[out=90,in=270] (2.25,1.75) node[right]{\large $\mathbb{X}_m$};
    \node at (-.25,.3) {$\dots$};
    \node at (1.5,1.8) {$\dots$};
    \draw[webs] (.75,1.75) to[out=90,in=180] (1.5,2.5);
    \draw[webs] (2.25,1.75) to[out=90,in=0] (1.5,2.5);
    \draw[webs] (1.5,2.5) to (1.5,3) node[above]{\large $\mathbb{X}_n$};
    \draw[webs] (-1,2) to (-1,3) node[above]{\large $\mathbb{X}_{\ell}$};
\end{tikzpicture}
; \quad CM_{\blam, c}^n = 
\begin{tikzpicture}[scale=.5,anchorbase,tinynodes]
    \draw[webs] (-1,4) node[above]{\large $\mathbb{X}_{\ell}$} to[out=270,in=90] (2,1.5);
    \draw[webs] (2,1.5) to (2,0) node[below,yshift=-2pt]{\large $\mathbb{X}'_{\ell}$};
    \draw[line width=5pt, color=white] (1.5,4) to[out=270,in=90] (0,2);
    \draw[webs] (1.5,4) node[above]{\large $\mathbb{X}_n$} to[out=270,in=90] (0,2);
    \draw[webs] (0,2) to (0,1.5) node[left,yshift=5pt]{\large $\mathbb{X}'_n$};
    \draw[webs] (0,1.5) to (0,1);
    \draw[webs] (0,1) to[out=180,in=90] (-.75,.25);
    \draw[webs] (0,1) to[out=0,in=90] (.75,.25);
    \draw[webs] (-.75,.25) to (-.75,0) node[below,yshift=-2pt]{\large $\mathbb{X}'_1$};
    \draw[webs] (.75,.25) to (.75,0) node[below,yshift=-2pt]{\large $\mathbb{X}'_m$};
    \node at (0,.3) {$\dots$};
\end{tikzpicture}
\end{gather}
\end{conv}

\begin{conv} \label{conv: cabled_fork_dot_slides}
We give the following names to the dot-sliding homotopies on $MCC_{\blam, c}^n$:
\begin{itemize}
	\item $\xi_{jk} \in \End^{-1}_{\CS(\SSBim)}(MCC_{\blam, c}^n)$ satisfies $[d, \xi_{jk}] = e_k(\mathbb{X}_j) - e_k(\mathbb{X}'_j)$.
	
	\item $\xi_{\ell k} \in \End^{-1}_{\CS(\SSBim)}(MCC_{\blam, c}^n)$ satisfies $[d, \xi_{\ell k}] = e_k(\mathbb{X}_{\ell}) - e_k(\mathbb{X}'_{\ell})$.
\end{itemize}

Similarly, we give the following names to the dot-sliding homotopies on $CM_{\blam, c}^n$:
\begin{itemize}
	\item $\tilde{\zeta}_i \in \End^{-1}_{\CS(\SSBim)}(CM_{\blam, c}^n)$ satisfies $[d, \tilde{\zeta}_i] = e_i(\mathbb{X}_n) - e_i(\mathbb{X}'_n)$.
	
	\item $\tilde{\xi}_{\ell k} \in \End^{-1}_{\CS(\SSBim)}(CM_{\blam, c}^n)$ satisfies $[d, \tilde{\xi}_{\ell k}] = e_k(\mathbb{X}_{\ell}) - e_k(\mathbb{X}'_{\ell})$.
\end{itemize}

\end{conv}

The families $\{\tilde{\zeta}_i\}_{i = 1}^n$ and $\{\tilde{\xi}_{\ell k}\}_{k = 1}^c$ square to $0$ and pairwise anticommute, and hence each of these families and their union are deforming families for $CM_{\blam, c}^n$ as in Definition \ref{def: strict_def_family}. We denote the deformations arising from these families as follows.

\begin{defn} \label{def: cm_deformed}

We denote by 

\[
\left( CM_{\blam, c}^n \right)^{u_{\ell}} \in \YS_{F_{u_{\ell}}^{(c)}} \left( \SSBim_{(\blam, c)}^{(c, n)}; R[\mathbb{U}_{\ell}] \right); \quad \left( CM_{\blam, c}^n \right)^{u_n} \in \YS_{F_{u_n}^{(n)}} \left( \SSBim_{(\blam, c)}^{(c, n)}; R[\mathbb{U}_n] \right)
\]
the strict deformations of $CM_{\blam, c}^n$ arising from the families $\{\tilde{\xi}_{\ell k}\}$ and $\{\tilde{\zeta}_i\}$, respectively. We denote the corresponding alphabets of deformation parameters by $\mathbb{U}_{\ell}$ and $\mathbb{U}_n$ and the connections by $\tilde{\Delta}_{\ell}$ and $\tilde{\Delta}_n$. Similarly, we denote by 

\[
\left( CM_{\blam, c}^n \right)^{u_{\ell}, u_n} \in \YS_{F_{u_{\ell}}^{(c)} + F_{u_n}^{(n)}} \left( \SSBim_{(\blam, c)}^{(c, n)}; R[\mathbb{U}_{\ell}, \mathbb{U}_n] \right)
\]
the strict deformation of $CM_{\blam, c}^n$ arising from the union of the families $\{\tilde{\xi}_{\ell k}\}$ and $\{\tilde{\zeta}_i\}$, with connection $\tilde{\Delta}_{\ell} + \tilde{\Delta}_n$.
\end{defn}

There are corresponding deformations on $MCC_{\blam, c}^n$. The deformation with $F_{u_{\ell}}^{(c)}$ curvature arises from the deforming family $\{\xi_{\ell k}\}$; we obtain a deformation with $F_{u_n}^{(n)}$ curvature via the following family.

\begin{lem} \label{lem: bulk_curv_MCC}
Let $a_{ijk}^{\blam}(\mathbb{X}_n, \mathbb{X}'_n) \in \mathrm{Sym}^{\blam}(\mathbb{X}_n, \mathbb{X}'_n)$ be as in Lemma \ref{lem: a_ijk_exist}, and for each $1 \leq i \leq n$, set

\[
\zeta_i := \sum_{j = 1}^m \sum_{k = 1}^{\lambda_j} a_{ijk}^{\blam}(\mathbb{X}_n, \mathbb{X}'_n) \xi_{jk} \in \End^1_{\CS(\SSBim)} \left( MCC_{\blam, c}^n \right).
\]
Then $\{\zeta_i\}$ is an $F_{u_n}^{(n)}$-deforming family on $MCC_{\blam, c}^n$.
\end{lem}

\begin{proof}
That $\zeta_i$ pairwise anti-commute follows from the same property of $\xi_{jk}$. To see that $[d, \zeta_i] = e_i(\mathbb{X}_n) - e_i(\mathbb{X}_n')$ is a straightforward computation:

\[
[d, \zeta_i] = \sum_{j = 1}^m \sum_{k = 1}^{\lambda_j} a_{ijk}^{\blam} [d, \xi_{jk}] = \sum_{j = 1}^m \sum_{k = 1}^{\lambda_j} a_{ijk}^{\blam} (e_k(\mathbb{X}_j) - e_k(\mathbb{X}'_j)) = e_i(\mathbb{X}_n) - e_i(\mathbb{X}_n').
\]
\end{proof}

\begin{defn}

We denote by 

\[
\left( MCC_{\blam, c}^n \right)^{u_{\ell}} \in \YS_{F_{u_{\ell}}^{(c)}} \left( \SSBim_{(\blam, c)}^{(c, n)}; R[\mathbb{U}_{\ell}] \right); \quad \left( MCC_{\blam, c}^n \right)^{u_n} \in \YS_{F_{u_n}^{(n)}} \left( \SSBim_{(\blam, c)}^{(c, n)}; R[\mathbb{U}_n] \right)
\]
the strict deformations of $MCC_{\blam, c}^n$ arising from the families $\{\xi_{\ell k}\}$ and $\{\zeta_i\}$, respectively. We denote the corresponding alphabets of deformation parameters again by $\mathbb{U}_{\ell}$ and $\mathbb{U}_n$ and the connections by $\Delta_{\ell}$ and $\Delta_n$. Similarly, we denote by

\[
\left( MCC_{\blam, c}^n \right)^{u_{\ell}, u_n} \in \YS_{F_{u_{\ell}}^{(c)} + F_{u_n}^{(n)}} \left( \SSBim_{(\blam, c)}^{(c, n)}; R[\mathbb{U}_{\ell}, \mathbb{U}_n] \right)
\]
the strict deformation of $MCC_{\blam, c}^n$ arising from the union of the families $\{\xi_{\ell k}\}$ and $\{\zeta_i\}$, with connection $\Delta_{\ell} + \Delta_r$.  
\end{defn}

At this point, \textcolor{revisions}{Proposition \ref{prop: lifting_maps} and Corollary \ref{cor: inv_uniq_yify}} suffice to show $\left( MCC_{\blam, c}^n \right)^{u_{\ell}, u_n} \simeq \left( CM_{\blam, c}^n \right)^{u_{\ell}, u_n}$ for \textit{some} homotopy equivalence, but it would be nice to use exactly the morphisms of Proposition \ref{prop: cabled_fork_slides}. To establish this stronger result, we first modify the connection on $\left( MCC_{\blam, c}^n \right)^{u_{\ell}, u_n}$ using Proposition \ref{prop: sdr_deformation_lifting}.

\begin{prop} \label{prop: conn_change_mcc}
Let $\{\xi'_{\ell k}\}$, $\{\zeta'_i\}$ be the endomorphisms of $MCC_{\blam, c}^n$ obtained by lifting the families $\{ \tilde{\xi}_{\ell k}\}$, $\{\tilde{\zeta}_i\}$ along the strong deformation retraction of Proposition \ref{prop: cabled_fork_slides}. Denote by $\left( MCC_{\blam, c}^n \right)^{u_{\ell}'}$, $\left( MCC_{\blam, c}^n \right)^{u_n'}$, and $\left( MCC_{\blam, c}^n \right)^{u_{\ell}', u_n'}$ the strict deformations of $MCC_{\blam, c}^n$ arising from $\{\xi'_{\ell k}\}$, $\{\zeta'_i\}$, and their union, respectively.

Then there are homotopy equivalences of curved complexes 

\[
\left( MCC_{\blam, c}^n \right)^{u_{\ell}', u_n'} \simeq \left( MCC_{\blam, c}^n \right)^{u_{\ell}, u_n}; \quad \left( MCC_{\blam, c}^n \right)^{u_{\ell}} \simeq \left( MCC_{\blam, c}^n \right)^{u_{\ell}'}; \quad \left( MCC_{\blam, c}^n \right)^{u_n} \simeq \left( MCC_{\blam, c}^n \right)^{u_n'}.
\]
\end{prop}

\begin{proof}
That $\left( MCC_{\blam, c}^n \right)^{u_{\ell}', u_n'}$, $\left( MCC_{\blam, c}^n \right)^{u_{\ell}'}$, and $\left( MCC_{\blam, c}^n \right)^{u_n'}$ are curved complexes with the appropriate curvature follows immediately from Lemma \ref{lem: sdr_lifting}. \textcolor{revisions}{The homotopy equivalences then follow from right quasi-invertibility of $MCC_{\blam, c}^n$ (with right quasi-inverse given by a cabled crossing) and Corollary \ref{cor: inv_uniq_yify}.}
\end{proof}

As an immediate application of Proposition \ref{prop: sdr_deformation_lifting}, we obtain the following:

\begin{prop} \label{prop: deformed_forkslide}
The strong deformation retraction of Proposition \ref{prop: cabled_fork_slides} lifts without modification to strong deformation retractions from $\left( MCC_{\blam, c}^n \right)^{u_{\ell}', u_n'}$ to $\left( CM_{\blam, c}^n \right)^{u_{\ell}, u_n}$, from $\left( MCC_{\blam, c}^n \right)^{u_{\ell}'}$ to $\left( CM_{\blam, c}^n \right)^{u_{\ell}}$, and from $\left( MCC_{\blam, c}^n \right)^{u_n'}$ to $\left( MCC_{\blam, c}^n \right)^{u_n}$.
\end{prop}

\section{Fray Functors} \label{sec: fray_functors}

In this section, we investigate the behavior of complexes of singular Soergel bimodules under extension of scalars to rings of partially symmetric polynomials, as introduced in Section \ref{sec: symm_poly}, as well as several deformations of these functors. Each of these functors is designed to mimic a step in constructing the finite and infinite Abel--Hogancamp projectors studied in \cite{AH17} as well as their deformations studied in \cite{Con23}. The full implications of this relationship are explored at length in Section \ref{sec: cat_idem} (see in particular Theorems \ref{thm: proj_agreement} and \ref{thm: yproj_agreement}); see also Section \ref{sec: proof_strat} in the Introduction for a further discussion of this point.

\begin{conv} \label{conv: fray_vals}
	Unless otherwise noted, we fix positive integers
	$N \geq n \geq 1$ and a \textcolor{revisions}{composition} $\blam = (\lambda_1, \dots, \lambda_m) \vdash n$ throughout this section.
	We also fix a domain object $\aa \vdash N$ and codomain object $\bb \vdash N$
	of $\SSBim_N$ satisfying $a_r = b_s = n$. In particular, when considering the refinement $\aa^{\blam}$ (resp. $\bb^{\blam}$), we always replace the element $a_r$ (resp. $b_s$); see Remark \ref{rem: fray_ambig}.
\end{conv}

\subsection{Reduced Fray Functor}

\begin{defn} \label{def: rfray}
    Let $\rfray{n}{\blam}$ denote the functor

    \[
    \rfray{n}{\blam} = \left( _{\bb^{\blam}}S_{\bb} \right) \star - \star \left( _{\aa}M_{\aa^{\blam}} \right)
    \colon \SSBim_{\aa}^{\bb} \to \SSBim_{\aa^{\blam}}^{\bb^{\blam}}.
    \]

Pictorially, given a bimodule $M \in \SSBim_{\aa}^{\bb}$, $\rfray{n}{\blam}(M)$
is the same bimodule with the designated $n$-labeled strands $a_r, b_s$ ``frayed"
to a $\blam$-labeled (blackboard framed) cable as indicated below:

\begin{gather} \label{fig: red_fray}
    \rfray{n}{\blam} \colon \  
    \begin{tikzpicture}[scale=.5,anchorbase]
    \draw (0,1) rectangle++(4,2);
    \node at (2,2) {$M$};
    \draw[webs] (2,1) to (2,0) node[below]{$n$};
    \draw[webs] (2,3) to (2,4) node[above]{$n$};
    \node at (1, .5) {$\dots$};
    \node at (3, .5) {$\dots$};
    \node at (1, 3.5) {$\dots$};
    \node at (3, 3.5) {$\dots$};
    \node at (5.5,2) {\huge $\longmapsto$};
    \draw (7,1) rectangle++(4,2);
    \node at (9,2) {$M$};
    \draw[webs] (9,1) to (9,0);
    \draw[webs] (9,0) to[out=180,in=90] (7,-1) node[below]{$\lambda_1$};
    \draw[webs] (9,0) to[out=180,in=90] (8,-1) node[below]{$\lambda_2$};
    \node at (9,-.5) {$\dots$};
    \draw[webs] (9,0) to[out=0,in=90] (10,-1) node[below,xshift=-2pt]{$\lambda_{m - 1}$};
    \draw[webs] (9,0) to[out=0,in=90] (11,-1) node[below,xshift=2pt]{$\lambda_m$};
    \draw[webs] (9,3) to (9,4);
    \node at (8,.5) {$\dots$};
    \node at (10,.5) {$\dots$};
    \draw[webs] (9,4) to[out=180,in=270] (7,5) node[above]{$\lambda_1$};
    \draw[webs] (9,4) to[out=180,in=270] (8,5) node[above]{$\lambda_2$};
    \draw[webs] (9,4) to[out=0,in=270] (10,5) node[above,xshift=-2pt]{$\lambda_{m - 1}$};
    \draw[webs] (9,4) to[out=0,in=270] (11,5) node[above,xshift=2pt]{$\lambda_m$};
    \node at (9,4.5) {$\dots$};
    \node at (8,3.5) {$\dots$};
    \node at (10,3.5) {$\dots$};
    \end{tikzpicture}
\end{gather}

We refer to the $n$-labeled strands corresponding to $a_r, b_s$ (in either the domain or codomain object) in Figure \eqref{fig: red_fray} as \textit{bulk} strands and the $\lambda_j$-labeled strands in the $\blam$-labeled cables as \textit{separated} strands. We refer to all other strands corresponding to elements $a_i \in \aa$, $i \neq r$ or $b_j \in \bb$, $j \neq s$ as \textit{passive} strands; we suppress these strands in Figure \eqref{fig: red_fray} for visual clarity.

Since $\rfray{n}{\blam}$ is an additive functor, it immediately extends to functors on the category of chain complexes $\CS \left( \SSBim_{\aa}^{\bb} \right)$ and the homotopy category $K \left( \SSBim_{\aa}^{\bb} \right)$. We denote each of these functors also by $\rfray{n}{\blam}$.
\end{defn}

\begin{rem}
    In the special case that $\blam = (n)$, $\rfray{n}{\blam}$ is just the identity functor.
    Despite this, we still formally distinguish between $C$ and $\rfray{n}{\blam}(C)$,
    referring to the distinguished $n$-labeled strand in $C$ as a bulk strand and the distinguished
    $n$-labeled strands in $\rfray{n}{\blam}(C)$ as both bulk and separated strands.
\end{rem}

\begin{example} \label{ex: rfray_merge_split}
Suppose $N = n$, $\aa = \bb = (n)$, and $\blam = (1^n) := (1, \dots, 1) \vdash n$. 
Then $\rfray{n}{\blam}(\strand{(n)}) = W_{(1^n)}$ is the full merge-split bimodule in $\SBim_n$.
This bimodule is often denoted $B_{w_0}$ in the literature,
as its class in the Grothendieck group $K^0(\SBim_n)$ maps
to the Kazhdan-Lusztig basis element $b_{w_0} \in H_n$ corresponding to the longest word $w_0 \in \SG_n$ under the Soergel isomorphism.
\end{example}

\begin{example}
    Given any $\blam \vdash n$, we have $\rfray{n}{\blam}(\strand{(n)}) = W_{\blam} \in \SSBim_{\blam}^{\blam}$. More generally, for any $\aa, \bb \vdash N$ satisfying $a_r = b_s = n$ for some $r, s$ and any $\blam \vdash n$, we have $\rfray{n}{\blam} \left(W^{\bb}_{\aa} \right) = W^{\bb^{\blam}}_{\aa^{\blam}}$.
\end{example}

\begin{prop} \label{prop: rfray_exact}
    $\rfray{n}{\blam}$ is an exact functor.
\end{prop}

\begin{proof}
    It is well known that $\mathrm{Sym}^{\bb^{\blam}}(\mathbb{X})$ is free, and therefore flat, as a left or right $\mathrm{Sym}^{\bb}(\mathbb{X})$-module via the inclusion of the latter into the former (ditto for $\aa$ and $\aa^{\blam}$); see e.g. \cite{ESW14}. Up to a quantum shift, $\rfray{n}{\blam}$ is \textcolor{revisions}{a composition of two extensions of scalars along inclusions of this form (one each for the left and right actions), each of which is exact by the above flatness, and compositions of exact functors are exact.}
\end{proof}

In Section \ref{sec: link_hom}, we will use the various fray functors considered in this section to relate the Rickard complexes assigned to colored braids to complexes obtained from these braids by cabling along the \textcolor{revisions}{composition} $\blam \vdash n$. Rickard complexes are \textit{local} in the sense that they are assigned to individual crossings, while cabling is a \textit{global} operation performed on an entire strand of a braid. To establish \textcolor{revisions}{the desired relationships requires} that we investigate the interaction between these functors and horizontal composition with cabled braids. We establish some useful notation for this investigation below.

\begin{conv}
When considering complexes of the form $\rfray{n}{\blam}(C) \star (_{\aa} \beta)^{\blam}$ or $\rfray{n}{\blam}(C \star _{\aa} \beta)$ for some colored braid $_{\aa} \beta$, we will be interested in the actions of (symmetric polynomials in) a variety of alphabets. We will always label the top-most alphabets with no apostrophes, the alphabets over which we take a tensor product in defining $\star$ with a single apostrophe, and the bottom-most alphabets with two apostrophes.

In labeling bulk strands, separated strands, and passive strands involved in crossings below $C$, we adhere to the subscript conventions established in Convention \ref{conv: cabled_fork_alphs}. Finally, we denote the total alphabet acting on the codomain by $\mathbb{X}$ and on the domain by $\mathbb{X}''$.

We summarize these conventions in the diagram below:

\begin{gather} \label{fig: x_alphs}
\rfray{n}{\blam}(C) \star (_{\aa} \sigma_r)^{\blam} =
\begin{tikzpicture}[scale=.5,anchorbase,tinynodes]
    \draw[webs] (-1,2) node[left]{\large $\mathbb{X}'_{\ell}$} to[out=270,in=90] (2,0) node[below,yshift=-2pt]{\large $\mathbb{X}''_{\ell}$};
    \draw[line width=5pt, color=white] (-1,0) to[out=90,in=270] (.75,1.75);
    \draw[webs] (-1,0) node[below,yshift=-2pt]{\large $\mathbb{X}''_1$}
    to[out=90,in=270] (.75,1.75) node[left]{\large $\mathbb{X}'_1$};
    \draw[line width=5pt, color=white] (.5,0) to[out=90,in=270] (2.25,1.75);
    \draw[webs] (.5,0) node[below,yshift=-2pt]{\large $\mathbb{X}''_m$}
    to[out=90,in=270] (2.25,1.75) node[right]{\large $\mathbb{X}'_m$};
    \node at (-.25,.3) {$\dots$};
    \node at (1.5,1.8) {$\dots$};
    \draw[webs] (.75,1.75) to[out=90,in=180] (1.5,2.5);
    \draw[webs] (2.25,1.75) to[out=90,in=0] (1.5,2.5);
    \draw[webs] (1.5,2.5) to (1.5,3.25) node[right]{\large $\mathbb{X}'_n$};
    \draw[webs] (1.5,3.25) to (1.5,4);
    \draw[webs] (-1,2) to (-1,4);
    \draw (-1.2,4) rectangle++(3.4,1.5);
    \node at (.5,4.65) {\large $C$};
    \draw[webs] (-1,5.5) to (-1,7.75);
    \draw[webs] (2,5.5) to (2,7.75);
    \draw[webs] (.5,5.5) to (.5,6.25) node[right]{\large $\mathbb{X}_n$};
    \draw[webs] (.5,6.25) to (.5,7);
    \draw[webs] (.5,7) to[out=180,in=270] (-.25,7.75) node[above]{\large $\mathbb{X}_1$};
    \draw[webs] (.5,7) to[out=0,in=270] (1.25,7.75) node[above]{\large $\mathbb{X}_m$};
    \node at (.5,7.5) {$\dots$};
\end{tikzpicture}
; \quad \quad \rfray{n}{\blam}(C \star (_{\aa} \sigma_r)) = 
\begin{tikzpicture}[scale=.5,anchorbase,tinynodes]
    \draw[webs] (-1,4) node[left,xshift=3pt,yshift=-10pt]{\large $\mathbb{X}'_{\ell}$} to[out=270,in=90] (2,1.5);
    \draw[webs] (2,1.5) to (2,0) node[below,yshift=-2pt]{\large $\mathbb{X}''_{\ell}$};
    \draw[line width=5pt, color=white] (1.5,4) to[out=270,in=90] (0,2);
    \draw[webs] (1.5,4) node[right,yshift=-10pt]{\large $\mathbb{X}'_n$} to[out=270,in=90] (0,2);
    \draw[webs] (0,2) to (0,1.5) node[left,yshift=5pt]{\large $\mathbb{X}''_n$};
    \draw[webs] (0,1.5) to (0,1);
    \draw[webs] (0,1) to[out=180,in=90] (-.75,.25);
    \draw[webs] (0,1) to[out=0,in=90] (.75,.25);
    \draw[webs] (-.75,.25) to (-.75,0) node[below,yshift=-2pt]{\large $\mathbb{X}''_1$};
    \draw[webs] (.75,.25) to (.75,0) node[below,yshift=-2pt]{\large $\mathbb{X}''_m$};
    \node at (0,.3) {$\dots$};
    \draw (-1.2,4) rectangle++(3.4,1.5);
    \node at (.5,4.65) {\large $C$};
    \draw[webs] (-1,5.5) to (-1,7.75);
    \draw[webs] (2,5.5) to (2,7.75);
    \draw[webs] (.5,5.5) to (.5,6.25) node[right]{\large $\mathbb{X}_n$};
    \draw[webs] (.5,6.25) to (.5,7);
    \draw[webs] (.5,7) to[out=180,in=270] (-.25,7.75) node[above]{\large $\mathbb{X}_1$};
    \draw[webs] (.5,7) to[out=0,in=270] (1.25,7.75) node[above]{\large $\mathbb{X}_m$};
    \node at (.5,7.5) {$\dots$};
\end{tikzpicture}
\end{gather}
\end{conv}

The following proposition will be our main tool in this investigation.

\begin{prop} \label{prop: rfray_fork_slide}
The strong deformation retraction of Proposition \ref{prop: cabled_fork_slides} induces an $R[\mathbb{X}, \mathbb{X}'']$-equivariant natural equivalence of functors

\[
\rfray{n}{\blam}(-) \star (_{\aa}\sigma_{r - 1})^{\blam} \sim \rfray{n}{\blam}(- \star (_{\aa}\sigma_{r - 1})) \colon K \left( \SSBim_{\aa}^{\bb} \right) \to K \left( \SSBim_{(s_{r - 1} \cdot \aa)^{\blam}}^{\bb^{\blam}} \right).
\]
which acts via a strong deformation retraction satisfying the side conditions. The analogous maps involving $_{\aa} \sigma_r$ induce an analogous equivalence, as do
the analogous maps involving negative crossings and crossings above $C$.
\end{prop}

\begin{proof}

Expressed diagramatically, Proposition \ref{prop: rfray_fork_slide} asserts that the morphisms of Proposition \ref{prop: cabled_fork_slides} induce a natural homotopy equivalence

\begin{gather}
\begin{tikzpicture}[scale=.5,anchorbase,tinynodes] \label{fig: rfray_fork_slide}
    \draw[webs] (0,2) to[out=270,in=90] (2,0) node[below]{$a_{r - 1}$};
    \draw[line width=5pt, color=white] (0,0) to[out=90,in=270] (1,2);
    \draw[webs] (0,0) to[out=90,in=270] (1,2);
    \draw[line width=5pt, color=white] (1,0) to[out=90,in=270] (2,2);
    \draw[webs] (1,0) to[out=90,in=270] (2,2);
    \node at (.5,.2) {$\dots$};
    \node at (1.5,1.8) {$\dots$};
    \draw [decorate,decoration = {brace,mirror}] (0,-.2) -- (1,-.2)
    node[pos=.5,below,yshift=-2pt]{$\blam$};
    \draw[webs] (1,2) to[out=90,in=180] (1.5,2.5);
    \draw[webs] (2,2) to[out=90,in=0] (1.5,2.5);
    \draw[webs] (1.5,2.5) to (1.5,3);
    \draw[webs] (0,2) to (0,3);
    \draw (-.7,3) rectangle++(2.9,1.5);
    \node at (.75,3.65) {\large $C$};
    \draw[webs] (-.3,4.5) to (-.3,5.5);
    \draw[webs] (1.8,4.5) to (1.8,5.5);
    \draw [decorate,decoration = {brace}] (-.3,5.9) -- (1.8,5.9) node[pos=.5,above]{$\bb^{\blam}$};
    \draw[webs] (.75,4.5) to (.75,5);
    \draw[webs] (.75,5) to[out=180,in=270] (.15,5.5);
    \draw[webs] (.75,5) to[out=0,in=270] (1.35,5.5);
    \node at (.75,5.4) {$\dots$};
    \node at (3.5,3) {\LARGE $\simeq$};
\end{tikzpicture}
\quad
\begin{tikzpicture}[scale=.5,anchorbase,tinynodes]
    \draw[webs] (0,0) to[out=90,in=180] (.5,.5);
    \draw[webs] (1,0) to[out=90,in=0] (.5,.5);
    \node at (.5,.2) {$\dots$};
    \draw [decorate,decoration = {brace,mirror}] (0,-.2) -- (1,-.2)
    node[pos=.5,below,yshift=-2pt]{$\blam$};
    \draw[webs] (2,0) node[below]{$a_{r - 1}$} to (2,.5);
    \draw[webs] (0,2) to[out=270,in=90] (2,.5);
    \draw[line width=5pt, color=white] (.5,.5) to[out=90,in=270] (1.5,2);
    \draw[webs] (.5,.5) to[out=90,in=270] (1.5,2);
    \draw[webs] (0,2) to (0,3);
    \draw[webs] (1.5,2) to (1.5,3);
    \draw (-.7,3) rectangle++(2.9,1.5);
    \node at (.75,3.65) {\large $C$};
    \draw[webs] (-.3,4.5) to (-.3,5.5);
    \draw[webs] (1.8,4.5) to (1.8,5.5);
    \draw[webs] (.75,4.5) to (.75,5);
    \draw[webs] (.75,5) to[out=180,in=270] (.15,5.5);
    \draw[webs] (.75,5) to[out=0,in=270] (1.35,5.5);
    \node at (.75,5.4) {$\dots$};
    \draw [decorate,decoration = {brace}] (-.3,5.9) -- (1.8,5.9) node[pos=.5,above]{$\bb^{\blam}$};
\end{tikzpicture}
\end{gather}
For visual clarity, we have suppressed all passive strands distant from the crossing $\sigma_d$ in Figure \ref{fig: rfray_fork_slide}.

Notice that the left-hand side of Figure \ref{fig: rfray_fork_slide} is exactly the complex

\[
\rfray{n}{\blam}(C) \star (_{\aa} \sigma_{r - 1})^{\blam} = \left( _{\bb^{\blam}}S_{\bb} \right) \star C \star MCC_{\blam, a_{r - 1}}^n;
\]
while the right-hand side is exactly the complex

\[
\rfray{n}{\blam}(C \star (_{\aa} \sigma_{r - 1})) = \left( _{\bb^{\blam}}S_{\bb} \right) \star C \star CM_{\blam, a_{r - 1}}^n.
\]
In particular, the morphisms $1 \star 1 \star \nu$, $1 \star 1 \star \mu$, $1 \star 1 \star h$ constitute a strong deformation retraction from the left-hand side to the right-hand side satisfying the side conditions by Proposition \ref{prop: cabled_fork_slides}. Since these are morphisms of bimodules, they are certainly $R[\mathbb{X}, \mathbb{X}'']$-equivariant. Finally, naturality comes from the middle interchange law, as any morphism $f \colon C \to D$ evaluates to $1 \star f \star 1$ when passed through either of these functors.
\end{proof}

In the special case that $C$ is a Rickard complex, repeated application of
Proposition \ref{prop: rfray_fork_slide} gives the following.

\begin{cor} \label{cor: rfray_cable}
    Let $\beta_{\aa} \in \BGpd{\aa}{\bb}$ be given. Then $\rfray{n}{\blam}(\beta)
    \simeq W_{\blam} \star \beta_{\aa}^{\blam} \simeq \beta_{\aa}^{\blam} \star W_{\blam}$.
\end{cor}

We refer to the equivalence of Proposition \ref{prop: rfray_fork_slide} as the \textit{fork-sliding equivalence for reduced fray functors}. We will make use of this fork-sliding equivalence repeatedly in what follows to establish similar equivalences for more complex fray functors. The following Proposition justifies this use:

\begin{prop} \label{prop: rfray_to_def_forks}
Let $\mathbb{U}$ be an alphabet of deformation parameters, and suppose $\Delta \in R[\mathbb{X}, \mathbb{X}''] \otimes_R R[\mathbb{U}]$ satisfies $\Delta^2 = F$ for some $F \in Z(\SSBim[\mathbb{U}])$. Then applying the fork-sliding equivalence for reduced fray functors \eqref{fig: rfray_fork_slide} in each $\mathbb{U}$-degree induces an $R[\mathbb{X}, \mathbb{X}'']$-equivariant natural equivalence of functors

\[
\mathrm{tw}_{\Delta} \left( \left( \rfray{n}{\blam} (-) \star \left( _{\aa} \sigma_{r - 1} \right)^{\blam} \right) \otimes_R R[\mathbb{U}] \right)
\sim
\mathrm{tw}_{\Delta} \left( \left( \rfray{n}{\blam} \left( - \star \left(_{\aa} \sigma_{r - 1} \right) \right) \right) \otimes_R R[\mathbb{U}] \right)
\]
from $K \left( \SSBim_{\aa}^{\bb} \right)$ to $K \YS_F \left( \SSBim_{(s_{r - 1} \cdot \aa)^{\blam}}^{\bb^{\blam}} ; R[\mathbb{U}] \right)$.
\end{prop}

\begin{proof}
This follows immediately from $R[\mathbb{X}, \mathbb{X}'']$-equivariance and naturality of \eqref{fig: rfray_fork_slide}.
\end{proof}

In practice, Proposition \ref{prop: rfray_fork_slide} can often be used to dramatically
simplify complexes. Indeed, for a colored braid $\beta_{\aa}$ with $c$ crossings involving
the distinguished $n$-colored strand, passing from $\rfray{n}{\blam}(C) \star (_{\aa} \beta)^{\blam}$
to $\rfray{n}{\blam}(C \star _{\aa} \beta)$ reduces the number of crossings by $(m - 1)c$.
On the other hand, applying this equivalence in the other direction allows us to decrease
the strand colors in $_{\aa} \beta$; in the extreme case $\blam = (1^n)$, this corresponds to
replacing Rickard complexes of \textit{singular} Soergel bimodules with the
better-understood Rouquier complexes of Soergel bimodules. Our primary motivation in investigating the increasingly complex fray functors
in the following sections is to leverage these crossing-reduction and color-reduction
techniques in increasingly complex situations. 

\subsection{Fray Functor}

For each complex $C \in \CS(\SSBim_{\aa}^{\bb})$, \textcolor{revisions}{because the differential $d_C$ consists of bimodule homomorphisms,} for each $1 \leq j \leq m$, and each $1 \leq k \leq \lambda_j$, there is a closed endomorphism 

\[
e_k(\mathbb{X}_j) - e_k(\mathbb{X}'_j) \in \End_{\CS \left( \SSBim_{\aa^{\blam}}^{\bb^{\blam}} \right) } \left( \rfray{n}{\blam}(C) \right)
\]
of degree $\mathrm{deg}(e_k(\mathbb{X}_j) - e_k(\mathbb{X}'_j)) = q^{2k}$. These endomorphisms pairwise commute, and so we may consider their Koszul complex as in Example \ref{ex: kosz_comp}. We label the corresponding alphabet of deformation parameters by $\Theta_{\blam} := \{\theta_{jk}\}_{\substack{1 \leq j \leq m \\ 1 \leq k \leq \lambda_j}}$, where $\mathrm{deg}(\theta_{jk}) = q^{-2k}t$.

\begin{defn} \label{def: fray}
    Let $\fray{n}{\blam}$ denote the functor taking a bimodule
    $M \in \SSBim_{\aa}^{\bb}$ to the $0$-curved complex

    \begin{align*}
    \fray{n}{\blam}(M) & := \mathrm{tw}_{\alpha}
    \left( \rfray{n}{\blam}(M) \otimes_R R[\Theta_{\blam}] \right)
    \in \YS_0 \left( \SSBim^{\bb^{\blam}}_{\aa^{\blam}}; R[\Theta_{\blam}] \right); \\
    \alpha & = \sum \limits_{j = 1}^m \sum \limits_{k = 1}^{\lambda_j}
    \left(e_k(\mathbb{X}_j) - e_k(\mathbb{X}_j') \right) \otimes \theta_{jk}
    \end{align*}
    
\textcolor{revisions}{and a bimodule morphism $f \colon M \to N$ to $\rfray{n}{\blam}(f) \otimes 1_{R[\Theta_{\blam}]}$.} We refer to the twist $\alpha$ as the \textit{Koszul} differential on $\fray{n}{\blam}(M)$. \textcolor{revisions}{As was the case for $\rfray{n}{\blam}$ above, since $\fray{n}{\blam}$ is an additive functor, it immediately} extends to functors on $\CS \left( \SSBim_{\aa}^{\bb} \right)$ and $K \left( \SSBim_{\aa}^{\bb} \right)$. We denote these functors also by $\fray{n}{\blam}$.
\end{defn}

\begin{example} \label{ex: fray_fin_proj}
    Consider the special case $N = n$, $\aa = \bb = (n)$. In this case,
    if $\blam = (1^n)$, then $\fray{n}{\blam}(\strand{n})$ is a Koszul complex on
    $W_{\blam}$ for (the action of) the elements $x_1 - x_1', \dots, x_n - x_n'$. More generally,
    for an arbitrary $\blam \vdash n$, $\fray{n}{\blam}(\strand{n})$ is a Koszul complex
    on the full merge-split bimodule $W_{\blam} \in \SSBim_{\blam}^{\blam}$
    for (the action of) differences of elementary
    symmetric polynomials in the alphabets associated to corresponding top and bottom strands.
\end{example}

\begin{defn} \label{def: fray_fin_proj}
    Let $n = N$, $\aa = \bb = (n) \vdash N$ as in Example \ref{ex: fray_fin_proj} \textcolor{revisions}{and $\blam \vdash n$ be an arbitrary \textcolor{revisions}{composition}}.
    In this case, we denote $\fray{n}{\blam}(\strand{n})$ by $\finproj{\blam}$ and call $\finproj{\blam}$ a \textit{finite frayed projector}.
\end{defn}

\begin{rem} \label{rem: ah_fin_proj}
    We point out to the reader familiar with \cite{AH17, EH17b, Con23} that even in the case $\blam = (1^n)$, our $\finproj{\blam}$ differs slightly from the
    finite column projectors $K_{1^n}$ considered in those works. Precisely,
    $\finproj{\blam}$ is a Koszul complex on $K_{1^n}$ for (the action of) the
    difference $x_1 - x_1'$. These two complexes are \textbf{not} homotopy equivalent;
    indeed, since $x_1 - x_1'$ acts nullhomotopically on $K_{1^n}$, a straightforward application
    of Proposition \ref{prop: kosz_base_change} gives $\finproj{\blam}
    \simeq (1 + q^{-2}t)K_{1^n}$.

    For our purposes, there are two major advantages in working with $\finproj{\blam}$.
    First, this more ``balanced" treatment of the difference $x_1 - x_1'$ generalizes
    immediately to arbitrary \textcolor{revisions}{composition}s $\blam \vdash n$ and to arbitrary complexes $C$.
    One can imagine adapting the construction of $K_{1^n}$ to this setting via
    an ``unbalanced" variant $\mathrm{Fray}_{unb}$ of our functor
    $\fray{n}{\blam}$ in which the component of the Koszul differential involving (differences
    of symmetric polynomials in) $\mathbb{X}_1$ and $\mathbb{X}_1'$ is omitted. Even when
    $\blam = (1^n)$, if the action of $e_1(\mathbb{X}_n) - e_1(\mathbb{X}'_n)$ on $C$ is not
    nullhomotopic to begin with, then the action of $x_1 - x_1'$ on $\mathrm{Fray}_{unb}(C)$
    is not nullhomotopic. (This issue does not arise for $K_{1^n} =
    \mathrm{Fray}_{unb}(\strand{n})$, since $e_i(\mathbb{X}_n) = e_i(\mathbb{X}'_n)
    \in \End(\strand{n})$ for all $1 \leq i \leq n$.)
    
    Second, the balanced functors in this work enjoy almost tautological good behavior
    after taking Hochschild homology. This allows for much easier contact with
    colored link homology results in Section \ref{sec: link_hom} below. The reader familiar
    with \cite{EH19} should compare the
    distinction between the balanced and unbalanced constructions with the passage to
    ``reduced" finite (row) projectors $\tilde{K}_n$ used in
    their recursion.
\end{rem}

We again wish to investigate the interaction between $\fray{n}{\blam}$ and horizontal composition with cabled braids. We cannot immediately apply Proposition \ref{prop: rfray_to_def_forks}, as the Koszul differential $\alpha$ involves the alphabet $\mathbb{X}'$, which does \textit{not} intertwine the fork-sliding equivalence for reduced fray functors. We remedy this situation using the dot-sliding homotopies of Proposition \ref{prop: dot_slide}.

\begin{prop} \label{prop: fray_dot_slide}
Set $\tilde{\alpha} := \sum \limits_{j = 1}^m \sum \limits_{k = 1}^{\lambda_j}
    \left(e_k(\mathbb{X}_j) - e_k(\mathbb{X}_j'') \right) \otimes \theta_{jk}$. Then there is an $R[\mathbb{X}, \mathbb{X}'']$-equivariant natural equivalence of functors

\begin{align*}
\fray{n}{\blam}(-) \star (_{\aa} \sigma_{r - 1})^{\blam} \sim \mathrm{tw}_{\tilde{\alpha}} \left( \left( \rfray{n}{\blam}(-) \star (_{\aa} \sigma_{r - 1})^{\blam} \right) \otimes_R R[\Theta_{\blam}] \right)
\end{align*}
from $\CS \left( \SSBim_{\aa}^{\bb} \right)$ to $\YS_0 \left( \SSBim_{(s_{r - 1} \cdot \aa)^{\blam}}^{\bb^{\blam}}; R[\Theta_{\blam}] \right)$.
\end{prop}

\begin{proof}
By definition, $\fray{n}{\blam}(C) \star (_{\aa} \sigma_{r - 1})^{\blam}$ is the strict deformation of
$\rfray{n}{\blam}(C) \star (_{\aa} \sigma_{r - 1})^{\blam}$ arising from the $0$-deforming family $\{e_k(\mathbb{X}_j) - e_k(\mathbb{X}'_j)\}$.
Recall the dot-slide endomorphisms $\xi_{jk} \in \End \left( MCC^n_{\blam, a_{r - 1}} \right)$ of Convention \ref{conv: cabled_fork_dot_slides}.
Since $\rfray{n}{\blam}(C) \star (_{\aa} \sigma_{r - 1})^{\blam} = \left( _{\bb^{\blam}}S_{\bb} \right) \star C \star MCC_{\blam, a_{r - 1}}^n$,
we may consider the family $\{\xi_{jk}\}$ as acting on $\rfray{n}{\blam}(C) \star (_{\aa}\sigma_{r - 1})^{\blam}$ as well by the endomorphisms $1 \star 1 \star \xi_{jk}$. These endomorphisms pairwise anti-commute, square to $0$, and satisfy

\[
[d, -\xi_{jk}] = \left( e_k(\mathbb{X}_j) - e_k(\mathbb{X}'_j) \right) - \left(e_k(\mathbb{X}_j) - e_k(\mathbb{X}''_j) \right).
\]

This is exactly the situation of Corollary \ref{cor: kosz_base_change_rep}. Explicitly, the desired isomorphisms are

\begin{equation} \label{eq: fray_dot_slides}
\Psi = \prod_{j = 1}^m \prod_{k = 1}^{\lambda_j} (1 \otimes 1 - \xi_{jk} \otimes \theta_{jk}); \quad \Psi^{-1} = \prod_{j = 1}^m \prod_{k = 1}^{\lambda_j} (1 \otimes 1 + \xi_{jk} \otimes \theta_{jk}).
\end{equation}
Since $\xi_{jk}$ occurs on a distant tensor factor from $C$, \textcolor{revisions}{$\Psi$ and $\Psi^{-1}$ give rise to a natural equivalence of functors}; finally, $\Psi$ and $\Psi^{-1}$ must be $R[\mathbb{X}, \mathbb{X}'']$-equivariant because they are bimodule morphisms.
\end{proof}

We refer to the equivalence of Proposition \ref{prop: fray_dot_slide} as the \textit{dot-sliding equivalence for fray functors}. Again, we will wish to make use of this dot-sliding equivalence in situations involving more complex deformations; we do so using the following Proposition.

\begin{prop} \label{prop: fray_to_def_dots}
Let $\mathbb{U}$ be an alphabet of deformation parameters, and suppose $\Delta$ is an endonatural transformation of the functor $\left( \fray{n}{\blam}(-) \star \left(_{\aa} \sigma_{r - 1} \right)^{\blam} \right) \otimes_R R[\mathbb{U}]$ satisfying $\Delta^2 = F$ for some $F \in Z(\SSBim[\mathbb{U}])$. Then the dot-sliding equivalence for fray functors \eqref{eq: fray_dot_slides} induces an $R[\mathbb{X}, \mathbb{X}'']$-equivariant natural equivalence of functors

\[
\mathrm{tw}_{\Delta} \left( \left( \fray{n}{\blam} (-) \star \left( _{\aa} \sigma_{r - 1} \right)^{\blam} \right) \otimes_R R[\mathbb{U}] \right)
\sim
\mathrm{tw}_{\Psi \Delta \Psi^{-1} + \tilde{\alpha}} \left( \left( \rfray{n}{\blam} (-) \star \left( _{\aa} \sigma_{r - 1} \right)^{\blam} \right) \otimes_R R[\Theta_{\blam}, \mathbb{U}] \right)
\]
from $\CS \left( \SSBim_{\aa}^{\bb} \right)$ to $\YS_F \left( \SSBim_{(s_{r - 1} \cdot \aa)^{\blam}}^{\bb^{\blam}} ; R[\Theta_{\blam}, \mathbb{U}] \right)$.
\end{prop}

\begin{proof}
This is a direct application of Proposition \ref{prop: hpt_curv_iso}. All that needs to be checked is that both functors are well-defined and produce legitimate $F$-deformations, but this follows immediately from naturality of $\Delta$ and the fact that $\Delta^2 = F$.
\end{proof}

Now, notice that the modified Koszul differential $\tilde{\alpha}$ on the right-hand side of Proposition \ref{prop: fray_dot_slide} lives in $R[\mathbb{X}, \mathbb{X}''] \otimes_R R[\Theta_{\blam}]$. This is exactly the situation of Proposition \ref{prop: rfray_to_def_forks}. It follows that composing the dot-sliding equivalence of Proposition \ref{prop: fray_dot_slide} with the fork-sliding equivalence of Proposition \ref{prop: rfray_fork_slide} furnishes the following.

\begin{prop} \label{prop: fray_fork_slide}
There is an $R[\mathbb{X}, \mathbb{X}'']$-equivariant natural equivalence of functors

\[
\fray{n}{\blam}(-) \star (_{\aa} \sigma_{r - 1})^{\blam} \sim \fray{n}{\blam}(- \star (_{\aa} \sigma_{r - 1})) \colon K \left( \SSBim_{\aa}^{\bb} \right) \to K \left( \SSBim_{(s_{r - 1} \cdot \aa)^{\blam}}^{\bb^{\blam}}[\Theta_{\blam}] \right)
\]
which acts via a strong deformation retraction satisfying the side conditions. The analogous maps involving $_{\aa} \sigma_r$ induce an analogous equivalence, as do
the analogous maps involving negative crossings and crossings above $C$.
\end{prop}

\begin{cor} \label{cor: braid_fin_slide}
    Let $\beta_{\aa} \in \BGpd{\aa}{\bb}$ be given. Then $\fray{n}{\blam}(\beta_{\aa})
    \simeq \finproj{\blam} \star \beta_{\aa}^{\blam} \simeq \beta_{\aa}^{\blam} \star \finproj{\blam}$.
\end{cor}

In Section \ref{sec: link_hom}, we will be especially interested in the result of Corollary \ref{cor: braid_fin_slide} when $\aa = (n, c)$, $\bb = (c, n)$, $\beta = (\sigma_1)_{\aa}$. Graphically, the resulting homotopy equivalence takes the form

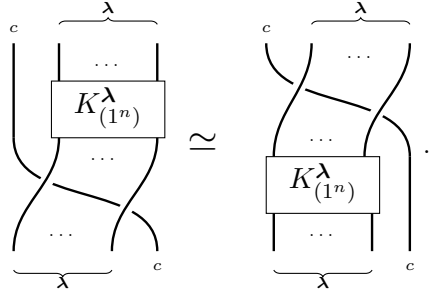
\begin{gather*}
\begin{tikzpicture}[scale=.5,anchorbase,tinynodes]
    \draw(0,0) rectangle++(3,1.5);
    \node at (1.5,.75) {\large $\finproj{\blam}$};
    \draw[webs] (.2,1.5) to (.2,2.5);
    \node at (1.5,2) {$\dots$};
    \draw[webs] (2.8,1.5) to (2.8,2.5);
    \draw[webs] (-1,2.5) node[above]{$c$} to (-1,0);
    \draw[webs] (-1,0) to[out=270,in=90] (2.8,-3) node[below]{$c$};
    \draw[line width=5pt, color=white] (.2,-.1) to[out=270,in=90] (-1,-3);
    \draw[webs] (.2,0) to[out=270,in=90] (-1,-3);
    \draw[line width=5pt, color=white] (2.8,-.1) to[out=270,in=90] (1.6,-3);
    \draw[webs] (2.8,0) to[out=270,in=90] (1.6,-3);
    \node at (1.4,-.5) {$\dots$};
    \node at (.3,-2.5) {$\dots$};
    \draw [decorate,decoration = {brace}] (.2,3) -- (2.8,3)
    node[pos=.5,above]{$\blam$};
    \draw [decorate,decoration = {brace,mirror}] (-1,-3.5) -- (1.6,-3.5)
    node[pos=.5,below]{$\blam$};
    \node at (4,-.25) {\LARGE $\simeq$};
\end{tikzpicture}
\quad 
\begin{tikzpicture}[scale=.5,anchorbase,tinynodes]
    \draw[webs] (-1,2.5) node[above]{$c$} to[out=270,in=90] (2.8,-.5);
    \draw(-1,-2) rectangle++(3,1.5);
    \draw[webs] (-.8,-2) to (-.8,-3);
    \draw[webs] (1.8,-2) to (1.8,-3);
    \node at (.5,-1.25) {\large $\finproj{\blam}$};
    \draw[webs] (2.8,-.5) to (2.8,-3) node[below]{$c$};
    \node at (.5,-2.5) {$\dots$};
    \draw[line width=5pt,color=white] (.2,2.5) to[out=270,in=90] (-.8,-.4);
    \draw[webs] (.2,2.5) to[out=270,in=90] (-.8,-.5);
    \draw[line width=5pt,color=white] (2.8,2.5) to[out=270,in=90] (1.6,-.4);
    \draw[webs] (2.8,2.5) to[out=270,in=90] (1.6,-.5);
    \draw [decorate,decoration = {brace}] (.2,3) -- (2.8,3)
    node[pos=.5,above]{$\blam$};
    \draw [decorate,decoration = {brace,mirror}] (-.8,-3.5) -- (1.8,-3.5)
    node[pos=.5,below]{$\blam$};
    \node at (1.5, 2.2) {$\dots$};
    \node at (.5,0) {$\dots$};
\end{tikzpicture}
.
\end{gather*}
Whenever a complex $C$ satisfies the homotopy equivalence enjoyed by $\finproj{\blam}$ depicted above, we say that $C$ \textit{slides past crossings}.

\begin{cor} \label{cor: fin_proj_crossing_slide}
    $\finproj{\blam}$ slides past crossings.
\end{cor}

\begin{inter}
In the following sections, we turn our attention to various deformed versions of the functor $\fray{n}{\blam}$ with the goal of establishing deformed versions of Propositions \ref{prop: rfray_fork_slide} and \ref{prop: fray_fork_slide}. There are three distinct deformations that will be relevant to us: the \textit{separated}, \textit{passive}, and \textit{bulk} deformations associated to separated, passive, and bulk strands, respectively. We have already encountered most of the tools needed to treat each of these deformations; the main challenges are notational, with each deformation being entitled to its own connection and its own alphabet of deformation parameters.

We call the alphabet of deformation parameters associated to bulk strands $\mathbb{U}_n$ and the alphabet associated to passive strands $\mathbb{U}_{\ell}$, in keeping with the notation of Section \ref{sec: fork_slyde}. This leaves the question of what to call the alphabet associated to separated strands. When $\aa = (1^N) \vdash N$, Definition \ref{def: y-ification} recovers the notion of a $y$-ification as introduced by Gorsky--Hogancamp in \cite{GH22} after renaming the variables $u_{j, 1}$ as $y_j$.

With this precedent in mind, we call the alphabet of deformation parameters associated to separated strands $\mathbb{Y}_{\blam}$, since they are closer in spirit to the usual deformation parameters of $y$-ified homology in e.g. \cite{GH22, GHM21, Con23}. We wish to emphasize that this should \textit{not} imply to the reader that we work in the fully refined setting $\aa^{\blam} = (1^N)$; rather, we view the labels of $\aa^{\blam}$ as ``thinner" than those of $\aa$ and accordingly use variables that have typically been associated with thinner strands.
\end{inter}

\subsection{Deformation on Separated Strands} \label{sec: yfray}
In this section, we consider deformations of complexes $\fray{n}{\blam}(C)$ with $\Delta e$-curvature associated to separated strands. The relevant curvature in this setting is

\[
F_y^{\blam}(\mathbb{X}, \mathbb{X}') = \sum \limits_{j = 1}^m \sum \limits_{k = 1}^{\lambda_j}
(e_k(\mathbb{X}_j) - e_k(\mathbb{X}_j')) \otimes y_{jk} \in \mathrm{Sym}^{\blam}(\mathbb{X}, \mathbb{X}') \otimes_R R[\mathbb{Y}_{\blam}].
\]

We will also use the alphabet $\mathbb{Y}_{\blam}$ to label deformation parameters associated to cabled strands in cabled braids, writing for example $(_{\aa} \sigma_{r - 1})^{\blam, y}$.

\begin{prop}
For each $C \in \CS(\SSBim_{\aa}^{\bb})$, the collection of endomorphisms $\{1 \otimes \theta_{jk}^{\vee}\}$ is an $F_y^{\blam}(\mathbb{X}, \mathbb{X}')$-deforming family on $\fray{n}{\blam}(C)$.
\end{prop}

\begin{proof}
What needs to be checked is that 
\[
\left[ \rfray{n}{\blam}(d_C) \otimes 1 + \sum_{i = 1}^m \sum_{\ell = 1}^{\lambda_i} \left( e_{\ell}(\mathbb{X}_i) - e_{\ell}(\mathbb{X}_i') \right) \otimes \theta_{i \ell}, 1 \otimes \theta_{jk}^{\vee} \right] = \left( e_k(\mathbb{X}_j) - e_k(\mathbb{X}'_j) \right) \otimes 1
\]
for each $j, k$ and that the family $\{1 \otimes \theta_{jk}^{\vee}\}$ is graded commutative. The former property is a straightforward computation, and the latter follows from graded commutativity of $R[\Theta_{\blam}]$.
\end{proof}

\begin{defn} \label{def: y_fray}
    Let $\yfray{n}{\blam}$ denote the functor taking a complex
    $C \in \CS(\SSBim_{\aa}^{\bb})$ to the $F_y^{\blam}$-curved complex

    \begin{align*}
    \yfray{n}{\blam}(C) & := \mathrm{tw}_{\delta} \left(\fray{n}{\blam}(C)
    \otimes_R R[\mathbb{Y}_{\blam}] \right) \in \YS_{F_y^{\blam}} \left( \SSBim_{\aa^{\blam}}^{\bb^{\blam}}; R[\Theta_{\blam}, \mathbb{Y}_{\blam}] \right); \\
    \delta & = \sum_{j = 1}^m \sum_{k = 1}^{\lambda_j} 1 \otimes \theta_{jk}^{\vee} y_{jk}
    \end{align*}
    
    As in Definitions \ref{def: rfray} and \ref{def: fray}, $\yfray{n}{\blam}$ automatically extends to a functor on $K(\SSBim_{\aa}^{\bb})$, which we also denote by $\yfray{n}{\blam}$.
\end{defn}

\begin{defn} \label{def: yfray_fin_proj}
    Let $N, n, \aa, \bb$ be as in Definition \ref{def: fray_fin_proj}.
    In this case, we denote $\yfray{n}{\blam}(\strand{n})$ by $\yfinproj{\blam}$ and call $\yfinproj{\blam}$ a \textit{deformed finite frayed projector}.
\end{defn}

\begin{rem}
    In \cite{Elb22, Con23}, it was observed that there is a natural $y$-ification
    of the Abel--Hogancamp finite column projector $K_{1^n}$. Our
    $\yfinproj{\blam}$ is a straightforward generalization of that $y$-ification from the case
    $\blam = (1^n)$ to an arbitrary \textcolor{revisions}{composition}, and the twist $\delta$ above is a straightforward
    generalization to a more general complex $C$ and \textcolor{revisions}{composition} $\blam$.
\end{rem}

\begin{prop} \label{prop: yfray_dot_slide}
Let $\tilde{\alpha}$ and $\delta$ be as in Proposition \ref{prop: fray_dot_slide} and Definition \ref{def: y_fray}. Then the natural equivalence of Proposition \ref{prop: fray_dot_slide} induces an $R[\mathbb{X}, \mathbb{X}'']$-equivariant natural equivalence of functors

\begin{align*}
\yfray{n}{\blam}(-) \star (_{\aa} \sigma_{r - 1})^{\blam, y} \sim \mathrm{tw}_{\delta + \tilde{\alpha}} \left( \left( \rfray{n}{\blam}(-) \star (_{\aa} \sigma_{r - 1})^{\blam} \right) \otimes_R R[\Theta_{\blam}, \mathbb{Y}_{\blam}] \right)
\end{align*}
from $\CS \left( \SSBim_{\aa}^{\bb} \right)$ to $\YS_{F_y^{\blam}} \left( \SSBim_{(s_{r - 1} \cdot \aa)^{\blam}}^{\bb^{\blam}}; R[\Theta_{\blam}, \mathbb{Y}_{\blam}] \right)$.
\end{prop}

\begin{proof}
Given a complex $C \in \CS(\SSBim_{\aa}^{\bb})$, we may express the functor on the left-hand side applied to $C$ explicitly as

\begin{align*}
\yfray{n}{\blam}(C) \star (_{\aa} \sigma_{r - 1})^{\blam, y} & = \mathrm{tw}_{\Delta} \left( \left( \fray{n}{\blam}(C) \star (_{\aa} \sigma_{r - 1})^{\blam} \right) \otimes_R R[\mathbb{Y}_{\blam}] \right); \\
\Delta & = \sum_{j = 1}^m \sum_{k = 1}^{\lambda_j} (1 \otimes \theta_{jk}^{\vee} + \xi_{jk} \otimes 1) y_{jk}
\end{align*}
Here $\xi_{jk}$ are considered as endomorphisms of $\rfray{n}{\blam}(C) \star (_{\aa} \sigma_r)^{\blam}$ as in the proof of Proposition \ref{prop: fray_dot_slide}. By Proposition \ref{prop: fray_to_def_dots}, it suffices to check that $\Psi \Delta \Psi^{-1} = \delta$.

This identity is checked by direct computation; we sketch this computation below, omitting some details which are are exactly analogous to those of previous computations (for example, those in the proof of Proposition \ref{prop: sdr_deformation_lifting}). We conjugate each summand of $\Delta$ by $\Psi$ separately. By Remark \ref{rem: kosz_base_order_ind}, we may choose any ordering of factors in the product defining $\Psi$ and $\Psi^{-1}$ to carry out this computation. To conjugate the summand with subscripts $j,k$, we begin by using the factors of $\Psi$ and $\Psi^{-1}$ with the same subscripts. Repeated application of the middle interchange law and the identity $\theta_{jk} \theta_{jk}^{\vee} + \theta_{jk}^{\vee} \theta_{jk} = 1$ gives

\[
(1 \otimes 1 - \xi_{jk} \otimes \theta_{jk}) \left( (1 \otimes \theta_{jk}^{\vee} + \xi_{jk} \otimes 1) y_{jk} \right) (1 \otimes 1 + \xi_{jk} \otimes \theta_{jk}) = 1 \otimes \theta_{jk}^{\vee} y_{jk}.
\]

Any other factor of $\Psi$ involves terms of the form $\xi_{i, \ell} \otimes \theta_{i, \ell}$ for $(i, \ell) \neq (j, k)$. In particular, conjugating $1 \otimes \theta_{jk}^{\vee} y_{jk}$ by these terms has no effect, as everything in sight is graded-commutative. In total, we obtain

\[
\Psi \left( (1 \otimes \theta_{jk}^{\vee} + \xi_{jk} \otimes 1) y_{jk} \right) \Psi^{-1} = 1 \otimes \theta_{jk}^{\vee} y_{jk}
\]
for each $j, k$. Adding these terms together gives exactly $\delta$, as desired.
\end{proof}

Again, the twist $\delta + \tilde{\alpha}$ of the right-hand side of Proposition \ref{prop: yfray_dot_slide} involves only the alphabets $\mathbb{X}, \mathbb{X}'', \Theta_{\blam}$, and $\mathbb{Y}_{\blam}$. By Proposition \ref{prop: rfray_to_def_forks}, applying the fork-sliding equivalence \eqref{fig: rfray_fork_slide} in each $\Theta_{\blam}$ and each $\mathbb{Y}_{\blam}$ degree and composing with the equivalence of Proposition \ref{prop: yfray_dot_slide} immediately gives the following.

\begin{prop} \label{prop: yfray_fork_slide}
There is an $R[\mathbb{X}, \mathbb{X}'']$-equivariant natural equivalence of functors

\[
\yfray{n}{\blam}(-) \star (_{\aa} \sigma_{r - 1})^{\blam, y} \sim \yfray{n}{\blam}(- \star (_{\aa} \sigma_{r - 1})) \colon K \left( \SSBim_{\aa}^{\bb} \right) \to K\YS_{F_y^{\blam}} \left( \SSBim_{(s_{r - 1} \cdot \aa)^{\blam}}^{\bb^{\blam}}; R[\Theta_{\blam}, \mathbb{Y}_{\blam}] \right)
\]
which acts via a strong deformation retraction satisfying the side conditions. The analogous maps involving $_{\aa} \sigma_r$ induce an analogous equivalence, as do
the analogous maps involving negative crossings and crossings above $C$.
\end{prop}

\begin{cor} \label{cor: braid_yfin_slide}
    Let $\beta_{\aa} \in \BGpd{\aa}{\bb}$ be given. Then $\yfray{n}{\blam}(\beta_{\aa})
    \simeq \yfinproj{\blam} \star \beta^{\blam, y}_{\aa} \simeq
    \beta^{\blam, y}_{\aa} \star \yfinproj{\blam}$.
\end{cor}

\begin{cor} \label{cor: yfin_proj_crossing_slide}
    $\yfinproj{\blam}$ slides past crossings with $\Delta e$-deformation on the strands passing through $\yfinproj{\blam}$.
\end{cor}

\subsection{Deformation on Passive Strands} \label{subsec: unfray_curv}

We turn our attention next to establishing a version of Proposition \ref{prop: yfray_fork_slide} with $\Delta e$-curvature on the passive strand which crosses under the separated strands. Recalling the alphabet conventions of \eqref{fig: x_alphs}, the relevant curvature in this setting is

\[
F_{u_{\ell}}^{(a_{r - 1})}(\mathbb{X}'_{\ell}, \mathbb{X}''_{\ell}) = \sum_{k = 1}^{a_{r - 1}} (e_k(\mathbb{X}'_{\ell}) - e_k(\mathbb{X}''_{\ell})) \otimes u_{\ell, k}.
\]

In keeping with the notation of Section \ref{sec: fork_slyde}, we denote the passive strand $\Delta e$-deformations of $_{\aa} \sigma_{r - 1}$ and $(_{\aa} \sigma_{r - 1})^{\blam}$ as $(_{\aa} \sigma_{r - 1})^{u_{\ell}}$ and $(_{\aa} \sigma_{r - 1})^{\blam, u_{\ell}}$, respectively.

\begin{rem}
Strictly speaking, $F_{u_{\ell}}^{(a_{r - 1})}(\mathbb{X}'_{\ell}, \mathbb{X}''_{\ell})$ is not a valid curvature, since the presence of the alphabet $\mathbb{X}'_{\ell}$ prevents this morphism from acting on complexes not of the form $\rfray{n}{\blam}(C) \star (_{\aa} \sigma_{r - 1})^{\blam}$ or $\rfray{n}{\blam}(C \star (_{\aa} \sigma_{r - 1}))$.
This issue can be solved by evaluating the functors of Proposition \ref{prop: unfray_fork_slide} on an $F$-curved complex, where $F \in Z(\SSBim_{\aa}^{\bb}[\mathbb{U}_{\ell}])$
satisfies $\rfray{n}{\blam}(F) \star 1 + 1 \star F_{u_{\ell}}^{(a_{r - 1})} \in Z(\SSBim_{(s_{r - 1} \cdot \aa)^{\blam}}^{\bb^{\blam}}[\mathbb{U}_{\ell}])$.
We deliberately ignore this issue in formulating Proposition \ref{prop: unfray_fork_slide} to avoid a notational nightmare. In all applications of this Proposition (for example, in Corollary \ref{cor: finproj_ycrossing_slide}) this is a non-issue, as we always take $C$ to be a deformed Rickard complex.
\end{rem}

\begin{prop} \label{prop: unfray_fork_slide}
There is an $R[\mathbb{X}, \mathbb{X}'']$-equivariant natural equivalence of functors

\begin{align*}
\yfray{n}{\blam}(-) \star (_{\aa} \sigma_{r - 1})^{\blam, y, u_{\ell}} & \sim \yfray{n}{\blam}(- \star (_{\aa} \sigma_{r - 1})^{u_{\ell}}) \colon \\
 K \left( \SSBim_{\aa}^{\bb} \right) & \to K\YS_{F_y^{\blam} + F_{u_{\ell}}^{(a_{r - 1})}} \left( \SSBim_{(s_{r - 1} \cdot \aa)^{\blam}}^{\bb^{\blam}}; R[\Theta_{\blam}, \mathbb{Y}_{\blam}, \mathbb{U}_{\ell}] \right)
\end{align*}
The analogous equivalence involving $_{\aa} \sigma_r$ also holds, as do
the analogous equivalences involving negative crossings and crossings above $C$.
\end{prop}

\begin{proof}
As in the proof of Proposition \ref{prop: fray_dot_slide}, we consider the dot-sliding endomorphisms $\xi_{\ell k}$, $\xi'_{\ell k} \in \End(MCC^n_{\blam, a_{r - 1}})$ of Section \ref{sec: fork_slyde}, respectively $\tilde{\xi}_{\ell k} \in \End(CM^n_{\blam, a_{r - 1}})$, as acting on $\rfray{n}{\blam}(C) \star (\sigma_{r - 1})_{\aa}^{\blam}$, respectively $\rfray{n}{\blam}(C \star (\sigma_{r - 1})_{\aa})$.

Given a complex $C \in \CS(\SSBim_{\aa}^{\bb})$, we may express the functor on the left-hand side applied to $C$ explicitly as

\begin{align*}
\yfray{n}{\blam}(C) \star (_{\aa} \sigma_{r - 1})^{\blam, y, u_{\ell}} & = \mathrm{tw}_{\Delta_{\ell}} \left( \left( \yfray{n}{\blam}(C) \star (_{\aa} \sigma_{r - 1})^{\blam, y} \right) \otimes \mathbb{R}[\mathbb{U}_{\ell}] \right); \\
\Delta_{\ell} & = \sum_{k = 1}^{a_{r - 1}} \xi_{\ell k} \otimes u_{\ell k}.
\end{align*}
By Proposition \ref{prop: fray_to_def_dots}, the natural equivalence $\Psi, \Psi^{-1}$ of \eqref{eq: fray_dot_slides} induces an isomorphism

\begin{multline} \label{eq: unfray_dot_slide}
\mathrm{tw}_{\Delta_{\ell}} \left( \left( \yfray{n}{\blam}(C) \star (_{\aa} \sigma_{r - 1})^{\blam, y} \right) \otimes_R R[\mathbb{U}_{\ell}] \right) \\
\cong \mathrm{tw}_{\Psi \Delta_{\ell} \Psi^{-1} + \delta + \tilde{\alpha}} \left( \left( \rfray{n}{\blam}(-) \star (_{\aa} \sigma_{r - 1})^{\blam} \right) \otimes_R R[\Theta_{\blam}, \mathbb{Y}_{\blam}, \mathbb{U}_{\ell}] \right).
\end{multline}

Now, $\Psi$ is constructed from dot-slide endomorphisms $\xi_{jk}$ and (multiplication by) deformation parameters $\theta_{jk}$, while $\Delta_{\ell}$ is constructed from dot-slide endomorphisms $\xi_{\ell k}$ and deformation parameters $u_{\ell k}$. All of these morphisms form a graded-commutative family, so $\Psi$ and $\Delta_{\ell}$ commute, and $\Psi \Delta_{\ell} \Psi^{-1} = \Delta_{\ell}$. Making this substitution in the right-hand side of \eqref{eq: unfray_dot_slide} and reassociating, we obtain

\begin{multline} \label{eq: unfray_dot_slide2}
\mathrm{tw}_{\Delta_{\ell}} \left( \left( \yfray{n}{\blam}(C) \star (_{\aa} \sigma_{r - 1})^{\blam, y} \right) \otimes_R R[\mathbb{U}_{\ell}] \right) \\
\cong \mathrm{tw}_{\delta + \tilde{\alpha}} \left( \left( \left( ((_{\bb^{\blam}} S_{\bb}) \star C) \otimes_R R[\mathbb{U}_{\ell}] \right) \star (MCC_{\blam, a_{r - 1}}^n)^{u_{\ell}} \right) \otimes_R R[\Theta_{\blam}, \mathbb{Y}_{\blam}] \right).
\end{multline}

By Proposition \ref{prop: conn_change_mcc}, we may replace $(MCC_{\blam, a_{r - 1}}^n)^{u_{\ell}}$ with $(MCC_{\blam, a_{r - 1}}^n)^{u_{\ell}'}$ up to homotopy equivalence. Finally, we apply the equivalence of Proposition \ref{prop: rfray_fork_slide} in every $\Theta_{\blam}$ and $\mathbb{Y}_{\blam}$-degree. In each such degree, the maps involved are exactly those of the homotopy equivalence of Proposition \ref{prop: deformed_forkslide}, resulting in a homotopy equivalence of curved complexes

\[
\left( ((_{\bb^{\blam}} S_{\bb}) \star C) \otimes_R R[\mathbb{U}_{\ell}] \right) \star (MCC_{\blam, a_{r - 1}}^n)^{u_{\ell}'} \simeq \left( ((_{\bb^{\blam}} S_{\bb}) \star C) \otimes_R R[\mathbb{U}_{\ell}] \right) \star (CM_{\blam, a_{r - 1}}^n)^{u_{\ell}}.
\]

Now, upon passing to a tensor product with $R[\Theta_{\blam}, \mathbb{Y}_{\blam}]$ and twisting by $\delta + \tilde{\alpha}$, the right-hand side of this equivalence becomes exactly $\yfray{n}{\blam}(C \star (\sigma_{r - 1})_{\aa}^{u_{\ell}})$, as can be seen by comparing each summand in the connections of these two expressions.
\end{proof}

\begin{cor} \label{cor: finproj_ycrossing_slide}
    $\yfinproj{\blam}$ slides past $\Delta e$-deformed crossings.
\end{cor}

\subsection{Deformation on Bulk Strands} \label{sec: ufray}

Finally, we turn our attention to deformations with $\Delta e$-curvature associated to bulk strands. The relevant curvature here is

\[
F_u^{(n)}(\mathbb{X}_n, \mathbb{X}'_n) = \sum_{i = 1}^n (e_i(\mathbb{X}_n) - e_i(\mathbb{X}'_n)) \otimes u_i \in \mathrm{Sym}^{(n)}(\mathbb{X}_n, \mathbb{X}'_n) \otimes_R R[\mathbb{U}_n].
\]
We also denote the curved complex associated to the crossing $_{\aa} \sigma_{r - 1}$ with $\Delta e$-curvature along only the bulk $n$-labeled strand by $(_{\aa} \sigma_{r - 1})^{u_n}$. Explicitly, $(_{\aa} \sigma_{r - 1})^{u_n}$ has curvature

\[
F_u^{(n)}(\mathbb{X}'_n, \mathbb{X}''_n) = \sum_{i = 1}^n (e_i(\mathbb{X}'_n) - e_i(\mathbb{X}''_n)) \otimes u_i \in \mathrm{Sym}^{(n)}(\mathbb{X}'_n, \mathbb{X}''_n) \otimes_R R[\mathbb{U}_n].
\]

\begin{prop} \label{prop: thick_curv_homotopies}
    Let $C \in \CS(\SSBim_{\bb}^{\aa})$ be given. Recall the polynomials $a_{ijk}^{\blam}(\mathbb{X}_n, \mathbb{X}'_n) \in \mathrm{Sym}^{\blam}(\mathbb{X}_n, \mathbb{X}'_n)$ of Lemma \ref{lem: a_ijk_exist}. For each $1 \leq i \leq n$, set

    \[
    \eta_i := \sum_{j = 1}^m \sum_{k = 1}^{\lambda_j} a_{ijk}^{\blam}(\mathbb{X}_n, \mathbb{X}'_n)
    \otimes \theta_{jk}^{\vee} \in \END^{-1}(\fray{n}{\blam}(C)).
    \]

    Then $\{\eta_i\}_{i = 1}^n$ is an $F_u^{(n)}$-deforming family on $\fray{n}{\blam}(C)$.
\end{prop}

\begin{proof}
    This is a direct computation; the details are identical to those in the proof of Lemma \ref{lem: bulk_curv_MCC}.
\end{proof}

As in the previous Subsection, we may consider the strict deformation of $\fray{n}{\blam}(C)$ arising from the family of Proposition \ref{prop: thick_curv_homotopies}. In fact we find it more convenient to negate this curvature, instead considering the deformation arising from the family $\{-\eta_i\}$; see Remark \ref{rem: bulk_neg} below for an explanation of this point. We will also find it convenient to first consider this deformation independent of the deformation on separated strands considered in Subsection \ref{sec: yfray}; we consider the two deformations in conjunction at the end of this Subsection.

\begin{defn} \label{def: ufray}
    Let $\ufray{n}{\blam}$ denote the functor taking a complex $C \in K(\SSBim_{\aa}^{\bb})$
    to the complex

    \[
    \ufray{n}{\blam}(C) := \mathrm{tw}_{-\gamma} \left( \fray{n}{\blam}(M) \otimes_R
    R[\mathbb{U}_n] \right) \in \YS_{-F_u^{(n)}} \left( \SSBim_{\aa^{\blam}}^{\bb^{\blam}}; R[\mathbb{U}_n] \right);
    \quad \gamma = \sum_{i = 1}^n \eta_i \otimes u_i
    \]
As usual, $\ufray{n}{\blam}$ immediately descends to a functor on the homotopy category $K(\SSBim_{\aa}^{\bb})$, which we also denote by $\ufray{n}{\blam}$.
\end{defn}

\begin{defn} \label{def: fray_inf_proj}
    Let $N, n, \aa, \bb$ be as in Definition \ref{def: fray_fin_proj}.
    In this case, we denote $\ufray{n}{\blam}(\strand{n})$ by $\infproj{\blam}$ and call $\infproj{\blam}$ an \textit{infinite frayed projector}. Notice that the curvature $F_u^{(n)}(\mathbb{X}, \mathbb{X}')$ vanishes on $\infproj{\blam}$; as a consequence, we may consider $\infproj{\blam}$ as a genuine chain complex by unrolling as in Remark \ref{rem: unrolling}.
\end{defn}

\begin{rem} \label{rem: ah_inf_proj}
    To the reader familiar with \cite{AH17, Con23},
    in the case $\blam = (1^n)$ our $\infproj{\blam}$ will look reminiscent of a larger
    (i.e. one additional variable in each of the alphabets $\Theta_{\blam}$ and $\mathbb{U}_n$) version of the unital column-colored idempotent
    $P_{(1^n)}^{\vee}$ described in that work.
    In contrast with Remark \ref{rem: ah_fin_proj}, these two complexes
    \textit{are} homotopy equivalent up to an overall quantum degree shift; this is the content of Theorem \ref{thm: proj_agreement} below.
\end{rem}

There is a straightforward extension of $\ufray{n}{\blam}(C)$ to \textit{curved} complexes of the form $\mathrm{tw}_{\Delta_C}(C \otimes_R R[\mathbb{U}_n])$.\footnote{In principle one could fomulate an analogous extension of the functor $\yfray{n}{\blam}$ to curved complexes of the form $\mathrm{tw}_{\Delta_C}(C \otimes_R R[\mathbb{Y}_{\blam}])$; however, since the alphabet of deformation parameters $\mathbb{Y}_{\blam}$ in this case depends on the \textcolor{revisions}{composition} $\blam \vdash N$, this extension is less natural. This reflects the fact that the deformation parameters $\mathbb{U}_n$ are associated to a \textit{bulk} $n$-labeled strand, which already exists as an exterior edge of $C$ before applying any fray functors.}

\begin{defn}
    Fix $X, Y \in \CS(\SSBim_{\aa}^{\bb})$. Given any $f \in \mathrm{Hom}_{\SSBim_{\aa}^{\bb} \otimes_R R[\mathbb{U}_n]}(X \otimes_R R[\mathbb{U}_n], Y \otimes_R R[\mathbb{U}_n])$, we may decompose $f$ as

    \[
    f = \sum_{{\bf{v}} \in \mathbb{Z}_{\geq 0}^n} f_{u^{\bf{v}}} \otimes u^{\bf{v}}.
    \]
    Then we define $\fray{n}{\blam}(f) \colon \fray{n}{\blam}(X) \otimes_R R[\mathbb{U}_n] \to \fray{n}{\blam}(Y) \otimes_R R[\mathbb{U}_n]$ to be the morphism

    \[
    \fray{n}{\blam}(f) = \sum_{{\bf{v}} \in \mathbb{Z}_{\geq 0}^n} \fray{n}{\blam}(f_{u^{\bf{v}}}) \otimes u^{\bf{v}}.
    \]
\end{defn}

\begin{defn} \label{def: ufray_curved}
    For each central element $Z_C \in Z(\SSBim_{\aa}^{\bb} \otimes_R R[\mathbb{U}_n])$, let $\ufray{n}{\blam}$ denote the functor taking a curved complex $\overline{C} := \mathrm{tw}_{\Delta_C}(C \otimes_R R[\mathbb{U}_n]) \in \YS_{Z_C}(\SSBim_{\aa}^{\bb}; R[\mathbb{U}_n])$ to the curved complex

    \[
    \ufray{n}{\blam}(\overline{C}) := \mathrm{tw}_{\fray{n}{\blam}(\Delta_C) - \gamma} \left( \fray{n}{\blam}(C) \otimes_R R[\mathbb{U}_n] \right) \in \YS_{\fray{n}{\blam}(Z_C) - F_u^{(n)}} \left( \SSBim_{\aa^{\blam}}^{\bb^{\blam}}; R[\Theta_{\blam}, \mathbb{U}_n] \right)
    \]
\end{defn}

\begin{rem}
    Implicit in Definition \ref{def: ufray_curved} is the claim that $\ufray{n}{\blam}(\overline{C})$
    \textit{is} in fact a curved complex with curvature $\fray{n}{\blam}(Z_C) - F_u^{(n)}$. This follows from
    the fact that $\gamma$ and $\fray{n}{\blam}(\Delta_C)$ commute, as they act on distant tensor factors.
\end{rem}

\begin{rem} \label{rem: bulk_neg}
    In defining $\ufray{n}{\blam}$ we do \textbf{not} curve by $\gamma$,
    but rather by its negation $-\gamma$, contributing a summand of
    $-F_u^{(n)} = \sum_{i = 1}^n (e_i(\mathbb{X}'_n) - e_i(\mathbb{X}_n)) \otimes u_i$
    to the total curvature rather than $F_u^{(n)}$. We prefer this sign convention for two reasons.
    
    First, this choice accounts for the form of Proposition \ref{prop: ufray_fork_slide} below.
    One could formulate an analog of Proposition \ref{prop: ufray_fork_slide}
    in which the opposite sign convention is used; to match curvatures,
    the adapted equivalence would then contain a deformed crossing on the left-hand side
    and an undeformed crossing on the right-hand side. This would require a discussion
    of ``bulk" $\Delta e$-curvature on the cabled crossing $(\sigma_r)_{\aa}^{\blam}$,
    which we wish to avoid in the interest of clarity of presentation (though the details
    are not difficult to sort through and are treated implicitly in Subsection \ref{sec: fork_slyde}).
    Alternatively, one could replace the $\Delta e$-curvature in $(\sigma_r)_{\aa}^u$ with
    \textit{negative} $\Delta e$-curvature, but this is undesirable
    because it runs counter to well-established conventions in $y$-ified homology literature.

    Second, Proposition \ref{prop: tr_yufray} and Theorem \ref{thm: infproj_fray_agree}
    depend on the cancellation of the curvature obtained from a deformed braid $\beta^u$
    and the application of $\ufray{n}{\blam}$. One could again adapt these statements at
    the cost of negating the connection on $\beta^u$, but again, we prefer not to
    run counter to the existing literature.
\end{rem}

For the remainder of this Subsection, we turn our attention to establishing an analog of Proposition \ref{prop: yfray_fork_slide} in the presence of bulk curvature, following largely the same outline as in Subsection \ref{sec: yfray}.

\begin{prop} \label{prop: ufray_dot_slide}
Let $C \in \CS(\SSBim_{\aa}^{\bb})$ be given and $\overline{C} = \mathrm{tw}_{\Delta_C}(C \otimes_R R[\mathbb{U}_n]) \in \YS_{F_C}(\SSBim_{\aa}^{\bb}; R[\mathbb{U}_n])$ an arbitrary $F_C$-deformation of $C$ for some $F_C \in Z(\SSBim_{\aa}^{\bb}[\mathbb{U}_n])$. Then the natural isomorphism of Proposition \ref{prop: fray_dot_slide} induces a natural isomorphism

\begin{equation}
\ufray{n}{\blam}(\overline{C}) \star \left( (_{\aa} \sigma_{r - 1})^{\blam} \otimes_R R[\mathbb{U}_n] \right) \cong \mathrm{tw}_{\fray{n}{\blam}(\Delta_C) + \tilde{\alpha} - \tilde{\gamma}} \left( \left( \rfray{n}{\blam}(C) \star (_{\aa} \sigma_{r - 1})^{\blam} \right) \otimes_R R[\Theta_{\blam}, \mathbb{U}_n] \right)
\end{equation}
where here $\tilde{\alpha}$ is as in Proposition \ref{prop: fray_dot_slide} and

\[
\tilde{\gamma} := \sum_{i = 1}^n \sum_{j = 1}^m \sum_{k = 1}^{\lambda_j} a_{ijk}^{\blam}(\mathbb{X}_n, \mathbb{X}'_n) (1 \otimes \theta_{jk}^{\vee} - \xi_{jk} \otimes 1) u_i.
\]
\end{prop}

\begin{proof}
We mimic the proof of Proposition \ref{prop: yfray_dot_slide}. We may express the left-hand side of the desired isomorphism explicitly as

\begin{align*}
\ufray{n}{\blam}(\overline{C}) \star \left( (_{\aa} \sigma_{r - 1})^{\blam} \otimes_R R[\mathbb{U}_n] \right) & = \mathrm{tw}_{\fray{n}{\blam}(\Delta_C) - \gamma} \left( \left( \fray{n}{\blam}(C) \star (_{\aa} \sigma_{r - 1})^{\blam} \right) \otimes_R R[\mathbb{U}_n] \right).
\end{align*}
By Proposition \ref{prop: fray_to_def_dots}, the natural equivalence $\Psi, \Psi^{-1}$ of \eqref{eq: fray_dot_slides} induces an isomorphism

\begin{multline*}
\mathrm{tw}_{\fray{n}{\blam}(\Delta_C) - \gamma} \left( \left( \fray{n}{\blam}(C) \star (_{\aa} \sigma_{r - 1})^{\blam} \right) \otimes_R R[\mathbb{U}_n] \right) \\
\cong \mathrm{tw}_{\Psi (\fray{n}{\blam}(\Delta_C) - \gamma) \Psi^{-1} + \tilde{\alpha}} \left( \left( \rfray{n}{\blam}(C) \star (_{\aa} \sigma_{r - 1})^{\blam} \right) \otimes_R R[\Theta_{\blam}, \mathbb{U}_n] \right).
\end{multline*}

It remains to compute $\Psi (\fray{n}{\blam}(\Delta_C) - \gamma) \Psi^{-1}$. Since $\fray{n}{\blam}(\Delta_C)$ is concentrated in a different tensor factor than $\Psi$ and $\Psi^{-1}$, a quick verification using the middle interchange law gives that $\Psi \fray{n}{\blam}(\Delta_C) \Psi^{-1} = \fray{n}{\blam}(\Delta_C)$. A direct computation exactly analogous to the proof of Proposition \ref{prop: yfray_dot_slide} gives $\Psi \gamma \Psi^{-1} = \tilde{\gamma}$, as desired.
\end{proof}

Here we run into a difficulty in mimicking the program of Subsection \ref{sec: yfray}. Since $\tilde{\gamma}$ contains terms from the alphabet $\mathbb{X}'_n$, we cannot immediately apply Proposition \ref{prop: rfray_to_def_forks}. Instead, we proceed by first modifying the connection on the right-hand side of the equivalence of Proposition \ref{prop: ufray_dot_slide}.

\begin{lem} \label{lem: ufray_conn_change}
Let $\overline{C}$, $\tilde{\alpha}$, $\tilde{\gamma}$ be as in Proposition \ref{prop: ufray_dot_slide}. For each $1 \leq i \leq n$, let $\tilde{\gamma}_i$ denote the $u_i$-component of $\tilde{\gamma}$. Let $\zeta_i$ be as in Lemma \ref{lem: bulk_curv_MCC}, and consider $\zeta_i$ as an endomorphism of $\rfray{n}{C} \star (_{\aa} \sigma_{r - 1})^{\blam}$ as usual. Set $\rho_i := \left( \sum_{j = 1}^m \sum_{k = 1}^{\lambda_j} a_{ijk}^{\blam}(\mathbb{X}_n, \mathbb{X}_n'') \otimes \theta_{jk}^{\vee} \right) - \zeta_i \otimes 1$.

Then for each $1 \leq i \leq n$, there exist endomorphisms 

\[
h_i \in \End^{-2}_{\SSBim[\Theta_{\blam}]} \left( \mathrm{tw}_{\tilde{\alpha}} \left( \rfray{n}{C} \star (_{\aa} \sigma_{r - 1})^{\blam} \right) \otimes_R R[\Theta_{\blam}] \right)
\]
satisfying

\[
[d, h_i] = \rho_i - \tilde{\gamma}_i
\]
such that the $R$-subalgebra of $\End_{\SSBim[\Theta_{\blam}]}(\fray{n}{C} \star (_{\aa} \sigma_{r - 1})^{\blam})$
generated by $\fray{n}{\blam}(\Delta_C)$, the family $\{\tilde{\gamma}_i\}$,
and the family $\{\rho_i\}$ is graded-commutative.
\end{lem}

\begin{proof}
Let $\mathcal{A}_{\blam}$ denote the $R$-subalgebra of $\End_{\SSBim[\Theta_{\blam}]}(\fray{n}{C} \star (_{\aa} \sigma_{r - 1})^{\blam})$ generated by $R[\mathbb{X}_n, \mathbb{X}'_n, \mathbb{X}''_n]$ and the families $\{\xi_{jk} \otimes 1\}$ and $\{1 \otimes \theta_{jk}^{\vee}\}$. All of these generators pairwise graded-commute, so $\mathcal{A}_{\blam}$ is itself graded-commutative. Furthermore, $\mathcal{A}_{\blam}$ is in fact a sub dg-algebra with differential $d_{\mathcal{A}}$ vanishing on $R[\mathbb{X}_n, \mathbb{X}'_n, \mathbb{X}''_n]$ and satisfying

\[
d_{\mathcal{A}}(\xi_{jk} \otimes 1) = \left( e_k(\mathbb{X}'_j) - e_k(\mathbb{X}''_j) \right) \otimes 1; \quad d_{\mathcal{A}}(1 \otimes \theta_{jk}^{\vee}) = \left( e_k(\mathbb{X}_j) - e_k(\mathbb{X}''_j) \right) \otimes 1.
\]
This is exactly the Koszul complex\footnote{Appropriately renormalized as in Remark \ref{rem: kosz_grading_shift}.} on the polynomial algebra $R[\mathbb{X}_n, \mathbb{X}_n', \mathbb{X}_n'']$ for the family of symmetric differences $e_k(\mathbb{X}'_j) - e_k(\mathbb{X}''_j)$ and $e_k(\mathbb{X}_j) - e_k(\mathbb{X}''_j)$. This family forms a regular sequence in $\mathcal{A}_{\blam}$, so by Proposition \ref{prop: kosz_reg_seq}, we have

\[
H^{\bullet}(\mathcal{A}_{\blam}) \cong R[\mathbb{X}_n].
\]
In particular, $H^{\bullet}(\mathcal{A}_{\blam})$ is concentrated in homological degree $0$.

Observe that $\tilde{\gamma}_i$ and $\rho_i$ are both homological degree $-1$ elements of $\mathcal{A}_{\blam}$ for each $i$. We can compute their differentials explicitly:

\begin{align*}
d_{\mathcal{A}}(\tilde{\gamma}_i) & = \sum_{j = 1}^m \sum_{k = 1}^{\lambda_j} a_{ijk}^{\blam}(\mathbb{X}_n, \mathbb{X}'_n) \left( ((e_k(\mathbb{X}_j) - e_k(\mathbb{X}_j'')) \otimes 1 ) - ((e_k(\mathbb{X}_j') - e_k(\mathbb{X}_j'')) \otimes 1) \right) \\
& = \sum_{j = 1}^m \sum_{k = 1}^{\lambda_j} a_{ijk}^{\blam}(\mathbb{X}_n, \mathbb{X}'_n) \left( (e_k(\mathbb{X}_j) - e_k(\mathbb{X}_j') \right) \otimes 1 \\
& = e_i(\mathbb{X}_n) - e_i(\mathbb{X}'_n); \\
d_{\mathcal{A}}(\rho_i) & = \left( \sum_{j = 1}^m \sum_{k = 1}^{\lambda_j} a_{ijk}^{\blam}(\mathbb{X}_n, \mathbb{X}''_n) \left(e_k(\mathbb{X}_j) - e_k(\mathbb{X}''_j) \right) \otimes 1 \right) - [d, \zeta_i] \otimes 1  \\
& =  \left(e_i(\mathbb{X}_n) - e_i(\mathbb{X}''_n) \right) - \left(e_i(\mathbb{X}'_n) - e_i(\mathbb{X}''_n) \right)\\
& = e_i(\mathbb{X}_n) - e_i(\mathbb{X}'_n)
\end{align*}
Since $d_{\mathcal{A}}(\tilde{\gamma}_i) = d_{\mathcal{A}}(\rho_i)$ and $\mathcal{A}_{\blam}$ has no nontrivial homology in degree $-1$, there must be some $h_i \in \mathcal{A}_{\blam}$ such that $d_{\mathcal{A}}(h_i) = \rho_i - \tilde{\gamma}_i$. Finally, everything in sight is graded-commutative, and $\fray{n}{\blam}(\Delta_C)$ occurs on a distinct tensor factor from any generator of $\mathcal{A}_{\blam}$, so it must commute with all of $\mathcal{A}_{\blam}$ as well.
\end{proof}

\begin{cor} \label{cor: ufray_internal_change}
Suppose $\overline{C} = \mathrm{tw}_{\Delta_C}(C \otimes_R R[\mathbb{U}_n])$ is a \textit{strict} deformation of $C$, and set $\rho := \sum_{i = 1}^n \rho_i \otimes u_i$. Then there is a natural isomorphism

\begin{multline*}
\mathrm{tw}_{\fray{n}{\blam}(\Delta_C) + \tilde{\alpha} - \tilde{\gamma}} \left( \left( \rfray{n}{\blam}(C) \star (_{\aa} \sigma_{r - 1})^{\blam} \right) \otimes_R R[\Theta_{\blam}, \mathbb{U}_n] \right) \\
\cong \mathrm{tw}_{\fray{n}{\blam}(\Delta_C) + \tilde{\alpha} - \rho}
\left( \left( \rfray{n}{\blam}(C) \star (_{\aa} \sigma_{r - 1})^{\blam} \right) \otimes_R R[\Theta_{\blam}, \mathbb{U}_n] \right).
\end{multline*}
\end{cor}

\begin{proof}
Apply Corollary \ref{cor: kosz_base_change_rep} to the strict deformations of the complex $\mathrm{tw}_{\tilde{\alpha}} \left( \left( \rfray{n}{\blam}(C) \star (_{\aa} \sigma_{r - 1})^{\blam} \right) \otimes_R R[\Theta_{\blam}] \right)$ arising from the families $\{(\Delta_C)_{u_i} - \tilde{\gamma}_i\}$ and $\{(\Delta_C)_{u_i} - \rho_i\}$.
\end{proof}

Finally, we arrive at the main result of this Subsection:

\begin{prop} \label{prop: ufray_fork_slide}
	Let $C \in \CS(\SSBim_{\aa}^{\bb})$ be given, and suppose $\overline{C} = \mathrm{tw}_{\Delta_C}(C \otimes_R R[\mathbb{U}_n]) \in \YS_{F_C}(\SSBim_{\aa}^{\bb}; R[\mathbb{U}_n])$ is a strict $F_C$-deformation of $C$ for some $F_C \in Z(\SSBim_{\aa}^{\bb}[\mathbb{U}_n])$. Then there is a natural homotopy equivalence
	
	\[
	\ufray{n}{\blam} \left(\overline{C} \right) \star \left( (_{\aa} \sigma_{r - 1})^{\blam} \otimes_R R[\mathbb{U}_n] \right) \simeq \ufray{n}{\blam} \left( \overline{C} \star (_{\aa} \sigma_{r - 1})^{u_n} \right).
	\]
The analogous equivalence involving $_{\aa} \sigma_r$ also holds, as do
the analogous equivalences involving negative crossings and crossings above $C$.
\end{prop}

\begin{proof}
Composing the equivalences of Proposition \ref{prop: ufray_dot_slide} and Corollary \ref{cor: ufray_internal_change} gives a natural isomorphism

\begin{equation} \label{eq: ufray_int_change}
\ufray{n}{\blam} \left(\overline{C} \right) \star \left( (_{\aa} \sigma_{r - 1})^{\blam} \otimes_R R[\mathbb{U}_n] \right) \cong \mathrm{tw}_{\fray{n}{\blam}(\Delta_C) + \tilde{\alpha} - \rho}
\left( \left( \rfray{n}{\blam}(C) \star (_{\aa} \sigma_{r - 1})^{\blam} \right) \otimes_R R[\Theta_{\blam}, \mathbb{U}_n] \right).
\end{equation}
Set $\overline{\rho} := \sum_{i = 1}^n \sum_{j = 1}^m \sum_{k = 1}^{\lambda_j} a_{ijk}(\mathbb{X}, \mathbb{X}'') \otimes \theta_{jk}^{\vee} u_i$. By reassociating, we can express the right-hand side of Equation \eqref{eq: ufray_int_change} as

\begin{equation} \label{eq: ufray_reassoc}
\mathrm{tw}_{\tilde{\alpha} + \fray{n}{\blam}(\Delta_C) - \overline{\rho}}
\left( \left( \left( \left( (_{\bb^{\blam}} S _{\bb}) \star C \right) \otimes_R R[\mathbb{U}_n] \right) \star \left( MCC^n_{\blam, a_{r - 1}} \right)^{u_n} \right) \otimes_R R[\Theta_{\blam}] \right).
\end{equation}

By Proposition \ref{prop: conn_change_mcc}, we may replace $\left( MCC^n_{\blam, a_{r - 1}} \right)^{u_n}$ by $\left( MCC^n_{\blam, a_{r - 1}} \right)^{u_n'}$ up to a homotopy equivalence (which acts distantly from $\tilde{\alpha}$, $\fray{n}{\blam}(\Delta_C)$, and $\overline{\rho}$, and hence commutes with each of these twists).

Finally, we apply the equivalence of Proposition \ref{prop: rfray_fork_slide} in every $\Theta_{\blam}$ and $\mathbb{U}_n$-degree. In each such degree, the maps involved are exactly those of the homotopy equivalence of Proposition \ref{prop: deformed_forkslide}, which take $(MCC_{\blam, a_{r - 1}}^n)^{u_n'}$ to $(CM_{\blam, a_{r - 1}}^n)^{u_n}$. These maps also commute with both $\fray{n}{\blam}(\Delta_C)$, which is concentrated on a distant tensor factor, and, because they are $R[\mathbb{X}, \mathbb{X}'']$-equivariant, $\overline{\rho}$. In total, we see a homotopy equivalence of curved complexes

\begin{multline*}
\mathrm{tw}_{\fray{n}{\blam}(\Delta_C) - \overline{\rho}} \left( \left( ((_{\bb^{\blam}} S_{\bb}) \star C) \otimes_R R[\mathbb{U}_n] \right) \star (MCC_{\blam, a_{r - 1}}^n)^{u_n'} \right) \\
\simeq \mathrm{tw}_{\fray{n}{\blam}(\Delta_C) - \overline{\rho}} \left( \left( ((_{\bb^{\blam}} S_{\bb}) \star C) \otimes_R R[\mathbb{U}_n] \right) \star (CM_{\blam, a_{r - 1}}^n)^{u_n} \right).
\end{multline*}

Now, upon passing to a tensor product with $R[\Theta_{\blam}]$ and twisting by $\tilde{\alpha}$, the right-hand side of this equivalence becomes exactly $\ufray{n}{\blam} \left( \overline{C} \star (\sigma_{r - 1})_{\aa}^{u_n} \right)$, as can be seen by comparing each summand in the connections of these two expressions.
\end{proof}

\begin{rem} \label{rem: uncurving_ufray}
    Each side of the homotopy equivalence of Proposition \ref{prop: ufray_fork_slide}
    is a curved complex with total curvature
    $\fray{n}{\blam}(Z_C) - F_u(\mathbb{X}, \mathbb{X}')$
    (\textbf{not} total curvature $\fray{n}{\blam}(Z_C) - F_u(\mathbb{X}, \mathbb{X}'')$). In particular, if $\fray{n}{\blam}(Z_C) = F_u(\mathbb{X}, \mathbb{X}')$; for example, if $\overline{C}$ is a $\Delta e$-deformed Rickard complex; then Proposition \ref{prop: ufray_fork_slide} furnishes a homotopy equivalence of $0$-curved complexes. By unrolling as in Remark \ref{rem: unrolling}, we may consider this as a homotopy equivalence of chain complexes, and we do so implicitly in what follows.
\end{rem}

\begin{cor} \label{cor: braid_inf_slide}
    Let $\beta_{\aa} \in \BGpd{\aa}{\bb}$ be given. Then $\ufray{n}{\blam}(\beta_{\aa}^{u_n})
    \simeq \infproj{\blam} \star \beta_{\aa}^{\blam} \simeq \beta_{\aa}^{\blam} \star \infproj{\blam}$.
\end{cor}

\begin{cor} \label{cor: inf_proj_crossing_slide}
    $\infproj{\blam}$ slides past crossings.
\end{cor}

\subsection{Bulk and Separated Curvature}

In Subsections \ref{sec: yfray} and \ref{sec: ufray}, we considered deformations of the functor $\fray{n}{\blam}$ with curvature associated to separated and bulk strands, respectively. We will also wish to study both of these types of curvature simultaneously, and we give a name to the corresponding functor.

\begin{defn} \label{def: yufray}
    Retain notation as in Definitions \ref{def: y_fray} and \ref{def: ufray_curved}. Then we set

    \[
    \yufray{n}{\blam}(\overline{C}) := \mathrm{tw}_{\fray{n}{\blam}(\Delta_C) + \delta - \gamma}
    \left(\fray{n}{\blam}(C) \otimes_R R[\mathbb{Y}_{\blam}, \mathbb{U}_n] \right)
    \in \YS_{\fray{n}{\blam}(Z_C) + F_y^{\blam} - F_u^{(n)}}(\SSBim_{\aa^{\blam}}^{\bb^{\blam}}; R[\mathbb{Y}_{\blam}, \mathbb{U}_n]).
    \]
\end{defn}

\begin{defn} \label{def: fray_yinf_proj}
    Let $N, n, \aa, \bb$ be as in Definition \ref{def: fray_fin_proj}.
    In this case, we denote $\yufray{n}{\blam}(\strand{n})$ by $\yinfproj{\blam}$ and refer to $\yinfproj{\blam}$ as a \textit{deformed infinite frayed projector}.
\end{defn}

\begin{rem}
    As in Remark \ref{rem: ah_inf_proj}, when $\blam = (1^n)$, our $\yinfproj{\blam}$ is homotopy
    equivalent to the $y$-ified, unital, column-colored, infinite idempotent
    $P_{(1^n)}^{\vee, y}$ of \cite{Con23} up to an overall quantum degree shift; this is the content of Theorem \ref{thm: yproj_agreement} below.
\end{rem}

The functor $\yufray{n}{\blam}$ enjoys similar interactions with deformed, cabled crossings as do the functors $\yfray{n}{\blam}$ and $\ufray{n}{\blam}$.

\begin{prop} \label{prop: yufray_fork_slide}
    Let $C \in \CS(\SSBim_{\aa}^{\bb})$ be given, and suppose $\overline{C} = \mathrm{tw}_{\Delta_C}(C \otimes_R R[\mathbb{U}_n]) \in \YS_{F_C}(\SSBim_{\aa}^{\bb}; R[\mathbb{U}_n])$ is a strict $F_C$-deformation of $C$ for some $F_C \in Z(\SSBim_{\aa}^{\bb}[\mathbb{U}_n])$. Then there is a natural homotopy equivalence
	
	\[
	\yufray{n}{\blam} \left(\overline{C} \right) \star \left( (_{\aa} \sigma_{r - 1})^{\blam, y, u_{\ell}} \otimes_R R[\mathbb{U}_n] \right) \simeq \yufray{n}{\blam} \left( \overline{C} \star (_{\aa} \sigma_{r - 1})^{u_{\ell}, u_n} \right).
	\]
The analogous equivalence involving $_{\aa} \sigma_r$ also holds, as do
the analogous equivalences involving negative crossings and crossings above $C$.
\end{prop}

\begin{proof}
We have essentially already proved this in Propositions \ref{prop: unfray_fork_slide} and \ref{prop: ufray_fork_slide}; the only challenge in merging those proofs is parsing notation while carrying around four distinct alphabets $\Theta_{\blam}$, $\mathbb{Y}_{\blam}$, $\mathbb{U}_{\ell}$, and $\mathbb{U}_n$ of deformation parameters.. More precisely, the only place where the equivalences of Propositions \ref{prop: unfray_fork_slide} and \ref{prop: ufray_fork_slide} diverge is in applying Proposition \ref{prop: conn_change_mcc} to modify the various connections on $MCC_{\blam, a_{r - 1}}^n$. But that Proposition is equally adapted to handle the alphabets $\mathbb{Y}_{\blam}$ and $\mathbb{U}_{\ell}$ at the same time. Using the homotopy equivalence treating both of these alphabets at that stage in the proof then gives the desired equivalence.
\end{proof}

\begin{cor} \label{cor: braid_yinf_slide}
    Let $\beta_{\aa} \in \BGpd{\aa}{\bb}$ be given. Then $\yufray{n}{\blam}(\beta_{\aa}^{u_n})
    \simeq \yinfproj{\blam} \star \beta_{\aa}^{\blam, y} \simeq \beta_{\aa}^{\blam, y}
    \star \yinfproj{\blam}$.
\end{cor}

\begin{cor} \label{cor: yinf_proj_crossing_slide}
    $\yinfproj{\blam}$ slides past fully $\Delta e$-deformed crossings.
\end{cor}

\section{Frayed Projectors} \label{sec: cat_idem}

In this section we undertake a deeper investigation of the frayed projectors $\finproj{\blam}$, $\yfinproj{\blam}$, $\infproj{\blam}$, and $\yinfproj{\blam}$ for $\blam 
= (1^N) \vdash N$. We will see that the functors of Section \ref{sec: fray_functors} allow us to recover analogs of the topological recursions studied in \cite{EH19, EH17b, Con23} in a purely formal way. In the presence of $\Delta e$-deformation on bulk strands, we also prove that $\infproj{\blam}$ and $\yinfproj{\blam}$ are (unital) idempotents in $K^+(\SBim_N)$ and $\YS_{F^{\blam}_y}(\SBim_N; R[\mathbb{Y}_{\blam}])$, respectively. Uniqueness of such idempotents then guarantees that $\infproj{\blam}$ (resp.  $\yinfproj{\blam}$) is homotopy equivalent to the complex $P_{(1^n)}^{\vee}$ (resp. $(P_{(1^n)}^{\vee})^y$) considered in \cite{Con23}.

\begin{conv}
    Throughout this Section, we fix $n \geq 1$, $N = n + 1$, $\aa = \bb = (n, 1) \vdash N$, and $\blam = (1^n) \vdash n$. As in Section \ref{sec: SSBim}, we denote by $W_{\aa}$ and $\strand{\aa}$ the full merge-split and identity bimodules in $\SSBim_{\aa}^{\bb}$:

    \begin{gather*}
    W_{\aa} := 
    \begin{tikzpicture}[scale=.5,anchorbase,tinynodes]
    \draw[webs] (0,0) node[below]{$n$} to[out=90,in=180] (.5,.5);
    \draw[webs] (1,0) node[below]{$1$} to[out=90,in=0] (.5,.5);
    \draw[webs] (.5,.5) to (.5,1.5);
    \draw[webs] (.5,1.5) to[out=180,in=270] (0,2) node[above,yshift=-2pt]{$n$};
    \draw[webs] (.5,1.5) to[out=0,in=270] (1,2) node[above,yshift=-2pt]{$1$};
    \end{tikzpicture}
    ; \quad \quad \strand{\aa} :=
    \begin{tikzpicture}[scale=.5,anchorbase,tinynodes]
    \draw[webs] (0,0) node[below]{$n$} to (0,2);
    \draw[webs] (1,0) node[below]{$1$} to (1,2);
    \end{tikzpicture}
    \end{gather*}

    We will write the alphabets involved in the usual action of $\mathrm{Sym}^{\aa}(\mathbb{X}, \mathbb{X}')$ on $\SSBim_{\aa}^{\bb}$ as follows. We label the total alphabet in $n + 1$ letters associated to the top (resp. bottom) of each complex as $\mathbb{X}_{n + 1}$ (resp. $\mathbb{X}'_{n + 1}$). These alphabets decompose along $\aa = (n, 1) \vdash N$ into alphabets $\mathbb{X}_{n + 1} := \mathbb{X}_n \sqcup \{x_{n + 1}\}$, resp. $\mathbb{X}'_{n + 1} := \mathbb{X}'_n \sqcup \{x'_{n + 1}\}$.

    We also let $FT_{\aa} \in Br_{\aa}^{\bb}$ denote the $(n, 1)$-colored full twist braid
    on two strands and $J_{n + 1} = \sigma_n \sigma_{n - 1} \dots \sigma _2 \sigma_1 \sigma_1
    \sigma_2 \dots \sigma_{n - 1} \sigma_n \in Br_{n + 1}$ denote the
    Jucys--Murphy braid on $n + 1$ strands:

    \begin{gather} \label{fig: braids}
    FT_{\aa} :=
    \begin{tikzpicture}[scale=.5,anchorbase,tinynodes]
    \draw[webs] (1,0) node[below]{$1$} to[out=90,in=270] (0,1);
    \draw[line width=5pt, color=white] (0,0) to[out=90,in=270] (1,1);
    \draw[webs] (0,0) node[below]{$n$} to[out=90,in=270] (1,1);
    \draw[webs] (1,1) to[out=90,in=270] (0,2) node[above,yshift=-2pt]{$n$};
    \draw[line width=5pt, color=white] (0,1) to[out=90,in=270] (1,2);
    \draw[webs] (0,1) to[out=90,in=270] (1,2) node[above,yshift=-2pt]{$1$};
    \end{tikzpicture}
    ; \quad \quad J_{n + 1} :=
    \begin{tikzpicture}[scale=.5,anchorbase,tinynodes]
    \draw[webs] (3,0) to[out=90,in=0] (1.5,.75);
    \draw[webs] (1.5,.75) to[out=180,in=270] (0,1.5);
    \draw[line width=5pt, color=white] (.8,0) to (.8,1.5);
    \draw[webs] (.8,0) to (.8,1.5);
    \draw[line width=5pt, color=white] (2.2,0) to (2.2,1.5);
    \draw[webs] (2.2,0) to (2.2,1.5);
    \node at (1.5,.2){$\dots$};
    \draw [decorate,decoration = {brace,mirror}] (.8,-.2) -- (2.2,-.2)
    node[pos=.5,below,yshift=-5pt]{$n$};
    \draw[webs] (.8,1.5) to (.8,3);
    \draw[webs] (2.2,1.5) to (2.2,3);
    \node at (1.5,2.8){$\dots$};
    \draw[line width=5pt, color=white] (0,1.5) to[out=90,in=180] (1.5,2.25);
    \draw[line width=5pt, color=white] (1.5,2.25) to[out=0,in=270] (3,3);
    \draw[webs] (0,1.5) to[out=90,in=180] (1.5,2.25);
    \draw[webs] (1.5,2.25) to[out=0,in=270] (3,3);
    \end{tikzpicture}
    \end{gather}
\end{conv}

Finally, we will often make use of the minimal complex for $FT_{\aa}$ introduced below.

\begin{defn} \label{def: c_n}
Let $C_n$ denote the three-term complex

\begin{center}
\begin{tikzcd}[sep=huge]
C_n := q^nW_{\aa} \arrow[r, "x_{n + 1} - x_{n + 1}'"]
& tq^{n - 2} W_{\aa} \arrow[r, "unzip"] & t^2q^{-2} \strand{\aa}.
\end{tikzcd}
\end{center}
\end{defn}

\begin{rem} \label{rem: hrw_min_model}
By a special case of Theorem 3.24 in \cite{HRW21b}\footnote{\textcolor{revisions}{In the notation of \cite{HRW21b}, the first term of the differential in $C_n$ is $\delta^v$ and the second is $\delta^h$. The SDR property follows from the fact that all homotopy equivalences involved in the proof of their Theorem 3.24 arise from Gaussian elimination.}}, there is a strong deformation retraction from $FT_{\aa}$ onto $C_n$; we denote the maps involved as follows:
\begin{center}
\begin{tikzcd}[sep=huge]
FT_{\aa} \arrow[r, "p", harpoon, shift left] & C_n
\arrow[l, "s", harpoon, shift left].
\end{tikzcd}
\end{center}
\end{rem}

\subsection{Topological Recursions} \label{sec: top_rec}

We consider each of the frayed projectors $\finproj{\blam}$, $\yfinproj{\blam}$, $\infproj{\blam}$, and $\yinfproj{\blam}$ in turn. In the finite, undeformed case, our goal will be to construct natural maps
$\kappa \colon \left( K_{(1^n)}^{(1^n)} \boxtimes \strand{(1)} \right) \to \left( K_{(1^n)}^{(1^n)} \boxtimes \strand{(1)} \right) \star J_{n + 1}$
satisfying $\mathrm{Cone}(\kappa) \simeq K_{(1^{n + 1})}^{(1^{n + 1})}$. Our goals in the infinite and deformed cases are similar; deforming will come along more or less for free, while passing to the infinite frayed projectors will require us to add in some periodic behavior by hand.

We constructed similar maps by hand in \cite{Con23} involving the Abel--Hogancamp projectors $K_n$ using explicit models of these complexes. This construction was challenging, especially in the infinite case, requiring several pages of delicate obstruction theory and ultimately giving a non-constructive proof that such maps exist. Even worse, that proof requires appealing to the parity of the triply-graded homology of the colored Hopf link established in \cite{HRW21}, which is quite a challenging computation in its own right.

In contrast, our construction here will only involve the very well-understood, finite complex $FT_{\aa}$ and its deformations. The economy of the present construction lies in the fact that $\fray{n}{\blam}$, $\yfray{n}{\blam}$, $\ufray{n}{\blam}$, and $\yufray{n}{\blam}$ are \textit{functors}. We view this simplification in establishing skein relations as the most significant computational advantage of working with fray functors.

\subsubsection{Finite Frayed Projector} \label{subsec: fin_proj}

Recall the three term complex $C_n$ of Definition \ref{def: c_n}. There is an obvious chain map $\iota \colon t^2q^{-2} \strand{\aa} \hookrightarrow C_n$ given
by inclusion of the final term.

\begin{lem} \label{lem: iota_gauss_elim}
There is a homotopy equivalence

\begin{center}
    \begin{tikzcd}[sep=huge]
        \mathrm{Cone}(\iota) \simeq q^n W_{\aa} \arrow[r, "x_{n + 1} - x_{n + 1}'"] & tq^{n - 2} W_{\aa}
    \end{tikzcd}
    \end{center}
\end{lem}

\begin{proof}
This is an immediate application of Gaussian elimination \eqref{prop: gauss_elim}.
\end{proof}

\begin{prop} \label{prop: fin_proj_recurs_thin}
There is a homotopy equivalence 
\[
\mathrm{Cone}(\fray{n}{\blam}(\iota)) \simeq q^n K_{(1^{n + 1})}^{(1^{n + 1})}.
\]
\end{prop}

\begin{proof}
By \textcolor{revisions}{Remark \ref{rem: cone_morphism_commute}}, we can rewrite the left-hand side as $\fray{n}{\blam}(\mathrm{Cone}(\iota))$. The right-hand side of the equivalence of Lemma \ref{lem: iota_gauss_elim} is exactly a shift of the Koszul complex \textcolor{revisions}{(Example \ref{ex: kosz_comp})} on $W_{\aa}$ for the action of $x_{n + 1} - x_{n + 1}'$. By definition, applying $\fray{n}{\blam}$ to this Koszul complex results in exactly the appropriate shift of the finite frayed projector $K_{(1^{n + 1})}^{(1^{n + 1})}$.
\end{proof}

We wish to repackage the results of Proposition \ref{prop: fin_proj_recurs_thin} to emphasize the role of the finite frayed projector $K_{(1^n)}^{(1^n)}$. Recall the map $s \colon C_n \to FT_{\aa}$ of Remark \ref{rem: hrw_min_model}. Since $s$ is a homotopy equivalence, it induces a homotopy equivalence $\mathrm{Cone}(s \circ \iota) \simeq \mathrm{Cone}(\iota)$ \textcolor{revisions}{by Proposition \ref{prop: cone_invariance}}. Applying the functor $\fray{n}{\blam}$, by Proposition \ref{prop: fin_proj_recurs_thin}, we obtain a homotopy equivalence

\[
\mathrm{Cone}(\fray{n}{\blam}(s \circ \iota)) \simeq q^n K_{(1^{n + 1})}^{(1^{n + 1})}.
\]

Now, $\fray{n}{\blam}(s \circ \iota)$ is by construction a chain map from $\fray{n}{\blam}(t^2q^{-2} \strand{\aa})$ to $\fray{n}{\blam}(FT_{\aa})$. The former is exactly $t^2q^{-2} \left( K_{(1^n)}^{(1^n)} \boxtimes \strand{(1)} \right)$; by Corollary \ref{cor: braid_fin_slide}, the latter is homotopy equivalent to $\left( K_{(1^n)}^{(1^n)} \boxtimes \strand{(1)} \right) \star J_{n + 1}$, and post-composing with this homotopy equivalence \textcolor{revisions}{again induces a homotopy equivalence of mapping cones by Proposition \ref{prop: cone_invariance}}. In total, we see the following (after shifting in $q$-degree to eliminate the factor of $q^n$ in Proposition \ref{prop: fin_proj_recurs_thin}):

\begin{cor} \label{cor: fin_proj_skein}
    There is a chain map
    
    \[
    \kappa \colon t^2q^{-2 - n} \left( \finproj{\blam} \boxtimes \strand{(1)} \right)
    \to q^{-n} \left( \finproj{\blam} \boxtimes \strand{1} \right) \star J_{n + 1}
    \]
    satisfying $K_{(1^{n + 1})}^{(1^{n + 1})} \simeq \mathrm{Cone}(\kappa)$.

\end{cor}

\subsubsection{Deformed Finite Frayed Projector} \label{subsec: passive_rec}

The entire story of Section \ref{subsec: fin_proj} admits a deformed analog. Let $FT_{\aa}^y$ denote the deformed full twist on 2 strands with $\Delta e$-curvature
on the $1$-labeled strand; we label the deformation parameter for this complex by $y_{n + 1}$. Explicitly, the curvature associated to this complex is of the form

\[
F_y^{(1)}(x_{n + 1}, x'_{n + 1}) = (x_{n + 1} - x'_{n + 1}) \otimes y_{n + 1} \in \mathrm{Sym}^{\aa}(\mathbb{X}_{n + 1}, \mathbb{X}_{n + 1}') \otimes_R R[y_{n + 1}].
\]

The three-term chain complex $C_n$ of Definition \ref{def: c_n} also admits a deformation with this curvature.

\begin{lem} \label{lem: c_ny}
Let $C_n^y$ be the strict deformation of $C_n$ via the twist indicated in blue below:

\begin{center}
\begin{tikzcd}[sep=huge]
C_n^y := q^nW_{\aa} \otimes_R R[y_{n + 1}] \arrow[r, harpoon, shift left, "x_{n + 1} - x'_{n + 1}"]
& tq^{n - 2}W_{\aa} \otimes_R R[y_{n + 1}] \arrow[l, blue, harpoon, shift left, "y_{n + 1}"] \arrow[r, "unzip"]
& t^2q^{-2} \strand{\aa} \otimes_R R[y_{n + 1}].
\end{tikzcd}
\end{center}
Then the homotopy equivalence of Remark \ref{rem: hrw_min_model} admits a curved lift to a homotopy equivalence

\begin{center}
\begin{tikzcd}[sep=huge]
FT_{\aa}^y \arrow[r, "p^y", harpoon, shift left] & C_n^y \arrow[l, "s^y", harpoon, shift left]
\end{tikzcd}
\end{center}
of $F_y^{(1)}$-curved complexes.
\end{lem}

\begin{proof}
A direct computation verifies that $C_n^y$ is indeed a $F_y^{(1)}$-deformation of $C_n$. Because $FT_{\aa}$ and $C_n$ are invertible up to homotopy equivalence, there is a homotopy equivalence of $\Hom$ complexes $\Hom_{\CS(\SSBim)}(FT_{\aa}, C_n) \simeq \End_{\CS(\SSBim)}(\strand{\aa})$. The latter complex is concentrated in degree $0$, so in particular, its homology is concentrated in non-negative degree. The desired result then follows immediately from Proposition \ref{prop: lifting_maps}.
\end{proof}

A quick check shows that the map $\iota$ lifts without modification to a map of curved complexes
$\iota \colon t^2q^{-2}\strand{\aa} \otimes_R R[y_{n + 1}] \hookrightarrow C_n^y$. We can once again compute $\mathrm{Cone}(\iota)$ by Gaussian elimination.

\begin{lem} \label{lem: iota_gauss_elim_y}
There is a homotopy equivalence of $F_y^{(1)}(x_{n + 1}, x'_{n + 1})$-curved complexes

\begin{center}
\begin{tikzcd}[sep=huge]
    \mathrm{Cone}(\iota) \simeq q^n W_{\aa} \otimes_R R[y_{n + 1}]
    \arrow[r, "x_{n + 1} - x'_{n + 1}", harpoon, shift left]
    & tq^{n - 2} W_{\aa} \otimes_R R[y_{n + 1}] \arrow[l, shift left, blue, harpoon, "y_{n + 1}"].
\end{tikzcd}
\end{center}
\end{lem}

The connection on the right-hand side of the equivalence of Lemma \ref{lem: iota_gauss_elim_y} once again resembles the connection on $K^{(1^{n + 1}), y}_{(1^{n + 1})}$. By exactly the same reasoning as in Proposition \ref{prop: fin_proj_recurs_thin}, we obtain the following deformed analog:

\begin{prop} \label{prop: yfin_proj_recurs_thin}
There is a homotopy equivalence of $F_y^{(1^{n + 1})}(\mathbb{X}_{n + 1}, \mathbb{X}'_{n + 1})$-curved complexes
\[
\mathrm{Cone}(\yfray{n}{\blam}(\iota)) \simeq q^n K^{(1^{n + 1}), y}_{(1^{n + 1})}.
\]
\end{prop}

We can again repackage the statement of Proposition \ref{prop: yfin_proj_recurs_thin} to emphasize the role of $K_{(1^n), y}^{(1^n)}$. Post-composing with $s^y$ induces a homotopy equivalence $\mathrm{Cone}(s^y \circ \iota) \simeq \mathrm{Cone}(\iota)$. After applying $\yfray{n}{\blam}$, we may consider $\yfray{n}{\blam}(s^y \circ \iota)$ as a map from $t^2q^{-2} \yfray{n}{\blam}(\strand{\aa}) \otimes_R R[y_{n + 1}]$ to $\yfray{n}{\blam}(FT^y_{\aa})$. The former is exactly $t^2q^{-2} \left( \left( \yfinproj{\blam} \boxtimes \strand{(1)} \right) \otimes_R R[y_{n + 1}] \right)$, and by Corollary \ref{cor: braid_yfin_slide}, the latter is homotopy equivalent to $\left( \left( \yfinproj{\blam} \boxtimes \strand{(1)} \right) \otimes_R R[y_{n + 1}] \right) \star J_{n + 1}^y$, where $J_{n + 1}^y$ is the fully $\Delta e$-deformed Rickard complex associated to $J_{n + 1}$. After shifting, we immediately obtain the following:

\begin{cor}
    There is a map of $F_y^{(1^{n + 1})}(\mathbb{X}_{n + 1}, \mathbb{X}'_{n + 1})$-curved complexes
    
    \[
    \kappa^y \colon t^2q^{-2 - n} \left( \left( \yfinproj{\blam} \boxtimes \strand{(1)} \right) \otimes_R R[y_{n + 1}] \right)
    \to q^{-n} \left( \left( \yfinproj{\blam} \boxtimes \strand{(1)} \right) \otimes_R R[y_{n + 1}] \right) \star J_{n + 1}^y
    \]
    satisfying $K_{(1^{n + 1})}^{(1^{n + 1}), y} \simeq \mathrm{Cone}(\kappa^y)$.

\end{cor}

\subsubsection{Infinite Frayed Projector} \label{subsec: inf_proj}

Next, we turn our attention to the infinite frayed projector $(P^{(1^{n + 1})}_{(1^{n + 1})})^{\vee}$. Our goal will be to express this projector as the cone of a morphism between two complexes, as in Proposition \ref{prop: fin_proj_recurs_thin}, and to interpret this result as a skein relation involving $(P^{(1^n)}_{(1^n)})^{\vee}$, as in Corollary \ref{cor: fin_proj_skein}.

Adapting these two statements to the infinite frayed projector requires three modifications. First, we use the functor $\ufray{n}{\blam}$ instead of $\fray{n}{\blam}$. Second, in order to get an \textit{uncurved} chain complex out the other side of this functor, we need to input $\Delta e$-deformed complexes to begin with (recall Remark \ref{rem: uncurving_ufray}). Finally, applying $\ufray{n}{\blam}$ only accounts for the first $n$ deformation parameters $u_1, \dots, u_n$; we will need to add the final parameter $u_{n + 1}$ in by hand via an infinite ladder construction.

\begin{rem} \label{rem: infinite_ladder}
The presence of the infinite ladder construction involving $u_{n + 1}$ is common in the categorified projector literature. It was first used in \cite{Hog18} to establish a skein relation for the row-colored projector, then again in \cite{EH17a, EH17b} in the more general framework of categorical diagonalization. The main advancement in \cite{Con23} that allowed us to compute column-colored homology of torus knots was adapting this framework to the column-colored projector. The results in this Subsection should be viewed as streamlining that adaptation using computations on two-strand braids.
\end{rem}

Let $FT_{\aa}^u$ denote the deformed full twist on 2 strands with $\Delta e$-curvature along only the $n$-labeled strand; we label the alphabet of deformation parameters
for this complex $\mathbb{U}_n := \{u_1, \dots, u_n\}$, where $\mathrm{deg}(u_i) = t^2q^{-2i}$. Explicitly, the curvature associated to this complex is of the form

\[
F_u^{(n)} = \sum_{i = 1}^n (e_i(\mathbb{X}_n) - e_i(\mathbb{X}'_n)) \otimes u_i \in \mathrm{Sym}^{\aa}(\mathbb{X}_{n + 1}, \mathbb{X}_{n + 1}') \otimes_R R[\mathbb{U}_n].
\]

Let $I_{\aa} \subset \mathrm{Sym}^{\aa}(\mathbb{X}_{n + 1}, \mathbb{X}'_{n + 1})$ denote the ideal
generated by the differences $e_i(\mathbb{X}_{n + 1}) - e_i(\mathbb{X}'_{n + 1})$
for $1 \leq i \leq n + 1$, and recall that up to an overall grading shift, we have an algebra isomorphism $W_{\aa} \cong \mathrm{Sym}^{\aa}(\mathbb{X}_{n + 1}, \mathbb{X}'_{n + 1})/I_{\aa}$.

\begin{lem} \label{lem: n1hopf_thick_twist}
    For each $1 \leq i \leq n$, set

    \[
    g_i(\mathbb{X}_{n + 1}, \mathbb{X}'_{n + 1}) := \sum_{j = 1}^i (-x_{n + 1}')^{j - 1} e_{i - j}(\mathbb{X}_n)
    \in \mathrm{Sym}^{\aa}(\mathbb{X}_{n + 1}, \mathbb{X}'_{n + 1}).
    \]

    Then $(x_{n + 1} - x_{n + 1}') g_i \equiv e_i(\mathbb{X}'_n) - e_i(\mathbb{X}_n)$
    modulo $I_{\aa}$ for each $i$.
\end{lem}

\begin{proof}
    We proceed by induction. When $i = 1$, we have $g_i = 1$, so $(x_{n + 1} - x_{n + 1}')
    g_1 = x_{n + 1} - x_{n + 1}'$. Since $e_1(\mathbb{X}_n + x_{n + 1}) =
    e_1(\mathbb{X}_n) + x_{n + 1}$ (and similarly for the primed alphabets), we have
    $x_{n + 1} - x_{n + 1}' \equiv e_1(\mathbb{X}'_n) - e_1(\mathbb{X}_n)$ modulo $I_{\aa}$.

    Now suppose $i \geq 2$, and observe that $g_i = e_{i - 1}(\mathbb{X}_n) -
    x_{n + 1}'g_{i - 1}$. By the inductive hypothesis, we have

    \begin{align*}
        (x_{n + 1} - x'_{n + 1}) g_i & = (x_{n + 1} - x'_{n + 1})
        (e_{i - 1}(\mathbb{X}_n) - x'_{n + 1}g_{i - 1}) \\
        & = (x_{n + 1} - x'_{n + 1}) e_{i - 1}(\mathbb{X}_n) - x'_{n + 1}
        (x_{n + 1} - x'_{n + 1}) g_{i - 1} \\
        & \equiv (x_{n + 1} - x'_{n + 1}) e_{i - 1}(\mathbb{X}_n) -
        x'_{n + 1} (e_{i-1}(\mathbb{X}'_n) - e_{i-1}(\mathbb{X}_n)) \ \ \mathrm{modulo } \ I \\
        & = x_{n + 1} e_{i - 1}(\mathbb{X}_n) - x'_{n + 1} e_{i - 1}(\mathbb{X}'_n) \\
        & = (e_i(\mathbb{X}_n + x_{n + 1}) - e_i(\mathbb{X}_n))
        - (e_i(\mathbb{X}'_n + x'_{n + 1}) - e_i(\mathbb{X}'_n)) \\
        & \equiv e_i(\mathbb{X}'_n) - e_i(\mathbb{X}_n) \ \ \mathrm{modulo} \ I_{\aa}.
    \end{align*}

    \vspace{1em}
\end{proof}

\begin{lem} \label{lem: c_nu}
Let $C_n^u$ be the strict deformation of $C_n$ via the twist indicated in blue below:

\begin{center}
\begin{tikzcd}[sep=huge]
C^u_n := q^n W_{\aa} \otimes_R R[\mathbb{U}_n] \arrow[r, shift left, harpoon, "x_{n + 1} - x'_{n + 1}"]
& tq^{n - 2}W_{\aa} \otimes_R R[\mathbb{U}_n] \arrow[l, shift left, blue, harpoon, "-\sum_{i = 1}^n g_i u_i"] \arrow[r, "unzip"]
& t^2q^{-2}\strand{\aa} \otimes_R R[\mathbb{U}_n]
\end{tikzcd}
\end{center}
Then the homotopy equivalence of Remark \ref{rem: hrw_min_model} admits a curved lift to a homotopy equivalence

\begin{center}
\begin{tikzcd}[sep=huge]
FT^u_{\aa} \arrow[r, shift left, harpoon, "p^u"] & C^u_n \arrow[l, shift left, harpoon, "s^u"]
\end{tikzcd}
\end{center}
of $F_u^{(n)}$-curved complexes.
\end{lem}

\begin{proof}
    That $C^u_n$ is an $F_u^{(n)}$-curved complex follows immediately from Lemma \ref{lem: n1hopf_thick_twist}. The desired homotopy equivalence then follows exactly as in the proof of Lemma \ref{lem: c_ny}.
\end{proof}

The map $\iota$ lifts without modification to a map of curved complexes
$\iota \colon t^2q^{-2} \strand{\aa} \otimes_R R[\mathbb{U}_n] \hookrightarrow C^u_n$, and we can once again compute its cone by Gaussian elimination.

\begin{lem} \label{lem: iota_gauss_elim_u}
There is a homotopy equivalence of $F_u^{(n)}(\mathbb{X}_{n + 1}, \mathbb{X}'_{n + 1})$-curved complexes

\begin{center}
\begin{tikzcd}[sep=huge]
\mathrm{Cone}(\iota) \simeq q^n W_{\aa} \otimes_R R[\mathbb{U}_n] \arrow[r, shift left, harpoon, "x_{n + 1} - x'_{n + 1}"]
& tq^{n - 2}W_{\aa} \otimes_R R[\mathbb{U}_n] \arrow[l, shift left, blue, harpoon, "-\sum_{i = 1}^n g_i u_i"] .
\end{tikzcd}
\end{center}

\end{lem}

Next, we turn our attention to upgrading the statement of Lemma \ref{lem: iota_gauss_elim_u} to include a new parameter $u_{n + 1}$. We begin with a technical digression to establish the existence of some morphisms involved in this upgrade.

\begin{lem} \label{lem: n1_hopf_homology}
    $H^{\bullet}(FT_{\aa}) \cong q^{-2n} \strand{\aa}$ as $\mathbb{Z}_q \times \mathbb{Z}_t$-graded bimodules.
\end{lem}

\begin{proof}
    Repeated application of the digon removal isomorphism of Proposition \ref{prop: web_rels} and the fork-sliding homotopy equivalence of Proposition \ref{prop: fork_slide} results in a sequence of homotopy equivalences

    \begin{gather*}
    [n]!
    \begin{tikzpicture}[anchorbase,scale=.5,tinynodes]
    \draw[webs] (1.5,0) node[below]{$1$} to[out=90,in=270] (0,1.5);
    \draw[line width=5pt, color=white] (0,0) to[out=90,in=270] (1.5,1.5);
    \draw[webs] (0,0) node[below]{$n$} to[out=90,in=270] (1.5,1.5);
    \draw[webs] (1.5,1.5) to[out=90,in=270] (0,3) node[above,yshift=-2pt]{$n$};
    \draw[line width=5pt, color=white] (0,1.5) to[out=90,in=270] (1.5,3);
    \draw[webs] (0,1.5) to[out=90,in=270] (1.5,3) node[above,yshift=-2pt]{$1$};
    \end{tikzpicture}
    \cong
    \begin{tikzpicture}[anchorbase,scale=.5,tinynodes]
    \draw[webs] (1.5,0) node[below]{$1$} to[out=90,in=270] (0,1.5);
    \draw[line width=5pt,color=white] (0,0) to[out=90,in=270] (1.5,1.5);
    \draw[webs] (0,0) node[below]{$n$} to[out=90,in=270] (1.5,1.5);
    \draw[webs] (1.5,1.5) to[out=90,in=270] (0,3);
    \draw[line width=5pt, color=white] (0,1.5) to[out=90,in=270] (1.5,3);
    \draw[webs] (0,1.5) to[out=90,in=270] (1.5,3);
    \draw[webs] (0,3) to[out=180,in=270] node[pos=.5,left]{$1$} (-.7,3.7);
    \node at (0,3.8) {$\dots$};
    \draw[webs] (0,3) to[out=0,in=270] node[pos=.5,right]{$1$} (.7,3.7);
    \draw[webs] (-.7,3.7) to[out=90,in=180] (0,4.4);
    \draw[webs] (.7,3.7) to[out=90,in=0] (0,4.4);
    \draw[webs] (0,4.4) to (0,5) node[above]{$n$};
    \draw[webs] (1.5,3) to (1.5,5) node[above]{$1$};
    \end{tikzpicture}
    \simeq
    \begin{tikzpicture}[scale=.5,anchorbase,tinynodes]
    \draw[webs] (3,0) to[out=90,in=0] (1.5,.75);
    \draw[webs] (1.5,.75) to[out=180,in=270] (0,1.5);
    \draw[line width=5pt, color=white] (.8,0) to (.8,1.5);
    \draw[webs] (.8,0) to (.8,1.5);
    \draw[line width=5pt, color=white] (2.2,0) to (2.2,1.5);
    \draw[webs] (2.2,0) to (2.2,1.5);
    \node at (1.5,.2){$\dots$};
    \draw[webs] (.8,1.5) to (.8,3);
    \draw[webs] (2.2,1.5) to (2.2,3);
    \node at (1.5,2.8){$\dots$};
    \draw[line width=5pt, color=white] (0,1.5) to[out=90,in=180] (1.5,2.25);
    \draw[line width=5pt, color=white] (1.5,2.25) to[out=0,in=270] (3,3);
    \draw[webs] (0,1.5) to[out=90,in=180] (1.5,2.25);
    \draw[webs] (1.5,2.25) to[out=0,in=270] (3,3);
    \draw[webs] (.8,0) to[out=270,in=180] (1.5,-.7);
    \draw[webs] (2.2,0) to[out=270,in=0] (1.5,-.7);
    \draw[webs] (1.5,-.7) to (1.5,-1.2) node[below]{$n$};
    \draw[webs] (.8,3) to[out=90,in=180] (1.5,3.7);
    \draw[webs] (2.2,3) to[out=90,in=0] (1.5,3.7);
    \draw[webs] (1.5,3.7) to (1.5,4.2) node[above]{$n$};
    \draw[webs] (3,0) to (3,-1.2) node[below]{$1$};
    \draw[webs] (3,3) to (3,4.2) node[above]{$1$};
    \end{tikzpicture}
    = (_{\aa} M_{(\blam, 1)}) \star J_{n + 1} \star (_{(\blam, 1)} S_{\aa}).
    \end{gather*}

    At the level of cohomology, this becomes an isomorphism

    \[
    [n]! H^{\bullet}(FT_{\aa}) \cong H^{\bullet} \left( (_{\aa} M_{(\blam, 1)}) \star J_{n + 1} \star (_{(\blam, 1)} S_{\aa}) \right).
    \]
    
    Since $(_{\aa} M_{(\blam, 1)})$ (resp. $(_{(\blam, 1)} S_{\aa})$) is free as a right (resp. left) $\mathrm{Sym}^{(\blam, 1)}(\mathbb{X})$-module, the functor $(_{\aa} M_{(\blam, 1)}) \star -$ (resp. $- \star (_{(\blam, 1)} S_{\aa})$) is exact. As a consequence, we obtain an isomorphism

    \[
    H^{\bullet} \left( (_{\aa} M_{(\blam, 1)}) \star J_{n + 1} \star (_{(\blam, 1)} S_{\aa}) \right) \cong (_{\aa} M_{(\blam, 1)}) \star H^{\bullet}(J_{n + 1}) \star (_{(\blam, 1)} S_{\aa}).
    \]
    By Lemma 19.40 of \cite{EMTW20}, $H^{\bullet}(J_{n + 1}) \cong q^{-2n} \strand{(1^{n + 1})}$ as $\mathbb{Z}_q \times \mathbb{Z}_t$-graded bimodules. Again by the digon removal relation of Proposition \ref{prop: blamgon_removal}, we obtain

    \[
    [n]!H^{\bullet}(FT_{\aa}) \cong q^{-2n} (_{\aa} M_{(\blam, 1)}) \star (_{(\blam, 1)} S_{\aa}) \cong [n]! q^{-2n} \strand{\aa}.
    \]

    The desired result then follows from the fact that $\SSBim_{\aa}^{\aa}$ is a Krull--Schmidt category.
\end{proof}

\begin{lem} \label{lem: qi}
Retain notation as in Lemma \ref{lem: n1hopf_thick_twist}, and set
$g_{n + 1}(\mathbb{X}_{n + 1}, \mathbb{X}'_{n + 1}) := e_n(\mathbb{X}_n) - x'_{n + 1}g_n(\mathbb{X}_{n + 1}, \mathbb{X}'_{n + 1})$. Let $\psi \colon q^{2n} \strand{\aa} \to C_n$
be the map depicted in blue below:

\begin{center}
\begin{tikzcd}[sep=huge]
    q^{2n} \strand{\aa} \arrow[d, blue, "g_{n + 1}"] \\
    q^nW_{\aa} \arrow[r, "x_{n + 1} - x'_{n + 1}"] & tq^{n - 2} W_{\aa}
    \arrow[r, "unzip"] & t^2 q^{-2} \strand{\aa}
\end{tikzcd}
\end{center}

Then $\psi$ is a quasi-isomorphism.
\end{lem}

\begin{proof}
By the same computation as in the proof of Lemma \ref{lem: n1hopf_thick_twist},
we have $(x_{n + 1} - x'_{n + 1}) g_{n + 1} \equiv e_{n + 1}(\mathbb{X}_n) -
e_{n + 1}(\mathbb{X}'_n) \equiv 0$ modulo $I_{\aa}$, so $\psi$ is a chain map.
Since $g_{n + 1} \neq 0 \in W_{\aa}$, $\psi$ is nonzero as well.

On the other hand, by Lemma \ref{lem: n1_hopf_homology}, there is a chain map
$\psi'$ with the same domain and codomain as $\psi$ given by inclusion of $0^{th}$ cohomology.
A straightforward computation gives that the space of degree $q^n$ bimodule homomorphisms
from $\strand{\aa}$ to $W_{\aa}$ is $1$-dimensional, so $\psi = c \psi'$ for some scalar $c \in R$. In particular, $\psi$ and $\psi'$ induce the same maps on cohomology up to a scalar. Now, $\psi'$ induces an isomorphism on cohomology by design. Since $R$ is a field, $\psi$ must induce an isomorphism on cohomology as well.
\end{proof}

A quick check shows that $\psi$ lifts without modification to a map of $F_u^{(n)}$-curved complexes

\[
\psi \colon q^{2n}\strand{\aa} \otimes_R R[\mathbb{U}_n] \to C_n^u.
\]
Now, let $u_{n + 1}$ be a deformation parameter of degree $\mathrm{deg}(u_{n + 1}) = t^2q^{-2(n + 1)}$, and set $\mathbb{U}_{n + 1} := \mathbb{U}_n \cup \{u_{n + 1}\}$. We may consider the same curvature $F_u^{(n)}$ as constructed from this larger alphabet:

\[
F_u^{(n)}(\mathbb{X}_{n + 1}, \mathbb{X}_{n + 1}') = \sum_{i = 1}^n (e_i(\mathbb{X}_n) - e_i(\mathbb{X}'_n)) \otimes u_i \in \mathrm{Sym}^{\aa}(\mathbb{X}_{n + 1}, \mathbb{X}'_{n + 1}) \otimes_R R[\mathbb{U}_{n + 1}].
\]
The curved complexes $\strand{\aa} \otimes_R R[\mathbb{U}_n], C_n^u \in \YS_{F_u^{(n)}}(\SSBim; R[\mathbb{U}_n])$ lift automatically to curved complexes

\[
\strand{\aa} \otimes_R R[\mathbb{U}_{n + 1}], C_n^u \otimes_R R[u_{n + 1}] \in \YS_{F_u^{(n)}}(\SSBim; R[\mathbb{U}_{n + 1}]).
\]

\begin{prop} \label{prop: inf_proj_recurs_thin}
	Set 
	
	\[
	\Phi := \iota \otimes 1 + \psi \otimes u_{n + 1} \colon t^2q^{-2} \strand{\aa} \otimes_R R[\mathbb{U}_{n + 1}] \to C_n^u \otimes_R R[u_{n + 1}].
	\]
	Then there is a homotopy equivalence of chain complexes 
    
    \[
    \mathrm{Cone}(\ufray{n}{\blam}(\Phi)) \simeq q^n \left( P^{(1^{n + 1})}_{(1^{n + 1})} \right)^{\vee} 
    \]
    in $\CS(\SBim_n)$.
\end{prop}

\begin{proof}
    We consider $\ufray{n}{\blam}(\Phi)$ as a chain map between uncurved complexes by unrolling, as described in Remark \ref{rem: infinite_ladder}. In fact, because the variable $u_{n + 1}$ does not play a role in the curvature $F_u^{(n)}$, we can unroll along this variable \textit{before} applying $\ufray{n}{\blam}$ and consider $\Phi$ as a morphism in the category $\YS_{F_u^{(n)}}(\SSBim; R[\mathbb{U}_n])$. Having done this, we can begin simplifying $\mathrm{Cone}(\Phi)$ using Gaussian elimination in that category. We can depict this cone as a semi-infinite ladder in which the degree of $u_{n + 1}$ increases as we travel downwards:

    \begin{center}
    \begin{tikzcd}[row sep=small]
        & q^nW_{\aa} \otimes_R R[\mathbb{U}_n] \arrow[dr, "x_{n + 1} - x'_{n + 1}", harpoon, shift left]
        & & \\
        & & tq^{n - 2} W_{\aa} \otimes_R R[\mathbb{U}_n]
        \arrow[ul, harpoon, shift left, blue, "-\sum_{i = 1}^n g_i u_i"] \arrow[dr, "unzip"] & \\
        tq^{-2} \strand{\aa} \otimes_R R[\mathbb{U}_n] \arrow[rrr, "1"] \arrow[dr, "g_{n + 1} u_{n + 1}"']
        & & & t^2q^{-2} \strand{\aa} \otimes_R R[\mathbb{U}_n] \\
        & u_{n + 1} q^n W_{\aa} \otimes_R R[\mathbb{U}_n]
        \arrow[dr, "x_{n + 1} - x'_{n + 1}", harpoon, shift left] & & \\
        & & u_{n + 1} tq^{n - 2} W_{\aa} \otimes_R R[\mathbb{U}_n]
        \arrow[ul, harpoon, shift left, blue, "-\sum_{i = 1}^n g_i u_i"] \arrow[dr, "unzip"] & \\
        u_{n + 1} tq^{-2} \strand{\aa} \otimes_R R[\mathbb{U}_n] \arrow[rrr, "1"]
        \arrow[dr, "g_{n + 1} u_{n + 1}"'] & & & u_{n + 1} t^2q^{-2} \strand{\aa} \otimes_R R[\mathbb{U}_n] \\
        & u_{n + 1}^2 q^n W_{\aa} \otimes_R R[\mathbb{U}_n]
        \arrow[dr, "x_{n + 1} - x'_{n + 1}", harpoon, shift left] & & \\
        & & u_{n + 1}^2 tq^{n - 2} W_{\aa} \otimes_R R[\mathbb{U}_n]
        \arrow[ul, harpoon, shift left, blue, "-\sum_{i = 1}^n g_i u_i"] \arrow[dr, "unzip"] & \\
        u_{n + 1}^2 tq^{-2} \strand{\aa} \otimes_R R[\mathbb{U}_n] \arrow[rrr, "1"]
        \arrow[dr, "g_{n + 1} u_{n + 1}"'] & & & u_{n + 1}^2 t^2q^{-2} \strand{\aa} \otimes_R R[\mathbb{U}_n] \\
        \dots & \dots & \dots & \dots
    \end{tikzcd}
    \end{center}

    \vspace{1em}

    After Gaussian elimination along each horizontal identity arrow,
    we are left with a new ladder of the form

    \begin{center}
    \begin{tikzcd}[sep=huge]
    q^nW_{\aa} \otimes_R R[\mathbb{U}_n] \arrow[r, shift left, harpoon, "x_{n + 1} - x'_{n + 1}"]
    & tq^{n - 2} W_{\aa} \otimes_R R[\mathbb{U}_n]
    \arrow[l, shift left, harpoon, blue, "-\sum_{i = 1}^n g_iu_i"]
    \arrow[dl, "-g_{n + 1} u_{n + 1}"] \\
    u_{n + 1} q^nW_{\aa} \otimes_R R[\mathbb{U}_n]
    \arrow[r, shift left, harpoon, "x_{n + 1} - x'_{n + 1}"]
    & u_{n + 1} tq^{n - 2} W_{\aa} \otimes_R R[\mathbb{U}_n]
    \arrow[l, shift left, harpoon, blue, "-\sum_{i = 1}^n g_iu_i"]
    \arrow[dl, "-g_{n + 1} u_{n + 1}"] \\
    u_{n + 1}^2 q^nW_{\aa} \otimes_R R[\mathbb{U}_n]
    \arrow[r, shift left, harpoon, "x_{n + 1} - x'_{n + 1}"]
    & u_{n + 1}^2 tq^{n - 2} W_{\aa} \otimes_R R[\mathbb{U}_n]
    \arrow[l, shift left, harpoon, blue, "-\sum_{i = 1}^n g_iu_i"]
    \arrow[dl, "-g_{n + 1} u_{n + 1}"] \\
    \dots & \dots
    \end{tikzcd}
    \end{center}

    \vspace{1em}

    We can express this complex more compactly as follows.
    Let $\theta_{n + 1}$ be a formal odd variable of degree $\mathrm{deg}(\theta_{n + 1})
    = q^{-2}t$; then we have

    \[
    \mathrm{Cone}(\Phi) \simeq \mathrm{tw}_{\overline{\tau}}
    \left( q^n W_{\aa} \otimes_R R[\theta_{n + 1}] \otimes_R
    R[\mathbb{U}_{n + 1}] \right); \quad \overline{\tau} = (x_{n + 1} - x'_{n + 1})
    \otimes \theta_{n + 1} \otimes 1 - \sum_{i = 1}^{n + 1} g_i
    \otimes \theta_{n + 1}^{\vee} \otimes u_i.
    \]

    After applying $\ufray{n}{\blam}$ and regrading, we obtain

    \begin{align*}
        \mathrm{Cone}(\ufray{n}{\blam}(\Phi)) & = \ufray{n}{\blam}(\mathrm{Cone}(\Phi)) \simeq
        q^n \mathrm{tw}_{\tau_n} \left( W_{(\blam, 1)} \otimes_R R[\Theta_{n + 1}, \mathbb{U}_{n + 1}] \right); \\
        \tau_n & = \sum_{j = 1}^{n + 1} (x_i - x'_i) \otimes \theta_j -
        \sum_{i = 1}^n \sum_{j = 1}^n a_{ij1}^{\blam}(\mathbb{X}_n, \mathbb{X}'_n)
        \otimes \theta^{\vee}_j u_i - \sum_{i = 1}^{n + 1} g_i \otimes
        \theta_{n + 1}^{\vee} u_i.
    \end{align*}
    Here $\Theta_{n + 1} = \{\theta_1, \dots, \theta_n, \theta_{n + 1}\}$, where $\theta_1, \dots, \theta_n$ are the usual deformation parameters of degree $\mathrm{deg}(\theta_i) = q^{-2i}t$ associated to the functor $\ufray{n}{\blam}$.
    
    For ease of notation, we set 
    
    \[
    a^{\blam}_{i, n+1, 1} := g_i(\mathbb{X}_{n + 1}, \mathbb{X}'_{n + 1})
    \]
    for each $1 \leq i \leq n$ and
    
    \[
    a^{\blam}_{n + 1, j, 1} := 0
    \]
    for each $1 \leq j \leq n$. In this notation, we may rewrite $\tau_n$ as
    
    \[
    \tau_n = \sum_{j = 1}^{n + 1} (x_i - x'_i) \otimes \theta_j
        - \sum_{i = 1}^{n + 1} \sum_{j = 1}^{n + 1} a_{ij1}^{\blam}
        \otimes \theta^{\vee}_j u_i
    \]

    It remains to compare $\mathrm{Cone}(\ufray{n}{\blam}(\Phi))$
    to $q^n (P^{(1^{n + 1})}_{(1^{n + 1})})^{\vee} = q^n \ufray{n + 1}{(1^{n + 1})}
    (\strand{(n + 1)})$. The latter is a convolution with the same underlying
    bimodule as the former but a different twist:

    \begin{align*}
        q^n (P^{(1^{n + 1})}_{(1^{n + 1})})^{\vee} & =
        q^n \mathrm{tw}_{\tau_{n + 1}} \left( W_{(\blam, 1)} \otimes_R
        R[\Theta_{n + 1}, \mathbb{U}_{n + 1}] \right); \\
        \tau_{n + 1} & = \sum_{j = 1}^{n + 1} (x_i - x'_i) \otimes \theta_j
        - \sum_{i = 1}^{n + 1} \sum_{j = 1}^{n + 1} a_{ij1}^{(\blam, 1)}
        (\mathbb{X}_{n + 1}, \mathbb{X}'_{n + 1})
        \otimes \theta^{\vee}_j u_i.
    \end{align*}
    Since the $\mathbb{U}_{n + 1}$-degree $0$ components of $\tau_n$ and $\tau_{n + 1}$ agree,
    to compare $\tau_n$ and $\tau_{n + 1}$, it suffices to compute
    $a^{(\blam, 1)}_{ij1} - a^{\blam}_{ij1}$ for each $i, j$.
    We use Lemma \ref{lem: thin_a_compare} throughout.

    \vspace{1em}
    
    \textbf{Case 1: $j = n + 1$.} In this case, for each $1 \leq i \leq n + 1$, we have

    \begin{align*}
    a^{(\blam, 1)}_{ij1} - a^{\blam}_{ij1} & = e_{i - 1}(\mathbb{X}_n) - g_i \\
    & = e_{i - 1}(\mathbb{X}_n) - (e_{i - 1}(\mathbb{X}_n) - x'_{n + 1} g_{i - 1}) \\
    & = x_{n + 1}' a^{\blam}_{i - 1, j, 1}.
    \end{align*}

    \vspace{1em}
    
    Here we take as convention that  $a_{0j1}^{\blam}(\mathbb{X}_n, \mathbb{X}'_n) = 0$ for each $j$.

    \textbf{Case 2: $i = n + 1$, $j \neq n + 1$.} In this case $a^{\blam}_{ij1} = 0$
    for each $j$, so we obtain

    \begin{align*}
    a^{(\blam, 1)}_{n + 1,j1} - a^{\blam}_{n + 1,j1} & = x'_{n + 1} a_{nj1}^{\blam}
    + a_{n + 1, j1}^{\blam} \\
    & = x_{n + 1}' a^{\blam}_{n, j, 1}.
    \end{align*}

    \vspace{1em}

    \textbf{Case 3: $i, j \leq n$.} In this case, Lemma \ref{lem: thin_a_compare}
    immediately gives

    \[
    a^{(\blam, 1)}_{ij1} - a^{\blam}_{ij1} = x_{n + 1}' a^{\blam}_{i - 1, j, 1}(\mathbb{X}_n, \mathbb{X}'_n)
    \]

    \vspace{1em}

    In all cases, we obtain
    $a^{(\blam, 1)}_{ij1} - a^{\blam}_{ij1} = x_{n + 1}' a^{\blam}_{i - 1, j, 1}$.
    Since this expression is uniform for each $i, j$, we may transform $\tau_n$ into
    $\tau_{n + 1}$ via a change of basis in the variables $\{u_i\}$.
    Formally, consider the dg-module endomorphism of $W_{(\blam, 1)} \otimes_R
    R[\mathbb{U}_{n + 1}]$ (with trivial differential) of the form
    
    \[
    u_i \mapsto \sum_{k = 0}^{n + 1 - i} (-x'_{n + 1})^k u_{j + k}.
    \]
    This is an isomorphism with
    inverse
    
    \[
    u_i \mapsto u_i + x'_{n + 1} u_{i + 1}.
    \]
    Applying this isomorphism termwise
    in each $\Theta_{\blam}$-degree to $q^{-n} \ufray{n}{\blam}(\mathrm{Cone}(\Phi))$
    preserves the underlying doubly-graded bimodule and takes $\tau_n$ to $\tau_{n + 1}$,
    hence exhibits the desired isomorphism.
\end{proof}

We can again repackage the statement of Proposition \ref{prop: inf_proj_recurs_thin} to emphasize the role of $\left( P_{(1^n)}^{(1^n)} \right)^{\vee}$. Post-composing with $s^u$ in each $u_{n + 1}$-degree induces a homotopy equivalence $\mathrm{Cone}(s^u \circ \Phi) \simeq \mathrm{Cone}(\Phi)$. After applying $\ufray{n}{\blam}$, we may consider $\ufray{n}{\blam}(s^u \circ \Phi)$ as a map from $t^2q^{-2} \ufray{n}{\blam}(\strand{\aa}) \otimes_R R[u_{n + 1}]$ to $\ufray{n}{\blam}(FT^u_{\aa}) \otimes_R R[u_{n + 1}]$. The former is exactly $t^2q^{-2} \left( \left( \infproj{\blam} \boxtimes \strand{(1)} \right) \otimes_R R[u_{n + 1}] \right)$, and by Corollary \ref{cor: braid_inf_slide}, the latter is homotopy equivalent to $\left( \left( \infproj{\blam} \boxtimes \strand{(1)} \right) \star J_{n + 1} \right) \otimes_R R[u_{n + 1}]$. After shifting, we immediately obtain the following:

\begin{cor} \label{cor: inf_proj_recurs}
    Let $u_{n + 1}$ be a formal variable of degree $\mathrm{deg}(u_{n + 1}) = t^2q^{-2(n + 1)}$.
    Then there is a map of chain complexes
    
    \[
    \kappa^u \colon t^2q^{-2 - n} \left( \left( \infproj{\blam} \boxtimes \strand{(1)} \right) \otimes_R R[u_{n + 1}] \right)
    \to q^{-n} \left( \left( \infproj{\blam} \boxtimes \strand{(1)} \right) \star J_{n + 1}^y \right) \otimes_R R[u_{n + 1}]
    \]
    satisfying $\left( P_{(1^{n + 1})}^{(1^{n + 1})} \right)^{\vee} \simeq \mathrm{Cone}(\kappa^u)$.
\end{cor}

\subsubsection{Deformed Infinite Frayed Projector}

Again, the entire story of Section \ref{subsec: inf_proj} admits a deformed analog. This construction essentially just merges the deformations of Sections \ref{subsec: passive_rec} and \ref{subsec: inf_proj}. Since the proofs in this construction exactly mimic the proofs in those Sections, we omit them throughout, commenting only when modifications are necessary.

Let $FT^{y, u}_{\aa}$ denote the
deformed full twist on $2$ strands with $\Delta e$-curvature on each strand. We
retain the notation $y_{n + 1}$ for the deformation parameter on the $1$-labeled strand
and $\mathbb{U}_n$ for the deformation alphabet on the $n$-labeled strand.

\begin{lem} \label{lem: c_nyu}
Let $C_n^{y, u}$ denote the strict deformation of $C_n$ via both twists indicated in blue in Lemmata \ref{lem: c_ny} and \ref{lem: c_nu}. Then the homotopy equivalence of Remark \ref{rem: hrw_min_model} admits a curved lift to a homotopy equivalence

\begin{center}
\begin{tikzcd}[sep=huge]
FT^{y, u}_{\aa} \arrow[r, shift left, harpoon, "p^{y, u}"] & C^{y, u}_n \arrow[l, shift left, harpoon, "s^{y, u}"]
\end{tikzcd}
\end{center}
of $F_y^{(1)} + F_u^{(n)}$-curved complexes.
\end{lem}

The map $\iota$ again lifts without modification to a map of curved complexes
$\iota \colon t^2q^{-2} \strand{(\aa)} \otimes_R R[y_{n + 1}, \mathbb{U}_n] \hookrightarrow C^{y, u}_n$, and we can once again compute its cone by Gaussian elimination.

\begin{lem} \label{lem: iota_gauss_elim_yu}
There is a homotopy equivalence of $F_y^{(1)} + F_u^{(n)}$-curved complexes

\begin{center}
\begin{tikzcd}[sep=huge]
\mathrm{Cone}(\iota) \simeq q^n W_{\aa} \otimes_R R[y_{n + 1}, \mathbb{U}_n] \arrow[r, shift left, harpoon, "x_{n + 1} - x'_{n + 1}"]
& tq^{n - 2}W_{\aa} \otimes_R R[y_{n + 1}, \mathbb{U}_n] \arrow[l, shift left, blue, harpoon, "y_{n + 1} -\sum_{i = 1}^n g_i u_i"] .
\end{tikzcd}
\end{center}
\end{lem}

A quick check shows that the map $\psi$ of Lemma \ref{lem: qi} lifts without modification to a map of $F_y^{(1)} + F_u^{(n)}$-curved complexes

\[
\psi \colon q^{2n}\strand{\aa} \otimes_R R[y_{n + 1}, \mathbb{U}_n] \to C_n^{y, u}.
\]

Now let $u_{n + 1}$, $\mathbb{U}_{n + 1}$ be defined as in Section \ref{subsec: inf_proj}. As in that Section, we automatically obtain curved complexes

\[
\strand{\aa} \otimes_R R[y_{n + 1}, \mathbb{U}_{n + 1}], C_n^{y, u} \otimes_R R[u_{n + 1}] \in \YS_{F_y^{(1)} + F_u^{(n)}}(\SSBim; R[y_{n + 1}, \mathbb{U}_{n + 1}]).
\]

\begin{prop} \label{prop: yinf_proj_recurs_thin}
	Set 
	
	\[
	\Phi := \iota \otimes 1 + \psi \otimes u_{n + 1} \colon t^2q^{-2} \strand{\aa} \otimes_R R[y_{n + 1}, \mathbb{U}_{n + 1}] \to C_n^{y, u} \otimes_R R[u_{n + 1}].
	\]
	Then there is a homotopy equivalence of $F_y^{(1^{n + 1})}$-curved complexes 
    
    \[
    \mathrm{Cone}(\yufray{n}{\blam}(\Phi)) \simeq q^n \left( P^{(1^{n + 1}), y}_{(1^{n + 1})} \right) ^{\vee}.
    \]
\end{prop}

\begin{proof}
Identical to the proof of Proposition \ref{prop: inf_proj_recurs_thin}. The presence of the variable $y_{n + 1}$ does not affect the Gaussian elimination involved in simplifying $\mathrm{Cone}(\Phi)$. Applying $\yufray{n}{\blam}$, rather than $\ufray{n}{\blam}$, to $\mathrm{Cone}(\Phi)$ only adds $\mathbb{U}_{n + 1}$-degree $0$ components to the resulting connection; these components agree exactly with the corresponding components of the connection on $\left( P^{(1^{n + 1}), y}_{(1^{n + 1})} \right)^{\vee}$. Finally, the same change of basis in the variables $\{u_i\}$ also works to establish this homotopy equivalence without modification.
\end{proof}

We can once again repackage the statement of Proposition \ref{prop: yinf_proj_recurs_thin} to emphasize the role of $\left( P^{(1^n), y}_{(1^n)} \right)^{\vee}$. Post-composing with $s^{y, u}$ in each $u_{n + 1}$-degree induces a homotopy equivalence $\mathrm{Cone}(s^{y, u} \circ \Phi) \simeq \mathrm{Cone}(\Phi)$. After applying $\yufray{n}{\blam}$, we may consider $\yufray{n}{\blam}(s^{y, u} \circ \Phi)$ as a map from $t^2q^{-2} \yufray{n}{\blam}(\strand{\aa}) \otimes_R R[y_{n + 1}, u_{n + 1}]$ to $\yufray{n}{\blam}(FT^{y, u}_{\aa}) \otimes_R R[u_{n + 1}]$. The former is exactly 
\[
t^2q^{-2} \left( \left( \yinfproj{\blam} \boxtimes \strand{(1)} \right) \otimes_R R[y_{n + 1}] \right) \otimes_R R[u_{n + 1}],
\]
and by Corollary \ref{cor: braid_yinf_slide}, the latter is homotopy equivalent to 
\[
\left( \left( \left( \yinfproj{\blam} \boxtimes \strand{(1)} \right) \otimes_R R[y_{n + 1}] \right) \star J_{n + 1}^y \right) \otimes_R R[u_{n + 1}].
\]
After shifting, we immediately obtain the following:

\begin{cor} \label{cor: yinf_proj_recurs}
    Let $u_{n + 1}$ be a formal variable of degree $\mathrm{deg}(u_{n + 1}) = t^2q^{-2(n + 1)}$.
    Then there is a map of chain complexes
    
    \begin{multline*}
    \kappa^{y, u} \colon t^2q^{-2 - n} \left( \left( \yinfproj{\blam} \boxtimes \strand{(1)} \right) \otimes_R R[y_{n + 1}] \right) \otimes_R R[u_{n + 1}] \\
    \to q^{-n} \left( \left( \left( \yinfproj{\blam} \boxtimes \strand{(1)} \right) \otimes_R R[y_{n + 1}] \right) \star J_{n + 1}^y \right) \otimes_R R[u_{n + 1}]
    \end{multline*}
    satisfying $\left( P_{(1^{n + 1}), y}^{(1^{n + 1})} \right)^{\vee} \simeq \mathrm{Cone}(\kappa^{y, u})$.
\end{cor}

\subsection{Infinite Frayed Projectors are Categorical Idempotents} \label{subsec: idem}

We have promised throughout that in the case $\blam = (1^n)$, our infinite frayed projector
$\infproj{\blam}$ is homotopy equivalent to the (unital) Abel--Hogancamp projector
$P_{(1^n)}^{\vee}$. In this section we make good on this promise by showing that
both complexes satisfy the same universal property of unital idempotents.

We begin by establishing some notation. Recall that we have fixed $n \geq 1$, $N = n + 1$,
$\aa = \bb = (n, 1)$, $\blam = (1^n) \vdash n$. For each \textcolor{revisions}{composition} $\bnu \vdash n$ or
$\bnu \vdash N$, let $\mathcal{I}_{\bnu}$ denote the full subcategory of
$K(\SSBim_{\bnu}^{\bnu})$ generated by complexes whose chain bimodules are
(direct sums of shifts of) $W_{\bnu}$. It is well known that $\mathcal{I}_{\bnu}$ is a
two-sided tensor ideal of $K(\SSBim_{\bnu}^{\bnu})$ (under the horizontal composition $\star$).
Let $\mathcal{I}_{\bnu}^{\perp}$ (respectively $^{\perp} \mathcal{I}_{\bnu}$) denote
the full subcategory of $K(\SBim_n)$ of complexes $C$ satisfying
$C \star W_{\bnu} \simeq 0$ (respectively $W_{\bnu} \star C \simeq 0$).

The following is Theorem 2.17 in \cite{AH17}
\footnote{Adapted to describe a unital, rather than counital, idempotent.}:

\begin{thm} \label{thm: AH_projector}
    There is a complex $P_{1^n}^{\vee} \in K^+(\SBim_n)$ and a chain map
    $\eta_n \colon \strand{(1^n)} \to P_{1^n}^{\vee}$ satisfying

    \begin{itemize}
        \item[(P1)] $P_{1^n}^{\vee} \in \mathcal{I}_{\blam}$

        \item[(P2)] $\mathrm{Cone}(\eta_n) \in \mathcal{I}_{\blam}^{\perp}
        \cap ^{\perp} \mathcal{I}_{\blam} $.
    \end{itemize}

    Furthermore, the pair $(P_{1^n}^{\vee}, \eta)$ is uniquely characterized by (P1) and (P2)
    up to canonical homotopy equivalence, in the sense that for any pair $(Q, \eta')$
    satisfying (P1) and (P2), there is a unique (up to homotopy) chain map
    $\phi \colon P_{1^n}^{\vee} \to Q$ satisfying $\eta' \sim \phi \circ \eta_n$,
    and $\phi$ is a homotopy equivalence.
\end{thm}

In the language of \cite{Hog17}, Properties (P1) and (P2) of Theorem
\ref{thm: AH_projector} can be rephrased as the statement that $P_{1^n}^{\vee}$
is a \textit{unital idempotent} projecting onto the cell ideal $\mathcal{I}_{\blam}$.
The general theory of such idempotents then gives the resulting uniqueness up to unique
(up to homotopy) homotopy equivalence;
we refer the interested reader to \cite{Hog17} for further details.

Our goal in this section is to prove that (a shift of) our $\infproj{\blam}$
is homotopy equivalent to $P_{1^n}^{\vee}$ by showing that $\infproj{\blam}$
also satisfies the conditions of Theorem \ref{thm: AH_projector}.

\begin{thm} \label{thm: inf_proj_idem}
    There is a chain map $\tilde{\eta}_n \colon \strand{(1^n)} \to q^{-\frac{n(n - 1)}{2}} \infproj{\blam}$
    satisfying the conditions of Theorem \ref{thm: AH_projector}.
\end{thm}

That $q^{-\frac{n(n - 1)}{2}} \infproj{\blam} = q^{-\frac{n(n - 1)}{2}} \ufray{n}{\blam}(\strand{(n)})$ satisfies (P1) follows immediately from the definition of $\ufray{n}{\blam}$. The construction of a map $\tilde{\eta}_n$ satisfying (P2) is lengthy and occupies the remainder of this Section.

We begin by establishing a useful criterion for membership in $\mathcal{I}_{\bnu}^{\perp}
\cap ^{\perp} \mathcal{I}_{\bnu}$ for an arbitrary \textcolor{revisions}{composition} $\bnu \vdash N$. The
following three Lemmata are taken directly from Section 2.2 of \cite{AH17} with
slight adaptations reflecting a shift from Soergel bimodules to singular
Soergel bimodules; setting $\bnu = (1^n)$ recovers the statements in that work.

\begin{lem} \label{lem: ah_perp_1}
    The functors $W_{\bnu} \star -$ and $- \star W_{\bnu} \colon \SSBim_{\bnu}^{\bnu}
    \to \SSBim_{\bnu}^{\bnu}$ are exact.
\end{lem}

\begin{proof}
    We consider the functor $W_{\bnu} \star -$; the corresponding proof for $- \star W_{\bnu}$ is exactly analogous. Since $W_{\bnu} = (_{\bnu} S_{(N)}) \star (_{(N)} M_{\bnu})$, we can express $W_{\bnu} \star -$ as the composition of functors

    \[
    W_{\bnu} \star - = \left( (_{\bnu} S_{(N)}) \star - \right) \circ \left( (_{(N)} M_{\bnu}) \star - \right).
    \]
    We have already observed that the left factor of this composition is exact in Proposition \ref{prop: rfray_exact}. Meanwhile, $(_{(N)} M_{\bnu}) \cong \mathrm{Sym}(\mathbb{X}_N) \otimes_{\mathrm{Sym}(\mathbb{X}_N)} \mathrm{Sym}^{\bnu}(\mathbb{X}_N)$ is free (in fact, cyclic) as a right $\mathrm{Sym}^{\bnu}(\mathbb{X}_N)$-module, it is also flat, and the right factor is exact as well.
\end{proof}

\begin{lem} \label{lem: ah_perp_2}
    Let $C \in \mathcal{I}_{\bnu}^-$ be an arbitrary complex. Then $C$ is contractible if and only
    if $C$ is acyclic.
\end{lem}

\begin{proof}
    All contractible complexes are acyclic, so assume conversely that $C$ is acyclic.
    Up to homotopy equivalence, we may assume that each chain bimodule of $C$ is a direct sum
    of shifted copies of $W_{\bnu}$. The bimodule action of $\mathrm{Sym}^{\bnu}(\mathbb{X}_N, \mathbb{X}'_N)$
    on $C$ factors through $\mathrm{Sym}^{\bnu}(\mathbb{X}_N) \otimes_{\mathrm{Sym}(\mathbb{X}_N)} \mathrm{Sym}^{\bnu}(\mathbb{X}_N)$,
    so we may regard $C$ as a complex of $\mathrm{Sym}^{\bnu}(\mathbb{X}_N) \otimes_{\mathrm{Sym}(\mathbb{X}_N)} \mathrm{Sym}^{\bnu}(\mathbb{X}_N)$-modules.
    Recall that up to regrading, there is a ring isomorphism 
    
    \[
    W_{\bnu} \cong \mathrm{Sym}^{\bnu}(\mathbb{X}_N) \otimes_{\mathrm{Sym}(\mathbb{X}_N)} \mathrm{Sym}^{\bnu}(\mathbb{X}_N),
    \]
    so $C$ is in fact a bounded above acyclic complex of free
    $W_{\bnu}$-modules. Such complexes are known to be contractible.

    Now, let $h \in \END^{-1}(C)$ denote the contracting homotopy satisfying $[d_C, h] = 0$ in the category
    $\mathrm{Ch}(W_{\bnu} - \mathrm{Mod})$. Each component
    of $h$ is a homomorphism of $W_{\bnu}$-modules,
    and so it commutes with the $\mathrm{Sym}^{\bnu}(\mathbb{X}_N) \otimes_{\mathrm{Sym}(\mathbb{X}_N)} \mathrm{Sym}^{\bnu}(\mathbb{X}_N)$-action on
    each chain bimodule of $C$. Then $h$ \textit{also} commutes with the original
    $\mathrm{Sym}^{\bnu}(\mathbb{X}_N, \mathbb{X}'_N)$ action on each chain bimodule of $C$,
    so $h$ is a perfectly good contracting homotopy in the
    category $\CS(\SSBim_{\bnu}^{\bnu})$ as well.
\end{proof}

\begin{cor} \label{cor: ah_perp_3}
    Let $Z \in \mathrm{K}^- \left( \SSBim_{\bnu}^{\bnu} \right)$ be acyclic.
    Then $Z \in \mathcal{I}_{\bnu}^{\perp} \cap ^{\perp} \mathcal{I}_{\bnu}$.
\end{cor}

\begin{proof}
    Since $\mathcal{I}_{\bnu}$ is a two-sided tensor ideal,
    we have $W_{\bnu} \star Z, Z \star W_{\bnu} \in \mathcal{I}_{\bnu}^-$.
    Each of these complexes is acylcic by Lemma \ref{lem: ah_perp_1},
    and is therefore contractible by Lemma \ref{lem: ah_perp_2}.
\end{proof}

We will use the following consequence of Corollary \ref{cor: ah_perp_3}
repeatedly in our proof of Theorem \ref{thm: inf_proj_idem}.

\begin{prop} \label{prop: rfray_perp_cond}
    Let $C, C' \in \CS^-(\SSBim_{\aa}^{\bb})$ be given,
    and suppose $f \colon C \to C'$ is a quasi-isomorphism.
    Then $\mathrm{Cone}(\rfray{n}{\bnu}(f)) \in \mathcal{I}_{(\bnu, 1)}^{\perp}
    \cap ^{\perp} \mathcal{I}_{(\bnu, 1)}$ for each \textcolor{revisions}{composition} $\bnu \vdash n$.
\end{prop}

\begin{proof}
    Since $f$ is a quasi-isomorphism, $\mathrm{Cone}(f)$ is acyclic.
    $\rfray{n}{\bnu}$ is exact by Proposition \ref{prop: rfray_exact}, so $\rfray{n}{\bnu}(\mathrm{Cone}(f))$
    is acyclic as well. The desired result follows from a direct application of
    Corollary \ref{cor: ah_perp_3} in the category
    $K^- \left( \SSBim_{(\bnu, 1)}^{(\bnu, 1)} \right)$.
\end{proof}

We wish to promote the statement of Proposition \ref{prop: rfray_perp_cond} to its analog
for the functors $\fray{n}{\bnu}$ and $\ufray{n}{\bnu}$. To do this, we use the technology of
\textit{simultaneous simplification}.

\begin{defn}
    A poset is \textit{upper-finite} if it satisfies the ascending chain condition (ACC):
    any chain $i_1 < i_2 < \dots $ is finite. A poset is \textit{lower-finite} if it
    satisfies the descending chain condition (DCC): any chain $i_1 > i_2 > \dots $ is finite.
\end{defn}

The following is Proposition 4.20 from \cite{EH17a}; we simply state the result
without proof and invite the interested reader to see that work for further discussion.

\begin{prop} \label{prop: simul_simp}
    Let $(I, \leq)$ be an upper finite poset and
    $C = \mathrm{tw} \left( \bigoplus_{i \in I} C_i \right)$ for some complexes $C_i$ and
    some unspecified twist. Suppose for each $i \in I$ there are complexes $D_i$ satisfying
    $C_i \simeq D_i$. Then $C \simeq \mathrm{tw} \left( \bigoplus_{i \in I} D_i \right)$
    for some twist.
    
    Similarly, let $(I, \leq)$ be a lower finite poset
    and $C = \mathrm{tw} \left( \prod_{i \in I} C_i \right)$ for some complexes $C_i$ and
    some unspecified twist. Suppose for each $i \in I$ there are complexes $D_i$ satisfying
    $C_i \simeq D_i$. Then $C \simeq \mathrm{tw} \left( \prod_{i \in I} D_i \right)$
    for some twist.
\end{prop}

\begin{rem} \label{rem: locally_finite}
    Strictly speaking, the convolutions we consider below are always (twists of)
    direct \textit{sums} of the form $\mathrm{tw} \left( \bigoplus_{i \in I} C_i \right)$,
    never direct products. That said, we will only ever apply the version of
    Proposition \ref{prop: simul_simp} dealing with lower finite posets to
    \textit{homologically locally finite} complexes in the sense of \cite{EH17a}.
    Roughly speaking, this condition guarantees that each direct product chain bimodule has
    only finitely many nonzero factors. This renders the distinction between
    direct product and direct sum irrelevant up to isomorphism.
\end{rem}

\begin{prop} \label{prop: fray_perp_cond}
    Let $C, C', f$ be as in Proposition \ref{prop: rfray_perp_cond}.
    Then $\mathrm{Cone}(\fray{n}{\bnu}(f)) \in \mathcal{I}_{(\bnu, 1)}^{\perp}
    \cap ^{\perp} \mathcal{I}_{(\bnu, 1)}$ for each \textcolor{revisions}{composition} $\bnu \vdash n$.
\end{prop}

\begin{proof}
    By definition, we have

    \[
    \mathrm{Cone}(\fray{n}{\bnu}(f)) \cong \fray{n}{\bnu}
    \left( \mathrm{Cone}(f) \right) = \mathrm{tw}_{\alpha}
    \left( \rfray{n}{\bnu} \left(\mathrm{Cone}(f) \right) \otimes_R R[\Theta_{\bnu}] \right).
    \]
    
    Consider the monomial basis $I$ of $R[\Theta_{\bnu}]$ as a poset ordered by homological degree.
    Then $I$ is finite, so it is certainly upper finite. Since $\mathrm{Cone}(f) \simeq 0$
    by Propositon \ref{prop: rfray_perp_cond}, Proposition \ref{prop: simul_simp} gives that
    $\mathrm{Cone}(\fray{n}{\bnu}(f))$ is homotopy equivalent to a convolution of $0$
    complexes and so is itself contractible.
\end{proof}

The same reasoning does \textit{not} immediately extend to the functor $\ufray{n}{\bnu}$,
as the monomial basis of $R[\mathbb{U}_n]$ is not upper finite, but \textit{lower} finite. Application of simultaneous simplification over this poset to the sort
of convolutions we consider here requires us to restrict our attention to
homologically locally finite complexes; see Remark \ref{rem: locally_finite}.

\begin{prop} \label{prop: ufray_perp_cond}
    Let $C, C' \in \CS^b(\SSBim_{\aa}^{\bb})$ be given,
    and suppose $f \colon C \to C'$ is a quasi-isomorphism. Then
    $\mathrm{Cone}(\ufray{n}{\bnu}(f)) \in 
    \mathcal{I}_{(\bnu, 1)}^{\perp} \cap ^{\perp} \mathcal{I}_{(\bnu, 1)}$
    for each \textcolor{revisions}{composition} $\bnu \vdash n$.
\end{prop}

\begin{proof}
    By definition, we have

    \[
    \mathrm{Cone}(\ufray{n}{\bnu}(f)) \cong \ufray{n}{\bnu}
    \left( \mathrm{Cone}(f) \right) = \mathrm{tw}_{\gamma}
    \left( \fray{n}{\bnu} \left(\mathrm{Cone}(f) \right) \otimes_R
    R[\mathbb{U}_n] \right).
    \]

    \vspace{1em}
    
    Let $J$ denote the monomial basis of $R[\mathbb{U}_n]$, considered as a poset ordered by homological degree.
    Then the homological degree of $u_i$ is non-negative for each $1 \leq i \leq n$,
    so $J$ is lower finite. In fact $J$ is \textit{strictly} lower finite:
    for each $d \in \mathbb{Z}$, all but finitely many elements of $J$ have
    homological degree exceeding $d$. The complex
    $\fray{n}{\bnu} \left( \mathrm{Cone}(f) \right)$ is bounded below in homological degree,
    so the $J$-indexed convolution considered above is in fact homologically locally finite.
    Since $\fray{n}{\bnu} \left( \mathrm{Cone}(f) \right) \simeq 0$ by Proposition
    \ref{prop: fray_perp_cond}, Proposition \ref{prop: simul_simp} gives that
    $\mathrm{Cone}(\ufray{n}{\bnu}(f))$ is homotopy equivalent to a convolution
    of $0$ complexes and so is itself contractible.
\end{proof}

We are now prepared to define the map $\tilde{\eta}_n$. We proceed by induction. When $n = 1$, unrolling the chain complex $\ufray{n}{\blam}$ gives

    \begin{center}
    \begin{tikzcd}
        \infproj{(1)} = \strand{(1)} \arrow[r, "0"] & tq^{-2} \strand{(1)} \arrow[r, "u_1"] & u_1 \strand{(1)} \arrow[r, "0"] & u_1 tq^{-2} \strand{(1)} \arrow[r, "u_1"] & u_1^2 \strand{(1)} \arrow[r, "0"] & u_1^2 tq^{-2} \strand{(1)} \arrow[r, "u_1"] & \dots
    \end{tikzcd}
    \end{center}
    We can perform Gaussian elimination along each of the nonzero components of the differential, furnishing a homotopy equivalence

    \[
    \tilde{\eta}_1 \colon \strand{(1)} \to \infproj{(1)}.
    \]
    Note that $\mathrm{Cone}(\tilde{\eta}_1) \simeq 0$, so (P2) is satisfied.
    
    For the inductive step, recall the chain map $\psi \colon q^{2n} \strand{(n, 1)} \to C^u_n$ of Lemma \ref{lem: qi}. Applying $\ufray{n}{\blam}$, we obtain a chain map
	
	\[
	\ufray{n}{\blam}(\psi) \colon q^{2n} \left( \infproj{\blam} \boxtimes \strand{(1)} \right) \to \ufray{n}{\blam}(C_n^u).
	\]
    Let $j \colon C^u_n \hookrightarrow \mathrm{Cone}(\Phi)$ denote the obvious degree $0$ chain map given by inclusion in $u_{n + 1}$-degree $0$. Finally, let
    
    \begin{center}
    \begin{tikzcd}
    \mathrm{Cone}(\ufray{n}{\blam}(\Phi)) \arrow[r, "\upsilon"] & q^n \left( P^{(1^{n + 1})}_{(1^{n + 1})} \right)^{\vee}
    \end{tikzcd}
    \end{center}
    denote the homotopy equivalence of Proposition \ref{prop: inf_proj_recurs_thin}.
    
    \begin{prop} \label{prop: building_eta}
    Fix $n \geq 1$, and suppose $\tilde{\eta}_n \colon \strand{(1^n)} \to q^{-\frac{n(n - 1)}{2}} \infproj{\blam}$ satisfies (P2). Define $\tilde{\eta}_{n + 1}$ to be the composition
    
    \begin{center}
	\begin{tikzcd}[column sep=large, row sep=small]
\strand{(1^{n + 1})} \arrow[r, "\tilde{\eta}_n \boxtimes 1"]
& q^{-\frac{n(n - 1)}{2}} \left( \infproj{\blam} \boxtimes \strand{(1)} \right) \arrow[r, "\ufray{n}{\blam}(\psi)"]
\arrow[d, phantom, ""{coordinate, name=Z}]
& q^{-\frac{n(n + 3)}{2}} \ufray{n}{\blam}(C^u_n) \arrow[dll,
"\ufray{n}{\blam}(j)",
rounded corners,
to path={ -- ([xshift=2ex]\tikztostart.east)
|- (Z) [near end]\tikztonodes
-| ([xshift=-2ex]\tikztotarget.west)
-- (\tikztotarget)}] \\
q^{-\frac{n(n + 3)}{2}} \mathrm{Cone}(\ufray{n}{\blam}(\Phi)) \arrow[r, "\upsilon"]
& q^{-\frac{(n + 1)n}{2}} (P_{(1^{n + 1})}^{(1^{n + 1})})^{\vee}
\end{tikzcd}
\end{center}
    
    Then $\tilde{\eta}_{n + 1}$ satisfies (P2).
    \end{prop}

\begin{proof}
	Define $\eta_{n + 1}'$ to be the composition of all the maps defining $\tilde{\eta}_{n + 1}$ except $\upsilon$. Because $\upsilon$ is a homotopy equivalence, we have a homotopy equivalence $\mathrm{Cone}(\eta'_{n + 1}) \simeq \mathrm{Cone}(\tilde{\eta}_{n + 1})$. Because $\mathcal{I}_{\blam}^{\perp} \cap ^{\perp} \mathcal{I}_{\blam}$ is closed under homotopy equivalence, it therefore suffices to show that $\eta'_{n + 1}$ satisfies (P2).
	
	The map $\eta'_{n + 1}$ fits nicely into the infinite ladder description of $\mathrm{Cone}(\ufray{n}{\blam}(\Phi))$, as indicated graphically by the blue arrow in the diagram below:

    \begin{center}
    \begin{tikzcd}[column sep=large]
        \strand{(1^{n + 1})} \arrow[dr, blue, "\ufray{n}{\blam}(\psi) \circ (\tilde{\eta}_n \boxtimes 1)"] & \\
        tq^{-\frac{n(n + 3)}{2} - 2} \infproj{\blam} \boxtimes \strand{(1)}
        \arrow[r, "\ufray{n}{\blam}(\iota)"] \arrow[dr, "u_{n + 1} \ufray{n}{\blam}(\psi)"]
        & q^{-\frac{n(n + 3)}{2}} \ufray{n}{\blam}(C^u_n) \\
        u_{n + 1} tq^{-\frac{n(n + 3)}{2} - 2} \infproj{\blam} \boxtimes \strand{(1)}
        \arrow[r, "\ufray{n}{\blam}(\iota)"] \arrow[dr, "u_{n + 1} \ufray{n}{\blam}(\psi)"]
        & u_{n + 1} q^{-\frac{n(n + 3)}{2}} \ufray{n}{\blam}(C^u_n) \\
        u_{n + 1}^2 tq^{-\frac{n(n + 3)}{2} - 2} \infproj{\blam} \boxtimes \strand{(1)}
        \arrow[r, "\ufray{n}{\blam}(\iota)"] \arrow[dr, "u_{n + 1} \ufray{n}{\blam}(\psi)"]
        & u_{n + 1}^2 q^{-\frac{n(n + 3)}{2}} \ufray{n}{\blam}(C^u_n) \\
        \dots & \dots
    \end{tikzcd}
    \end{center}

    By reassociating, we may write $\mathrm{Cone}(\eta'_{n + 1})$ as a convolution of the form

    \begin{center}
    \begin{tikzcd}[sep=tiny]
        \mathrm{Cone}(\eta'_{n + 1}) \simeq
        \dots \arrow[r] & u_{n + 1} \mathrm{Cone} \left( \ufray{n}{\blam}(\psi) \right) \arrow[r]
        & \mathrm{Cone} \left( \ufray{n}{\blam}(\psi) \right) \arrow[r]
        & \mathrm{Cone} \left( \ufray{n}{\blam}(\psi) \circ (\tilde{\eta}_n \boxtimes 1) \right).
    \end{tikzcd}
    \end{center}
    This is an upper finite convolution, so by Proposition \ref{prop: simul_simp},
    it suffices to show that each term in this convolution lives in
    $\mathcal{I}_{(1^{n + 1})}^{\perp} \cap ^{\perp} \mathcal{I}_{(1^{n + 1})}$.

    Both $\strand{\aa}$ and $C_n$ are bounded in homological degree,
    so $\mathrm{Cone} \left( \ufray{n}{\blam} (\psi) \right) \in
    \mathcal{I}_{(1^{n + 1})}^{\perp} \cap ^{\perp} \mathcal{I}_{(1^{n + 1})}$ by
    Proposition \ref{prop: ufray_perp_cond}. It remains to show that $\mathrm{Cone}(\ufray{n}{\blam}(\psi) \circ (\tilde{\eta}_n \boxtimes 1))
    \in \mathcal{I}_{(1^{n + 1})}^{\perp} \cap ^{\perp} \mathcal{I}_{(1^{n + 1})}$.
    
	We show inclusion in $\mathcal{I}_{(1^{n + 1})}^{\perp}$; inclusion in $^{\perp} \mathcal{I}_{(1^{n + 1})}$ is exactly analogous. To see this inclusion, we need to show that the complex
	
	\[
	W_{(1^{n + 1})} \star \mathrm{Cone}(\ufray{n}{\blam}(\psi) \circ (\tilde{\eta}_n \boxtimes 1)) \simeq \mathrm{Cone} \left( 1_{W_{(1^{n + 1})}} \star \left( \ufray{n}{\blam}(\psi) \circ (\tilde{\eta}_n \boxtimes 1) \right) \right)
	\]
    is contractible. This in turn is true if and only if the chain map
    
    \[
    1_{W_{(1^{n + 1})}} \star \left( \ufray{n}{\blam}(\psi) \circ (\tilde{\eta}_n \boxtimes 1) \right) = \left( 1_{W_{(1^{n + 1})}} \star \ufray{n}{\blam}(\psi) \right) \circ \left( 1_{W_{(1^{n + 1})}} \star (\tilde{\eta}_n \boxtimes 1) \right)
    \]
    is a homotopy equivalence. To show this composition is a homotopy equivalence, it suffices to show that each factor is a homotopy equivalence.
    
    We have already seen that $\left( 1_{W_{(1^{n + 1})}} \star \ufray{n}{\blam}(\psi) \right)$ is a homotopy equivalence (since $\ufray{n}{\blam}(\psi) \in \mathcal{I}_{(1^{n + 1})}^{\perp} \cap ^{\perp} \mathcal{I}_{(1^{n + 1})}$). To see that $\left( 1_{W_{(1^{n + 1})}} \star (\tilde{\eta}_n \boxtimes 1) \right)$ is a homotopy equivalence, it again suffices to show that $W_{(1^{n + 1})} \star \mathrm{Cone}(\tilde{\eta}_n \boxtimes 1)$ is contractible. Repeated removal of digons via Proposition \ref{prop: web_rels} gives a chain isomorphism

    \[
    [n]! W_{(1^{n + 1})} \star \mathrm{Cone}(\tilde{\eta}_n \boxtimes 1)
    \cong W_{(1^{n + 1})} \star  (W_{\blam} \boxtimes \strand{(1)})
    \star \mathrm{Cone}(\tilde{\eta}_n \boxtimes 1).
    \]
    Recall that $(W_{\blam} \boxtimes \strand{(1)}) \star \mathrm{Cone}(\tilde{\eta}_n \boxtimes 1)$
    is contractible by the inductive hypothesis. Then $W_{(1^{n + 1})}
    \star \mathrm{Cone}(\tilde{\eta}_n \boxtimes 1)$ is a summand of a contractible complex,
    so must itself be contractible.
\end{proof}

Before moving on, we pause to establish some straightforward consequences of Theorem \ref{thm: inf_proj_idem}.

\begin{thm} \label{thm: proj_agreement}
    There is a (unique up to homotopy) chain map $\phi \colon P_{1^n}^{\vee}
    \to q^{-\frac{n(n - 1)}{2}} \infproj{\blam}$ satisfying $\tilde{\eta}_n \sim \phi \circ \eta_n$, and $\phi$ is
    a homotopy equivalence.
\end{thm}

\begin{proof}
    Immediate consequence of the uniqueness portion of Theorem \ref{thm: AH_projector}.
\end{proof}

\begin{thm} \label{thm: yproj_agreement}
    The homotopy equivalence $\phi$ of Theorem \ref{thm: proj_agreement} lifts to a homotopy equivalence of $F_y^{\blam}$-curved complexes $\phi^y \colon \left( P_{(1^n)}^{\vee} \right)^y \simeq q^{-\frac{n(n - 1)}{2}} \yinfproj{\blam}$.
\end{thm}

\begin{proof}
    By Theorem 1.9 of \cite{AH17}, the complex

    \[
    \Hom_{\CS(\SBim_n)}(P_{(1^n)}, \infproj{\blam}) \simeq \End_{\CS(\SBim_n)}(P_{(1^n)})
    \]
    has homology concentrated in non-negative degrees. The desired result then follows immediately from Proposition \ref{prop: lifting_maps}.
\end{proof}

\section{Applications to Link Homology} \label{sec: link_hom}

In this section, we use the various Fray functors of Section \ref{sec: fray_functors} to construct several families of homological invariants of colored links parametrized by \textcolor{revisions}{composition}s $\blam \vdash n$. We demonstrate that (the triply-graded dimensions of) these invariants are related to the colored HOMFLYPT homology due to Webster--Williamson \cite{WW17} and its deformation due to Hogancamp--Rose--Wedrich \cite{HRW21} by multiplication by a polynomial in $q$. In the special case $\blam = (1^n) \vdash n$, we recover the various incarnations ($y$-ified/un $y$-ified, finite/infinite) of projector-colored HOMFLYPT homology. As a result, we are able to explicitly relate these two invariants, realizing the conjectural program of Section 1.4.1 of \cite{Con23}.

\subsection{Intrinsically-Colored Link Homology} \label{sec: int_col_hom}

\begin{defn} \label{def: hoch_hom}
    Let $S$ be a $\mathbb{Z}_q$-graded $R$-algebra and $M$ a $\mathbb{Z}_q$-graded $(S, S)$-bimodule. Consider $M$ as a graded module over the enveloping algebra $S^e := S \otimes_R S^{\mathrm{op}}$. Then the \textit{$i^{\text{th}}$ Hochschild homology} of $M$ is the $\mathbb{Z}_q$-graded $R$-module

    \[
    HH_i(M) := \mathrm{Tor}_i^{S^e}(S, M).
    \]
    
    For each $i$, consider $HH_i(M)$ as a $\mathbb{Z}_a \times \mathbb{Z}_q$-graded $R$-module with trivial $\mathbb{Z}_a$-grading. The \textit{total Hochschild homology} of $M$ is the $\mathbb{Z}_a \times \mathbb{Z}_q$-graded $R$-module

    \[
    HH_{\bullet}(M) := \bigoplus_{i \geq 0} a^{-i} HH_i(M).
    \]

    We view $HH_{\bullet}$ as a functor from $(S, S)-\mathrm{Bim}[\mathbb{Z}_q]$ to $R-\mathrm{Mod}[\mathbb{Z}_a \times \mathbb{Z}_q]$.
\end{defn}

The functor $HH_{\bullet}$ has a natural extension to complexes of graded $(S, S)$-bimodules. Explicitly, given a complex $(X, d_X) = \mathrm{tw}_{d_X} \left( \bigoplus_{j \in \mathbb{Z}} t^j X^j \right) \in \CS((S, S)-\mathrm{Bim}[\mathbb{Z}_q])$, we set

\[
HH_{\bullet}(X, d_X) := \mathrm{tw}_{HH_{\bullet}(d_X)} \left( \bigoplus_{j \in \mathbb{Z}} t^j HH_{\bullet}(X^j) \right) \in \CS(R-\mathrm{Mod}[\mathbb{Z}_a \times \mathbb{Z}_q]).
\]

Objects of $\CS(R-\mathrm{Mod}[\mathbb{Z}_a \times \mathbb{Z}_q])$ are pairs consisting of an underlying $\mathbb{Z}_a \times \mathbb{Z}_q \times \mathbb{Z}_t$-graded $R$-module and a degree $t$ differential, where the grading conventions are as in Example \ref{ex: tri_grading}. Since $HH_{\bullet}$ is additive, it also extends to a functor between homotopy categories

\[
HH_{\bullet} \colon K((S, S)-\mathrm{Bim}[\mathbb{Z}_q]) \to K(R-\mathrm{Mod}[\mathbb{Z}_a \times \mathbb{Z}_q])
\]

The following is Proposition 5.3 in \cite{HRW21}, adapted to our notation.

\begin{prop} \label{prop: hh_is_trace}
    Let $S, S'$ be $\mathbb{Z}_q$-graded $R$-algebras. Suppose $X \in \CS((S, S') - \mathrm{Bim}[\mathbb{Z}_q])$, $Y \in \CS((S', S) - \mathrm{Bim}[\mathbb{Z}_q])$ are complexes of $\mathbb{Z}_q$-graded bimodules which are projective as $S$ or $S'$-modules. Then there is a natural isomorphism of $\mathbb{Z}_a \times \mathbb{Z}_q \times \mathbb{Z}_t$-graded complexes of $R$-modules

    \[
    HH_{\bullet}(X \otimes_{S'} Y) \cong HH_{\bullet}(Y \otimes_S X).
    \]
\end{prop}

We now restrict our attention to complexes of singular Soergel bimodules with the usual coefficient algebras $S = \mathrm{Sym}^{\aa}(\mathbb{X}_N)$ for some \textcolor{revisions}{composition} $\aa \vdash N$. Note that the conditions of Proposition \ref{prop: hh_is_trace} are always satisfied in this case, since every singular Soergel bimodule is free as a left or right module.

\begin{defn} \label{def: balanced}
    We call a complex $X \in \CS(\SSBim_N)$ \textit{balanced} if $X \in \CS(\SSBim_{\aa}^{\aa})$ for some \textcolor{revisions}{composition} $\aa \vdash N$.
    Similarly, we call a colored braid $\beta_{\aa} \in \BGpd{\aa}{\bb}$ \textit{balanced} if $\bb = \aa$.
\end{defn}

Notice that when $\beta_{\aa}$ is a balanced braid, its Rickard complex $C(\beta_{\aa})$ is a balanced complex. In particular, $C(\beta_{\aa})$ is a complex of graded $\mathrm{Sym}^{\aa}(\mathbb{X})$-bimodules and is therefore a potential target for $HH_{\bullet}$ with coefficients $S = \mathrm{Sym}^{\aa}(\mathbb{X})$. In keeping with the established literature, the homology of the resulting complex (appropriately normalized) will be our first model of colored, triply-graded link homology. First, in order to correctly describe the normalization, we describe some statistics associated to $\beta_{\aa}$.

\begin{defn} \label{def: braid_stats}
    Given a colored Artin braid generator $x = (\sigma_i)^{\pm 1}_{(a, b)}$, we define its \textcolor{revisions}{\textit{colored writhe}} by $\epsilon(x) = \pm ab$. Given a colored braid $\beta_{\aa} \in \mathfrak{Br}_n$, we define the \textit{colored writhe} $\epsilon(\beta_{\aa})$ as the sum of the colored writhes of the Artin generators in a braid word\footnote{One can check that this sum is preserved by the braid relations, so $\epsilon(\beta_{\aa})$ is independent of a choice of braid word.} representation of $\beta_{\aa}$. If $\beta_{\aa}$ is a balanced braid for $\aa = (a_1, \dots, a_n)$, we set

    \[
    N(\beta_{\aa}) := \sum_{i = 1}^n a_i; \quad \eta(\beta_{\aa}) = \sum_{i \in \pi_0(\beta)} a_{[i]}
    \]
    Here $\pi_0(\beta)$ denotes the set of orbits under the action of the permutation associated to $\beta$ acting on the set $\{1, \dots, n\}$.
\end{defn}

\begin{defn} \label{def: kr_hom}
    Let $\beta_{\aa}$ be a balanced colored braid. Then the \textit{intrinsically-colored, triply-graded Khovanov--Rozansky complex} associated to $\beta_{\aa}$ is

    \[
    C_{KR}(\beta_{\aa}) := (at^{-1})^{\frac{1}{2} (\epsilon(\beta_{\aa}) + N(\beta_{\aa}) - \eta(\beta_{\aa}))} q^{-\epsilon(\beta_{\aa})} HH_{\bullet}(C(\beta_{\aa})).
    \]
    The \textit{intrinsically-colored, triply-graded Khovanov--Rozansky homology} of $\beta_{\aa}$ is the $\mathbb{Z}_a \times \mathbb{Z}_q \times \mathbb{Z}_t$-graded $R$-module

    \[
    H_{KR}(\beta_{\aa}) := H^{\bullet}(C_{KR}(\beta_{\aa})).
    \]
\end{defn}

Given a balanced colored braid $\beta_{\aa}$, its topological closure $\hat{\beta}_{\aa}$ is an oriented colored link. The following is Theorem 5.19 in \cite{HRW21}.

\begin{thm} \label{thm: HKR_link_inv}
    A change of braid representatives $\beta_{\aa}$ for a framed, oriented, colored link $\mathcal{L}$ induces a homotopy equivalence between the corresponding complexes $C_{KR}(\beta_{\aa})$. Consequently, $H_{KR}(\mathcal{L}) := H_{KR}(\beta_{\aa})$ is a well-defined invariant of the framed, oriented, colored link $\mathcal{L}$ up to isomorphism. Changing framing by $\pm 1$ on a $b$-colored component shifts $H_{KR}(\mathcal{L})$ by $(at^{-1})^{\pm \frac{1}{2} b(b - 1)}$.
\end{thm}

\subsection{Deformed Intrinsically-Colored Link Homology} \label{sec: def_int_col_hom}

Hogancamp--Rose--Wedrich in \cite{HRW21} define a deformation of the intrinsically-colored, triply-graded Khovanov--Rozansky homology assigned to a colored link. Their procedure largely mimics that of Definition \ref{def: kr_hom} but takes as input a version of the \textit{deformed} Rickard complex of Corollary \ref{cor: curved_braids}. To properly define their invariant requires that we briefly pass from $\Delta e$-curvature modeled on differences of elementary symmetric functions to $\Delta p$-curvature modeled on differences of power sum symmetric functions.

\begin{defn} \label{def: delta_p_curv}
    Let $\aa = (a_1, \dots, a_n) \vdash N$, $\sigma \in \SG_n$ be given, and recall the decomposition $\mathbb{X} = \mathbb{X}_1 \sqcup \dots \sqcup \mathbb{X}_n$ of an alphabet $\mathbb{X}$ of size $N$ as in Section \ref{sec: symm_poly}. For each pair of positive integers $1 \leq j \leq n$ and $1 \leq k \leq a_j$, let $\dot{v}_{jk}$ be a formal variable of degree $\mathrm{deg}(\dot{v}_{jk}) = q^{-2k}t^2$. Let $\dot{\mathbb{V}}_{\aa}$ denote the alphabet of all such variables. Then the \textit{$\Delta p$-curvature} associated to the pair $(\aa, \sigma)$ is the polynomial

    \[
    F_{\sigma, \dot{v}}^{\aa}(\mathbb{X}, \mathbb{X}') := \sum_{j = 1}^n \sum_{k = 1}^{a_j} \frac{1}{k} \left(p_k(\mathbb{X}_{\sigma(j)}) - p_k(\mathbb{X}'_j) \right) \otimes \dot{v}_{jk} \in \mathrm{Sym}^{\aa}_{\sigma}(\mathbb{X}, \mathbb{X}') \otimes_R R[\dot{\mathbb{V}}_{\aa}].
    \]

    As in Definition \ref{def: delta_e_curv}, we reserve the notation $F_{\dot{v}}^{\aa}(\mathbb{X}, \mathbb{X}')$ for the special case $\sigma = e$. We also suppress the alphabets $\mathbb{X}, \mathbb{X}'$, writing just $F_{\sigma, {\dot{v}}}^{\aa}$, when they are clear from context.
\end{defn}

\begin{defn} \label{def: y-ification_power_sums}
    We call a curved complex $\mathrm{tw}_{\Delta_X}(X \otimes R[\dot{\mathbb{V}_{\aa}}]) \in \YS_{F_{\sigma, {\dot{v}}}^{\aa}}(\Bim_N; R[\dot{\mathbb{V}}_{\aa}])$ a \textit{$\Delta p$-deformation} of the underlying complex $X$.
\end{defn}

\textit{A priori} it is not guaranteed that a complex which admits a $\Delta e$-deformation also admits a $\Delta p$-deformation and vice versa. The following Lemma gives a dictionary for transitioning between these two notions. Here $h_i(\mathbb{X})$ denotes the degree $i$ complete symmetric polynomial in the alphabet $\mathbb{X}$, and $\mathbb{X} - \mathbb{X}'$ is a formal difference of alphabets; see Section 2 of \cite{HRW21} for precise definitions.

\begin{lem} \label{lem: alph_soup_dict}
    Fix $a \geq 1$, and let $\mathbb{X}, \mathbb{X}', \mathbb{U}, \dot{\mathbb{V}}$ be alphabets of size $a$ with degrees $\mathrm{deg}(x_i) = \mathrm{deg}(x'_i) = q^2$, $\mathrm{deg}(u_i) = \mathrm{deg}(\dot{v}_i) = t^2q^{-2i}$ for each $1 \leq i \leq a$. Consider the polynomials
    
    \[
    \psi_i^a(\mathbb{X}, \mathbb{X}', \mathbb{U}) := \sum_{j = i}^a \sum_{k = j}^a \frac{i}{j} h_{j - i}(\mathbb{X} - \mathbb{X}') e_{k - j}(\mathbb{X}) u_k \in \mathrm{Sym}(\mathbb{X}, \mathbb{X}') \otimes_R R[\mathbb{U}];
    \]
    
    \[
    \rho_i^a(\mathbb{X}, \mathbb{X}', \dot{\mathbb{V}}) := \sum_{j = 1}^a \sum_{k = j}^a (-1)^{i + k - j - 1} \frac{j}{k} h_{j - i}(\mathbb{X}) e_{k - j}(\mathbb{X} - \mathbb{X}') \dot{v}_k \in \mathrm{Sym}(\mathbb{X}, \mathbb{X}') \otimes_R R[\dot{\mathbb{V}}].
    \]
    Then the assignment

    \[
    \dot{v}_i := \psi_i^a(\mathbb{X}, \mathbb{X}', \mathbb{U}), \quad u_i := \rho_i^a(\mathbb{X}, \mathbb{X}', \dot{\mathbb{V}})
    \]
    is a mutually inverse change of variables inducing a $\mathbb{Z}_q \times \mathbb{Z}_t$-graded $R$-algebra isomorphism $R[\mathbb{X}, \mathbb{X}', \mathbb{U}] \cong R[\mathbb{X}, \mathbb{X}', \dot{\mathbb{V}}]$ taking $F^{(a)}_u(\mathbb{X}, \mathbb{X}')$ to $F^{(a)}_{\dot{v}}(\mathbb{X}, \mathbb{X}')$.
\end{lem}

\begin{proof}
    This follows immediately from Corollary 4.6 and Lemma 5.34 of \cite{HRW21}.
\end{proof}

\begin{cor} \label{cor: alph_soup_full}
    Let $\aa = (a_1, \dots, a_n) \vdash N$, $\sigma \in \SG_n$ be given, and recall the alphabets $\mathbb{U}_{\aa}$, $\dot{\mathbb{V}}_{\aa}$ of Definitions \ref{def: delta_e_curv} and \ref{def: delta_p_curv}. Then the assignment

    \[
    \dot{v}_{jk} := \psi_k^{a_j}(\mathbb{X}_{\sigma(j)}, \mathbb{X}'_j, \mathbb{U}_j); \quad u_{jk} := \rho_k^{a_j}(\mathbb{X}_{\sigma(j)}, \mathbb{X}'_j, \dot{\mathbb{V}}_j)
    \]
    is a mutually inverse change of variables inducing a $\mathbb{Z}_q \times \mathbb{Z}_t$-graded $R$-algebra isomorphism
    
    \begin{center}
    \begin{tikzcd}
    \mathrm{Sym}^{\aa}_{\sigma}(\mathbb{X}, \mathbb{X}') \otimes_R R[\mathbb{U}_{\aa}] \arrow[r, harpoon, shift left, "\rho_{\aa}"] & \mathrm{Sym}^{\aa}_{\sigma}(\mathbb{X}, \mathbb{X}') \otimes_R R[\dot{\mathbb{V}}_{\aa}] \arrow[l, harpoon, shift left, "\psi_{\aa}"]
    \end{tikzcd}
    \end{center}
    which interchanges $F_{\sigma, u}^{\aa}$ and $F_{\sigma, \dot{v}}^{\aa}$.
\end{cor}

\begin{cor} \label{cor: e_to_p_curvature}
    Let $\mathrm{tw}_{\Delta_X}(X \otimes_R R[\mathbb{U}_{\aa}])$ be a strict $\Delta e$-deformation of $X$, and consider the decomposition of $\Delta_X$ into $\mathbb{U}_{\aa}$-homogeneous components

    \[
    \Delta_X = \sum_{j = 1}^n \sum_{k = 1}^{a_j} h_{jk} \otimes u_{jk}.
    \]
    Set

    \[
    \rho_{\aa}(\Delta_X) := \sum_{j = 1}^n \sum_{k = 1}^{a_j} h_{jk} \otimes \rho_k^{a_j} \left( \mathbb{X}_{\sigma(j)}, \mathbb{X}'_j, \dot{\mathbb{V}}_j \right) \in \End(X) \otimes_{\mathrm{Sym}^{\aa}_{\sigma}(\mathbb{X}, \mathbb{X}')} (\mathrm{Sym}^{\aa}_{\sigma}(\mathbb{X}, \mathbb{X}') \otimes_R R[\dot{\mathbb{V}}_{\aa}]).
    \]
    Then $\mathrm{tw}_{\rho_{\aa}(\Delta_X)}(X \otimes_R R[\dot{\mathbb{V}}_{\aa}])$ is a $\Delta p$-deformation of $X$.
\end{cor}

\begin{defn}
    Let $\beta_{\aa}$ be a \textcolor{revisions}{colored} braid, and let $C^u(\beta_{\aa}) = \mathrm{tw}_{\Delta_{\beta}}(C(\beta_{\aa}) \otimes_R R[\mathbb{U}_{\aa}])$ denote its $\Delta e$-deformed Rickard complex. Set

    \[
    C^{\dot{v}}(\beta_{\aa}) := \mathrm{tw}_{\rho_{\aa}(\Delta_{\beta})}(C(\beta_{\aa}) \otimes_R R[\dot{\mathbb{V}}_{\aa}]).
    \]
    By Corollary \ref{cor: e_to_p_curvature}, $C^{\dot{v}}(\beta_{\aa})$ is a $\Delta p$-deformation of $C(\beta_{\aa})$, called the \textit{$\Delta p$-deformed Rickard complex} for the colored braid $\beta_{\aa}$.
\end{defn}

We can extend Hochschild homology to curved complexes as follows.

\begin{defn} \label{def: hh_curved}
	Suppose $\mathbb{U} = \{u_1, \dots, u_n\}$ is an alphabet of deformation parameters of even homological degree as in Definition \ref{def: curved_lift}, and let $\mathrm{tw}_{\Delta_X}(X \otimes_R S) \in \YS_F(\SSBim_{\bb}^{\bb}; R[\mathbb{U}])$ be an $F$-deformation of some complex $X$. We decompose

    \[
    \Delta_X = \sum_{{\bf{v}} \in \mathbb{Z}_{\geq 0}^n} (\Delta_X)_{u^{\bf{v}}} \otimes u^{\bf{v}}
    \]
    as in Section \ref{subsec: def}. Set

    \[
    HH_{\bullet}(\Delta_X) := \sum_{{\bf{v}} \in \mathbb{Z}_{\geq 0}^n} HH_{\bullet}((\Delta_X)_{u^{\bf{v}}}) \otimes u^{\bf{v}}.
    \]
    Then the \textit{Hochschild homology} of $\mathrm{tw}_{\Delta_X}(X \otimes_R S)$ is the curved complex

    \[
    HH_{\bullet}(\mathrm{tw}_{\Delta_X}(X \otimes_R S)) := \mathrm{tw}_{HH_{\bullet}(\Delta_X)}(HH_{\bullet}(X) \otimes_R S) \in \YS_{HH_{\bullet}(F)}(R-\mathrm{Mod}[\mathbb{Z}_a \times \mathbb{Z}_q]; R[\mathbb{U}]).
    \]
\end{defn}

By construction, Hochschild homology identifies the action of (symmetric polynomials in) $\mathbb{X}_i$ with the action of (symmetric polynomials in) $\mathbb{X}'_i$. Extending this identification $\dot{\mathbb{V}}_{\aa}$-linearly takes the usual $\Delta p$-curvature $F_{\sigma, \dot{v}}^{\aa}(\mathbb{X}, \mathbb{X}')$ to

\begin{align*}
HH_{\bullet}(F_{\sigma, \dot{v}}^{\aa}(\mathbb{X}, \mathbb{X}')) & = \sum_{j = 1}^n \sum_{k = 1}^{a_j} \frac{1}{k} \left(p_k(\mathbb{X}_{\sigma(j)}) - p_k(\mathbb{X}_j) \right) \otimes \dot{v}_{jk} \\
& = \sum_{j = 1}^n \sum_{k = 1}^{a_j} \frac{1}{k} p_k(\mathbb{X}_j) \otimes (\dot{v}_{\sigma^{-1}(j)k} - \dot{v}_{jk}).
\end{align*}

This vanishes upon identifying $\dot{v}_{\sigma^{-1}(j)k} \sim \dot{v}_{jk}$. Motivated by this observation, the authors of \cite{HRW21} define \textit{bundled} $\Delta p$-curvature as follows.

\begin{defn} \label{def: bundled_delta_p_curv}
    Let $\aa = (a_1, \dots, a_n) \vdash N$, $\sigma \in \SG_n$ be given such that $\sigma \cdot \aa = \aa$, and let $\Omega(\sigma)$ denote the orbits of the permutation $\sigma$ acting on the set $\{1, \dots, n\}$. For each pair of positive integers $1 \leq j \leq |\Omega(\sigma)|$ and $1 \leq k \leq a_j$, let $\dot{v}_{[j]k}$ be a formal variable of degree $\mathrm{deg}(\dot{v}_{[j]k}) = q^{-2k}t^2$. Let $\dot{\mathbb{V}}_{\aa}^{\Omega(\sigma)}$ denote the alphabet of all such variables. Then the \textit{bundled $\Delta p$-curvature} associated to the pair $(\aa, \sigma)$ is the polynomial

    \[
    F_{\sigma, [\dot{v}]}^{\aa}(\mathbb{X}, \mathbb{X}') := \sum_{j = 1}^n \sum_{k = 1}^{a_j} \frac{1}{k} \left(p_k(\mathbb{X}_{\sigma(j)}) - p_k(\mathbb{X}'_j) \right) \otimes \dot{v}_{[j]k} \in \mathrm{Sym}^{\aa}_{\sigma}(\mathbb{X}, \mathbb{X}') \otimes_R R[\dot{\mathbb{V}}_{\aa}^{\Omega(\sigma)}].
    \]

    As in Definition \ref{def: delta_e_curv}, we reserve the notation $F_{[\dot{v}]}^{\aa}(\mathbb{X}, \mathbb{X}')$ for the special case $\sigma = e$.
\end{defn}

\begin{defn} \label{def: bundled_y-ification_power_sums}
    We call a curved complex $\mathrm{tw}_{\Delta_X}(X \otimes_R R[\dot{\mathbb{V}}_{\aa}^{\Omega(\sigma)}]) \in \YS_{F_{\sigma, {[\dot{v}]}}^{\aa}}(\Bim_N; R[\dot{\mathbb{V}}_{\aa}^{\Omega(\sigma)}])$ a \textit{bundled $\Delta p$-deformation} of the underlying complex $X$.
\end{defn}

\begin{lem} \label{lem: delta_p_bundling}
    Let $\mathrm{tw}_{\Delta_X}(X \otimes_R R[\dot{\mathbb{V}}_{\aa}])$ be a $\Delta p$-deformation of $X$. Let $\pi_{\dot{v}} \colon R[\dot{\mathbb{V}}_{\aa}] \to R[\dot{\mathbb{V}}_{\aa}^{\Omega(\sigma)}]$ denote the quotient map $\dot{v}_{jk} \mapsto \dot{v}_{[j]k}$. Define $\pi_{\dot{v}}(\Delta_X)$ by acting on $\dot{\mathbb{V}}$-homogeneous tensor factors as usual. Then $\mathrm{tw}_{\pi_{\dot{v}}(\Delta_X)} \left( X \otimes_R R[\dot{\mathbb{V}}_{\aa}^{\Omega(\sigma)}] \right)$is a bundled $\Delta p$-deformation of $X$.
\end{lem}

We are now prepared to define the deformed invariant of \cite{HRW21}. The following is a power-sum analog of their Definition 5.27.

\begin{defn}
    Let $\mathcal{L}$ be a colored, oriented, framed link which is presented as a balanced braid closure $\mathcal{L} = \hat{\beta}_{\aa}$, and let $\sigma$ denote the permutation induced by $\beta$. Let

    \[
    \overline{\Delta}_{\beta} \in \End_{\CS(\SSBim)}(C(\beta_{\aa})) \otimes_R R[\dot{\mathbb{V}}_{\aa}^{\Omega(\sigma)}]
    \]
    denote the connection on the bundled $\Delta p$-deformation of $C(\beta_{\aa})$ obtained from the $\Delta p$-deformed Rickard complex through Lemma \ref{lem: delta_p_bundling}. Set

    \[
    \YS HH_{\bullet}(\beta_{\aa}) := \mathrm{tw}_{HH_{\bullet}(\overline{\Delta}_{\beta})} \left( HH_{\bullet}(C(\beta_{\aa})) \otimes_R R[\dot{\mathbb{V}}_{\aa}^{\Omega(\sigma)}] \right).
    \]
    Then the \textit{deformed, intrinsically-colored, triply-graded link homology} $\YS H_{KR}(\mathcal{L})$ is the $\mathbb{Z}_a \times \mathbb{Z}_q \times \mathbb{Z}_t$-graded $R$-module given by taking homology of the complex

    \[
    \YS C_{KR}(\beta_{\aa}) := (at^{-1})^{\frac{1}{2} (\epsilon(\beta_{\aa}) + N(\beta_{\aa}) - \eta(\beta_{\aa}))} q^{-\epsilon(\beta_{\aa})} \YS HH_{\bullet}(\beta_{\aa}).
    \]
\end{defn}

\begin{thm}[\cite{HRW21}, Theorem 5.30] \label{thm: def_kr_inv}
    $\YS H_{KR}(\mathcal{L})$ is a well-defined invariant of framed, oriented, colored links.
\end{thm}

In the comparisons of Section \ref{sec: comp_hom}, we will find it more convenient to work with a version of $\YS H_{KR}(\mathcal{L})$ modeled on bundled $\Delta e$-curvature. Hogancamp--Rose--Wedrich do not consider this form of their invariant in \cite{HRW21}; one can distill an implicit bundled $\Delta e$ model of their invariant from that work, but the relevant bundled curvature differs slightly from our Definition \ref{def: bundled_delta_e_curv}\footnote{Specifically, let $\mathbb{X}_{[j]}$ (respectively $\mathbb{X}'_{[j]}$) denote the formal sum of the alphabets associated to all strands of $\beta$ which close up to the component $[j]$ of $\hat{\beta}$. Then the bundled $\Delta e$-curvature implicit in \cite{HRW21} is $\sum_{[j] \in \Omega(\sigma)} \sum_{k = 1}^{a_j} \left(e_k(\mathbb{X}_{[\sigma(j)]}) - e_k(\mathbb{X}'_{[j]}) \right) \otimes u_{[j]k}$. Compare to Definition \ref{def: bundled_delta_e_curv}, in which we do \textit{not} bundle the $\mathbb{X}$ and $\mathbb{X}'$ alphabets in this way.}. We define a new version of this invariant that is more convenient for the applications of Section \ref{sec: comp_hom}; this definition and the proof that the two invariants are isomorphic occupy the remainder of this Section.

\begin{defn} \label{def: bundled_delta_e_curv}
    Let $\aa = (a_1, \dots, a_n) \vdash N$, $\sigma \in \SG_n$ be given such that $\sigma \cdot \aa = \aa$, and let $\Omega(\sigma)$ denote the orbits of the permutation $\sigma$ acting on the set $\{1, \dots, n\}$. For each pair $[j] \in \Omega(\sigma)$ and $1 \leq k \leq a_j$, let $u_{[j]k}$ be a formal variable of degree $\mathrm{deg}(u_{[j]k}) = q^{-2k}t^2$. Let $\mathbb{U}_{\aa}^{\Omega(\sigma)}$ denote the alphabet of all such variables. Then the \textit{bundled $\Delta e$-curvature} associated to the pair $(\aa, \sigma)$ is the polynomial

    \[
    F_{\sigma, [u]}^{\aa}(\mathbb{X}, \mathbb{X}') := \sum_{j = 1}^n \sum_{k = 1}^{a_j} \left(e_k(\mathbb{X}_{\sigma(j)}) - e_k(\mathbb{X}'_j) \right) \otimes u_{[j]k} \in \mathrm{Sym}^{\aa}_{\sigma}(\mathbb{X}, \mathbb{X}') \otimes_R R[\mathbb{U}_{\aa}^{\Omega(\sigma)}].
    \]

    As in Definition \ref{def: delta_e_curv}, we reserve the notation $F_{[u]}^{\aa}(\mathbb{X}, \mathbb{X}')$ for the special case $\sigma = e$.
\end{defn}

\begin{defn} \label{def: bundled_y-ification_elem}
    We call a curved complex $\mathrm{tw}_{\Delta_X}(X \otimes_R R[\mathbb{U}_{\aa}^{\Omega(\sigma)}]) \in \YS_{F_{\sigma, [u]}^{\aa}}(\Bim_N; R[\mathbb{U}_{\aa}^{\Omega(\sigma)}])$ a \textit{bundled $\Delta e$-deformation} of the underlying complex $X$.
\end{defn}

\begin{lem} \label{lem: delta_e_bundling}
    Let $\mathrm{tw}_{\Delta_X}(X \otimes_R R[\mathbb{U}_{\aa}])$ be a $\Delta e$-deformation of $X$. Let $\pi_u \colon R[\mathbb{U}_{\aa}] \to R[\mathbb{U}_{\aa}^{\Omega(\sigma)}]$ denote the quotient map $u_{jk} \mapsto u_{[j]k}$. Define $\pi_u(\Delta_X)$ by acting on $\mathbb{U}_{\aa}$-homogeneous summands as usual. Then $\mathrm{tw}_{\pi_u(\Delta_X)} \left( X \otimes_R R[\mathbb{U}_{\aa}^{\Omega(\sigma)}] \right)$is a bundled $\Delta e$-deformation of $X$.
\end{lem}

As was the case with bundled $\Delta p$-curvature, the bundled $\Delta e$-curvature $F_{\sigma, [u]}^{\aa}(\mathbb{X}, \mathbb{X}')$ vanishes upon identifying $\mathbb{X}_j$ with $\mathbb{X}_j'$ for each $j$. This leads us naturally to the following definition.

\begin{defn} \label{def: e_bun_hom}
    Let $\mathcal{L}$ be a colored, oriented, framed link which is presented as a balanced braid closure $\mathcal{L} = \hat{\beta}_{\aa}$, and let $\sigma$ denote the permutation induced by $\beta$. Let

    \[
    \overline{\Delta}_{\beta} \in \End_{\CS(\SSBim)}(C(\beta_{\aa})) \otimes_R R[\mathbb{U}_{\aa}^{\Omega(\sigma)}]
    \]
    denote the connection on the bundled $\Delta e$-deformation of $C(\beta_{\aa})$ obtained from the $\Delta e$-deformed Rickard complex through Lemma \ref{lem: delta_e_bundling}. Set

    \[
    \YS_u HH_{\bullet}(\beta_{\aa}) := \mathrm{tw}_{HH_{\bullet}(\overline{\Delta}_{\beta})} \left( HH_{\bullet}(C(\beta_{\aa})) \otimes_R R[\mathbb{U}_{\aa}^{\Omega(\sigma)}] \right).
    \]
    Then the \textit{$\Delta e$-deformed, colored, triply-graded link homology} $\YS_u H_{KR}(\mathcal{L})$ is the $\mathbb{Z}_a \times \mathbb{Z}_q \times \mathbb{Z}_t$-graded $R$-module given by taking homology of the complex

    \[
    \YS_u C_{KR}(\beta_{\aa}) := (at^{-1})^{\frac{1}{2} (\epsilon(\beta_{\aa}) + N(\beta_{\aa}) - \eta(\beta_{\aa}))} q^{-\epsilon(\beta_{\aa})} \YS_u HH_{\bullet}(\beta_{\aa}).
    \]
\end{defn}

Implicit in the statement of Definition \ref{def: e_bun_hom} is the claim that $\YS_u H_{KR}(\mathcal{L})$ is a well-defined invariant of colored links. We prove this by furnishing an isomorphism to the link invariant $\YS H_{KR}(\mathcal{L})$.

\begin{thm} \label{thm: e_to_p_bundled}
    Let $\mathcal{L}$ be a framed, oriented, colored link. Then $\YS_u H_{KR}(\mathcal{L}) \cong \YS H_{KR} (\mathcal{L})$ as triply-graded $R$-modules.
\end{thm}

\begin{proof}
	Present $\mathcal{L}$ as the closure of an $n$-strand colored braid $\beta_{\aa}$ with induced permutation $\sigma \in \SG_n$ for some sequence of colors $\aa = (a_1, \dots, a_n)$. Since $\YS_u C_{KR}(\beta_{\aa})$ and $\YS C_{KR}(\beta_{\aa})$ carry the same grading shifts, it suffices to show that the homology of $\YS_u HH_{\bullet}(\beta_{\aa})$ and that of $\YS HH_{\bullet}(\beta_{\aa})$ are isomorphic.
	
	To that end, let $\mathrm{tw}_{\overline{\Delta}_{\beta}}(C(\beta_{\aa}) \otimes_R R[\mathbb{U}_{\aa}^{\Omega(\sigma)}]$ be the (strict) bundled $\Delta e$-deformation of $C(\beta_{\aa})$ from Definition \ref{def: e_bun_hom}. We break $\overline{\Delta}_{\beta}$ into $\mathbb{U}_{\aa}^{\Omega(\beta)}$ components as follows:
	
	\[
	\overline{\Delta}_{\beta} = \sum_{[j] \in \Omega(\sigma)} \sum_{k = 1}^{a_j} \left( \sum_{j \in [j]} \zeta_{jk} \right) \otimes u_{[j] k}.
	\]
	Here $\zeta_{jk} \in \End^{-1}(C(\beta_{\aa}))$ are an $F_{\sigma, u}^{\aa}$-deforming family satisfying $[d, \zeta_{jk}] = e_k(\mathbb{X}_{\sigma(j)}) - e_k(\mathbb{X}'_j)$. Applying Hochschild homology, we obtain the uncurved complex
	
	\begin{align*}
	\YS_u C_{KR}(\beta_{\aa}) & = \mathrm{tw}_{HH_{\bullet}(\overline{\Delta}_{\beta})}(HH_{\bullet}(C(\beta_{\aa})) \otimes_R R[\mathbb{U}_{\aa}^{\Omega(\sigma)}]; \\
	HH_{\bullet}(\overline{\Delta}_{\beta}) & = \sum_{[j] \in \Omega(\sigma)} \sum_{k = 1}^{a_j} \left( \sum_{j \in [j]} HH_{\bullet}(h_{jk}) \right) \otimes u_{[j] k}.
	\end{align*}
	
	Now, for each orbit $[j] \in \Omega(\sigma)$, choose some representative $r_{[j]} \in \{1, \dots, n\}$ satisfying $[r_{[j]}] = [j]$. We make a change of variables from $R[\mathbb{U}_{\aa}^{\Omega(\sigma)}]$ to a polynomial ring $R[\dot{\mathbb{V}}_{\aa}^{\Omega(\sigma)}]$ by the mutually inverse assignment
	
	\[
    \dot{v}_{[j]k} := \psi_k^{a_j}(\mathbb{X}_{\sigma(r_{[j]})}, \mathbb{X}'_{r_{[j]}}, \mathbb{U}_{[j]}); \quad u_{[j]k} := \rho_k^{a_j}(\mathbb{X}_{\sigma(r_{[j]})}, \mathbb{X}'_{r_{[j]}}, \dot{\mathbb{V}}_{[j]}).
    \]
    Applying this change of variables to $ \YS_u HH_{\bullet}(\beta_{\aa})$ produces an isomorphic uncurved complex

    \begin{align*}
    \YS_u HH_{\bullet}(\beta_{\aa}) & \cong \mathrm{tw}_{\tilde{\Delta}_{HH, \beta}} \left( HH_{\bullet} ( C(\beta_{\aa})) \otimes_{\mathrm{Sym}^{\sigma}(\mathbb{X}, \mathbb{X}')} \left( \mathrm{Sym}^{\sigma}(\mathbb{X}, \mathbb{X}') \otimes_R R[\dot{\mathbb{V}}^{\Omega(\sigma)}] \right) \right); \\
    \tilde{\Delta}_{HH, \beta} & := \sum_{[j] \in \Omega(\sigma)} \sum_{k = 1}^{a_j} \left( \sum_{j \in [j]} HH_{\bullet}(h_{jk}) \right) \otimes \rho_k^{a_j}(\mathbb{X}_{\sigma(r_{[j]})}, \mathbb{X}'_{r_{[j]}}, \dot{\mathbb{V}}_{[j]}).
    \end{align*}
    We can express the connection $\tilde{\Delta}_{HH, \beta}$ more explicitly using the expression for $\rho_k^{a_j}$ from Lemma \ref{lem: alph_soup_dict}:
    
    \begin{equation} \label{eq: equiv_connection}
    \tilde{\Delta}_{HH, \beta} = \sum_{[j] \in \Omega(\sigma)} \sum_{k = 1}^{a_j} \left( \sum_{j \in [j]}  HH_{\bullet}(h_{jk}) \left( \sum_{\ell = 1}^{a_j} \sum_{m = \ell}^{a_j} (-1)^{k + m - \ell - 1} \frac{\ell}{m} h_{\ell - k}(\mathbb{X}_{\sigma(r_{[j]})}) e_{k - j}(\mathbb{X}_{\sigma(r_{[j]})} - \mathbb{X}'_{r_{[j]}}) \right) \right) \otimes \dot{v}_{[j]k}.
    \end{equation}
    
    Because we work over a field, we can fix a homotopy equivalence $HH_{\bullet}(C(\beta_{\aa})) \simeq H^{\bullet}(HH_{\bullet}(C(\beta_{\aa})))$ of complexes of triply-graded $R$-modules, where the latter is considered as a complex with $0$ differential. By the arguments in the proof of Lemma 5.49 of \cite{HRW21}, this induces a homotopy equivalence of triply-graded $R[\dot{\mathbb{V}}^{\Omega(\sigma)}]$-modules
    
    \begin{equation} \label{eq: field_simp}
    \mathrm{tw}_{\tilde{\Delta}_{HH, \beta}} \left( HH_{\bullet} ( C(\beta_{\aa})) \otimes_R R[\dot{\mathbb{V}}^{\Omega(\sigma)}] \right) \simeq \mathrm{tw}_{D''} \left( H^{\bullet}(HH_{\bullet}(C(\beta_{\aa}))) \otimes_R R[\dot{\mathbb{V}}^{\Omega(\sigma)}] \right)
    \end{equation}
    for some $R[\dot{\mathbb{V}}^{\Omega(\sigma)}]$-linear twist $D''$.
    
    Now, the twist $D''$ is some explicit function of the coefficients in $\dot{\Delta}_{HH, \beta}$. Recall that upon applying Hochschild homology, the actions of (symmetric polynomials in) $\mathbb{X}_j$ and $\mathbb{X}'_j$ on $HH_{\bullet}(C(\beta_{\aa}))$ are identified for each $j$. Since the actions of (symmetric polynomials in) $\mathbb{X}_{\sigma(j)}$ and $\mathbb{X}'_j$ are already homotopic at the level of $C(\beta_{\aa})$, the actions of (symmetric polynomials in) the alphabets $\mathbb{X}_j$, $\mathbb{X}'_j$, $\mathbb{X}_{\sigma(j)}$, and $\mathbb{X}'_{\sigma(j)}$ on $HH_{\bullet}(C(\beta_{\aa}))$ are all homotopic. In particular, they induce \textit{equal} endomorphisms of $H^{\bullet}(HH_{\bullet}(C(\beta_{\aa})))$. Because of this, replacing each occurence of the alphabets $\mathbb{X}_{\sigma(r_{[j]})}$ and $\mathbb{X}'_{r_{[j]}}$ by the alphabets $\mathbb{X}_{\sigma(j)}$ and $\mathbb{X}'_j$ in \eqref{eq: equiv_connection} does not affect the twist $D''$ in \eqref{eq: field_simp}.
    
    By the previous paragraph, we obtain a homotopy equivalence of triply-graded $R[\dot{\mathbb{V}}^{\Omega(\sigma)}]$-modules
    
    \begin{equation} \label{eq: almost_power_sum}
    \mathrm{tw}_{\tilde{\Delta}'_{HH, \beta}} \left( HH_{\bullet} ( C(\beta_{\aa})) \otimes_R R[\dot{\mathbb{V}}^{\Omega(\sigma)}] \right) \simeq \mathrm{tw}_{D''} \left( H^{\bullet}(HH_{\bullet}(C(\beta_{\aa}))) \otimes_R R[\dot{\mathbb{V}}^{\Omega(\sigma)}] \right)
    \end{equation}
    where
    
    \begin{align*} \label{eq: equiv_connection2}
    \tilde{\Delta}'_{HH, \beta} & = \sum_{[j] \in \Omega(\sigma)} \sum_{k = 1}^{a_j} \left( \sum_{j \in [j]}  HH_{\bullet}(h_{jk}) \left( \sum_{\ell = 1}^{a_j} \sum_{m = \ell}^{a_j} (-1)^{k + m - \ell - 1} \frac{\ell}{m} h_{\ell - k}(\mathbb{X}_{\sigma(j)}) e_{k - j}(\mathbb{X}_{\sigma(j)} - \mathbb{X}'_{j}) \right) \right) \otimes \dot{v}_{[j]k} \\
    & = \sum_{[j] \in \Omega(\sigma)} \sum_{k = 1}^{a_j} \left( \sum_{j \in [j]} HH_{\bullet}(h_{jk}) \right) \otimes \rho_k^{a_j}(\mathbb{X}_{\sigma(j)}, \mathbb{X}'_{j}, \dot{\mathbb{V}}_{[j]}).
    \end{align*}
    
    By the definition of Hochschild homology of curved complexes, we can rewrite the left-hand side of \eqref{eq: almost_power_sum} as
    
    \begin{align*}
    \mathrm{tw}_{\tilde{\Delta}'_{HH, \beta}} \left( HH_{\bullet} ( C(\beta_{\aa})) \otimes_R R[\dot{\mathbb{V}}^{\Omega(\sigma)}] \right) & = HH_{\bullet} \left( \mathrm{tw}_{\tilde{\Delta}'_{\beta}} \left( C(\beta_{\aa}) \otimes_{\mathrm{Sym}^{\sigma}(\mathbb{X}, \mathbb{X}')} \left( \mathrm{Sym}^{\sigma}(\mathbb{X}, \mathbb{X}') \otimes_R R[\dot{\mathbb{V}}^{\Omega(\sigma)}] \right) \right) \right); \\
    \tilde{\Delta}'_{\beta} & := \sum_{[j] \in \Omega(\sigma)} \sum_{k = 1}^{a_j} \left( \sum_{j \in [j]} h_{jk} \right) \otimes \rho_k^{a_j}(\mathbb{X}_{\sigma(j)}, \mathbb{X}'_{j}, \dot{\mathbb{V}}_{[j]}).
    \end{align*}
	
	Now, $\mathrm{tw}_{\tilde{\Delta}'_{\beta}} \left( C(\beta_{\aa}) \otimes_{\mathrm{Sym}^{\sigma}(\mathbb{X}, \mathbb{X}')} \left( \mathrm{Sym}^{\sigma}(\mathbb{X}, \mathbb{X}') \otimes_R R[\dot{\mathbb{V}}^{\Omega(\sigma)}] \right) \right)$ is a deformation of $C(\beta_{\aa})$ with curvature

    \begin{align*}
    \sum_{[j] \in \Omega(\sigma)} \sum_{k = 1}^{a_j} \left( \sum_{j \in [j]} [d, h_{jk}] \right) \otimes \rho_k^{a_j}(\mathbb{X}_{\sigma(j)}, \mathbb{X}'_{j}, \dot{\mathbb{V}}_{[j]}) & = \sum_{[j] \in \Omega(\sigma)} \sum_{k = 1}^{a_j} \left( \sum_{j \in [j]} e_k(\mathbb{X}_{\sigma(j)}) - e_k(\mathbb{X}'_j) \right) \otimes \rho_k^{a_j}(\mathbb{X}_{\sigma(j)}, \mathbb{X}'_{j}, \dot{\mathbb{V}}_{[j]}) \\
    & = \sum_{[j] \in \Omega(\sigma)} \sum_{i = 1}^{a_j} \left( \frac{1}{k} \left( p_k(\mathbb{X}_{\sigma(j)}) - p_k(\mathbb{X}'_j) \right) \right) \otimes \dot{v}_{[j]k}.
    \end{align*}
    
    This is \textit{exactly} the bundled $\Delta p$-curvature $F_{\sigma, [\dot{v}]}^{\aa}(\mathbb{X}, \mathbb{X}')$. Such $\Delta p$-deformations of $C(\beta_{\aa})$ are unique up to homotopy equivalence by \textcolor{revisions}{Corollary} \ref{cor: inv_uniq_yify}. In particular, we must have a homotopy equivalence
    
    \[
    HH_{\bullet} \left( \mathrm{tw}_{\tilde{\Delta}'_{\beta}} \left( C(\beta_{\aa}) \otimes_{\mathrm{Sym}^{\sigma}(\mathbb{X}, \mathbb{X}')} \left( \mathrm{Sym}^{\sigma}(\mathbb{X}, \mathbb{X}') \otimes_R R[\dot{\mathbb{V}}^{\Omega(\sigma)}] \right) \right) \right) \cong \YS HH_{\bullet}(\beta_{\aa}).
    \]
\end{proof}

\subsection{Projector-Colored Link Homology} \label{sec: proj_link_hom}

So far, we have only considered families of colored link invariants obtained by modifying the complexes assigned to a crossing based on the colors of the strands involved. In this Section, we consider a well-known alternative construction given by \textit{cabling and insertion}. For our purposes, the input for this procedure consists of the following data:

\begin{itemize}
    \item A framed, oriented, colored link $\mathcal{L}$ presented as a balanced braid closure $\mathcal{L} = \beta_{\aa}$ with associated permutation $\sigma$ for some sequence of colors $\aa = (a_1, \dots, a_n)$.

    \item For each link component $[j] \in \Omega(\sigma)$, a \textcolor{revisions}{composition} $\blam_j \vdash a_j$.

    \item For each such \textcolor{revisions}{composition} $\blam_j$, a (potentially curved) complex $C_{[j]}$ over $\SSBim_{\blam_j}^{\blam_j}$ which slides past (potentially deformed) crossings. 
\end{itemize}

Given this input data, the following procedure produces a $\mathbb{Z}_a \times \mathbb{Z}_q \times \mathbb{Z}_t$-graded $R$-module:

\begin{enumerate}
    \item Choose one marked point on each component of $\mathcal{L}$ away from the crossings of $\beta_{\aa}$.

    \item For each $1 \leq j \leq n$, replace the $j^{th}$ strand of $\beta_{\aa}$ with a $\blam_j$-colored cable.

    \item Consider the (potentially deformed) Rickard complex associated to the colored braid obtained in the previous step. At each marked point, insert the complex $C_{[j]}$ into this complex (by horizontal composition $\star$).

    \item Apply $H_{KR}$ (or $\YS H_{KR}$) to the (potentially curved) complex obtained in the previous step.
\end{enumerate}

\textcolor{revisions}{This procedure requires a choice of braid presentation $\beta_{\aa}$ and marked points on each component of $\mathcal{L}$. Because the complexes $C_{[j]}$ slide past crossings up to homotopy equivalence, different choices of marked points on a given braid strand do not affect the homotopy equivalence class of the complex resulting from Step (3). Further, by Proposition \ref{prop: hh_is_trace}, moving a marked point through the ``closure" segements of $\mathcal{L}$ to different braid strands does not affect the module resulting from Step (4). These two facts show independence of the resulting module from the choice of marked points, and independence of braid representatives then follows from the usual arguments for Rickard complexes, giving a well-defined link invariant. We refer the interested reader to \cite{CK12} for further discussion of this construction.}

We have already established in Section \ref{sec: fray_functors} that each of our frayed projectors slides past (potentially deformed) crossings (see Corollaries \ref{cor: fin_proj_crossing_slide}, \ref{cor: yfin_proj_crossing_slide}, \ref{cor: inf_proj_crossing_slide}, \ref{cor: yinf_proj_crossing_slide}). It follows that choosing $C_{[j]}$ to be each type (finite/infinite, deformed/undeformed) of frayed projector gives rise to a whole \textit{family} of frayed projector-colored, triply-graded link homologies parametrized by \textcolor{revisions}{composition}s of the colors assigned to each strand.

Mixing and matching various types of frayed projectors on various components does give rise to well-defined link invariants, but we will be most interested in the case in which the same type of frayed projector is chosen for each strand. We describe each of these frayed projector-colored invariants more precisely below. 

\begin{thm} \label{thm: frayed_proj_homologies}
    Let $\mathcal{L} = \hat{\beta}_{\aa}$ be a framed, oriented, colored link with braid colors $\aa = (a_1, \dots, a_n)$. Each of the following procedures gives rise to a family of well-defined link invariants indexed by choices of \textcolor{revisions}{composition} $\blam_j \vdash a_{[j]}$.

    \begin{itemize}
        \item Cabling and inserting $C_{[j]} := K^{\blam_j}_{(1^j)}$ into the \textit{undeformed} cabled Rickard complex. We refer to this family of invariants as \textit{finite frayed projector-colored homology}.

        \item Cabling and inserting $C_{[j]} := K^{\blam_j, y}_{(1^j)}$ into the $\Delta e$-\textit{deformed} cabled Rickard complex. We refer to this family of invariants as \textit{deformed finite frayed projector-colored homology}.

        \item Cabling and inserting $C_{[j]} := (P^{\blam_j}_{(1^j)})^{\vee}$ into the \textit{undeformed} cabled Rickard complex. We refer to this family of invariants as \textit{infinite frayed projector-colored homology}.

        \item Cabling and inserting $C_{[j]} := (P^{\blam_j, y}_{(1^j)})^{\vee}$ into the $\Delta e$-\textit{deformed} cabled Rickard complex. We refer to this family of invariants as \textit{deformed infinite frayed projector-colored homology}.
    \end{itemize}
\end{thm}

\subsection{Comparing Colored Homologies} \label{sec: comp_hom}

In this Section we relate the various frayed projector-colored homologies of Section \ref{sec: proj_link_hom} to the intrinsically-colored, triply-graded Khovanov--Rozansky homologies of Sections \ref{sec: int_col_hom} and \ref{sec: def_int_col_hom}. Recall Corollaries \ref{cor: braid_fin_slide}, \ref{cor: braid_yfin_slide}, \ref{cor: braid_inf_slide}, and \ref{cor: braid_yinf_slide}, which state that cabling and inserting a frayed projector into a pure colored braid is equivalent to applying the corresponding fray functor. We show \textcolor{revisions}{in Section \ref{subsec: hoch_fray}} below that, in the presence of separated $\Delta e$-deformation, \textcolor{revisions}{cabling and inserting a frayed projector followed by taking Hochschild homology} has the overall effect of multiplication by a polynomial in $q$.

\textcolor{revisions}{Recall that Hochschild homology is a derived version of identifying the left and right module actions of partially symmetric polynomials. Our argument below proceeds by analyzing the effect of this identification on the various twists involved in the definitions of Fray functors. These twists utilize the left and right module action at particular external labels. In the setting of projector-colored homology, these labels occur at the top and bottom of the braid strand on which the projector is inserted. It would be particularly convenient if Hochschild homology exactly identified the module actions assigned to these two labels. This is the case for a certain class of braids called \textit{pure} braids, in which the underlying permutation is trivial. By contrast, given an impure braid, the two actions involved in fraying a particular braid strand are not necessarily identified.}

\textcolor{revisions}{With this in mind, we begin by illustrating the added complexity of dealing with impure braids and discussing how to reduce our comparison of the two types of homology to the case of pure braids. This discussion is rather technical, and the reader only interested in the comparison results is welcome to skip directly to Section \ref{subsec: hoch_fray}.}

\subsubsection{Dealing with Impure Braids}

We treat the example $\beta_{\aa} = (\sigma_1)_{(2, 2)}^3 \in \mathfrak{Br}_2$, $\blam = (1, 1) \vdash 2$ with marked point on the top left in detail; the general case differs only by more cumbersome notation. \textcolor{revisions}{Our task is to analyze the effect on $C^u(\beta_{\aa})$ of first cabling and inserting an infinite deformed frayed projector $P := (P^{(1, 1), y}_{(1^2)})^{\vee}$, then applying Hochschild homology. We depict this composition diagrammatically as follows:}

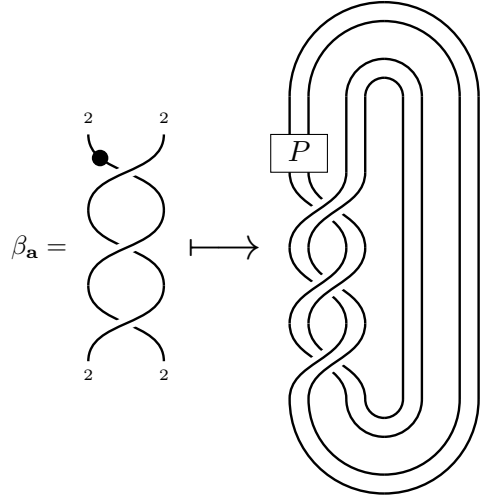
\begin{center}
\begin{gather*}
\beta_{\aa} =
\begin{tikzpicture}[anchorbase,tinynodes]
    \draw[webs] (1,0) node[below]{$2$} to[out=90,in=270] (0,1);
    \draw[line width=5pt,color=white] (0,0) to[out=90,in=270] (1,1);
    \draw[webs] (0,0) node[below]{$2$} to[out=90,in=270] (1,1);
    \draw[webs] (1,1) to[out=90,in=270] (0,2);
    \draw[line width=5pt,color=white] (0,1) to[out=90,in=270] (1,2);
    \draw[webs] (0,1) to[out=90,in=270] (1,2);
    \draw[webs] (0,3) node[above]{$2$} to[out=270,in=90] node[pos=.25, circle, fill=black, scale=.5] {} (1,2);
    \draw[line width=5pt,color=white] (0,2) to[out=90,in=270] (1,3);
    \draw[webs] (0,2) to[out=90,in=270] (1,3) node[above]{$2$};
    \node at (1.8,1.4){\huge $\longmapsto$};
\end{tikzpicture}
\begin{tikzpicture}[anchorbase,tinynodes,scale=.5]
    \draw[webs] (1.5,0) to[out=90,in=270] (0,2);
    \draw[webs] (2,0) to[out=90,in=270] (.5,2);
    \draw[line width=5pt,color=white] (0,0) to[out=90,in=270] (1.5,2);
    \draw[line width=5pt,color=white] (.5,0) to[out=90,in=270] (2,2);
    \draw[webs] (0,0) to[out=90,in=270] (1.5,2);
    \draw[webs] (.5,0) to[out=90,in=270] (2,2);
    \draw[webs] (1.5,2) to[out=90,in=270] (0,4);
    \draw[webs] (2,2) to[out=90,in=270] (.5,4);
    \draw[line width=5pt,color=white] (0,2) to[out=90,in=270] (1.5,4);
    \draw[line width=5pt,color=white] (.5,2) to[out=90,in=270] (2,4);
    \draw[webs] (0,2) to[out=90,in=270] (1.5,4);
    \draw[webs] (.5,2) to[out=90,in=270] (2,4);
    \draw[webs] (1.5,4) to[out=90,in=270] (0,6);
    \draw[webs] (2,4) to[out=90,in=270] (.5,6);
    \draw[line width=5pt,color=white] (0,4) to[out=90,in=270] (1.5,6);
    \draw[line width=5pt,color=white] (.5,4) to[out=90,in=270] (2,6);
    \draw[webs] (0,4) to[out=90,in=270] (1.5,6);
    \draw[webs] (.5,4) to[out=90,in=270] (2,6);
    \draw (-.5,6) rectangle++(1.5,1);
    \node at (.25,6.4) {\large $P$};
    \draw[webs] (0,7) to (0,8);
    \draw[webs] (.5,7) to (.5,8);
    \draw[webs] (1.5,6) to (1.5,8);
    \draw[webs] (2,6) to (2,8);
    \draw[webs] (2,8) to[out=90,in=180] (2.5,8.5);
    \draw[webs] (1.5,8) to[out=90,in=180] (2.5,9);
    \draw[webs] (.5,8) to[out=90,in=180] (2.5,10);
    \draw[webs] (0,8) to[out=90,in=180] (2.5,10.5);
    \draw[webs] (2.5,8.5) to[out=0,in=90] (3,8);
    \draw[webs] (2.5,9) to[out=0,in=90] (3.5,8);
    \draw[webs] (2.5,10) to[out=0,in=90] (4.5,8);
    \draw[webs] (2.5,10.5) to[out=0,in=90] (5,8);
    \draw[webs] (2,0) to[out=270,in=180] (2.5,-.5);
    \draw[webs] (1.5,0) to[out=270,in=180] (2.5,-1);
    \draw[webs] (.5,0) to[out=270,in=180] (2.5,-2);
    \draw[webs] (0,0) to[out=270,in=180] (2.5,-2.5);
    \draw[webs] (2.5,-.5) to[out=0,in=270] (3,0);
    \draw[webs] (2.5,-1) to[out=0,in=270] (3.5,0);
    \draw[webs] (2.5,-2) to[out=0,in=270] (4.5,0);
    \draw[webs] (2.5,-2.5) to[out=0,in=270] (5,0);
    \draw[webs] (3,0) to (3,8);
    \draw[webs] (3.5,0) to (3.5,8);
    \draw[webs] (4.5,0) to (4.5,8);
    \draw[webs] (5,0) to (5,8);
\end{tikzpicture}
\end{gather*}
\end{center}

\vspace{1em}

\textcolor{revisions}{We refer to strand of $\beta_{\aa}$ containing the marked point in the diagram above as the \textit{marked strand} and the other as the \textit{unmarked strand}. By Corollary \ref{cor: braid_yinf_slide}, we can replace the operation of cabling the marked strand and inserting $P$ by an application of $\yufray{2}{(1, 1)}$ to the marked strand, with distinguished label $2$ in the top left and bottom right. After making this replacement, we can depict the resulting composition diagrammatically as follows:}


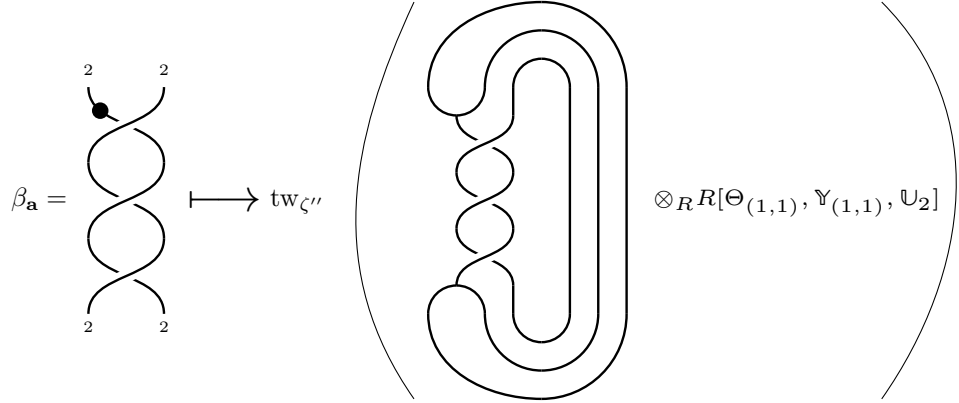
\begin{center}
\begin{gather*}
\beta_{\aa} =
\begin{tikzpicture}[anchorbase,tinynodes]
    \draw[webs] (1,0) node[below]{$2$} to[out=90,in=270] (0,1);
    \draw[line width=5pt,color=white] (0,0) to[out=90,in=270] (1,1);
    \draw[webs] (0,0) node[below]{$2$} to[out=90,in=270] (1,1);
    \draw[webs] (1,1) to[out=90,in=270] (0,2);
    \draw[line width=5pt,color=white] (0,1) to[out=90,in=270] (1,2);
    \draw[webs] (0,1) to[out=90,in=270] (1,2);
    \draw[webs] (0,3) node[above]{$2$} to[out=270,in=90] node[pos=.25, circle, fill=black, scale=.5] {} (1,2);
    \draw[line width=5pt,color=white] (0,2) to[out=90,in=270] (1,3);
    \draw[webs] (0,2) to[out=90,in=270] (1,3) node[above]{$2$};
    \node at (1.8,1.4){\huge $\longmapsto$};
\end{tikzpicture}
\mathrm{tw}_{\zeta}
\begin{tikzpicture}[anchorbase,scale=.5,tinynodes]
    \draw[webs] (1.25,-.5) to[out=90,in=180] (1.75,0);
    \draw[webs] (2.25,-.5) to[out=90,in=0] (1.75,0);
    \draw[webs] (1.75,0) to[out=90,in=270] (.25,1.5);
    \draw[webs] (0,-.5) to (0,0);
    \draw[webs] (.5,-.5) to (.5,0);
    \draw[line width=5pt,color=white] (0,0) to[out=90,in=270] (1.5,1.5);
    \draw[line width=5pt,color=white] (.5,0) to[out=90,in=270] (2,1.5);
    \draw[webs] (0,0) to[out=90,in=270] (1.5,1.5);
    \draw[webs] (.5,0) to[out=90,in=270] (2,1.5);
    \draw[webs] (1.5,1.5) to[out=90,in=270] (0,3);
    \draw[webs] (2,1.5) to[out=90,in=270] (.5,3);
    \draw[line width=5pt,color=white] (.25,1.5) to[out=90,in=270] (1.75,3);
    \draw[webs] (.25,1.5) to[out=90,in=270] (1.75,3);
    \draw[webs] (1.75,3) to[out=90,in=270] (.25,4.5);
    \draw[line width=5pt,color=white] (0,3) to[out=90,in=270] (1.5,4.5);
    \draw[line width=5pt,color=white] (.5,3) to[out=90,in=270] (2,4.5);
    \draw[webs] (0,3) to[out=90,in=270] (1.5,4.5);
    \draw[webs] (.5,3) to[out=90,in=270] (2,4.5);
    \draw[webs] (.25,4.5) to[out=180,in=270] (-.25,5);
    \draw[webs] (.25,4.5) to[out=0,in=270] (.75,5);
    \draw[webs] (1.5,4.5) to (1.5,5);
    \draw[webs] (2,4.5) to (2,5);
    \draw[webs] (2,5) to[out=90,in=180] (2.5,5.5);
    \draw[webs] (1.5,5) to[out=90,in=180] (2.5,6);
    \draw[webs] (.75,5) to[out=90,in=180] (2.5,6.75);
    \draw[webs] (-.25,5) to[out=90,in=180] (2.5,7.25);
    \draw[webs] (2.5,5.5) to[out=0,in=90] (3,5);
    \draw[webs] (2.5,6) to[out=0,in=90] (3.5,5);
    \draw[webs] (2.5,6.75) to[out=0,in=90] (4,5);
    \draw[webs] (2.5,7.25) to[out=0,in=90] (4.5,5);
    \draw[webs] (3,5) to (3,-.5);
    \draw[webs] (3.5,5) to (3.5,-.5);
    \draw[webs] (4,5) to (4,-.5);
    \draw[webs] (4.5,5) to (4.5,-.5);
    \draw[webs] (3,-.5) to[out=270,in=0] (2.5,-1);
    \draw[webs] (3.5,-.5) to[out=270,in=0] (2.5,-1.5);
    \draw[webs] (4,-.5) to[out=270,in=0] (2.5,-2);
    \draw[webs] (4.5,-.5) to[out=270,in=0] (2.5,-2.5);
    \draw[webs] (2.5,-1) to[out=180,in=270] (2.25,-.5);
    \draw[webs] (2.5,-1.5) to[out=180,in=270] (1.25,-.5);
    \draw[webs] (2.5,-2) to[out=180,in=270] (.5,-.5);
    \draw[webs] (2.5,-2.5) to[out=180,in=270] (0,-.5);
    \node at (9,2.5)[scale=1.4]{$\otimes_R R[\Theta_{(1,1)}, \mathbb{Y}_{(1,1)}, \mathbb{U}_{2}]$};
    \draw (11.5,-2) to[out=55,in=305] (11.5,7);
    \draw (-.75,-2) to[out=125,in=235] (-.75,7);
\end{tikzpicture}
\end{gather*}
\end{center}

\textcolor{revisions}{The twist $\zeta$ contains five groups of summands: three arising from applying the functor $\yufray{2}{(1,1)}$ (one each involving the alphabets $\Theta$, $\mathbb{Y}$, and $\mathbb{U}$), one arising from the $\Delta e$-deformation on the cabled unmarked strand (involving the alphabet $\mathbb{Y}$), and one arising from the $\Delta e$-deformation on the marked strand (involving the alphabet $\mathbb{U}$). The $\Theta$-portion of this twist involves the polynomial action in $\mathbb{X}$ variables assigned to the bottom right of this diagram. These are identified with the $\mathbb{X}$ variables in the top right after applying Hochschild homology, so we may freely replace $\zeta$ by a new twist $\zeta'$ by substituting the top right variables in their place. Having done this, we may now apply the natural isomorphism of Proposition \ref{prop: hh_is_trace} in each $\Theta, \mathbb{Y}$, and $\mathbb{U}$-degree to slide the merge bimodule in the bottom right of this picture to the top right. This results in the following complex:}

\begin{center}
\begin{gather*}
\beta_{\aa} =
\begin{tikzpicture}[anchorbase,tinynodes]
    \draw[webs] (1,0) node[below]{$2$} to[out=90,in=270] (0,1);
    \draw[line width=5pt,color=white] (0,0) to[out=90,in=270] (1,1);
    \draw[webs] (0,0) node[below]{$2$} to[out=90,in=270] (1,1);
    \draw[webs] (1,1) to[out=90,in=270] (0,2);
    \draw[line width=5pt,color=white] (0,1) to[out=90,in=270] (1,2);
    \draw[webs] (0,1) to[out=90,in=270] (1,2);
    \draw[webs] (0,3) node[above]{$2$} to[out=270,in=90] node[pos=.25, circle, fill=black, scale=.5] {} (1,2);
    \draw[line width=5pt,color=white] (0,2) to[out=90,in=270] (1,3);
    \draw[webs] (0,2) to[out=90,in=270] (1,3) node[above]{$2$};
    \node at (1.8,1.4){\huge $\longmapsto$};
\end{tikzpicture}
\mathrm{tw}_{\zeta'}
\begin{tikzpicture}[anchorbase,scale=.5,tinynodes]
	\draw[webs] (1.75,-.5) to (1.75,0);
    \draw[webs] (1.75,0) to[out=90,in=270] (.25,1.5);
    \draw[webs] (0,-.5) to (0,0);
    \draw[webs] (.5,-.5) to (.5,0);
    \draw[line width=5pt,color=white] (0,0) to[out=90,in=270] (1.5,1.5);
    \draw[line width=5pt,color=white] (.5,0) to[out=90,in=270] (2,1.5);
    \draw[webs] (0,0) to[out=90,in=270] (1.5,1.5);
    \draw[webs] (.5,0) to[out=90,in=270] (2,1.5);
    \draw[webs] (1.5,1.5) to[out=90,in=270] (0,3);
    \draw[webs] (2,1.5) to[out=90,in=270] (.5,3);
    \draw[line width=5pt,color=white] (.25,1.5) to[out=90,in=270] (1.75,3);
    \draw[webs] (.25,1.5) to[out=90,in=270] (1.75,3);
    \draw[webs] (1.75,3) to[out=90,in=270] (.25,4.5);
    \draw[line width=5pt,color=white] (0,3) to[out=90,in=270] (1.5,4.5);
    \draw[line width=5pt,color=white] (.5,3) to[out=90,in=270] (2,4.5);
    \draw[webs] (0,3) to[out=90,in=270] (1.5,4.5);
    \draw[webs] (.5,3) to[out=90,in=270] (2,4.5);
    \draw[webs] (.25,4.5) to[out=180,in=270] (-.25,5);
    \draw[webs] (.25,4.5) to[out=0,in=270] (.75,5);
    \draw[webs] (1.5,4.5) to[out=90,in=180] (1.75,4.75);
    \draw[webs] (2,4.5) to[out=90,in=0] (1.75,4.75);
    \draw[webs] (1.75,4.75) to (1.75,5);
    \draw[webs] (1.75,5) to[out=90,in=180] (2.5,5.75);
    \draw[webs] (.75,5) to[out=90,in=180] (2.5,6.75);
    \draw[webs] (-.25,5) to[out=90,in=180] (2.5,7.25);
    \draw[webs] (2.5,5.75) to[out=0,in=90] (3.25,5);
    \draw[webs] (2.5,6.75) to[out=0,in=90] (4,5);
    \draw[webs] (2.5,7.25) to[out=0,in=90] (4.5,5);
    \draw[webs] (3.25,5) to (3.25,-.5);
    \draw[webs] (4,5) to (4,-.5);
    \draw[webs] (4.5,5) to (4.5,-.5);
    \draw[webs] (3.25,-.5) to[out=270,in=0] (2.5,-1.25);
    \draw[webs] (4,-.5) to[out=270,in=0] (2.5,-2);
    \draw[webs] (4.5,-.5) to[out=270,in=0] (2.5,-2.5);
    \draw[webs] (2.5,-1.25) to[out=180,in=270] (1.75,-.5);
    \draw[webs] (2.5,-2) to[out=180,in=270] (.5,-.5);
    \draw[webs] (2.5,-2.5) to[out=180,in=270] (0,-.5);
    \node at (9,2.5)[scale=1.4]{$\otimes_R R[\Theta_{(1,1)}, \mathbb{Y}_{(1,1)}, \mathbb{U}_{2}]$};
    \draw (11.5,-2) to[out=55,in=305] (11.5,7);
    \draw (-.75,-2) to[out=125,in=235] (-.75,7);
\end{tikzpicture}
\end{gather*}
\end{center}

\textcolor{revisions}{The complex on the right-hand side is the Hochschild homology of some curved complex. The connection on this curved complex is a curved lift of the underlying Rickard differential by the twist $\zeta'$, which involves only the alphabets $\Theta, \mathbb{Y}, \mathbb{U}$, the $\mathbb{X}$-alphabets associated to the top of this diagram, and the various dot-sliding homotopies associated to this picture.} As a result, we may apply exactly the same fork-sliding arguments as in Section \ref{sec: ufray} to slide the merge bimodule from the top right to the bottom left. The effect on the connection is completely analogous to the various fork-sliding equivalences for Fray functors exhibited in Section \ref{sec: fray_functors}, but in different alphabets\textcolor{revisions}{; we denote this new connection by $\zeta''$. The result of these modifications is the following complex:}

\begin{center}
\begin{gather*}
\beta_{\aa} =
\begin{tikzpicture}[anchorbase,tinynodes]
    \draw[webs] (1,0) node[below]{$2$} to[out=90,in=270] (0,1);
    \draw[line width=5pt,color=white] (0,0) to[out=90,in=270] (1,1);
    \draw[webs] (0,0) node[below]{$2$} to[out=90,in=270] (1,1);
    \draw[webs] (1,1) to[out=90,in=270] (0,2);
    \draw[line width=5pt,color=white] (0,1) to[out=90,in=270] (1,2);
    \draw[webs] (0,1) to[out=90,in=270] (1,2);
    \draw[webs] (0,3) node[above]{$2$} to[out=270,in=90] node[pos=.25, circle, fill=black, scale=.5] {} (1,2);
    \draw[line width=5pt,color=white] (0,2) to[out=90,in=270] (1,3);
    \draw[webs] (0,2) to[out=90,in=270] (1,3) node[above]{$2$};
    \node at (1.8,1.4){\huge $\longmapsto$};
\end{tikzpicture}
\mathrm{tw}_{\zeta''}
\begin{tikzpicture}[anchorbase,tinynodes,scale=.75]
    \draw[webs] (1,0) to[out=90,in=270] (0,1);
    \draw[line width=5pt,color=white] (0,0) to[out=90,in=270] (1,1);
    \draw[webs] (0,0) to[out=90,in=270] (1,1);
    \draw[webs] (1,1) to[out=90,in=270] (0,2);
    \draw[line width=5pt,color=white] (0,1) to[out=90,in=270] (1,2);
    \draw[webs] (0,1) to[out=90,in=270] (1,2);
    \draw[webs] (1,2) to[out=90,in=270] (0,3);
    \draw[line width=5pt,color=white] (0,2) to[out=90,in=270] (1,3);
    \draw[webs] (0,2) to[out=90,in=270] (1,3);
    \draw[webs] (0,0) to[out=180,in=90] (-.5,-.5);
    \draw[webs] (0,0) to[out=0,in=90] (.5,-.5);
    \draw[webs] (1,0) to (1,-.5);
    \draw[webs] (0,3) to[out=180,in=270] (-.5,3.5);
    \draw[webs] (0,3) to[out=0,in=270] (.5,3.5);
    \draw[webs] (1,3) to (1,3.5);
    \draw[webs] (1,3.5) to[out=90,in=180] (1.5,4);
    \draw[webs] (.5,3.5) to[out=90,in=180] (1.5,4.5);
    \draw[webs] (-.5,3.5) to[out=90,in=180] (1.5,5);
    \draw[webs] (1.5,4) to[out=0,in=90] (2,3.5);
    \draw[webs] (1.5,4.5) to[out=0,in=90] (2.5,3.5);
    \draw[webs] (1.5,5) to[out=0,in=90] (3,3.5);
    \draw[webs] (2,3.5) to (2,-.5);
    \draw[webs] (2.5,3.5) to (2.5,-.5);
    \draw[webs] (3,3.5) to (3,-.5);
    \draw[webs] (2,-.5) to[out=270,in=0] (1.5,-1);
    \draw[webs] (2.5,-.5) to[out=270,in=0] (1.5,-1.5);
    \draw[webs] (3,-.5) to[out=270,in=0] (1.5,-2);
    \draw[webs] (1.5,-1) to[out=180,in=270] (1,-.5);
    \draw[webs] (1.5,-1.5) to[out=180,in=270] (.5,-.5);
    \draw[webs] (1.5,-2) to[out=180,in=270] (-.5,-.5);
    \node at (6,1.5)[scale=1.4]{$\otimes_R R[\Theta_{(1,1)}, \mathbb{Y}_{(1,1)}, \mathbb{U}_{2}]$};
    \draw (-.75,-2) to[out=125,in=245] (-.75,5);
    \draw (7.5,-2) to[out=55,in=315] (7.5,5);
\end{tikzpicture}
\end{gather*}
\end{center}

\textcolor{revisions}{Tracking the effect of each of these equivalences on the twist $\zeta''$, the result is exactly the twist arising from applying the functor $\yufray{2}{(1,1)}$ to the curved Rickard complex $C^u(\beta_{\aa})$ (with $\Delta e$-curvature on both strands), now with distinguished label $2$ in the top left and bottom \textit{left}. It follows that the complex on the right-hand side above is exactly $HH_{\bullet}(\yufray{2}{(1,1)}(C^u(\beta_{\aa}))$.}

\textcolor{revisions}{We emphasize again that the same analysis applies with more cumbersome notation for a general colored braid $\beta_{\aa}$ with arbitrary choice of marked point; further, by turning off various summands of the connection $\zeta$, one can easily adapt the above argument to account for undeformed and finite fray projectors. We summarize this claim below.}

\begin{prop} \label{prop: pure_suffices}
\textcolor{revisions}{Let $\beta_{\aa}$ be an arbitrary balanced colored braid which closes to a colored link $\mathcal{L}$, and pick one marked point on each component of $\mathcal{L}$. Let $\{ \blam_i \vdash k_i \}_{i \in \pi_0(\mathcal{L})}$ be compositions of the colors of the components of $\mathcal{K}$. Then the complexes obtained by the following two procedures are equivalent:}

\begin{itemize}
\item \textcolor{revisions}{Cable every strand of $\beta_{\aa}$ along the compositions $\blam_i$, insert a finite (resp. infinite, deformed finite, deformed infinite) frayed projector into the undeformed (resp. undeformed, deformed, deformed) Rickard complex of this cable, and apply Hochschild homology.}

\item \textcolor{revisions}{For each component $i \in \pi_0(\mathcal{L})$, pick the \textit{same} entry $a_i \in \aa$ in the domain and codomain which closes to the component $i$. Apply $\fray{k_i}{\blam_i}$ (resp. $\ufray{k_i}{\blam_i}$, $\yfray{k_i}{\blam_i}$, $\yufray{k_i}{\blam_i}$) with this distinguished entry $a_i$ to the undeformed (resp. deformed, undeformed, deformed) Rickard complex of $\beta_{\aa}$, then apply Hochschild homology.}
\end{itemize}
\end{prop}

\subsubsection{Hochschild Homology of Frayed Complexes} \label{subsec: hoch_fray}

Unless otherwise noted, we now fix a domain sequence $\aa$, an element $n$ of $\aa$, and a \textcolor{revisions}{composition} $\blam \vdash n$ throughout. \textcolor{revisions}{Potentially using Proposition \ref{prop: pure_suffices}, we assume throughout that the distinguished element $n$ to which we apply various Fray functors appears in the same position in the domain and codomain.}

\begin{defn}
    For each positive integer $n \geq 1$ and \textcolor{revisions}{composition} $\blam = (\lambda_1, \dots, \lambda_m) \vdash n$, set

    \[
    f_{n, \blam}(q) := \cfrac{[n]!}{\prod_{j = 1}^m [\lambda_j]!}.
    \]
\end{defn}

\begin{prop} \label{prop: tr_rfray}
    Let $C \in \CS(\SSBim_{\aa}^{\aa})$ be an arbitrary balanced complex of singular Soergel bimodules. Then there is a natural isomorphism of complexes \textcolor{revisions}{of} $\mathbb{Z}_a \times \mathbb{Z}_q \times \mathbb{Z}_t$-graded $R$-modules

    \[
    HH_{\bullet} \left( \rfray{n}{\blam}(C) \right) \cong f_{n, \blam}(q) HH_{\bullet}(C).
    \]
\end{prop}

\begin{proof}
    By definition, we have $\rfray{n}{\blam}(C) = (_{\blam}S_{(n)}) \star C \star (_{(n)} M_{\blam})$. By Proposition \ref{prop: hh_is_trace}, there is a natural graded $R$-module isomorphism

    \[
    HH_{\bullet} \left( (_{\blam}S_{(n)}) \star C \star (_{(n)} M_{\blam}) \right) \cong HH_{\bullet} \left( C \star (_{(n)} M_{\blam}) \star (_{\blam}S_{(n)}) \right).
    \]
    The desired isomorphism then follows immediately from the generalized digon removal of Proposition \ref{prop: blamgon_removal}; further, since this generalized digon removal occurs \textcolor{revisions}{on a distant tensor factor} from $C$, the resulting isomorphism is natural.
\end{proof}

\begin{prop} \label{prop: tr_fray}
    Let $C \in \CS(\SSBim_{\aa}^{\aa})$ be an arbitrary balanced complex of singular Soergel bimodules. Then there is an isomorphism of complexes of $\mathbb{Z}_a \times \mathbb{Z}_q \times \mathbb{Z}_t$-graded $R$-modules
    
    \[
    HH_{\bullet} \left(\fray{n}{\blam}(C) \right) \cong
    \left( \prod_{j = 1}^m \prod_{k = 1}^{\lambda_j} (1 + tq^{-2k}) \right) f_{n, \blam}(q) HH_{\bullet}(C).
    \]
\end{prop}

\begin{proof}
    By functoriality of $HH_{\bullet}$ \textcolor{revisions}{and the definition of $\fray{n}{\blam}$ (Definition \ref{def: fray})}, we can write the left-hand side of the desired isomorphism explicitly as

    \[
    HH_{\bullet} \left(\fray{n}{\blam}(C) \right) = \mathrm{tw}_{HH_{\bullet}(\alpha)} \left( HH_{\bullet} \left( \rfray{n}{\blam}(C) \right) \otimes_R R[\Theta_{\blam}] \right); \quad \alpha = \sum_{j = 1}^m \sum_{k = 1}^{\lambda_j} e_k(\mathbb{X}_j)
    - e_k(\mathbb{X}'_j).
    \]
    
    Applying $HH_{\bullet}$ identifies (the action of symmetric polynomials in) the alphabets
    $\mathbb{X}_j$ and $\mathbb{X}'_j$ for each $j$, so $HH_{\bullet}(\alpha) = 0$.
    It follows that $HH_{\bullet}(\fray{n}{\blam}(C))$ is a direct sum of copies of
    $HH_{\bullet}(\rfray{n}{\blam}(C))$ in each $\Theta_{\blam}$-degree\textcolor{revisions}{. Since $R[\Theta_{\blam}]$ is an exterior algebra on generators of degree $\mathrm{deg}(\theta_{jk}) = tq^{-2k}$, we can easily calculate its graded dimension as}
    
\[
\textcolor{revisions}{\mathrm{dim}(R[\Theta_{\blam}]) = \prod_{j = 1}^m \prod_{k = 1}^{\lambda_j} (1 + tq^{-2k}),}
\]

and we obtain the
    right-hand side by a direct application of \textcolor{revisions}{Proposition \ref{prop: tr_rfray}}.
\end{proof}

In the case that $C$ is a colored braid, Proposition \ref{prop: tr_fray} gives a direct relationship between the finite frayed projector-colored homology and the intrinsically-colored, triply-graded Khovanov-Rozansky homology of its closure.

\begin{cor}
    Let $\mathcal{L}$ be a framed, oriented, colored link with sequence of colors $a_1, \dots, a_r$. For each $1 \leq i \leq r$, let $\blam_i = (\lambda_1, \dots, \lambda_m) \vdash a_i$ be a \textcolor{revisions}{composition}, and set $g_i(q, t) := \prod_{j = 1}^m \prod_{k = 1}^{\lambda_j} (1 + tq^{-2k})$. Then the graded dimension of the finite frayed projector-colored homology of $\mathcal{L}$ given by these parameters is exactly $\prod_{i = 1}^r g_i(q, t) f_{a_i, \blam_i}(q)$ times the graded dimension of $H_{KR}(\mathcal{L})$.
\end{cor}

\begin{prop} \label{prop: tr_yfray}
    Let $C \in \CS(\SSBim_{\aa}^{\aa})$ be an arbitrary balanced complex of singular Soergel bimodules. Then there is an isomorphism of complexes of $\mathbb{Z}_a \times \mathbb{Z}_q \times \mathbb{Z}_t$-graded $R$-modules
    
    \[
    HH_{\bullet} \left(\yfray{n}{\blam}(C) \right) \cong
    f_{n, \blam}(q) HH_{\bullet}(C).
    \]
\end{prop}

\begin{proof}
    As in the proof of Proposition \ref{prop: tr_fray}, we can write the
    left-hand side explicitly as

    \[
    HH_{\bullet} \left(\yfray{n}{\blam}(C) \right) = \mathrm{tw}_{HH_{\bullet}(\delta)} \left( HH_{\bullet} \left( \rfray{n}{\blam}(C) \right) \otimes_R R[\Theta_{\blam}, \mathbb{Y}_{\blam}] \right); \quad \delta = \sum_{j = 1}^m \sum_{k = 1}^{\lambda_j} 
    1 \otimes \theta^{\vee}_{jk} y_{jk}.
    \]
    By reassociating, we can recognize the result as the tensor product of
    $HH_{\bullet}(\rfray{n}{\blam}(C))$ and a Koszul complex on
    the polynomial ring $R[\mathbb{Y}_{\blam}]$ for (the action of multiplication by) its generators $\{y_{jk}\}$.
    The latter is quasi-isomorphic to $R$ by Proposition \ref{prop: kosz_reg_seq}, since these generators form a regular sequence in $R[\mathbb{Y}_{\blam}]$. The desired result then follows immediately from Proposition \ref{prop: tr_rfray}.
\end{proof}

Again, in the case that $C$ is a colored braid, Proposition \ref{prop: tr_yfray} gives a direct relationship between the deformed finite frayed projector-colored homology and the intrinsically-colored, triply-graded Khovanov--Rozansky homology of its closure.

\begin{thm} \label{thm: comparing_finite_homs}
    Let $\mathcal{L}$ be a framed, oriented, colored link with sequence of colors $a_1, \dots, a_r$. For each $1 \leq i \leq r$, let $\blam_i = (\lambda_1, \dots, \lambda_m) \vdash a_i$ be a \textcolor{revisions}{composition}. Then the graded dimension of the deformed finite frayed projector-colored homology of $\mathcal{L}$ given by these parameters is exactly $\prod_{i = 1}^r f_{a_i, \blam_i}(q)$ times the graded dimension of $H_{KR}(\mathcal{L})$.
\end{thm}

We computed the deformed finite frayed projector-colored homology\footnote{The computations in \cite{Con23} actually color by (a deformation of)
the \textit{reduced} finite projector $K_{1^n}^y$ instead of $\yfinproj{(1^n)}$; c.f.
Remark \ref{rem: ah_fin_proj}. The recursion in that work is easily adapted to our
setting and does not affect the parity of the result.} of all positive torus links with one non-trivially colored component (in particular, all colored positive torus knots) in \cite{Con23}. Theorem \ref{thm: comparing_finite_homs} allows us to immediately lift that result to a computation of $H_{KR}(\mathcal{L})$. In particular, since the result of \cite{Con23} is concentrated in even homological degrees, we immediately obtain

\begin{thm}
    Let $\mathcal{L}$ be a colored positive torus knot. Then $H_{KR}(\mathcal{L})$ is concentrated in even homological degrees.
\end{thm}

This affirmatively resolves Conjecture 8.4 of \cite{HRW21} in the case of positive torus knots.

\begin{prop} \label{prop: tr_yufray}
    Let $C \in \CS(\SSBim_{\aa}^{\aa})$ be an arbitrary balanced complex of singular Soergel bimodules and $\overline{C} := \mathrm{tw}_{\Delta_C}(C \otimes_R R[\mathbb{U}_n])$ a curved twist of $C$. Then there is an isomorphism of curved complexes of $\mathbb{Z}_a \times \mathbb{Z}_q \times \mathbb{Z}_t$-graded $R$-modules
    
    \[
    HH_{\bullet} \left(\yufray{n}{\blam}(\overline{C}) \right) \cong
    f_{n, \blam}(q) HH_{\bullet}(\overline{C}).
    \]
\end{prop}

\begin{proof}
As in the proof of Proposition \ref{prop: tr_fray}, we can write the left-hand side explicitly as

\begin{align*}
    HH_{\bullet}(\yufray{n}{\blam}(\overline{C})) & = \mathrm{tw}_{HH_{\bullet}(\delta - \gamma + \rfray{n}{\blam}(\Delta_C))} \left(HH_{\bullet}(\rfray{n}{\blam}(C)) \otimes_R R[\Theta_{\blam}, \mathbb{Y}_{\blam}, \mathbb{U}_n] \right); \\
    \delta - \gamma & = \sum_{j = 1}^m \sum_{k = 1}^{\lambda_j} \left( 1 \otimes \theta_{jk}^{\vee} y_{jk} - \left( \sum_{i = 1}^n a_{ijk}(\mathbb{X}_n, \mathbb{X}_n') \otimes \theta_{jk}^{\vee} u_i \right) \right).
\end{align*}

Now, let $\mathcal{R}_{\blam} := \mathrm{Sym}^{\blam}(\mathbb{X}, \mathbb{X}') \otimes_R R[
\mathbb{Y}_{\blam}, \mathbb{U}_n]$, and let $\mathcal{K}_{\blam}$ denote the Koszul complex on
$\mathcal{R}_{\blam}$ for (the action of multiplication by) the family
$\left\{ 1 \otimes y_{jk} - \sum_{i = 1}^n a^{\blam}_{ijk}(\mathbb{X}, \mathbb{X}') \otimes u_i \right\}_{j,k}$.
By reassociating, we may rewrite the left-hand side of the desired equivalence as

\[
HH_{\bullet}(\yufray{n}{\blam}(\overline{C})) \cong \mathrm{tw}_{HH_{\bullet}(\rfray{n}{\blam}(\Delta_C))}
\Big( HH_{\bullet}(\rfray{n}{\blam}(C)) \otimes_R R[\mathbb{U}_n] \Big) \otimes_{\mathrm{Sym}^{\blam}(\mathbb{X}, \mathbb{X}') \otimes_R R[\mathbb{U}_n]} \mathcal{K}_{\blam}.
\]

Observe that the family
$\left\{ 1 \otimes y_{jk} - \sum_{i = 1}^n a^{\blam}_{ijk}(\mathbb{X}, \mathbb{X}') \otimes u_i \right\}_{j,k}$
is a regular sequence in $\mathcal{R}_{\blam}$
(e.g. since each element contains a distinct variable $y_{jk}$,
and these are algebraically independent). By Proposition \ref{prop: kosz_reg_seq}, we have a quasi-isomorphism

\[
\mathcal{K}_{\blam} \cong \mathcal{R}_{\blam}/ \langle 1 \otimes y_{jk} - \sum_{i = 1}^n a^{\blam}_{ijk}(\mathbb{X}, \mathbb{X}') \otimes u_i \rangle \cong \mathrm{Sym}^{\blam}(\mathbb{X}, \mathbb{X}') \otimes_R R[\mathbb{U}_n].
\]
As a consequence, we obtain a quasi-isomorphism

\[
HH_{\bullet}(\yufray{n}{\blam}(\overline{C})) \cong \mathrm{tw}_{HH_{\bullet}(\rfray{n}{\blam}(\Delta_C))}
\Big( HH_{\bullet}(\rfray{n}{\blam}(C)) \otimes_R R[\mathbb{U}_n] \Big).
\]
Now, since the isomorphism of Proposition \ref{prop: tr_rfray} is \textit{natural}, applying this isomorphism in each $\mathbb{U}_n$-degree gives a quasi-isomorphism

\[
HH_{\bullet}(\yufray{n}{\blam}(\overline{C})) \cong f_{n, \blam} \mathrm{tw}_{HH_{\bullet}(\Delta_C)}
\left( HH_{\bullet}(C) \otimes_R R[\mathbb{U}_n] \right).
\]
This is exactly $f_{n, \blam}(q)HH_{\bullet}(\overline{C})$.
\end{proof}

Once again, in the case that $C$ is a colored braid, Proposition \ref{prop: tr_yufray} gives a direct relationship between the deformed infinite frayed projector-colored homology and the deformed intrinsically-colored, triply-graded Khovanov--Rozansky homology of its closure.

\begin{thm} \label{thm: infproj_fray_agree}
    Let $\mathcal{L}$ be a framed, oriented, colored link with sequence of colors $a_1, \dots, a_r$. For each $1 \leq i \leq r$, let $\blam_i = (\lambda_1, \dots, \lambda_m) \vdash a_i$ be a \textcolor{revisions}{composition}. Then the graded dimension of the deformed infinite frayed projector-colored homology of $\mathcal{L}$ given by these parameters is exactly $\prod_{i = 1}^r f_{a_i, \blam_i}(q)$ times the graded dimension of $\YS H_{KR}(\mathcal{L})$.
\end{thm}

Notice that by Theorems \ref{thm: proj_agreement} and \ref{thm: yproj_agreement}, the undeformed/deformed infinite frayed projector-colored homology considered here coincides up to an overall grading shift with the non $y$-ified/$y$-ified column-colored homology considered in \cite{Con23}. Theorem \ref{thm: infproj_fray_agree} gives a direct relationship between this invariant and deformed intrinsically-colored homology, realizing the program suggested in Section 1.4.1 of \cite{Con23}.
\bibliographystyle{alpha}
\bibliography{fray_bib}

%
\end{document}